\documentclass{article}
\usepackage[table]{xcolor}
\usepackage[margin=1in]{geometry}
\usepackage{amsthm}
\usepackage{booktabs}
\usepackage{amsmath}
\usepackage{graphicx}
\usepackage{float}
\usepackage{subcaption}
\usepackage{animate}
\usepackage{minted}
\usepackage{empheq}
\usepackage{pdfpages}
\usepackage{tikz}
\usepackage{makecell}
\usepackage{enumerate}
\usepackage{enumitem}
\usepackage{amssymb}
\usepackage{algorithm}
\usepackage{algpseudocode}
\usepackage{MnSymbol}
\usepackage{multirow}
\usepackage{appendix}
\usepackage{MnSymbol}
\usepackage{titlesec}
\usepackage{fancyhdr}
\usepackage{extarrows}

\titleformat{\part}[display]
  {\normalfont\Large\bfseries}{\partname~\thepart}{1em}{}
\titlespacing*{\part}{0pt}{0pt}{20pt}

\newtheorem{lemma}{Lemma}[section]
\newtheorem{remark}{Remark}[section]

\newtheorem{thom}{Theorem}[section]

\newcommand\msc[1]{\textbf{Mathematical Subject Classifications}: #1}
\makeatletter

\newcommand{\Rmnum}[1]{\expandafter\@slowromancap\romannumeral #1@}
\makeatother
\usepackage{pifont}
\usepackage[colorlinks,
          linkcolor = cyan,		
          anchorcolor = blue,	
          urlcolor = blue,		
          citecolor = teal,
           ]{hyperref}
\usepackage{footnote}
\makesavenoteenv{table}   
\makesavenoteenv{tabular} 

\definecolor{redd}{RGB}{179,30,0}

\definecolor{navy}{RGB}{0,0,230}

\usepackage[framemethod=TikZ]{mdframed}

\mdfdefinestyle{myframe}{
linecolor=navy,
linewidth=1.5pt,
roundcorner=5pt,
backgroundcolor=blue!2,
innertopmargin=5pt,
innerbottommargin=5pt,
innerleftmargin=10pt,
innerrightmargin=10pt,
}
\tikzset{every picture/.style={line width=0.75pt}} 

\title{UGM: A Unified Framework and New Perspectives for Accelerated Gradient Methods in Smooth and Strongly Convex Optimization}
\author{Danqing Zhou\footnotemark[1],\; Shiqian Ma\footnotemark[2],\; Junfeng Yang\footnotemark[3]}

\date{\today}

\begin{document}

\footnotetext[1]{\, Email: {dqkyo@foxmail.com}.}
\footnotetext[2]{\, Data Science and AI Institute, and Department of Applied Mathematics and Statistics,  Johns Hopkins University, Baltimore, MD 21218. (Email: {sqma@jhu.edu.})}
\footnotetext[3]{\, School of Mathematics, Nanjing University, Nanjing, Jiangsu 210093, People's Republic of China. (Email: {jfyang@nju.edu.cn}).}

\maketitle

\begin{abstract}
In this paper, we propose a unified framework for accelerated gradient methods, dubbed UGM, which subsumes a wide range of accelerated and conventional gradient-type methods designed for minimizing  $L$-smooth and $\mu$-strongly convex functions. We demonstrate that the iteration update of the proposed framework can be intrinsically interpreted as a hybrid combination of the heavy-ball method and vanilla gradient descent. This interpretation reveals that classical accelerated gradient methods essentially integrate a conservative gradient descent step into the fast yet unstable heavy-ball dynamics, which achieves a favorable trade-off between acceleration and stability. We further establish a unified convergence analysis using Lyapunov functions. Guided by our analysis, we develop a family of enhanced accelerated gradient algorithms that leverage the inner product relationship between gradient information and iterative variables to optimize iterative updates. Extensive numerical experiments on unconstrained quadratic optimization and logistic regression validate that the proposed algorithms achieve superior performance compared with existing baseline methods under typical structural conditions.
\end{abstract}
\noindent \msc 15A18, 15A69, 65F15, 90C33

\section{Introduction}
Let $\mu$ and $L$ be constants satisfying $0<\mu < L < \infty$, and let 
$\mathcal{F}_{\mu,L}$ be the set of  $L$-smooth and $\mu$-strongly convex  functions on $\mathbb{R}^n$. 
In this paper, we consider the following unconstrained minimization problem
\begin{equation} \label{pro}
\min_{x\in \mathbb{R}^n} f(x), \quad f\in \mathcal{F}_{\mu,L}.
\end{equation}
Throughout this paper, we denote $\kappa :=L/\mu$ as the condition number, $x^*$ as the unique minimizer of $f$, and  $f_* := f(x^*)$.  
Gradient descent (GD) method is a standard algorithm for solving \eqref{pro}. With fixed step size $\alpha>0$ and initial point $x^0\in\mathbb{R}^n$, the GD method iterates as 
\begin{equation} \label{gd-alpha} 
x^{k+1} = x^{k} - \alpha \nabla f(x^k), \quad k \geq 0.
\end{equation}
A typical choice of the step size that is optimal for maximizing the per-iteration decrease in objective value
is $\alpha=1/L$, leading to the update
\begin{equation}  \label{gd} 
x^{k+1} = x^{k} - \frac{1}{L} \nabla f(x^{k}), \quad k \geq 0, 
\end{equation}
whereas $\alpha=2/(L+\mu)$ optimizes the decay of the squared optimality gap $\|x^{k}-x^*\|^2$ \cite{N87,MM20}; 
these two step sizes yield convergence rates of $\mathcal{O}((1-{1}/{\kappa})^k)$ and $\mathcal{O}( (\kappa-1)/(\kappa+1) )^k)$, respectively. 
Moreover, momentum-based accelerated algorithms have been extensively developed to solve the problem \eqref{pro}. The key principle of such methods lies in incorporating historical iterative information alongside the instantaneous gradient to construct the update direction, which enables improved convergence speed.
Representative algorithms within this category contain Polyak’s classic Heavy-Ball (HB) accelerated scheme \cite{P64}, Nesterov’s Fast Gradient Method (FGM) \cite{N83,N87}, as well as numerous subsequent variants proposed in \cite{AM09,DT14,KF16,VFL17,PPR23,WC25,TD23,CL19}.
While the body of related work is extensive, we focus, throughout this paper, our discussion on the FGM \cite{N83,N87}, the Strongly Convex Optimized Gradient Method (SC-OGM) \cite{PPR23}, the Triple Momentum (TM) method  \cite{VFL17}, and the Accelerated Over-Relaxation Heavy Ball method (AOR-HB) \cite{WC25}.

Given $x^0\in\mathbb{R}^n$, $\upsilon, \varsigma > 0$, and $x^{-1}=x^0$, the general form of Polyak's HB method iterates via 
\begin{equation} \label{hb-general} \tag{gHB}
    x^{k+1} = x^{k} -\upsilon\nabla f(x^{k}) + \varsigma(x^{k}-x^{k-1}), \quad k\geq 0.
\end{equation}
Setting $\upsilon=4/{(\sqrt{L}+\sqrt{\mu})^2}$ and $\varsigma={(\sqrt{L}-\sqrt{\mu})^2}/{(\sqrt{L}+\sqrt{\mu})^2}$ yields
\begin{equation} \label{hb-method} \tag{HB} 
    x^{k+1} = x^{k} -\frac{4}{(\sqrt{L}+\sqrt{\mu})^2}\nabla f(x^{k}) + \frac{(\sqrt{L}-\sqrt{\mu})^2}{(\sqrt{L}+\sqrt{\mu})^2}(x^{k}-x^{k-1}),
\end{equation}
which is theoretically proven optimal for convex quadratic minimization with the convergence rate $\frac{\sqrt{\kappa}-1}{\sqrt{\kappa}+1}$, since its convergence complexity attains the known lower complexity bounds \cite{N95}. 
However, \eqref{hb-method} is not guaranteed to converge for $f\in \mathcal{F}_{\mu,L}$; see counterexamples given in \cite{LRP16, BAA25}. 
Moreover, it is shown in \cite{BAA25} that for any $\kappa, \upsilon, \varsigma > 0$, the general HB method \eqref{hb-general} either fails to attain
an accelerated convergence rate  (i.e., slower than $1-\mathcal{O}(1/\kappa)$) over strongly convex quadratic functions, or admits some
$f\in \mathcal{F}_{\mu,L}$  together with an initialization for which the iteration diverges. 
Compared with the GD method \eqref{gd-alpha}, FGM achieves the faster convergence rates $\mathcal{O}(1/k^2)$ on $\mathcal{F}_{0,L}$ and 
$\mathcal{O}((1-1/\sqrt{\kappa})^{k})$ on $\mathcal{F}_{\mu,L}$ for $\mu>0$; by contrast, convergence of \eqref{hb-method} is guaranteed only for strongly convex quadratic functions. 
The convergence rate of FGM over $\mathcal{F}_{\mu,L}$ is further improved to $\mathcal{O}((1-2/\sqrt{\kappa})^{k})$ in \cite{VFL17, TD23}, which matches the exact lower complexity bound in \cite{DT22}.  Moreover, the idea of accelerating gradient methods via momentum was extended to the proximal point method in \cite{AA01} and to composite optimization problems in \cite{AM09, N13}. 
Furthermore, the fundamental question ``\textbf{Why can FGM achieve accelerated convergence?}" has attracted substantial research attention, and its acceleration mechanism has been interpreted from diverse viewpoints; we review several representative interpretations in what follows.
\begin{itemize}
    \item \textbf{ODE-based interpretation.} Su et al. \cite{SBV16} established an approximate equivalence between FGM and a second-order ordinary differential equation (ODE). Based on this ODE model, they found that oscillations in late iterations degrade the algorithm's convergence. Accordingly,  they proposed an adaptive restart strategy for cases where the gradient direction conflicts with the momentum direction. A large body of follow-up research has adopted the ODE perspective to analyze and design accelerated algorithms, including \cite{AC18, MM19, CL21, LC22, SD22, WC25}.
    
    \item \textbf{Linear-coupling-based interpretation.} Zhu et al. \cite{AZ17} attributed the acceleration of FGM to the linear coupling of two basic algorithms: GD \eqref{gd} and mirror descent. The combination integrates their respective strengths to attain accelerated convergence. Later, this linear-coupling framework was extended to design faster and more concise algorithms for packing and covering linear programs \cite{AO15, WR15}, semidefinite programming solvers \cite{ALO16}, and coordinate descent methods \cite{AQRY16}.
    
    \item \textbf{Dissipativity theory from the control perspective.} Lessard et al. \cite{LRP16} modeled FGM as a dynamical system and adopted the integral quadratic constraints (IQC) framework from robust control theory to rigorously analyze its accelerated convergence rate for strongly convex problems. Their work established a unified paradigm for automated algorithm design and performance evaluation. Later, this framework was extended to general convex settings via non-quadratic Lyapunov functions \cite{FRMP18}.  
\end{itemize}
 
\subsection{Our proposed framework and contributions}
Given $x^0\in\mathbb{R}^n$,  $q \in (0,1)$ and $\theta\neq -1$, we initialize $z^0=x^0$.
In this paper, we propose the following unified framework, dubbed UGM, for both non-accelerated and accelerated gradient methods
\begin{subequations}  \label{ugm}
\begin{empheq}[left={\text{UGM:}\qquad \empheqlbrace}]{align}
\quad y^{k+1} &= x^{k}-\frac{1}{L}\nabla f(x^{k}), \label{ugm-y}\\
\quad w^{k+1} &= x^{k} -\frac{1}{\mu} \nabla f(x^{k}), \label{ugm-w}\\
\quad z^{k+1} &= qz^{k}+(1-q)w^{k+1}, \label{ugm-z}\\
\quad x^{k+1} &= \frac{\theta}{1+\theta}y^{k+1} + \frac{1}{1+\theta}z^{k+1}. \label{ugm-x}
\end{empheq}
\end{subequations}
A more compact yet equivalent formulation of the UGM update \eqref{ugm} is given in \eqref{ugm-2}, which matches the general form of the optimized gradient method presented in \cite{KF16, KF18}, \cite[Eq. (2)]{TD23}.
Our main contributions are summarized below.
\begin{enumerate}
    \item[(i)] We show that UGM can be interpreted as a combination of \eqref{hb-method} and the GD method \eqref{gd}. Moreover, a broad class of popular accelerated and non-accelerated methods is unified within the UGM framework. Hence, accelerated methods such as FGM can be interpreted as integrating a conservative GD step \eqref{gd} into the fast yet unstable HB method \eqref{hb-method} via an appropriate combination of coefficients, thereby balancing acceleration and convergence.
    Furthermore, we elucidate the roles of the parameters $q$ and $\theta$: $q$ mainly regulates the convergence rate, while $\theta$ controls how closely the method resembles \eqref{gd}. These findings provide useful insights for constructing Lyapunov functions and are discussed in detail in Section \ref{sec:why_ugm}.
     
    \item[(ii)] In Section \ref{sec:uni-ana}, we conduct a unified convergence analysis for a broad class of popular accelerated methods with fixed parameters $q$ and $\theta$ under the UGM framework, using intuitively structured Lyapunov functions.
    The main convergence results are summarized in the fixed-parameter portion of Table~\ref{tab:convergence-overall}. Under our analysis framework, we develop new $Q$-linear convergence guarantees for FGM, TM, and AOR-HB methods.
    
    \item[(iii)] Since our analysis indicates that the choice of $\theta$ is, to some extent, independent of $q$, we propose new accelerated UGM algorithms with varying parameters. We also develop adaptive variants, UGM-II and UGM-III, which select $\theta$ based on the inner product $\langle \nabla f(x^{k+1}),x^{k+1}\!-z^{k+1} \rangle$ to improve the iteration updates. To the best of our knowledge, these adaptive selection mechanisms are new to the literature.
    Full details are given in Section \ref{sec:new_algs}, and the main convergence results are summarized in the varying parameters portion of Table \ref{tab:convergence-overall}.
    
\end{enumerate}

\begin{table}[!htbp]
    \centering
\begin{tabular}{cccccc}
\toprule
Param. & \multicolumn{1}{c}{Method}     & Lyapunov function\footnote{In this column, we only present the Lyapunov functions $V_k$ for Q-linear convergence, which satisfy $V_{k+1}\leq q V_k$ for some $q\in(0,1)$.} & Rate & \multicolumn{1}{l}{Type} & Ref. \\ \hline
\multirow{8}{*}{fixed}            
& \multirow{2}{*}{FGM}        &   ---     &  
\multirow{2}{*}{$1-1/\sqrt{\kappa}$}    & R\footnote{In \cite{N87}, convergence is established via the estimate sequence $\{\phi_k^*\}$, which is   lower bounded by $\{f(x^k)\}$, leading to the result 
$f(x^k) - f_* \leq \phi_k^* - f_* \leq \left( 1 - 1/\sqrt{\kappa} \right)^k \left( f(x^0) - f_* + \frac{\mu}{2} \|x^0 - x^*\|^2 \right)$. }    &\cite{N87}   \\
&   &   \eqref{ly-2} with $a=1$    &   & Q   & this work     \\ \cline{2-6} 
& \multirow{2}{*}{SC-OGM}     &   \multirow{2}{*}{\eqref{ly-2} with $a=1$}       &   \multirow{2}{*}{$1-\sqrt{2}/\sqrt{\kappa}$}
  & \multirow{2}{*}{Q}   &\cite{PPR23}   \\ 
&   &       &     &     & this work  \\ \cline{2-6} 
& \multirow{2}{*}{TM method}   &   ---    &  $1-2/\sqrt{\kappa}$  & R  & \cite{VFL17}  \\
&   &  \eqref{ly-2} with $a =2\sqrt{\frac{\sqrt{\kappa}(\kappa+1)}{(\kappa-1)(\sqrt{\kappa}-1)}}$   &  $1-1/(2\sqrt{\kappa})$ 
& Q  & this work  \smallskip \\ \cline{2-6} 
&  \multirow{2}{*}{AOR-HB}   &  see footnote\footnote{Here, the Lyapunov function is
$f(x^k)\!-f_*\!+{\mu\over 2}\|p^k\!-\! x^*\|^2 +\sqrt{\kappa} \, \langle \nabla f(x^k), p^k\!-\!x^*\rangle$, where $p^k = x^{k+1}+(x^{k+1}-x^{k})/\sqrt{\kappa}$. The authors also provide $1-1/\sqrt{\kappa}$ convergence analysis, but with a more complex Lyapunov function. In contrast, our analysis achieves the same rate using a simple Lyapunov function.}
   &  $1-1/(2\sqrt{\kappa})$ &  Q   &\cite{WC25} \\
&   &   \makecell{\eqref{ly-1} with $a=1\!+\!\frac{4\sqrt{\kappa} \!+ 1 \!+\! (2\sqrt{\kappa} \!+ \!1)\sqrt{4\sqrt{\kappa} \!+\! 5}}{2\kappa}$} &  $1-1/\sqrt{\kappa}$  & Q  &this work        \\ 
\midrule
\multirow{4}{*}{varying}  & ITEM    &   --- &      $1-2/\sqrt{\kappa}$   & R    &\cite{TD23}         \\ \cline{2-6} 
& UGM-I      &   \multirow{3}{*}{\makecell{\eqref{ada-ly-ana}}} 
&   \multirow{3}{*}{\makecell{$\big[\frac{\sqrt{1+8\kappa}-3}{\sqrt{1+8\kappa}+1}, \frac{\kappa-1}{\kappa+1}\big)$ }}    
& \multirow{3}{*}{\makecell{Q}}   & \multirow{3}{*}{\makecell{this work}}   \\
&  UGM-II     &   
                 &   

      &    &     \\
& UGM-III     &     &   &   &   \\ \bottomrule
\end{tabular}
    \caption{Convergence results for the main algorithms discussed in this paper. Parameters (Param. in the table) are classified as fixed or varying.  
    The column ``Type'' denotes the type of linear convergence: R-linear (R) or Q-linear (Q).}
    \label{tab:convergence-overall}
\end{table}

\subsection{Preliminaries and notation} \label{sec:pre}
Throughout the paper, we adopt the following notation. 
Let $\mathbb{R}^n$ be the $n$-dimensional Euclidean space, with inner product and induced norm denoted by $\langle\cdot, \cdot\rangle$ and $\|\cdot\| =\sqrt{\langle\cdot,\cdot\rangle}$, respectively. 
A function $f:\mathbb{R}^n \rightarrow \mathbb{R}$ is said to be $L$-smooth if it is differentiable and satisfies 
$\|\nabla f(x) - \nabla f(y)\|\leq L\|x-y\|$ for all $x,y\in \mathbb{R}^n$. Following \cite[Thm. 2.1.5]{N87}, for convex $f$, $L$-smoothness is equivalent to either of the following statements:
\begin{align}
    f(x)-f(y)-\langle \nabla f(y), x-y\rangle & \geq \tfrac{1}{2L}\|\nabla f(x)-\nabla f(y)\|^2,  \quad \forall x,y \in  \mathbb{R}^n; \label{jy-0706a}\\
    \langle \nabla f(x)-\nabla f(y), x-y \rangle & \geq {\tfrac{1}{L}}\|\nabla f(x)-\nabla f(y)\|^2,  \quad \forall x,y \in \mathbb{R}^n. \label{jy-0706b}
\end{align}
A differentiable function $f:\mathbb{R}^n \rightarrow \mathbb{R}$ is said to be $\mu$-strongly convex if
\[
f(x)-f(y)-\langle \nabla f(y), x-y\rangle \geq \frac{\mu}{2}\|x-y\|^2, \quad \forall x,y \in \mathbb{R}^n.    
\]
For any $f \in \mathcal{F}_{\mu,L}$, there holds
\begin{equation} \label{f-mu-L}
\langle \nabla f(x)-\nabla f(y),x-y\rangle  \geq \frac{\mu L}{\mu +L}\|x-y\|^2 + \frac{1}{\mu +L} \|\nabla f(x)-\nabla f(y)\|^2,  \quad \forall x,y \in \mathbb{R}^n.
\end{equation}
For any $u,v\in\mathbb{R}^n$ and $\alpha\in\mathbb{R}$, the following identity holds:
\begin{equation} \label{eq:convex-id}
\|\alpha u+(1-\alpha)v\|^2 = \alpha \|u\|^2+(1-\alpha)\|v\|^2-\alpha(1-\alpha)\|u-v\|^2.
\end{equation}

All algebraic identities and equivalent reformulations in this paper have been verified by manual calculation and symbolic computation using SymPy. The verification notebook is available at \url{https://github.com/dq-kyo/UGM-a-Unified-Framework-for-Gradient-Method}. Statements whose proofs rely on these code-verified algebraic identities are marked with footnotes in the text.

\section{UGM unifies diverse gradient methods} \label{sec:why_ugm}

We now explain the motivation behind UGM \eqref{ugm} from two perspectives.
\paragraph{Perspective 1: Embedding the GD step \eqref{gd} into \eqref{hb-method}.}
First, the UGM iteration can be viewed as a combination of \eqref{hb-method} and the GD method \eqref{gd}.
To see this, we can rewrite \eqref{hb-method} as
    \begin{subequations}  \label{hb}
    \begin{empheq}[left={\text{}\quad \empheqlbrace}]{align}
    \quad w^{k+1} &= x^{k} - \frac{1}{\mu} \nabla f\left(x^{k}\right), \label{hb-w}\\ 
    \quad z^{k+1} &= q z^{k} + (1-q) w^{k+1}, \label{hb-z}\\ 
    \quad x^{k+1} &= \frac{\theta}{1+\theta} x^{k} + \frac{1}{1+\theta} z^{k+1}. \label{hb-x}
    \end{empheq}
    \end{subequations}
    Here $z^0 = x^0$, $q = \frac{\sqrt{\kappa}-1}{\sqrt{\kappa}+1}$ and $\theta = \frac{\sqrt{\kappa}-1}{2}$.    
    This formulation shows that \eqref{hb-method} consists of a large GD step with step size $1/\mu$, followed by two moving-average updates; the proof of equivalence between  \eqref{hb-method} and \eqref{hb}
    is provided in Appendix \ref{proof-of-hb}.
    As shown in \cite{LRP16, BAA25}, \eqref{hb-method} fails to converge for general strongly convex and smooth functions. 
    To address this issue, we introduce the standard GD update \eqref{gd}, i.e.,
    $y^{k+1} = x^{k} - \frac{1}{L} \nabla f(x^{k})$,
    into the second moving-average step \eqref{hb-x}, which yields the UGM iteration \eqref{ugm}. 
    By properly tuning the combination coefficients for \eqref{gd} and \eqref{hb-method}, existing accelerated methods (e.g., FGM) can be naturally embedded into the UGM framework. Hence, these methods essentially integrate the stable GD step \eqref{gd} into the fast but possibly divergent scheme \eqref{hb-method}, achieving a favorable trade-off between acceleration and convergence guarantee.

\paragraph{Perspective 2: Unifying many existing gradient methods.} 
Second, the UGM framework unifies a large family of classic accelerated and non-accelerated gradient methods. This unified perspective helps us clarify the roles of parameters $q$ and $\theta$. Such understanding further guides the construction of Lyapunov functions and the design of new algorithms. Below, we list some representative algorithms encompassed by the proposed UGM framework. 
    \begin{enumerate}
        \item GD with step size $\alpha = 1/L$, namely, \eqref{gd}: UGM reduces to \eqref{gd} when setting $q=1-\frac{1}{\kappa}$ and any $\theta\neq 1$;
        see Lemma \ref{gd-1l} for the equivalence proof. Here we also have $y^{k}=x^k=z^k$ for all $k$;
        
        \item GD with step size $\alpha = 2/(L+\mu)$: UGM reduces to this scheme when setting $q=\frac{\kappa-1}{\kappa+1}$ and $\theta=0$;
        see Lemma \ref{gd-2l} for the equivalence proof. Here we also have $x^k=z^k$ for all $k$;
        
        \item FGM \cite{N87}: This method iterates as
        \begin{equation} \label{fgm}
        y^{k+1} = x^{k}-\frac{1}{L}\nabla f(x^{k}),  \;\; 
         x^{k+1} = y^{k+1} +\frac{\sqrt{\kappa}-1}{\sqrt{\kappa}+1}(y^{k+1} -y^{k}),    \quad k\geq 0,
        \end{equation}
        with $y^0=x^0$. UGM reduces to this method when setting $q=1-\frac{1}{\sqrt{\kappa}}$ and $\theta=\sqrt{\kappa}$;
        see Lemma \ref{sec-b2} for the equivalence proof.

        \item TM method \cite{VFL17}:  This method iterates as
    \begin{subequations}  \label{tm-method}
    \begin{empheq}{align}
    \quad \xi^{k+1} &= (1+\lambda)\xi^{k}-\lambda\xi^{k-1}-\nu\nabla f(x^k), \label{tm-xi}\\
    \quad x^{k} &= (1+\gamma)\xi^{k}-\gamma \xi^{k-1}, \label{tm-x}\\
    \quad z^{k} &= (1+\delta)\xi^{k}-\delta \xi^{k-1}, \label{tm-z}
    \end{empheq}
    \end{subequations}
     where $(\nu, \lambda, \gamma, \delta)=(\frac{1+\rho}{L}, \frac{\rho^2}{2-\rho}, \frac{\rho^2}{(1+\rho)(2-\rho)}, \frac{\rho^2}{1-\rho^2})$ and $\rho=1-\frac{1}{\sqrt{\kappa}}$. UGM reduces to this method when setting $q = \rho=1-\frac{1}{\sqrt{\kappa}}$ and $\theta = \frac{\rho}{2(1-\rho)}$;
     see Lemma \ref{sec-b3} for the equivalence proof. 
     
        \item SC-OGM \cite{PPR23}: This method iterates as
        \begin{equation}\label{sc-ogm}
        y^{k+1} = x^{k}-\frac{1}{L}\nabla f(x^{k}), \;\; 
         x^{k+1} = y^{k+1} +\frac{1}{2\gamma+1}(y^{k+1}-y^{k})+\frac{1}{2\gamma+1}(y^{k+1}-x^k),    \quad k\geq 0,
        \end{equation}
        with $y^0=x^0$ and $\gamma=\frac{\sqrt{8\kappa+1}+3}{2\kappa-2}$.
        UGM reduces to this method when setting $q=\frac{1}{\gamma+1}$ and $\theta=\frac{\gamma+1}{\gamma}$; see Lemma \ref{sec-b4} for the equivalence proof.
        
        \item AOR-HB \cite{WC25}: This method iterates as
        \begin{equation} \label{aor-hb}
         x^{k+1}=x^{k}-\frac{1}{(\sqrt{L}+\sqrt{\mu})^2}(2\nabla f(x^{k})-\nabla f(x^{k-1}))+\frac{L}{(\sqrt{L}+\sqrt{\mu})^2}(x^{k}-x^{k-1}),  \quad k\geq 1.
        \end{equation}
     UGM reduces to this method when setting $q = \frac{2\kappa + 2\sqrt{\kappa} - 1 + \sqrt{4\sqrt{\kappa} + 5}}{2(\sqrt{\kappa} + 1)^2}$ and 
     $\theta = \frac{\kappa}{2} \cdot \frac{2\sqrt{\kappa} + 3 - \sqrt{4\sqrt{\kappa} + 5}}{ \kappa + \sqrt{4\sqrt{\kappa} + 5} - 2}$;
     see Lemma \ref{sec-b5} for the equivalence proof. 
    \end{enumerate}

\begin{remark}[The role of $q$ and $\theta$]
Note that for the HB method \eqref{hb}, GD method \eqref{gd-alpha} with $\alpha=1/L$ or $2/(L+\mu)$, and FGM, the associated $q$ exactly matches their respective convergence rates. On the other hand, the convergence rates of SC-OGM, AOR-HB, and TM methods approximate their corresponding $q$ values.
From this observation, we conclude that $q$ primarily characterizes the convergence rate. On the other hand, the parameter $\theta$ governs how similar the iteration behaves to the GD method \eqref{gd}. In fact, as can be seen from \eqref{ugm-x}, taking
$\theta$ sufficiently large forces $x^{k+1}$ to mimic the GD iteration \eqref{gd}.
\end{remark}

\section{Unified analysis via Lyapunov functions} \label{sec:uni-ana}

In this section, we derive the accelerated convergence rates for all the aforementioned accelerated algorithms via an intuitive unified Lyapunov function framework.

When conducting convergence analysis for gradient-type methods over the class $\mathcal{F}_{\mu,L}$ with $\mu>0$,  it is standard to study two quantities: the objective gap $f(x^{k})-f_*$ and the squared distance to the optimal solution $\|x^k-x^*\|^2$. Standard proofs for Nesterov’s FGM require rather intricate technical arguments to balance these two quantities.

In our UGM analysis, we instead track the objective gap $f(x^{k})-f_*$ and the squared distance $\|z^k-x^*\|^2$ separately, and then balance these two terms.
Our choice to employ $\|z^k-x^*\|^2$ draws inspiration from \eqref{gd-alpha} with step sizes $\alpha = 1/L$ and $\alpha=2/(L+\mu)$, for which $z^k=x^k$ holds for all $k$ when regarded as special instances of UGM.
We further derive an inductive relation for $f(x^{k})-f_*-\frac{1}{2L}\|\nabla f(x^k)\|^2$, motivated by \cite{PPR23} and the inductive relation for $f(x^k)-f_*$ in \eqref{ef}.

\begin{lemma}[Inductive relation for $f(x^{k})-f_*$ and $f(x^{k})-f_*-\frac{1}{2L}\|\nabla f(x^k)\|^2$]
Let $\varrho\in (0,1)$ and $\theta\neq 0$. For all $k\geq 0$, the update $x^{k+1}$ defined in \eqref{ugm} satisfies
\begin{equation}\label{ef}
\begin{aligned}
   f(x^{k})-f_* & - \frac{1}{\varrho}\big(f(x^{k+1})-f_*\big) \geq \frac{1}{2L}\|\nabla f(x^{k})\|^2 + \frac{1}{2L}\|\nabla f(x^{k+1})\|^2 \\
 & +\frac{1}{\theta}\langle \nabla f(x^{k+1}),x^{k+1}-z^{k+1} \rangle  -\frac{1-\varrho}{\varrho}\left[
\langle \nabla f(x^{k+1}),x^{k+1}-x^{*}\rangle -\frac{\mu}{2}\|x^{k+1}-x^{*}\|^2
\right].
\end{aligned}
\end{equation}
Moreover, we also have 
\begin{equation}\label{ef-2}
\begin{aligned}
   f(x^{k})-f_* &-\frac{1}{2L}\|\nabla f(x^k)\|^2  - \frac{1}{\varrho}\left[f(x^{k+1})-f_*-\frac{1}{2L}\|\nabla f(x^{k+1})\|^2 \right] \geq  \frac{\varrho +1}{2L\varrho}\|\nabla f(x^{k+1})\|^2 \\
 & +\frac{1}{\theta}\langle \nabla f(x^{k+1}),x^{k+1}-z^{k+1} \rangle  -\frac{1-\varrho}{\varrho}\left[
\langle \nabla f(x^{k+1}),x^{k+1}-x^{*}\rangle -\frac{\mu}{2}\|x^{k+1}-x^{*}\|^2
\right].
\end{aligned}
\end{equation}
\end{lemma}
\begin{proof}
Since $f$ is $L$-smooth and $\mu$-strongly convex, there hold
\begin{equation}\label{jy-a}
\begin{aligned}
f(x^{k})-f(x^{k+1}) &\geq  \langle \nabla f(x^{k+1}),x^{k}-x^{k+1} \rangle +  \frac{1}{2L}\|\nabla f(x^{k})-\nabla f(x^{k+1})\|^2, \\
f(x^{k+1})-f_* & \leq \langle \nabla f(x^{k+1}),x^{k+1}-x^{*}\rangle -\frac{\mu}{2}\|x^{k+1}-x^{*}\|^2.
\end{aligned}
\end{equation}
As a result, we have 
\begin{align*}
&f(x^{k})-f_* - \frac{1}{\varrho}\big(f(x^{k+1})-f_*\big) 
= f(x^{k})-f(x^{k+1}) - \frac{1-\varrho}{\varrho}\big(f(x^{k+1})-f_*\big)\\
\stackrel{\eqref{jy-a}}{\geq} &\frac{1}{2L}\|\nabla f(x^{k})-\nabla f(x^{k+1})\|^2
+\langle \nabla f(x^{k+1}),x^{k}-x^{k+1} \rangle  -\frac{1-\varrho}{\varrho}\left[
\langle \nabla f(x^{k+1}),x^{k+1}-x^{*}\rangle -\frac{\mu}{2}\|x^{k+1}-x^{*}\|^2
\right] \\
\stackrel{\eqref{ugm-x}}{=}& \text{right-hand side of~}\eqref{ef}.
\end{align*}
By subtracting $\frac{1}{2L}\|\nabla f(x^k)\|^2$ and adding $\frac{1}{2L\varrho}\|\nabla f(x^{k+1})\|^2$ 
from both sides of \eqref{ef},  we obtain \eqref{ef-2}.
\end{proof}

\begin{lemma}[Inductive relation for $\|z^{k+1}-x^{*}\|^2$]
For all $k\geq 0$,  the update $z^{k+1}$ defined in \eqref{ugm} satisfies
\begin{equation} \label{ez}
\begin{aligned}
q\|z^{k}-x^{*}\|^2& - \|z^{k+1}-x^{*}\|^2 =\frac{2(1-q)}{\mu}\left[\langle \nabla f(x^{k}),x^{k}-x^{*}\rangle-\frac{\mu}{2}\|x^{k}-x^{*}\|^2 \right]  \\
&+q(1-q)\|x^{k}-z^{k}\|^2- \frac{(1-q)^2}{\mu^2}\|\nabla f(x^{k})\|^2 - \frac{2q(1-q)}{\mu}\langle \nabla f(x^{k}), x^{k}-z^{k}\rangle.
\end{aligned}
\end{equation}
\end{lemma}
\begin{proof}
It follows from the UGM framework \eqref{ugm} and the identity \eqref{eq:convex-id} that
\begin{align*}
\|z^{k+1}&-x^{*}\|^2 \stackrel{\eqref{ugm-w},\eqref{ugm-z}}{=}  \big\|(1-q)x^{k}+qz^{k}-x^{*}-\frac{1-q}{\mu}\nabla f(x^{k})\big\|^2   \\
= &\|(1-q)x^{k}+qz^{k}-x^{*}\|^2 + \frac{(1-q)^2}{\mu^2}\|\nabla f(x^{k})\|^2-\frac{2(1-q)}{\mu}\langle  (1-q)x^{k}+qz^{k}-x^{*},\nabla f(x^{k}) \rangle   \\
 \stackrel{\eqref{eq:convex-id}}{=}& (1-q)\|x^{k}-x^{*}\|^2 + q\|z^{k}-x^{*}\|^2- q(1-q)\|x^{k}-z^{k}\|^2+ \frac{(1-q)^2}{\mu^2}\|\nabla f(x^{k})\|^2 \\
&- \frac{2(1-q)}{\mu}\langle x^{k}-x^{*},\nabla f(x^{k})\rangle + \frac{2q(1-q)}{\mu}\langle x^{k}-z^{k},\nabla f(x^{k})\rangle.
\end{align*}
Rearranging the terms in the above equality, we obtain \eqref{ez}.
\end{proof}

Next, we introduce the following Lyapunov functions, which play a central role in our linear convergence analysis:
\begin{align}
E_{k}(a) &:= f(x^{k})-f_*+\frac{a\mu}{2}\|z^{k+1}-x^{*}\|^2, \label{ly-1}\\
F_{k}(a) &:= f(x^{k})-f_*-\frac{1}{2L}\|\nabla f(x^{k})\|^2+\frac{a\mu}{2}\|z^{k+1}-x^{*}\|^2. \label{ly-2}    
\end{align}
Here $a>0$ denotes an undetermined parameter serving as a balancing weight between the squared distance term $\|z^{k+1}-x^{*}\|^2$ and
either the objective gap $f(x^{k})-f_*$ or $f(x^{k})-f_*-\frac{1}{2L}\|\nabla f(x^{k})\|^2$.  
For $a=1$, we  write $E_k=E_k(1)$ and $F_k=F_k(1)$ for brevity. 
By combining the inductive bounds for $f(x^{k})-f_*$ from \eqref{ef}, for $f(x^{k})-f_*-\frac{1}{2L}\|\nabla f(x^{k})\|^2$ from \eqref{ef-2}, and for 
$\|z^{k+1}-x^{*}\|^2$ from \eqref{ez}, we arrive at the following theorem.

\begin{thom}
For the iterates defined in \eqref{ugm}, suppose $0 < q \leq \varrho < 1$ and $a>0$. Let $E_k(a)$ and $F_k(a)$ be defined in \eqref{ly-1} and \eqref{ly-2}, respectively. 
Then, for all $k\geq 0$, we have
\begin{align}
&E_{k}(a)-\frac{1}{\varrho}E_{k+1}(a) \geq \frac{1}{2L}\|\nabla f(x^{k})\|^2 +  \left(\frac{1-q}{q}a-\frac{1-\varrho}{\varrho}\right)\left[
\langle \nabla f(x^{k+1}),x^{k+1}-x^{*}\rangle -\frac{\mu}{2}\|x^{k+1}-x^{*}\|^2 \right] \label{ly-ana-1}\\
 & +\left(\frac{1}{2L}-\frac{a(1-q)^2}{2q\mu}\right)\|\nabla f(x^{k+1})\|^2  +\left(\frac{1}{\theta}-(1-q)a\right) \langle \nabla f(x^{k+1}),x^{k+1}-z^{k+1} \rangle  +\frac{\mu(1-q)a}{2}\|x^{k+1}-z^{k+1}\|^2, \nonumber
\end{align}
and 
\begin{align}
&F_{k}(a)-\frac{1}{\varrho}F_{k+1}(a) \geq \left(\frac{1-q}{q}a-\frac{1-\varrho}{\varrho}\right)\left[
\langle \nabla f(x^{k+1}),x^{k+1}-x^{*}\rangle -\frac{\mu}{2}\|x^{k+1}-x^{*}\|^2 \right] \label{ly-ana-2}\\
 & +\left(\frac{\varrho +1}{2L\varrho}-\frac{a(1-q)^2}{2q\mu}\right)\|\nabla f(x^{k+1})\|^2  +\left(\frac{1}{\theta}-(1-q)a\right) \langle \nabla f(x^{k+1}),x^{k+1}-z^{k+1} \rangle  +\frac{\mu(1-q)a}{2}\|x^{k+1}-z^{k+1}\|^2. \nonumber
\end{align}
\end{thom}
\begin{proof}
Let $k\geq 0$. 
It follows from the definition of $E_k(a)$ in \eqref{ly-1} and the condition $\varrho \geq q$ that
\begin{equation}\label{jy-b}
E_{k}(a)-\frac{1}{\varrho}E_{k+1}(a)   \geq  \Big(  f(x^{k})-f_*  - \frac{1}{\varrho}\big(f(x^{k+1})-f_*\big) \Big)
+ \frac{a\mu}{2\varrho}\Big(q\|z^{k+1}-x^{*}\|^2 - \|z^{k+2}-x^{*}\|^2 \Big).
\end{equation}
Using \eqref{ef} to bound the first term on the right-hand side of \eqref{jy-b}, substituting  \eqref{ez} (with $k$ replaced by $(k+1)$) into the second term, and rearranging all terms immediately yields the desired result \eqref{ly-ana-1}.
Similarly,  \eqref{ly-ana-2} can be derived using \eqref{ly-2}, \eqref{ef-2} and \eqref{ez}.
\end{proof}

\begin{remark}
Note that the only non-trivial term appearing in both \eqref{ly-ana-1} and \eqref{ly-ana-2} is $\langle x^{k+1} - z^{k+1}, \nabla f(x^{k+1}) \rangle$, whose sign is indeterminate. In particular,  this term vanishes when $x^{k+1} = z^{k+1}$. In the subsequent analysis, we either select appropriate parameters to eliminate this term or bound it via the weighted AM-GM inequality.
\end{remark}

\begin{remark}
Equations~\eqref{ly-1} and \eqref{ly-2} each define a Lyapunov function by combining a function-value-based error term with the squared distance $\|z^k-x^*\|^2$. Specifically, \eqref{ly-1} adopts the standard objective gap $f(x^k)-f_*$, whereas \eqref{ly-2} employs the modified residual term $f(x^k)-f_*-\frac{1}{2L}\|\nabla f(x^k)\|^2$.
Note that the nonnegativity of $f(x^k)-f_*-\frac{1}{2L}\|\nabla f(x^k)\|^2$ follows directly from the $L$-smoothness of $f$.
Alternatively, one may replace the function-value error term with its strongly convex counterpart $f(x^{k})-f_*-\frac{\mu}{2}\|x^k-x^*\|^2\geq 0$, or adopt the tighter inequality $f(x^k)-f_*-\frac{1}{2L}\|\nabla f(x^k)\|^2-\frac{\mu}{2(1-\mu/L)}\|x^k-\tfrac{1}{L}\nabla f(x^k)-x^*\|^2\geq 0$ as studied in \cite{WC25,TD23}. 
Nonnegativity of these two expressions hinges on the $\mu$-strong convexity of $f$. Moreover, both expressions explicitly couple the objective gap and squared distance, which conflicts with our aim to decouple these two quantities.
\end{remark}

Based on \eqref{ly-ana-1} and \eqref{ly-ana-2}, we subsequently present new 
convergence analysis for FGM, SC-OGM, TM method and AOR-HB, using the unified Lyapunov functions.
\begin{enumerate}[leftmargin=*]
    \item[(i)] FGM \cite{N83,N87}: As established in the analysis of Section \ref{sec:why_ugm}, FGM is a special case of UGM with 
    $\theta=\sqrt{\kappa}$ and $q=1-\frac{1}{\sqrt{\kappa}}$. 
    Under this parameter setup, the identities $\frac{(1-q)^2}{2q\mu}=\frac{1}{2Lq}$ and $\frac{1}{\theta}=1-q$ hold; these equalities have been verified both analytically and via the symbolic computation script \texttt{check\_proof.ipynb}. 
    Set $a=1$ and $\varrho=q$. Substituting these into \eqref{ly-ana-2}, we obtain
    \[
    F_{k}-\frac{1}{q}F_{k+1}\geq \frac{1}{2L}\|\nabla f(x^{k+1})\|^2 + \frac{\mu(1-q)}{2}\|x^{k+1}-z^{k+1}\|^2 \geq 0.
    \]
    From this inequality, we conclude that Nesterov’s FGM attains linear convergence with convergence rate $1-1/\sqrt{\kappa}$.

    \item[(ii)] SC-OGM \cite{PPR23}:  SC-OGM is a special case of UGM with $\gamma=\frac{\sqrt{8\kappa+1}+3}{2\kappa-2}$, $q=\frac{1}{\gamma+1}$ and $\theta=\frac{\gamma+1}{\gamma}$. With these parameter choices, the  equalities
    $\frac{\gamma+2}{L}=\frac{\gamma^2}{(\gamma+1)\mu}$, $\frac{q +1}{2Lq}=\frac{(1-q)^2}{2q\mu}$ and $\frac{1}{\theta}=1-q$ hold.
    Again, these identities have been verified both analytically and using the symbolic computation script \texttt{check\_proof.ipynb}. 
    Set $a=1$ and $\varrho=q$. Substituting these into  \eqref{ly-ana-2}, we obtain
    \[
    F_{k}-\frac{1}{q}F_{k+1}\geq  \frac{\mu(1-q)}{2}\|x^{k+1}-z^{k+1}\|^2 \geq 0.
    \]
    Hence, SC-OGM achieves linear convergence with rate $\frac{1}{1+\gamma}\sim 1-\sqrt{2/\kappa}$.
    
    \item[(iii)]TM method \cite{VFL17}: The TM method is a special instance of UGM with $q = \rho=1-\frac{1}{\sqrt{\kappa}}$ and {$\theta = \frac{\rho}{2(1-\rho)}$}.
    Then, using $\frac{1}{\theta} = \frac{2(1-q)}{q}$ and applying the weighted AM-GM inequality, we obtain
    \[
    \left(\frac{1}{\theta}-(1-q)a\right) \langle \nabla f(x^{k+1}),x^{k+1}-z^{k+1} \rangle 
    \geq -\frac{\mu(1-q)a}{2}\|x^{k+1}-z^{k+1}\|^2 -  \frac{1-q}{2\mu a} \left( \frac{2}{q} - a \right)^2\|\nabla f(x^{k+1})\|^2.
    \]
   Furthermore, invoking \eqref{f-mu-L} with $x=x^{k+1}$ and $y = x^*$, and using 
   $\|\nabla f(x^{k+1})\|\leq L\|x^{k+1}-x^{*}\|$,  
   we derive
    \[
    \langle \nabla f(x^{k+1}),x^{k+1}-x^{*}\rangle -\frac{\mu}{2}\|x^{k+1}-x^{*}\|^2  \geq  \left(\frac{1}{\mu +L}+ \frac{\mu(L - \mu)}{2(\mu + L)L^2}\right) \|\nabla f(x^{k+1})\|^2 \geq \frac{1}{\mu +L}\|\nabla f(x^{k+1})\|^2.
    \]
     For convenience, we define
    \begin{align}
    A(a, \varrho) & := \frac{\varrho + 1}{2L\varrho} - \frac{a(1-q)^2}{2q\mu} - \frac{1-q}{2\mu a} \left( \frac{2}{q} - a \right)^2 + \frac{1}{\mu + L} \left( \frac{(1-q)a}{q} - \frac{1-\varrho}{\varrho} \right) \label{Aarho}\\
    & = \left(\frac{1}{2L} + \frac{1}{\mu + L} + \frac{2(1 - q)}{\mu q}\right) + \left(\frac{1}{2L}-\frac{1}{\mu+L}\right)\frac{1}{\varrho}-\left[\frac{(L - \mu)(1 - q)}{2q\mu(\mu + L)}a+\frac{2(1 - q)}{\mu q^2}\frac{1}{a}\right]. \label{Aarho2}
    \end{align}
    If inequality
    \begin{equation} \label{con-eq-tm}
    \frac{1-q}{q}a-\frac{1-\varrho}{\varrho}\geq 0,    
    \end{equation}
    holds, then from \eqref{ly-ana-2} we obtain 
    \begin{equation} \label{fka}
    F_{k}(a)-\frac{1}{\varrho}F_{k+1}(a) \geq A(a, \varrho) \|\nabla f(x^{k+1})\|^2,
    \end{equation}
    where $A(a, \varrho)$ is defined in \eqref{Aarho}.
    Next, we show that there exists a pair $(a,\varrho)$ satisfying \eqref{con-eq-tm} and such that $A(a,\varrho)\geq 0$. 
    We focus on the nontrivial case $L>\mu$. First, we set  $a_*=2\sqrt{\frac{L+\mu}{L-\mu}\frac{1}{q}}$ to minimize the expression $\frac{(L - \mu)(1 - q)}{2q\mu(\mu + L)}a + \frac{2(1 - q)}{\mu q^2}\frac{1}{a}$.
    Since $q=1-\frac{1}{\sqrt{\kappa}}<\frac{L-\mu}{L+\mu}$, it follows that $a_*\geq \frac{2(L+\mu)}{L-\mu}>2$. 
    Substituting the variable transformation $t=1/\sqrt{\kappa}$
    into \eqref{Aarho2}, along with the relations $q = 1-t$ and $a_*=2\sqrt{\frac{1+t^2}{(1+t)(1-t)^2}}$, 
    and performing elementary algebraic computations, we obtain
    \begin{equation}\label{Aarho3}
     A(a_*, \varrho) 
     = 
     \frac{1}{\mu}\left(\frac{t^2}{2} + \frac{t^2}{1+t^2} + \frac{2t}{1-t} - \frac{2t\sqrt{1+t}}{(1-t)\sqrt{1+t^2}}- \frac{t^2(1-t^2)}{2(1+t^2)}\frac{1}{\varrho}\right)
      = \frac{t^2}{\mu}\Big(P(t)  - \frac{1}{\varrho}Q(t)\Big).
    \end{equation}
    where $P(t) := \frac{1}{2} + \frac{1}{1 + t^2} - \frac{2}{\sqrt{1 + t^2}(\sqrt{1 + t} + \sqrt{1 + t^2})}$ and $Q(t) := \frac{1-t^2}{2(1+t^2)}$. 
    Next, we choose $\varrho_*$ such that $A(a_*, \varrho_*)=0$, i.e., $\varrho_*=Q(t)\big/P(t)$. 
    From Lemma \ref{appendix-ineq-a}, we have $\frac{1}{2}\leq P(t) \leq \frac{1+t}{2(1+t^2)}$. Hence, via elementary algebraic calculations, we derive
    $q = 1-t \leq    \varrho_*  \leq \frac{1-t^2}{1+t^2}<1$.
    Since $a_*>1$ and $\varrho_*\geq q$, we get $\frac{1-q}{q}a_*-\frac{1-\varrho_*}{\varrho_*} \geq 0$, which verifies that 
    the condition \eqref{con-eq-tm} holds for $(a,\varrho) = (a_*, \varrho_*)$. 
    Substituting  $(a,\varrho) = (a_*, \varrho_*)$ into   \eqref{fka} and noting that $A(a_*, \varrho_*)=0$, we derive the linear convergence result
    $F_{k+1}(a_*) \leq \varrho_* F_{k}(a_*)$.
    Finally, we have
    \begin{equation}\label{Aarho4}
    \varrho_*  = \frac{Q(t)}{P(t)} =  1-\frac{2(t^2+1)\sqrt{1+t} + 2(t^2-1)\sqrt{1+t^2}}{(t^2+3)\sqrt{1+t} + (t^2-1)\sqrt{1+t^2}}
    \sim 1- \frac{t+\frac{11}{4}t^2+\mathcal{O}(t^3)}{2+\frac{3}{2}t+\frac{9}{8}t^2+\mathcal{O}(t^3)}\sim 1- \frac{t}{2} =1-\frac{1}{2\sqrt{\kappa}},
    \end{equation}
    where the Taylor expansion $\sqrt{1+v}=1 + \frac{v}{2} - \frac{v^2}{8}  + \mathcal{O}(v^3)$ is adopted for the asymptotic approximation.  
    Accordingly, we establish a Q-linear convergence rate of $1-{1}/{(2\sqrt{\kappa})}$, while the original result in \cite{VFL17}  achieves an R-linear convergence rate of $1-{2}/{\sqrt{\kappa}}$. 
    Note that the accelerated convergence behavior is retained.
    Finally, we note that all algebraic identities: the expression for $A(a,\varrho)$ in \eqref{Aarho}-\eqref{Aarho2}, 
    the form of $A(a_*,\varrho)$ in \eqref{Aarho3}, and the alternative representation of $\varrho_*$ in \eqref{Aarho4} 
    have been formally verified in the script \texttt{check\_proof.ipynb}.

    \item[(iv)] AOR-HB \cite{WC25}: From the discussion in Section \ref{sec:why_ugm}, AOR-HB is a  special instance of UGM with parameters $q = \frac{2\kappa + 2\sqrt{\kappa} - 1 + \sqrt{4\sqrt{\kappa} + 5}}{2(\sqrt{\kappa} + 1)^2}$ and $\theta = \frac{\kappa}{2} \cdot \frac{2\sqrt{\kappa} + 3 - \sqrt{4\sqrt{\kappa} + 5}}{ \kappa + \sqrt{4\sqrt{\kappa} + 5} - 2}$. Lemma \ref{sec-b5} guarantees  $q \in [\frac{3}{4},1)$ and $\theta>0$. 
    Let $a= \frac{2q-1}{\kappa(1-q)^2}$ and $\varrho=q$. One can manually verify the relation $\frac{1}{\theta} = \frac{2q-1}{\kappa(1 - q)} = a(1 - q)$.
    Alternatively, these equalities are also validated via symbolic computation in \texttt{check\_proof.ipynb}.
    Hence, we have
\begin{equation}\label{jy-0619a}
\frac{1}{L} - \frac{a(1-q)^2}{q\mu}  = \frac{1}{L} - \frac{2q-1}{q\kappa\mu} = \frac{1}{L} \Big( 1 - \frac{2q-1}{q} \Big) = \frac{1-q}{qL} 
\text{~~and~~}
a\mu(1-q)  = \frac{\mu(2q-1)}{\kappa(1-q)} = \frac{\mu}{\theta}.
\end{equation}
    For the given $q$, the following equivalence holds:
    \begin{equation*}
    a = \frac{2q-1}{\kappa(1-q)^2} > 1    \text{~~if and only if~~} 4\kappa^{7/2} + 21\kappa^3 + 48\kappa^{5/2} + 63\kappa^2 + 52\kappa^{3/2} + 27\kappa + 8\sqrt{\kappa} + 1 > 0.
    \end{equation*}
    This equivalence is verified via symbolic computation in \texttt{check\_proof.ipynb}.  
    Given $\kappa > 0$, the polynomial on the right-hand side is obviously strictly positive, which directly yields $a>1$.
    Recall our choice $\varrho=q$. We thus have $\frac{1-q}{q}a-\frac{1-\varrho}{\varrho} = \frac{1-q}{q} (a-1)>0$.
    Meanwhile, $\mu$-strongly convexity of $f$ delivers the bound
    $\langle \nabla f(x^{k+1}),x^{k+1}-x^{*}\rangle \geq \mu\|x^{k+1}-x^{*}\|^2$.
    This observation guarantees nonnegativity of the second term appearing on the right-hand side of \eqref{ly-ana-1}.
    Combining the relations \eqref{ly-ana-1}, \eqref{jy-0619a}, alongside the identity  $\frac{1}{\theta} = a(1 - q)$, we deduce that
    \[
    E_{k}(a)-\frac{1}{q}E_{k+1}(a) \geq \frac{1}{2L}\|\nabla f(x^{k})\|^2  +\frac{1-q}{2qL}\|\nabla f(x^{k+1})\|^2 +\frac{\mu}{2\theta}\|x^{k+1}-z^{k+1}\|^2\geq 0.
    \]
    Hence, we have established the linear convergence result $E_{k+1}(a) \leq q E_{k}(a)$.
    Furthermore,  one may check that our selected $q$ satisfies $q\sim  1- 1/\sqrt{\kappa}$. 
    In comparison, the original analysis of the AOR-HB algorithm in \cite{WC25}\footnote{The authors also mention a rate $1- 1/ \sqrt{\kappa}$ in their slides, yet the corresponding Lyapunov function is far more intricate than the one constructed here.} establishes a linear convergence rate of $1-1/(2\sqrt{\kappa})$.
\end{enumerate}

\section{UGM with varying parameters \texorpdfstring{$q_k$}{} and \texorpdfstring{$\theta_k$}{}} \label{sec:new_algs}
Let $x^0=z^0\in\mathbb{R}^n$. This section investigates the UGM framework \eqref{ugm} equipped with iteration-dependent parameters $q_k$ and $\theta_k$:
\begin{equation}
\label{ugm-ada}
\left\{ \quad 
\begin{aligned}
 y^{k+1} &= x^{k}-\frac{1}{L}\nabla f(x^{k}), \\
 w^{k+1} & = x^{k} -\frac{1}{\mu} \nabla f(x^{k}), \\
 z^{k+1} &= q_kz^{k}+(1-q_k)w^{k+1}, \\
 x^{k+1} &= \frac{\theta_k}{1+\theta_k}y^{k+1} + \frac{1}{1+\theta_k}z^{k+1}.
\end{aligned}
\right.
\end{equation}
In particular, Subsection~\ref{sec:ugm-I} presents simple iteration-varying parameter choices, while Subsections~\ref{sec:ugm-II}--\ref{sec:ugm-III} further develop adaptive, problem-dependent parameter variants.  First, using analogous arguments based on \eqref{ef-2} and \eqref{ez}, we derive the following iterate estimates for the updates in \eqref{ugm-ada}.

\begin{lemma}
For iterates generated by \eqref{ugm-ada}, the following two relations hold for all $k\geq 0$:
\begin{equation}\label{ef-2-2}
\begin{aligned}
   f(x^{k})-f_* &-\frac{1}{2L}\|\nabla f(x^k)\|^2  - \frac{1}{\varrho}\left[f(x^{k+1})-f_*-\frac{1}{2L}\|\nabla f(x^{k+1})\|^2 \right] \geq  \frac{\varrho +1}{2L\varrho}\|\nabla f(x^{k+1})\|^2 \\
 & +\frac{1}{\theta_k}\langle \nabla f(x^{k+1}),x^{k+1}-z^{k+1} \rangle  -\frac{1-\varrho}{\varrho}\left[
\langle \nabla f(x^{k+1}),x^{k+1}-x^{*}\rangle -\frac{\mu}{2}\|x^{k+1}-x^{*}\|^2
\right],
\end{aligned}
\end{equation}
and 
\begin{equation} \label{ez-2}
\begin{aligned}
\|z^{k}-x^{*}\|^2&-\frac{1}{q_k}\|z^{k+1}-x^{*}\|^2 =\frac{2(1-q_k)}{\mu q_k}\langle x^{k}-x^{*},\nabla f(x^{k})\rangle - \frac{1-q_k}{q_k}\|x^{k}-x^{*}\|^2    \\
&+(1-q_k)\|x^{k}-z^{k}\|^2- \frac{(1-q_k)^2}{q_k\mu^2}\|\nabla f(x^{k})\|^2 - \frac{2(1-q_k)}{\mu}\langle \nabla f(x^{k}), x^{k}-z^{k}\rangle.
\end{aligned}
\end{equation}
\end{lemma}

\begin{thom}
Let $0<\varrho<1$ and $a_k>0$ satisfy $a_{k+1}\leq \frac{\varrho a_k}{q_{k+1}}$ for all $k\geq 0$.
Define the Lyapunov function $G_{k}(a_k)$ by
\begin{equation} \label{ada-ly-ana}
G_{k}(a_k) := f(x^{k})-f_*-\frac{1}{2L}\|\nabla f(x^{k})\|^2+\frac{a_k\mu}{2}\|z^{k+1}-x^{*}\|^2.    
\end{equation}
Then, for the iterates generated by \eqref{ugm-ada}, it holds for all $k\geq 0$ that
\begin{equation} \label{ly-ana-2-2}
\begin{aligned}
G_{k}(a_k)-&\frac{1}{\varrho}G_{k+1}(a_{k+1}) \geq    \left(\frac{1-q_{k+1}}{q_{k+1}}a_k-\frac{1-\varrho}{\varrho}\right)\left[
\langle \nabla f(x^{k+1}),x^{k+1}-x^{*}\rangle -\frac{\mu}{2}\|x^{k+1}-x^{*}\|^2 \right]  \\
& +   \left(\frac{\varrho +1}{2L\varrho}\!-\!\frac{a_k(1\!-\!q_{k+1})^2}{2q_{k+1}\mu}\right)\|\nabla f(x^{k+1})\|^2  +\left(\frac{1}{\theta_k}\!-\!(1\!-\!q_{k+1})a_k\right) \langle \nabla f(x^{k+1}),x^{k+1}\!-\!z^{k+1} \rangle   \\ 
& +\frac{\mu(1\!-\!q_{k+1})a_k}{2}\|x^{k+1}\!-\!z^{k+1}\|^2.
\end{aligned}
\end{equation}
\end{thom}
\begin{proof}
Let $k\geq 0$.
Since $a_{k+1}\leq \frac{\varrho a_k}{q_{k+1}}$, we have
\begin{equation}\label{jy-0705-add}
\frac{a_k\mu}{2}\|z^{k+1}-x^{*}\|^2  -\frac{a_{k+1}\mu}{2\varrho}\|z^{k+2}-x^{*}\|^2  
\geq \frac{a_k\mu}{2}\left[\|z^{k+1}-x^{*}\|^2-\frac{1}{q_{k+1}}\|z^{k+2}-x^{*}\|^2\right] 
\end{equation}
Substituting \eqref{ez-2} (with $k$ replaced by $k+1$) into the right-hand side of \eqref{jy-0705-add}, summing both sides of the derived inequality with those of \eqref{ef-2-2}, and invoking the definition of $G_{k}(a_k)$ from \eqref{ada-ly-ana}, we arrive at the target estimate \eqref{ly-ana-2-2}.
\end{proof}

\subsection{UGM-I: Parameter rule \texorpdfstring{$(1-q_{k+1})a_k=1/\theta_k$}{}} \label{sec:ugm-I}
Consider the four terms on the right-hand side of \eqref{ly-ana-2-2}.
Among them and aside from their coefficients, the only term with an undetermined sign is ``$\langle \nabla f(x^{k+1}),x^{k+1}\!-\!z^{k+1} \rangle$", whose coefficient is 
$\left(1/\theta_k\!-\!(1\!-\!q_{k+1})a_k\right)$. 
We therefore set the coefficient preceding this term to zero, yielding $(1-q_{k+1})a_k=1/\theta_k$, and select the parameters involved in the iterative scheme \eqref{ugm-ada} and Lyapunov function \eqref{ada-ly-ana} as follows:
\begin{mdframed}[style=myframe]
\textbf{UGM-I.} Choose $\varrho \in (0,1)$ and any sequence $\{a_k, q_{k+1}, \theta_k\}_{k\geq 0}$ such that $a_k>0$, $\theta_k \in (0,1)$, and $q_{k+1}\in (0,1)$  for all $k\geq 0$, and 
\begin{subequations}  \label{scheme1}
\begin{empheq}[left=\empheqlbrace]{align}
\quad & a_{k+1}\leq \frac{\varrho a_k}{q_{k+1}} \label{con-1},\\
\quad & \frac{1-q_{k+1}}{q_{k+1}}a_k-\frac{1-\varrho}{\varrho} \geq 0 \label{con-2},\\
\quad & \frac{\varrho +1}{2L\varrho}-\frac{a_k(1\!-\!q_{k+1})^2}{2q_{k+1}\mu}+\left(\frac{1-q_{k+1}}{q_{k+1}}a_k-\frac{1-\varrho}{\varrho} \right)\frac{1}{\mu + L}\geq 0, \label{con-3}\\
\quad & 1/\theta_k = (1-q_{k+1})a_k. \label{con-4}
\end{empheq}
\end{subequations}
\end{mdframed}
\begin{thom}
For the iterates generated by \eqref{ugm-ada} with parameters satisfying \eqref{scheme1}, it holds for all $k\geq 0$ that $G_{k+1}\leq \varrho G_k$.
\end{thom}
\begin{proof}
For parameters satisfying \eqref{scheme1}, inequality \eqref{ly-ana-2-2} holds. Furthermore, since $f$ is $L$-smooth and $\mu$-strongly convex, we have
$\langle \nabla f(x^{k+1}),x^{k+1}-x^{*}\rangle -\frac{\mu}{2}\|x^{k+1}-x^{*}\|^2\geq \frac{1}{\mu + L} \|\nabla f(x^{k+1})\|^2$.
By further exploiting the parameter conditions in \eqref{scheme1}, we conclude that the right-hand side of \eqref{ly-ana-2-2} is nonnegative, which establishes the desired result.
\end{proof}

We next provide concrete parameter choices that satisfy the conditions in \eqref{scheme1}.

\begin{mdframed}[style=myframe]
\textbf{UGM-I (explicit).} Choose $\varrho $ satisfying $\frac{\sqrt{1+8\kappa}-3}{\sqrt{1+8\kappa}+1} \leq \varrho < \frac{\kappa-1}{\kappa+1}$,  and set 
$a_0\geq \frac{(1 - \varrho)[\kappa(1 - \varrho) - \varrho - 1]}{(\varrho + 1)\varrho}$.  The iteration-dependent parameters are updated via the following explicit rules:
\begin{subequations}  \label{explicit-form}
\begin{empheq}[left=\empheqlbrace]{align}
\quad & \theta_k  = \dfrac{\varrho}{1 - \varrho} + \dfrac{1}{a_k}, \label{sch-th}\\
\quad & a_{k+1}  = \varrho a_k + 1 - \varrho, \label{sch-a}\\
\quad & q_{k+1}  = \dfrac{\varrho a_k}{\varrho a_k + 1 - \varrho}. \label{sch-q}
\end{empheq}
\end{subequations}
\end{mdframed}

To verify that the parameters in \eqref{explicit-form} satisfy the constraints in \eqref{scheme1}, we first introduce a useful lemma.

\begin{lemma} \label{lem:ak_qk}
Let $\{a_k\}_{k\geq 0}$ be a sequence satisfying $a_{k+1} = \varrho a_k + 1 - \varrho$ for all $k \geq 0$, with $a_0 \geq \frac{(1 - \varrho)[\kappa(1 - \varrho) - \varrho - 1]}{(\varrho + 1)\varrho}$. If $\frac{\sqrt{1+8\kappa}-3}{\sqrt{1+8\kappa}+1} \leq \varrho < \frac{\kappa-1}{\kappa+1}$, then for all $k\geq 0$, we have
\begin{equation} \label{ak-lower}
a_k \geq \frac{(1 - \varrho)[\kappa(1 - \varrho) - \varrho - 1]}{(\varrho + 1)\varrho}.
\end{equation}
Moreover, for $q_{k+1} = \dfrac{\varrho a_k}{\varrho a_k + 1 - \varrho}$, we have $\min \left\{\varrho, \frac{\varrho a_0}{\varrho a_0 + 1-\varrho}\right\} \leq q_{k+1} < 1$.
\end{lemma}
\begin{proof}
For $q_{k+1} = \tfrac{\varrho a_k}{\varrho a_k + 1 - \varrho}$, the inequality $q_{k+1}<1$ trivially holds due to $\varrho<1$. 
Furthermore, solving the recurrence relation $a_{k+1} = \varrho a_k + 1 - \varrho$ for all $k \geq 0$ yields the closed-form expression $a_k = 1 + \varrho^k (a_0 - 1)$. 
The subsequent proof proceeds by considering two separate cases: $a_0\geq 1$ and $a_0 < 1$.
\begin{enumerate}
    \item[(i)] If $a_0 \geq 1$, the sequence $\{a_k\}$ is decreasing and converges to $1$, which implies $a_k \geq 1$ for all $k\geq 0$.
    The condition $\varrho \geq \frac{\sqrt{1+8\kappa}-3}{\sqrt{1+8\kappa}+1}$ ensures $\frac{(1 - \varrho)[\kappa(1 - \varrho) - \varrho - 1]}{(\varrho + 1)\varrho} \leq 1$ 
    (see Lemma \ref{sec3-detail-1} for a rigorous verification). Consequently, $a_k \geq \frac{(1 - \varrho)[\kappa(1 - \varrho) - \varrho - 1]}{(\varrho + 1)\varrho}$ holds for all $k\geq 0$. In addition, since the function $x\mapsto\frac{\varrho x}{\varrho x+1-\varrho}$ is monotonically increasing and $a_k\geq 1$, we further obtain  $q_{k+1}\geq \varrho$.
    
    \item[(ii)] If $a_0 < 1$,  the sequence $\{a_k\}$ is increasing and converges to $1$, so $a_k \geq a_0$ for all $k\geq 0$. This further implies $a_k \geq a_0 \geq \frac{(1 - \varrho)[\kappa(1 - \varrho) - \varrho - 1]}{(\varrho + 1)\varrho}$ and $q_{k+1}\geq \frac{\varrho a_0}{\varrho a_0 + 1-\varrho}$. 
\end{enumerate}
Thus, we have shown that $a_k \geq \frac{(1 - \varrho)[\kappa(1 - \varrho) - \varrho - 1]}{(\varrho + 1)\varrho}$ and $q_{k+1}\geq \min \left\{\varrho, \frac{\varrho a_0}{\varrho a_0 + 1-\varrho}\right\}$  for all $k\geq 0$.
\end{proof}

\begin{lemma}
The parameters given in \eqref{explicit-form} satisfy \eqref{scheme1}.
\end{lemma}
\begin{proof}
First, direct computation using \eqref{sch-q} yields $\frac{1-q_{k+1}}{q_{k+1}}a_k = \frac{1-\varrho}{\varrho}$, which verifies condition \eqref{con-2}.
Second, it is straightforward to verify from \eqref{explicit-form} that $a_{k+1}=\varrho a_{k}/q_{k+1}$ and $(1-q_{k+1})a_k=1/\theta_k$, which validates \eqref{con-1} and \eqref{con-4}. 
Finally, using the foregoing relations and the update rules in \eqref{explicit-form}, one can directly derive that
\[
\text{LHS of~}\eqref{con-3} = \frac{\varrho +1}{2L\varrho}-\frac{1-\varrho}{2\varrho\mu}\frac{1-\varrho}{1-\varrho+\varrho a_k} \geq  \frac{\varrho +1}{2L\varrho}-\frac{1-\varrho}{2\varrho\mu}\frac{\varrho+1}{\kappa(1-\varrho)} = 0,
\]
where the inequality follows from 
\eqref{ak-lower}. This verifies condition \eqref{con-3}.
\end{proof}

\subsection{UGM-II: \texorpdfstring{$\langle \nabla f(x^{k+1}),x^{k+1}\!-\!z^{k+1} \rangle$-aware parameter rule}{}} \label{sec:ugm-II}
In UGM-I, we choose $\theta_k$ so that the coefficient of  $\langle \nabla f(x^{k+1}),x^{k+1}\!-\!z^{k+1} \rangle$ vanishes, since this inner-product term carries an indefinite sign. 
A natural question arises: if, under suitable conditions, we can determine the sign of $\langle \nabla f(x^{k+1}),x^{k+1}\!-\!z^{k+1} \rangle$, do we gain greater flexibility in selecting $\{\theta_k\}_{k\geq 0}$?
Suppose $\theta_k \neq 0$ for all $k\geq 0$. Substituting $x=x^k$ and $y=x^{k+1}$ into \eqref{jy-0706a} and invoking \eqref{ugm-ada}, we can derive 
\begin{align}\label{sec:est-0}
f(x^{k})-f(x^{k+1}) 
\geq \frac{1}{2L}\|\nabla f(x^{k})\|^2 + \frac{1}{2L}\|\nabla f(x^{k+1})\|^2+ \frac{1}{\theta_k}\langle \nabla f(x^{k+1}), x^{k+1}-z^{k+1} \rangle.
\end{align}
Hence, the term $\frac{1}{\theta_k}\langle \nabla f(x^{k+1}), x^{k+1}-z^{k+1} \rangle$ is regarded as a ``good" term whenever it is nonnegative, since \eqref{sec:est-0} then guarantees a reduction in the objective function value. 

Below, we derive two estimates for $\langle \nabla f(x^{k+1}),x^{k+1}-z^{k+1} \rangle$. First, we have
\begin{equation}\label{jy-0718a}
\mu\|x^{k+1}-z^{k+1} \|^2\leq \langle \nabla f(x^{k+1})-\nabla f(z^{k+1}),x^{k+1}-z^{k+1} \rangle \leq L\|x^{k+1}-z^{k+1} \|^2.
\end{equation}
Second, rearranging $x^{k+1}= \frac{\theta_k}{1+\theta_k}y^{k+1} + \frac{1}{1+\theta_k}z^{k+1}$ yields $x^{k+1}-z^{k+1}= \frac{\theta_k}{1+\theta_k}(y^{k+1}-z^{k+1})$.
Substituting this identity into \eqref{jy-0718a} and rearranging the terms, we obtain
\begin{equation}
\begin{aligned} \label{sec:est}
&  \frac{\mu\theta_k^2}{(1+\theta_k)^2} \|y^{k+1}-z^{k+1} \|^2+ \frac{\theta_k}{1+\theta_k}\langle \nabla f(z^{k+1}),y^{k+1}-z^{k+1} \rangle \\
\leq\, &\langle \nabla f(x^{k+1}),x^{k+1}-z^{k+1} \rangle 
\leq  \frac{L\theta_k^2}{(1+\theta_k)^2} \|y^{k+1}-z^{k+1} \|^2 + \frac{\theta_k}{1+\theta_k}\langle \nabla f(z^{k+1}),y^{k+1}-z^{k+1} \rangle.
\end{aligned}
\end{equation}
\begin{lemma} \label{lem:est}
Let  $\bar{\theta}>\underline{\theta}>0$ be fixed constants. Suppose $\theta_k \in [\underline{\theta},\bar{\theta}]$ for all $k\geq 0$. Then the following hold:
\begin{itemize}
    \item[(i)] If $\langle \nabla f(z^{k+1}), y^{k+1}-z^{k+1} \rangle \leq -\frac{L\bar{\theta}}{1+\bar{\theta}}\|y^{k+1}-z^{k+1}\|^2$, then $\langle \nabla f(x^{k+1}),x^{k+1}-z^{k+1} \rangle \leq 0$;
    \item[(ii)] If $\langle \nabla f(z^{k+1}),z^{k+1}-y^{k+1} \rangle \leq \frac{\mu\underline{\theta}}{1+\underline{\theta}}\|y^{k+1}-z^{k+1}\|^2$, then $\langle \nabla f(x^{k+1}),x^{k+1}-z^{k+1} \rangle \geq 0$.
\end{itemize}
\end{lemma}

\begin{proof}
Fix $k\geq 0$. 
Since  $\theta_k \leq \bar{\theta}$, the condition of (i) implies
$$\langle \nabla f(z^{k+1}), y^{k+1}-z^{k+1} \rangle \leq -\frac{L\bar{\theta}}{1+\bar{\theta}}\|y^{k+1}-z^{k+1}\|^2 \leq -\frac{L\theta_k}{1+\theta_k}\|y^{k+1}-z^{k+1}\|^2.$$
It then follows from the second inequality in \eqref{sec:est} that $\langle \nabla f(x^{k+1}),x^{k+1}-z^{k+1} \rangle \leq 0$.
Similarly, since $\theta_k \geq \underline{\theta}$, the condition of (ii) implies 
$$\langle \nabla f(z^{k+1}),z^{k+1}-y^{k+1} \rangle \leq \frac{\mu\underline{\theta}}{1+\underline{\theta}}\|y^{k+1}-z^{k+1}\|^2 \leq \frac{\mu \theta_k}{1+\theta_k}\|y^{k+1}-z^{k+1}\|^2.$$
Then, $\langle \nabla f(x^{k+1}),x^{k+1}-z^{k+1} \rangle \geq  0$ follows from the first inequality from \eqref{sec:est}.
\end{proof}

In view of \eqref{ly-ana-2-2}, if $\langle \nabla f(x^{k+1}), x^{k+1} - z^{k+1} \rangle \geq 0$, then it is desirable to choose $\theta_k$ satisfying 
$1/\theta_k \geq (1 - q_{k+1})a_k$.  
Furthermore, \eqref{sec:est-0} suggests that a smaller $\theta_k$ may be preferable, as it tends to yield a larger decrease in the objective function value. 
Conversely, if $\langle \nabla f(x^{k+1}), x^{k+1} - z^{k+1} \rangle \leq 0$, then it is desirable to choose $\theta_k$ so that the reverse inequality
$1/\theta_k \leq (1 - q_{k+1})a_k$ holds.
On the other hand, the parameters $\{(q_{k+1}, a_k)\}_{k\geq 0}$ and $\varrho$ are selected following the scheme UGM-I (explicit) presented in \eqref{explicit-form}. 
Recall that from \eqref{explicit-form} and Lemma \ref{lem:ak_qk}, for all $k\geq 0$ we have 
\begin{equation}\label{jy-0718b}
\min\left\{\varrho, \frac{\varrho a_0}{\varrho a_0 + 1-\varrho}\right\} \leq q_{k+1} < 1.
\end{equation}
Then, we have the following lemma.
\begin{lemma}\label{lem:jy-0718}
Choose $\{(q_{k+1}, a_k)\}_{k\geq 0}$ and $\varrho$ by following the scheme UGM-I (explicit) presented in \eqref{explicit-form}. 
Assume $\bar{\theta}>\underline{\theta}>0$ be fixed constants satisfying $\underline{\theta} \leq \frac{\varrho}{1-\varrho}$ and $\bar{\theta}\geq \max\left\{\frac{1}{1-\varrho},\frac{\varrho a_0+1-\varrho}{(1-\varrho)a_0}\right\}$. For $k\geq 0$, define $\theta_k$ via
\begin{equation} \label{up:theta_k}
\theta_k=\left\{
\begin{array}{ll}
\bar{\theta}, & \text{if } \langle \nabla f(z^{k+1}),y^{k+1}-z^{k+1} \rangle \leq -\frac{L\bar{\theta}}{1+\bar{\theta}}\|y^{k+1}-z^{k+1}\|^2,\\ \smallskip
\underline{\theta}, & \text{if } \langle \nabla f(z^{k+1}),z^{k+1}-y^{k+1} \rangle \leq \frac{\mu\underline{\theta}}{1+\underline{\theta}}\|y^{k+1}-z^{k+1}\|^2,\\ \smallskip
\dfrac{\varrho}{1-\varrho} + \dfrac{1}{a_k}, & \text{otherwise.}
\end{array}
\right.
\end{equation}
Then, we have $\theta_{k} \in [\underline{\theta},\bar{\theta}]$ for all $k\geq 0$.
\end{lemma}
\begin{proof}
By the definition of $\theta_k$ in \eqref{up:theta_k}, we only need to treat the ``otherwise" case. 
Note that in the ``otherwise" case,  from  \eqref{sch-q} we have $\theta_k = \frac{\varrho}{1-\varrho} \!+ \!\frac{1}{a_k}
= \frac{\varrho}{1-\varrho} \frac{1}{q_{k+1}}$.
Then, it follows from \eqref{jy-0718b} that
\begin{equation*}
\underline{\theta} \leq \frac{\varrho}{1-\varrho} \stackrel{\eqref{jy-0718b}}{<} \theta_k \equiv \frac{\varrho}{1-\varrho} \frac{1}{q_{k+1}} 
\stackrel{\eqref{jy-0718b}}{\leq}  \max\left\{\frac{1}{1-\varrho},\frac{\varrho a_0+1-\varrho}{(1-\varrho)a_0}\right\} \leq \bar{\theta}.
\end{equation*}
This completes the proof.
\end{proof}

Based on Lemma \ref{lem:jy-0718}, we are now ready to present our $\langle \nabla f(x^{k+1}),x^{k+1}\!-\!z^{k+1} \rangle$-aware parameter rule as follows.
\begin{mdframed}[style=myframe]
\textbf{UGM-II: $\langle \nabla f(x^{k+1}),x^{k+1}\!-\!z^{k+1} \rangle$-aware parameter rule.} 
Choose parameters $(\varrho,a_0,\underline{\theta}, \bar{\theta})$ such that
$0< \underline{\theta}  \leq \frac{\varrho}{1-\varrho}$,
$\bar{\theta}  \geq \max\left\{\frac{1}{1-\varrho},\frac{\varrho a_0+1-\varrho}{(1-\varrho)a_0}\right\}$,
$\frac{\sqrt{1+8\kappa}-3}{\sqrt{1+8\kappa}+1} \leq \varrho < \frac{\kappa-1}{\kappa+1}$, 
$a_0\geq \frac{(1 - \varrho)[\kappa(1 - \varrho) - \varrho - 1]}{(\varrho + 1)\varrho}$. 
For all $k\geq 0$, compute $\theta_k$, $a_{k+1}$ and $q_{k+1}$ via
\eqref{up:theta_k}, \eqref{sch-a} and \eqref{sch-q}, respectively. 
\end{mdframed}

\begin{thom}
For the iterations generated by \eqref{ugm-ada} with the parameters specified in UGM-II, we have $G_{k+1}\leq \varrho G_k$ for all $k\geq 0$.
\end{thom}
\begin{proof}
It is straightforward to verify that, in all three cases defining $\theta_k$ in \eqref{up:theta_k}, the right-hand side of \eqref{ly-ana-2-2} is nonnegative, which 
gives the desired conclusion.
\end{proof}

\subsection{UGM-III: \texorpdfstring{$\langle \nabla f(x^{k+1}),x^{k+1}\!-\!z^{k+1} \rangle$- and $\|x^{k+1}-z^{k+1}\|^2$-aware parameter rule}{}} \label{sec:ugm-III}
We further note the term $\|x^{k+1}-z^{k+1}\|^2$ on the right-hand side of \eqref{ly-ana-2-2} also depends implicitly on the choice of $\theta_k$, since
$x^{k+1}-z^{k+1} =  \frac{\theta_k}{1+\theta_k} (y^{k+1}-z^{k+1})$.
In this subsection, we select $\theta_k$ by relating it to this previously neglected term; specifically, we seek to find $\theta_k$ such that
\begin{equation} \label{sat-con}
\left(1/\theta_k\!-\!(1\!-\!q_{k+1})a_k\right) \langle \nabla f(x^{k+1}),x^{k+1}\!-\!z^{k+1} \rangle  +
\frac{\mu(1\!-\!q_{k+1})a_k}{2} 
 \|x^{k+1}-z^{k+1} \|^2 
\geq 0.
\end{equation}
Next, we focus on the case $1/\theta_k - (1\!-\!q_{k+1})a_k \geq 0$.

\begin{lemma}\label{lem:jy-0719b}
Let $\{(a_k, q_{k+1})\}_{k\geq 0}$ be any sequence satisfying $a_k>0$ and $q_{k+1} \in (0,1)$ for all $k\geq 0$.
Define $A_k:=(1-q_{k+1})a_k$, $B_k:=\|y^{k+1}-z^{k+1}\|^2$, and $C_k:=\langle \nabla f(z^{k+1}),y^{k+1}-z^{k+1}\rangle$ for all $k\geq 0$.
Suppose that $\sup_{k\geq 0} A_k < \infty$ and $\underline{\theta}$ satisfies $0 < \underline{\theta} \leq 1/\sup_{k\geq 0}A_k$. 
For any $k\geq 0$, if $B_k=0$, then inequality \eqref{sat-con} holds trivially, and if $B_k>0$ then 
\begin{equation}\label{new_theta_k}
\tilde{\theta}_k
:=
\dfrac{-2C_k}
{\mu B_k+(1-A_k)C_k+
\sqrt{\mu^2B_k^2+2\mu B_kC_k+(1+A_k)^2C_k^2}} \leq \frac{1}{A_k},
\end{equation}
and inequality \eqref{sat-con} holds for any $\theta_k$ satisfying $\max\{\underline{\theta},  \tilde{\theta}_k\} \leq \theta_k \leq 1/A_k$.
\end{lemma}
\begin{proof}
Fix any $k\geq 0$. Recall that we always have $A_k>0$ and the identity $x^{k+1}-z^{k+1} =  \frac{\theta_k}{1+\theta_k} (y^{k+1}-z^{k+1})$.
If $B_k = 0$, then $y^{k+1}=z^{k+1}$, and hence $x^{k+1}-z^{k+1} = 0$, which implies that \eqref{sat-con} holds.
Next, we assume $B_k > 0$. 
From the definitions of $A_k$, $B_k$, $C_k$, the relation $x^{k+1}-z^{k+1} =  \frac{\theta_k}{1+\theta_k} (y^{k+1}-z^{k+1})$, and the left-hand side inequality in \eqref{sec:est}, we obtain
\begin{equation*}
\begin{aligned}
\text{LHS of~} \eqref{sat-con} = \, & \left(\frac{1}{\theta_k}\!-\!A_k\right) \langle \nabla f(x^{k+1}),x^{k+1}\!-\!z^{k+1} \rangle  + \frac{\mu A_k}{2}\frac{\theta_k^2}{(1+\theta_k)^2}\|y^{k+1}-z^{k+1} \|^2 \\
\geq \, & \left(\frac{1}{\theta_k}\!-\!A_k\right)\left( \frac{\mu\theta_k^2 B_k}{(1+\theta_k)^2}   \! +\!\frac{\theta_k C_k}{1+\theta_k} \right) + \frac{\mu A_k B_k}{2} \frac{\theta_k^2}{(1+\theta_k)^2}
=   \frac{I_1(\theta_k)}{(1+\theta_k)^2},
\end{aligned}
\end{equation*}
where $I_1(\theta) := - \frac{1}{2}A_k(2C_k+\mu B_k)\theta^2 + (C_k-A_kC_k+\mu B_k)\theta+C_k$ is a quadratic function of $\theta$.
Note that $I_1(1/A_k)=\frac{\mu B_k}{2A_k} > 0$, and the axis of symmetry of $I_1(\theta)$ is $\frac{C_k-A_kC_k+\mu B_k}{A_k(2C_k+\mu B_k)}$.
Next, we split the discussion into three mutually disjoint cases.
\begin{itemize}
    \item[(i)] $2C_k+\mu B_k>0$.  
    In this case, the coefficient of the quadratic term in $I_1(\theta)$ is negative, and
    straightforward calculations show that $\tilde{\theta}_k$ defined in \eqref{new_theta_k} is the smaller root of the quadratic equation $I_1(\theta)=0$. 
    Since $I_1(1/A_k)>0$, we conclude that $\tilde{\theta}_k \leq 1/A_k$, and for all $\theta_k \in [\tilde{\theta}_k, 1/A_k]$, we have $I_1(\theta_k) \geq 0$, hence \eqref{sat-con} follows;
    
    \item[(ii)]  $2C_k+\mu B_k<0$. 
    In this case, the coefficient of the quadratic term in $I_1(\theta)$ is positive,  and
    straightforward calculations show that $\tilde{\theta}_k$ defined in \eqref{new_theta_k} is the larger root of the quadratic equation $I_1(\theta)=0$. 
    Since $C_k<0$ and $A_k>0$, we have $C_k-A_kC_k+\mu B_k > 2C_k+\mu B_k$. Combined with $2C_k+\mu B_k < 0$, this implies that 
    the axis of symmetry of $I_1(\theta)$, namely $\frac{C_k-A_kC_k+\mu B_k}{A_k(2C_k+\mu B_k)}$, is bounded above $1/A_k$. 
    Recalling $I_1(1/A_k)>0$, we further obtain  $\tilde{\theta}_k \leq 1/A_k$, and $I_1(\theta_k) \geq 0$ for all $\theta_k \in [\tilde{\theta}_k, 1/A_k]$, so \eqref{sat-con} holds;
    
    \item[(iii)] $2C_k+\mu B_k=0$. In this case, $C_k < 0$ and straightforward calculations yields $\tilde{\theta}_k = 1/(A_k+1)$, which solves the linear equation 
       $I_1(\theta)= \frac{1}{2} \mu B_k \big((A_k+1)\theta-1\big) = 0$.
       Hence, $I_1(\theta_k) \geq 0$  for all $\theta_k \in [\tilde{\theta}_k, 1/A_k]$, and thus \eqref{sat-con} follows.
\end{itemize}
Thus, we complete the proof.
\end{proof}

Based on Lemma \ref{lem:jy-0719b} and in analogy with \eqref{up:theta_k}, we now propose the following more flexible update rule for selecting $\theta_k$:
\begin{equation} \label{up2:theta_k}
\theta_k \in \left\{
\begin{array}{ll}
\{\bar{\theta}\}, & \text{if } \langle \nabla f(z^{k+1}),y^{k+1}-z^{k+1} \rangle \leq -\frac{L\bar{\theta}}{1+\bar{\theta}}\|y^{k+1}-z^{k+1}\|^2,\\ \smallskip 
\{\underline{\theta}\}, & \text{if } \langle \nabla f(z^{k+1}),z^{k+1}-y^{k+1} \rangle \leq \frac{\mu\underline{\theta}}{1+\underline{\theta}}\|y^{k+1}-z^{k+1}\|^2,\\ \smallskip 
\left[\max\big\{\underline{\theta},\tilde{\theta}_k\big\}, \dfrac{\varrho}{1-\varrho} + \dfrac{1}{a_k}\right], & \text{otherwise.}
\end{array}
\right.
\end{equation}

We are now ready to present our $\langle \nabla f(x^{k+1}),x^{k+1}\!-\!z^{k+1} \rangle$- and $\|x^{k+1}-z^{k+1}\|^2$-aware parameter rule as follows.
\begin{mdframed}[style=myframe]
\textbf{UGM-III: $\langle \nabla f(x^{k+1}),x^{k+1}\!-\!z^{k+1} \rangle$ \& $\|x^{k+1}-z^{k+1}\|^2$-aware parameter rule.} 
Choose parameters $(\varrho,a_0,\underline{\theta}, \bar{\theta})$ such that
$0< \underline{\theta}  \leq \frac{\varrho}{1-\varrho}$, 
$\bar{\theta}  \geq \max\left\{\frac{1}{1-\varrho},\frac{\varrho a_0+1-\varrho}{(1-\varrho)a_0}\right\}$,
$\frac{\sqrt{1+8\kappa}-3}{\sqrt{1+8\kappa}+1} \leq \varrho < \frac{\kappa-1}{\kappa+1}$, 
$a_0\geq \frac{(1 - \varrho)[\kappa(1 - \varrho) - \varrho - 1]}{(\varrho + 1)\varrho}$.
For all $k\geq 0$, compute $\theta_k$, $a_{k+1}$ and $q_{k+1}$ via
\eqref{up2:theta_k}, \eqref{sch-a} and \eqref{sch-q}, respectively. 
\end{mdframed}
\begin{thom}
For the iterations generated by \eqref{ugm-ada} with the parameters specified in UGM-III, we have $G_{k+1}\leq \varrho G_k$ for all $k\geq 0$.
\end{thom}
\begin{proof}
It is straightforward to verify that, in all three cases defining $\theta_k$ in \eqref{up2:theta_k}, 
the right-hand side of \eqref{ly-ana-2-2} is nonnegative, which gives the desired conclusion. 
\end{proof}

\section{Numerical experiments} \label{sec:ne}

This section presents three groups of numerical experiments to comprehensively validate the performance and practical merits of the proposed algorithms. First, we test the proposed methods on unconstrained quadratic optimization problems to explore their favorable application scenarios and analyze the functional role of the parameter $\varrho$. Second, we evaluate the algorithm performance on the logistic regression task, verifying that our methods achieve competitive results in practical real-world learning problems. Third, we conduct experiments on softmax regression to investigate whether the iteration efficiency advantages of UGM-II and UGM-III can be converted into CPU time savings in distributed environments, leveraging cached-logit computation and scalar aggregation techniques. All experiments are implemented in Python 3.11.0 and run on a personal computer equipped with an AMD Ryzen 7 5800H processor, Radeon Graphics, and 16 GB of RAM.

\subsection{Unconstrained quadratic optimization} \label{sec:qp}

For the first experiment, we consider the unconstrained optimization of quadratic functions. Let $A\in\mathbb{R}^{n\times n}$ be a symmetric positive definite matrix and $b \in\mathbb{R}^n$. Consider minimizing $f(x) = \frac{1}{2} x^\top A x - b^\top x$ over $x\in\mathbb{R}^n$. We evaluate four accelerated gradient methods in this setting, including the classical FGM algorithm \cite{N87} and our three proposed schemes, namely UGM-I, UGM-II, and UGM-III, which are developed in Section \ref{sec:new_algs}.
To simplify the experimental setup, we construct $A$ and $b$ in a controlled manner. We first randomly generate an optimal solution $x^*$ whose entries are uniformly sampled from $[0,1]$, and then set $b = A x^*$. The matrix $A$ is defined as the diagonal matrix $A = \operatorname{diag}(\lambda_1,\lambda_2,\dots,\lambda_n)$ with ordered eigenvalues $\lambda_1\leq \lambda_2\leq \cdots\leq\lambda_n$. For a fixed dimension $n$, we set the extreme eigenvalues to $\lambda_1 = \mu$ and $\lambda_n = L$, which ensures that $f$ is $\mu$-strongly convex and $L$-smooth. The intermediate eigenvalues $\lambda_2,\ldots,\lambda_{n-1}$ are sampled from five predefined distributions, rescaled to the interval $[\mu, L]$, and finally sorted in ascending order. 
The detailed configurations of the five eigenvalue distributions are presented as follows.

\begin{enumerate}[itemindent=2em]
    \item[(Case 1)]  
    Geometric progression: set $\lambda_i = \mu r^{\,i-1}$ for $i=2,\ldots,n-1$ with $r = (L/\mu)^{1/(n-1)}$.
    
    \item[(Case 2)] U-shaped distribution: generate $n-2$ $i.i.d.$ samples from $\texttt{Beta}(0.01,0.01)$.
    
    \item[(Case 3)] Central-concentrated distribution: generate $n-2$ $i.i.d.$ samples from $\texttt{Beta}(50,50)$.
    
    \item[(Case 4)] Lower-end skewed distribution: generate $n-2$ $i.i.d.$  samples from $\texttt{Beta}(1,20)$.
    
    \item[(Case 5)] Upper-end skewed distribution: generate $n-2$ $i.i.d.$ samples from $\texttt{Beta}(20,1)$.
\end{enumerate}
For Cases 2-5,  all samples are linearly rescaled to the interval $[\mu, L]$ via the mapping $x \mapsto \mu + (L-\mu)x$, and the resulting values are then sorted in ascending order to form the intermediate eigenvalues. In this experiment, we set $n=1000$. The five eigenvalue distributions are visualized in Figure \ref{fig:eigendist}.

\begin{figure}[!htbp]
    \centering
    \includegraphics[width=0.8\textwidth]{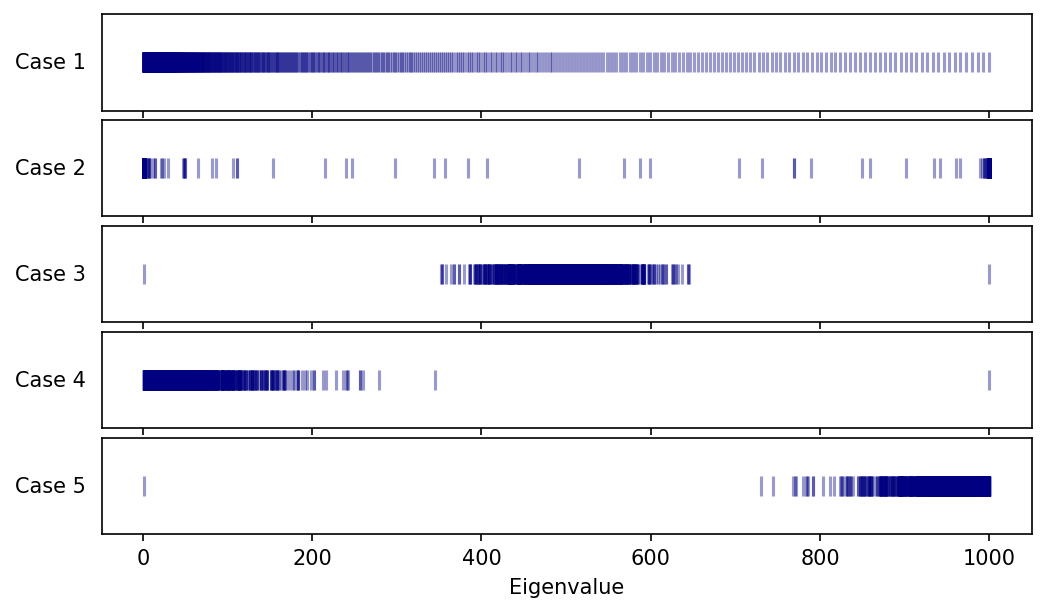}
    \caption{Illustration of eigenvalue distributions for Cases 1-5 with $\mu=1$ and $L=10^3$.}
    \label{fig:eigendist}
\end{figure}
We consider five combinations of strong convexity and smoothness parameters $(\mu, L)$: $(1,10^2)$, $(1,10^3)$, $(1,10^6)$, $(10,10^2)$, and $(10^2,10^4)$. These configurations produce condition numbers $\kappa = L/\mu$ ranging from $10$ to $10^6$. 
To ensure fair comparison across different conditioning levels, the maximum iteration number $K$ is adaptively determined according to $\kappa$. Specifically, we specify three anchor points $(\log_{10}\kappa, K) = (1,50),\;(3,300),\;(6,5000)$ and adopt linear interpolation on the $\log_{10}\kappa$ scale to compute the iteration budget for each tested condition number.

For the proposed UGMs, we adopt the following parameter settings. We set $\varrho=1-1/\sqrt{\kappa}$ for all the three UGMs, $q_0 = \frac{\sqrt{\kappa}-1}{\sqrt{\kappa}+1}$ for UGM-I and $q_0 = \varrho$ for both UGM-II and UGM-III.  The lower bound parameter $\underline{\theta}$ is set to $1$ for $\kappa \geq 10^3$ and $0.01$ otherwise, which satisfies the theoretical constraint $\underline{\theta}<\frac{\varrho}{1-\varrho}$. Furthermore, we set $a_0= \frac{(1 - \varrho)[\kappa(1 - \varrho) - \varrho - 1]}{(\varrho + 1)\varrho}$ and define the upper bound parameter $\bar{\theta}=\max\left\{\frac{1}{1-\varrho},\frac{\varrho a_0+1-\varrho}{(1-\varrho)a_0}\right\}=\sqrt{\kappa}$.  All compared algorithms are initialized at the origin. 

\paragraph{Comparison results.} 
Each $(\mu,L)$ configuration is run five times with independent random seeds. For each eigenvalue distribution and each tested $(\mu, L)$ pair, we report the objective suboptimality $f(x^k)-f_*$ against iteration count and CPU time (s), alongside the evolution of the extrapolation parameter $\theta_k$ for UGM-II and UGM-III. Numerical results for Cases 1–5 are presented respectively in Figures \ref{fig:case1}–\ref{fig:case5}.

\begin{figure}[!htbp]
\centering
\begin{subfigure}[b]{0.19\textwidth}\includegraphics[width=\textwidth]{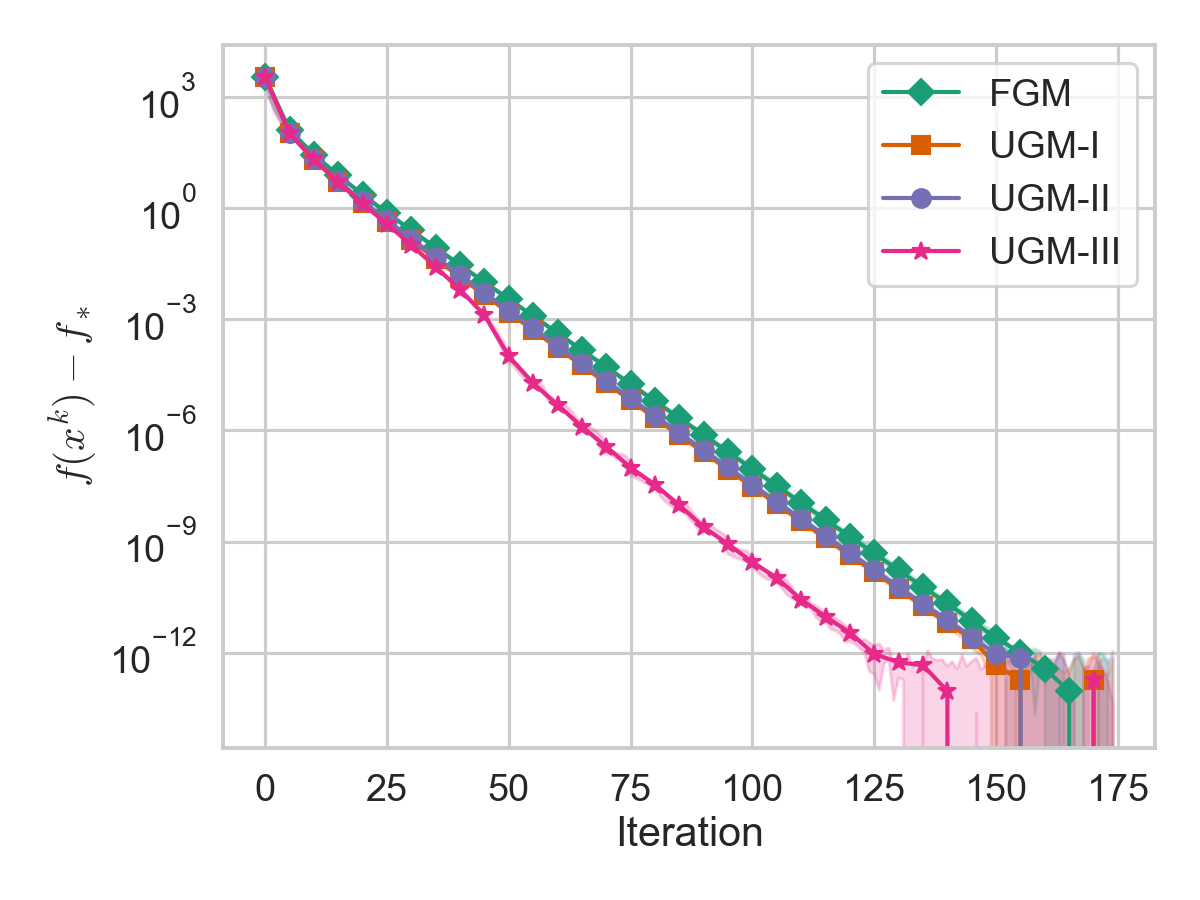}\end{subfigure}\hfill
\begin{subfigure}[b]{0.19\textwidth}\includegraphics[width=\textwidth]{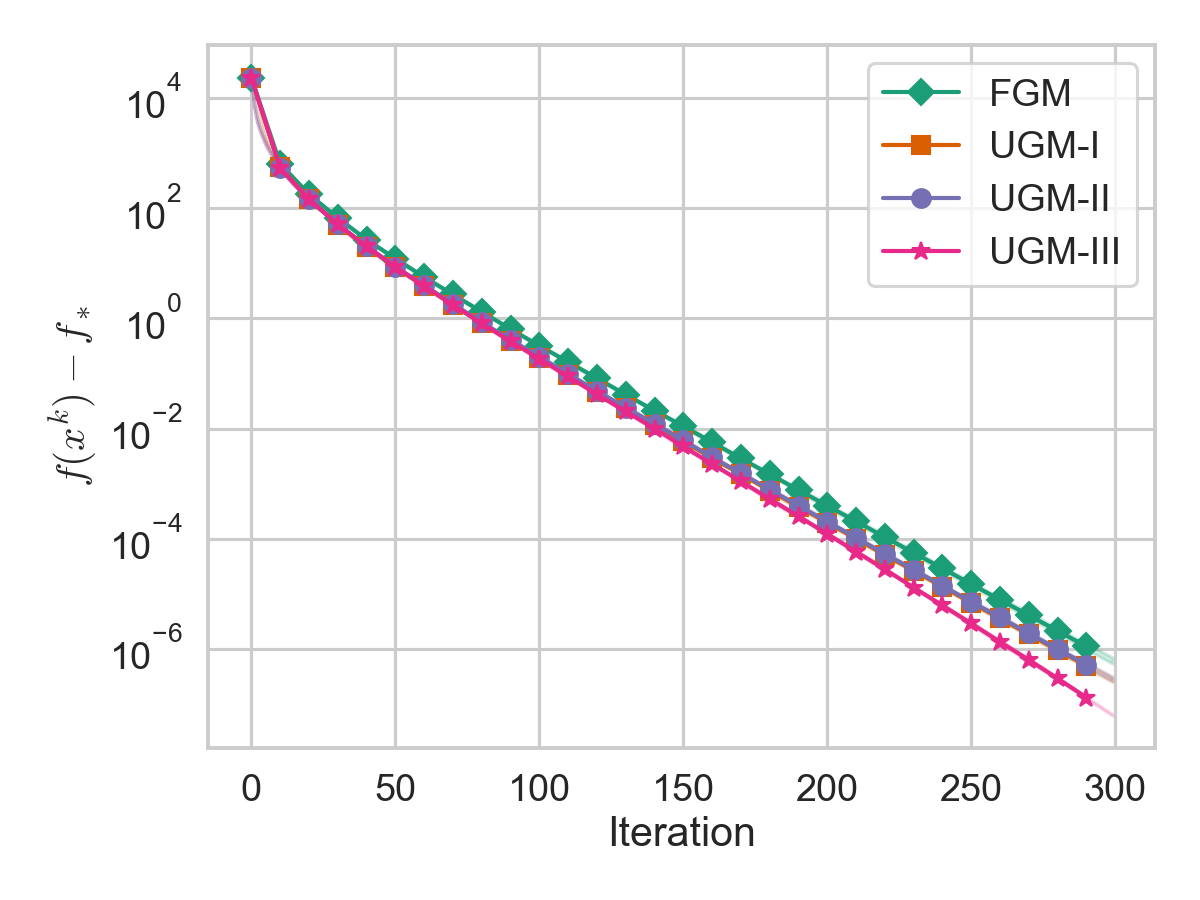}\end{subfigure}\hfill
\begin{subfigure}[b]{0.19\textwidth}\includegraphics[width=\textwidth]{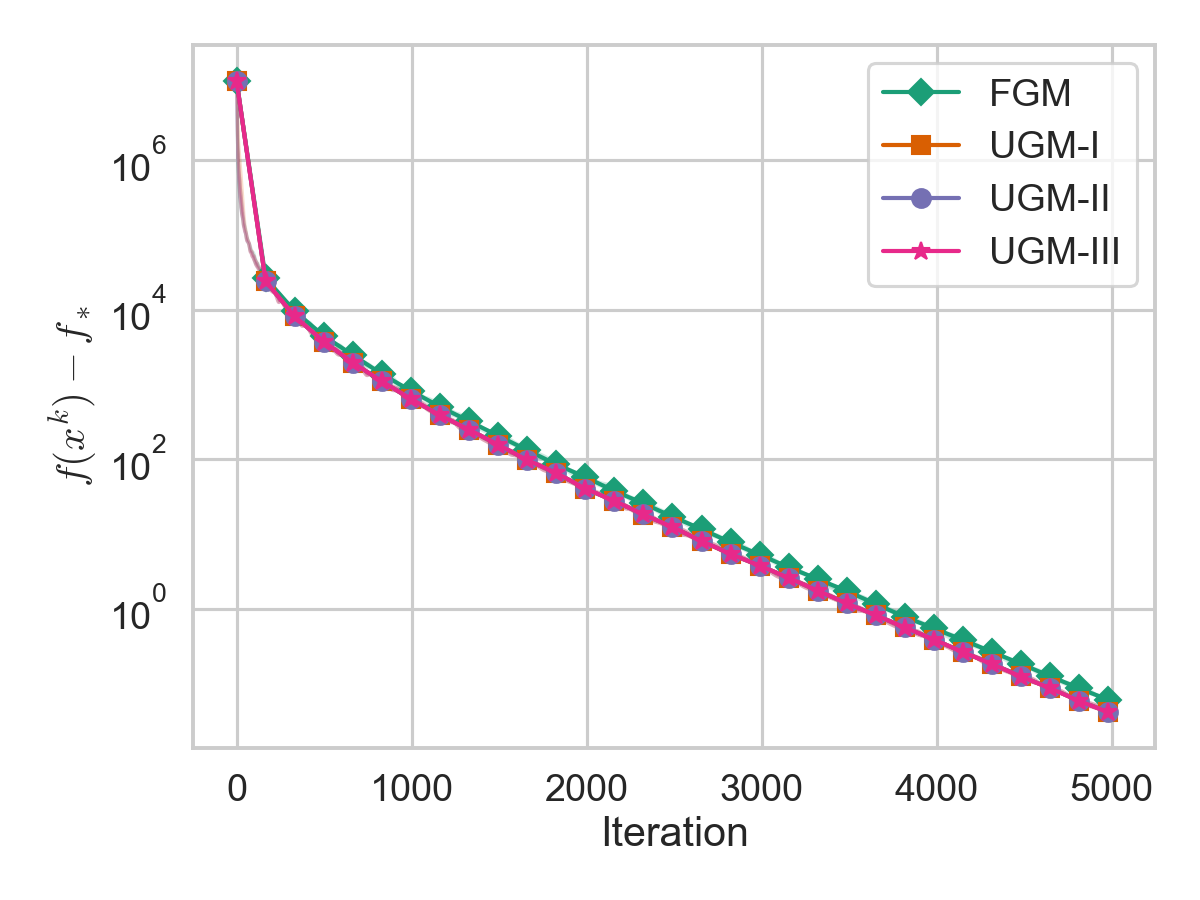}\end{subfigure}\hfill
\begin{subfigure}[b]{0.19\textwidth}\includegraphics[width=\textwidth]{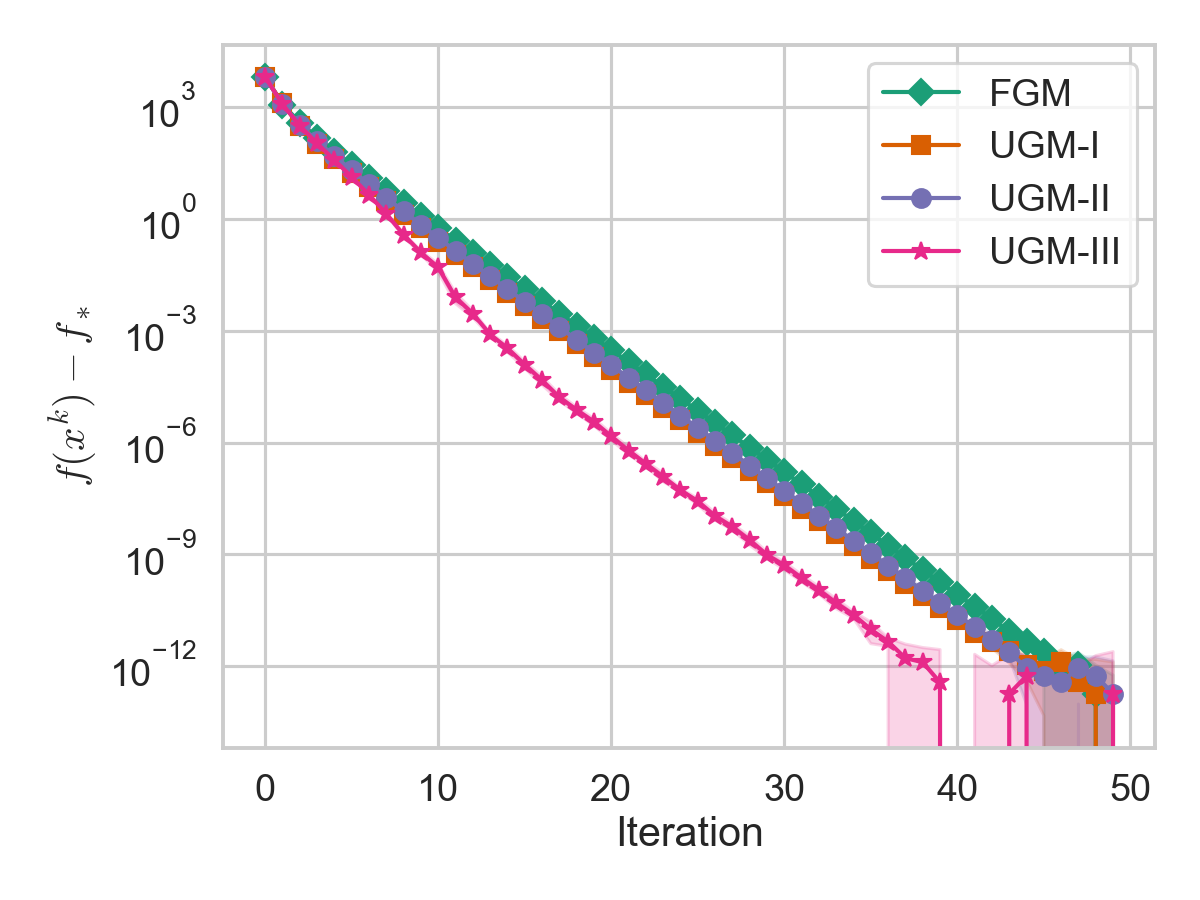}\end{subfigure}\hfill
\begin{subfigure}[b]{0.19\textwidth}\includegraphics[width=\textwidth]{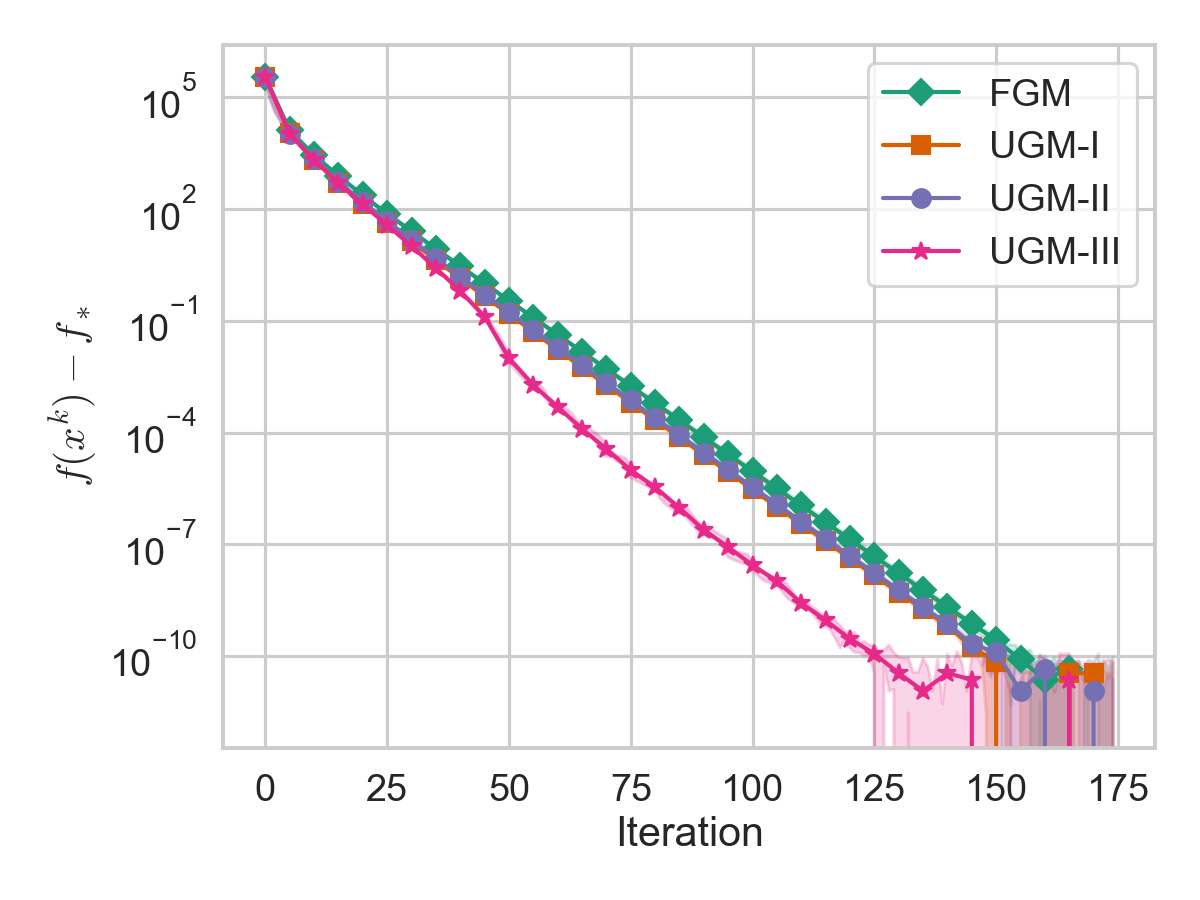}\end{subfigure}
\\
\begin{subfigure}[b]{0.19\textwidth}\includegraphics[width=\textwidth]{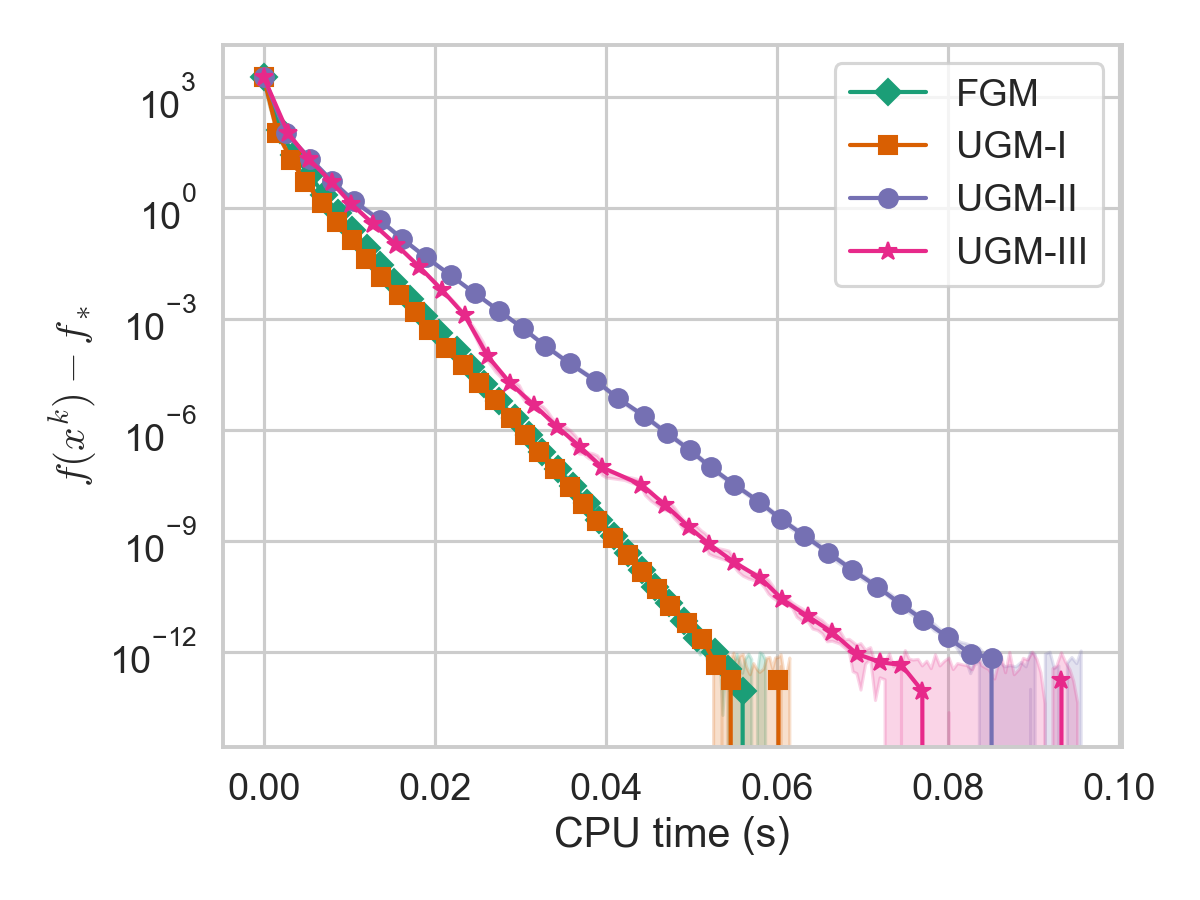}\end{subfigure}\hfill
\begin{subfigure}[b]{0.19\textwidth}\includegraphics[width=\textwidth]{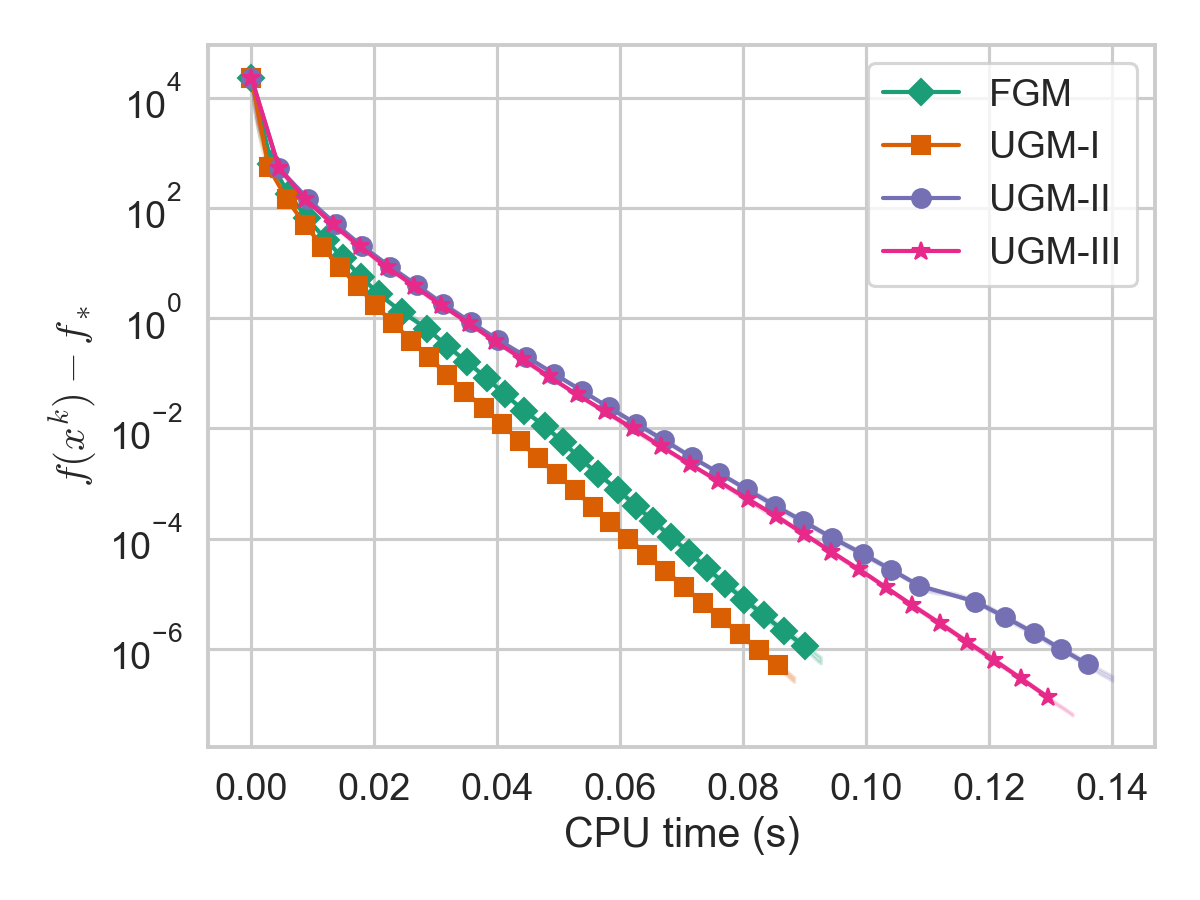}\end{subfigure}\hfill
\begin{subfigure}[b]{0.19\textwidth}\includegraphics[width=\textwidth]{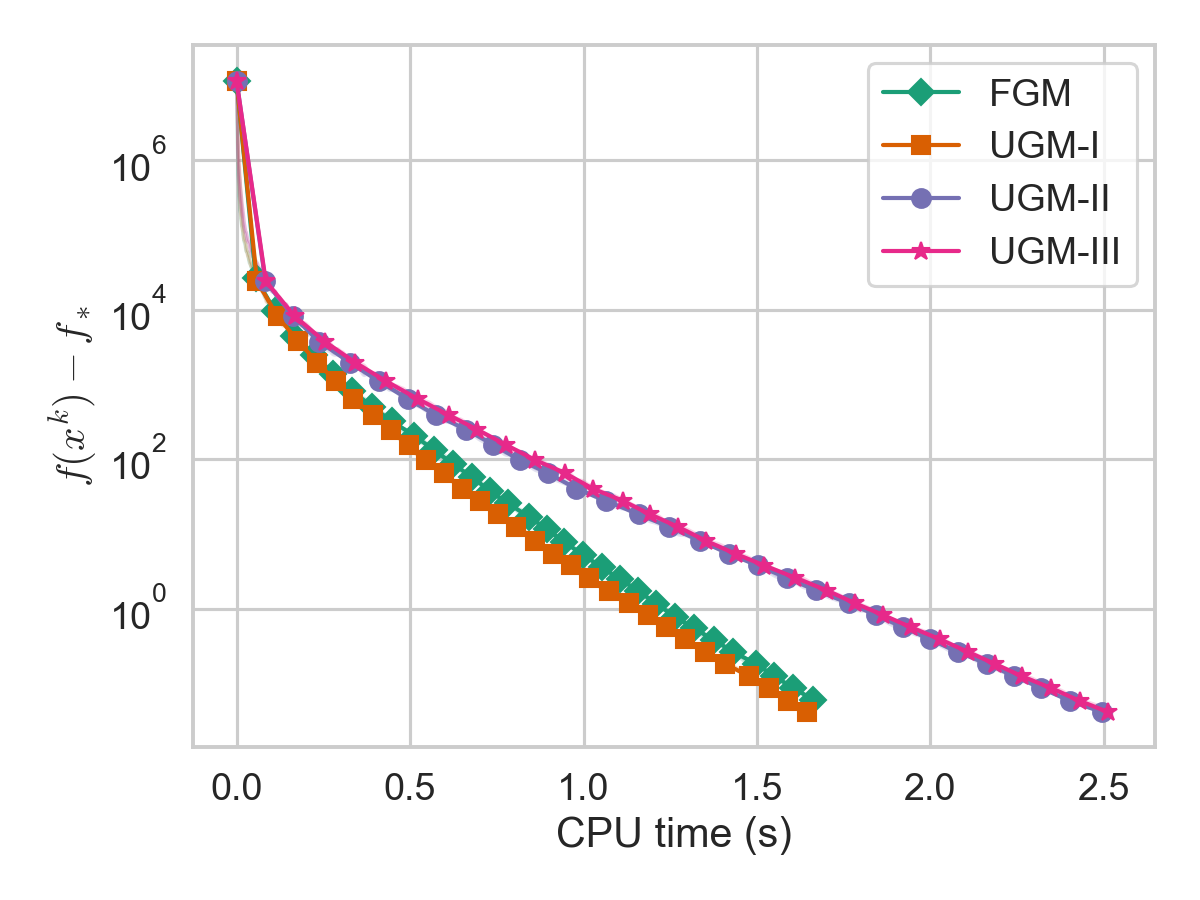}\end{subfigure}\hfill
\begin{subfigure}[b]{0.19\textwidth}\includegraphics[width=\textwidth]{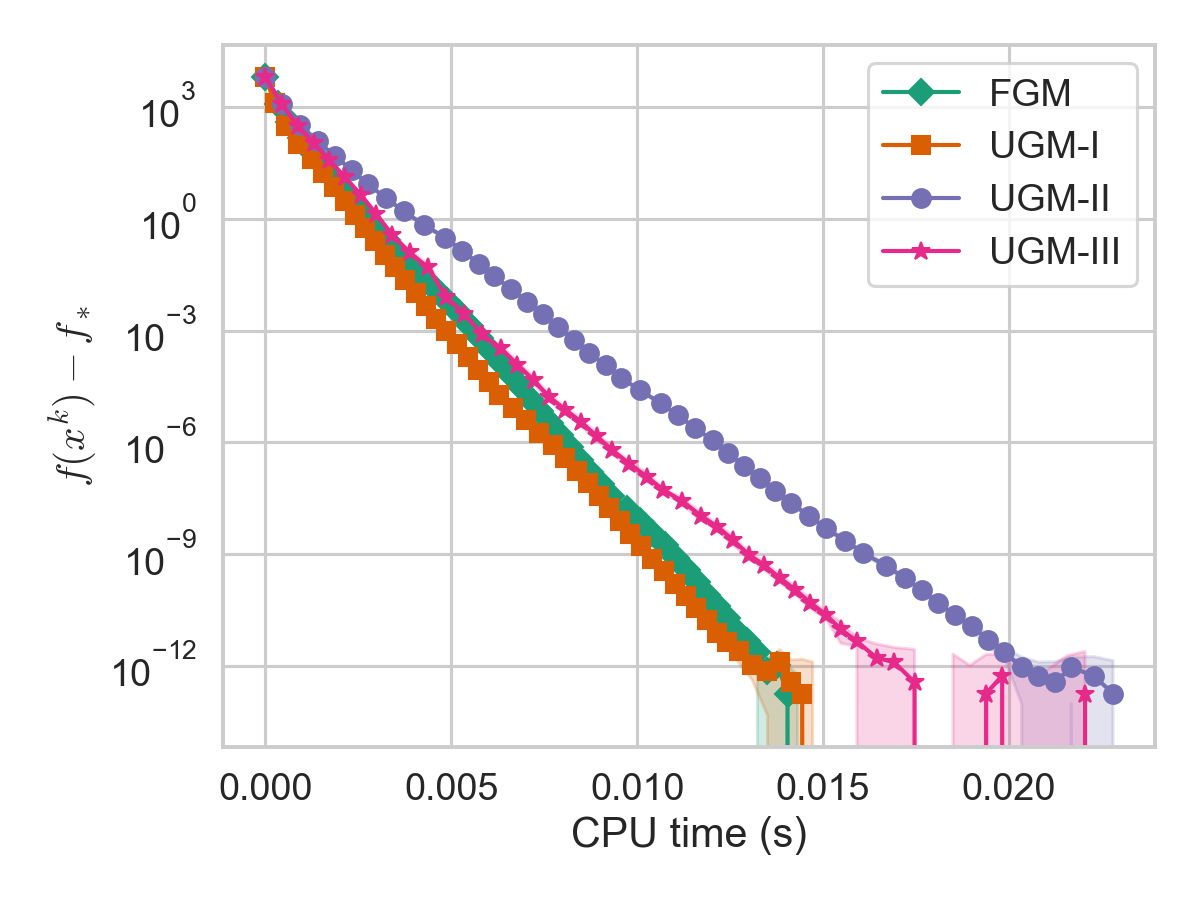}\end{subfigure}\hfill
\begin{subfigure}[b]{0.19\textwidth}\includegraphics[width=\textwidth]{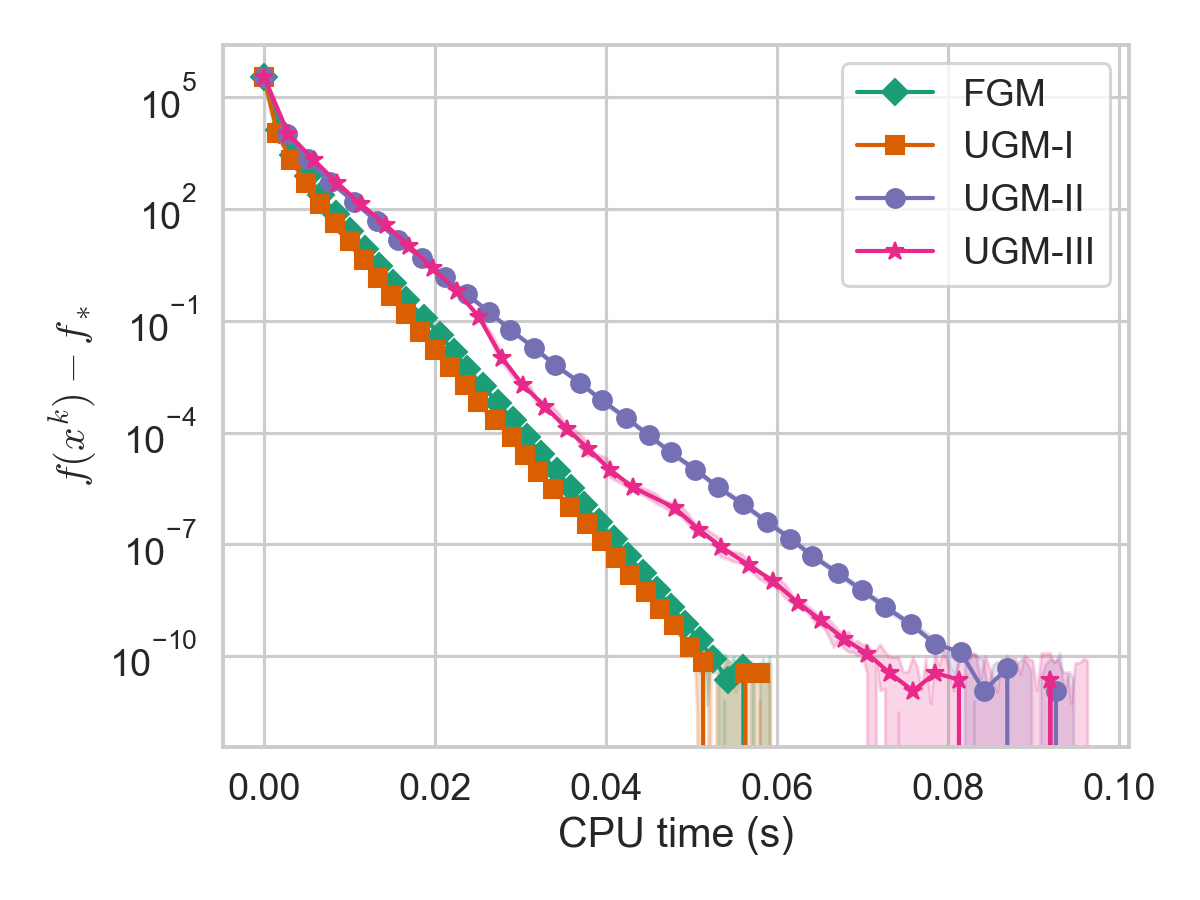}\end{subfigure}
\\
\begin{subfigure}[b]{0.19\textwidth}\includegraphics[width=\textwidth]{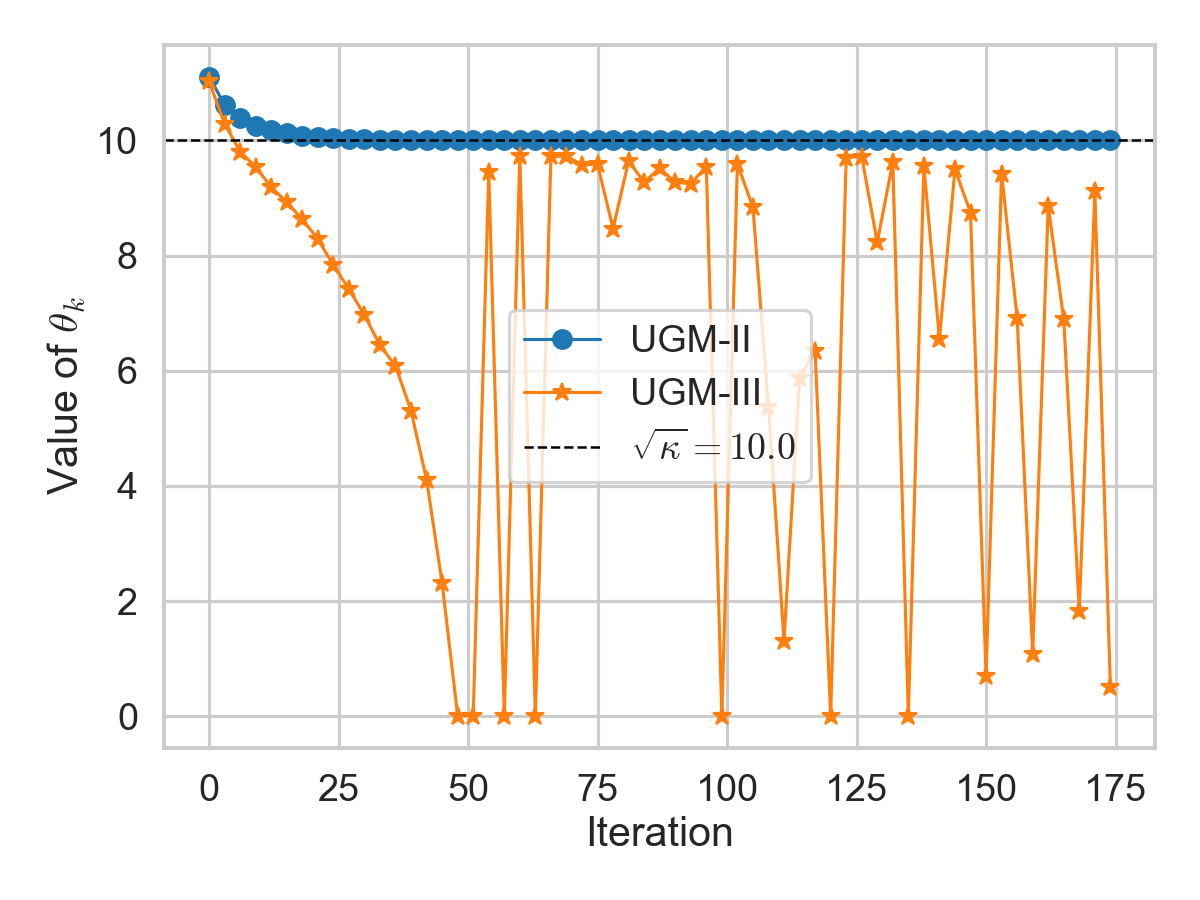}\caption{$\mu=1, L=10^2$}\end{subfigure}\hfill
\begin{subfigure}[b]{0.19\textwidth}\includegraphics[width=\textwidth]{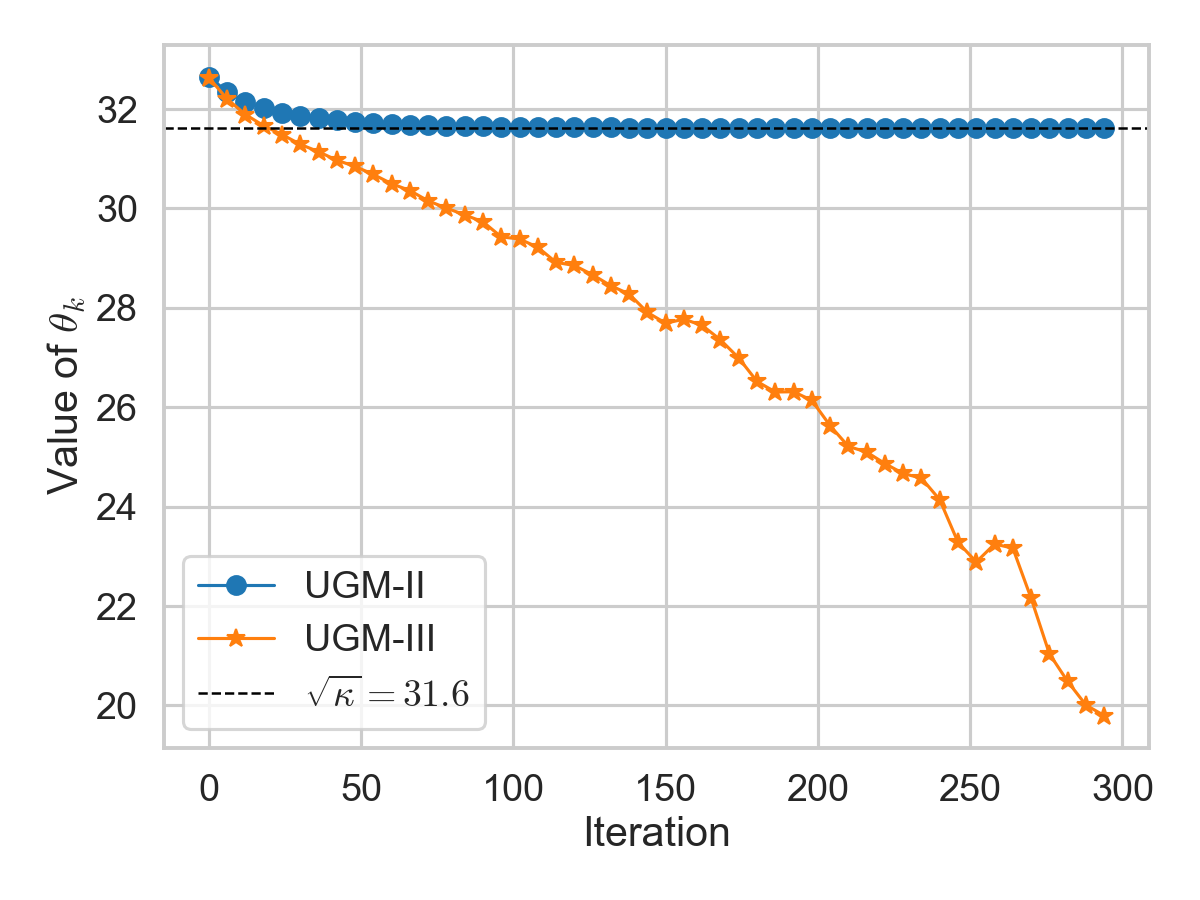}\caption{$\mu=1, L=10^3$}\end{subfigure}\hfill
\begin{subfigure}[b]{0.19\textwidth}\includegraphics[width=\textwidth]{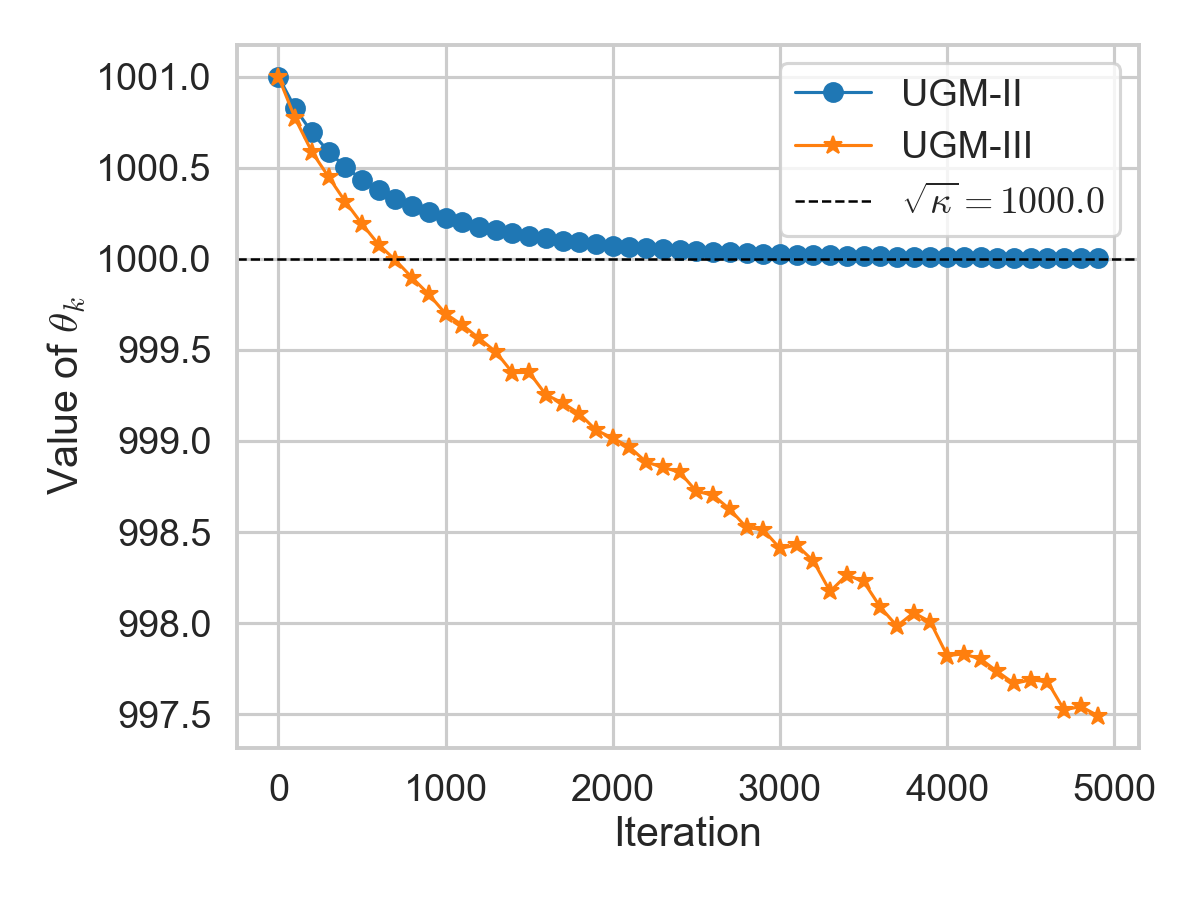}\caption{$\mu=1, L=10^6$}\end{subfigure}\hfill
\begin{subfigure}[b]{0.19\textwidth}\includegraphics[width=\textwidth]{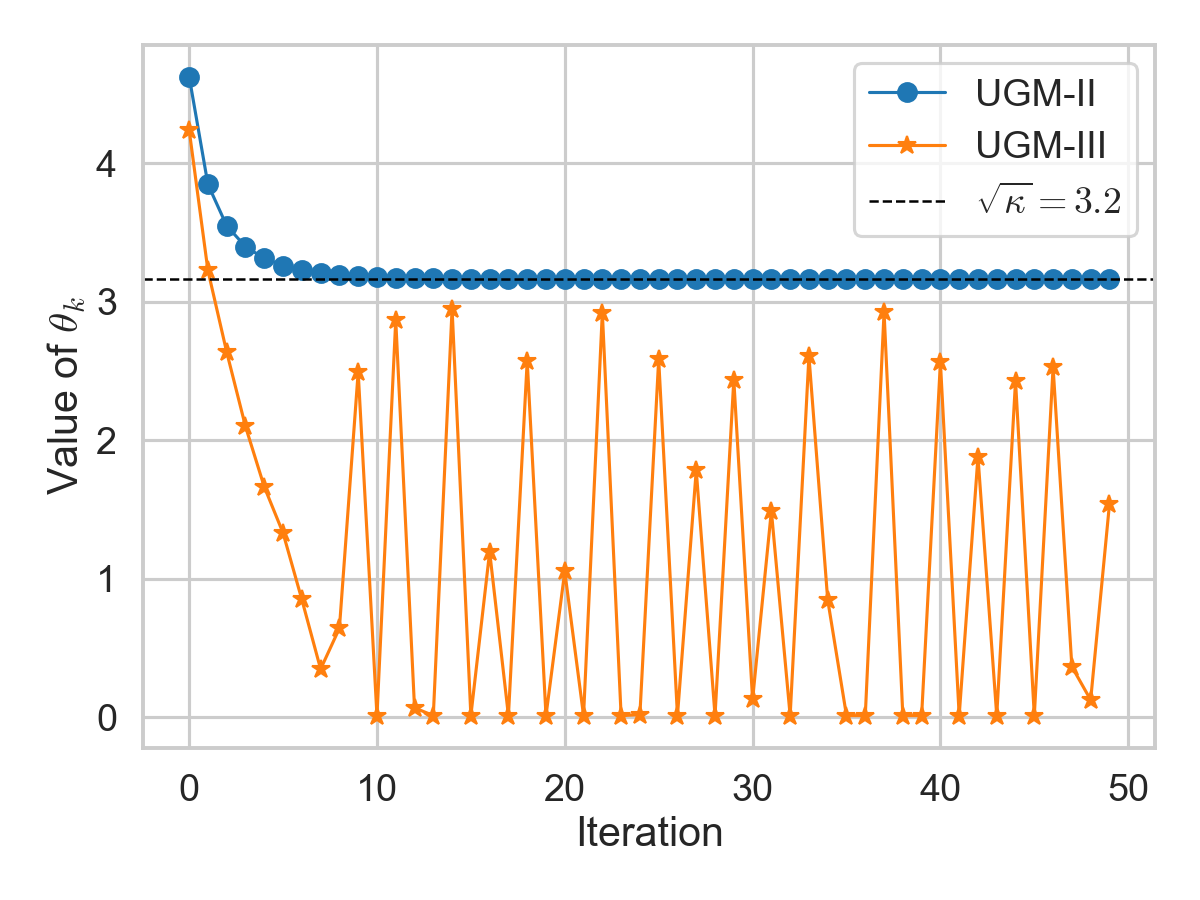}\caption{$\mu=10, L=10^2$}\end{subfigure}\hfill
\begin{subfigure}[b]{0.19\textwidth}\includegraphics[width=\textwidth]{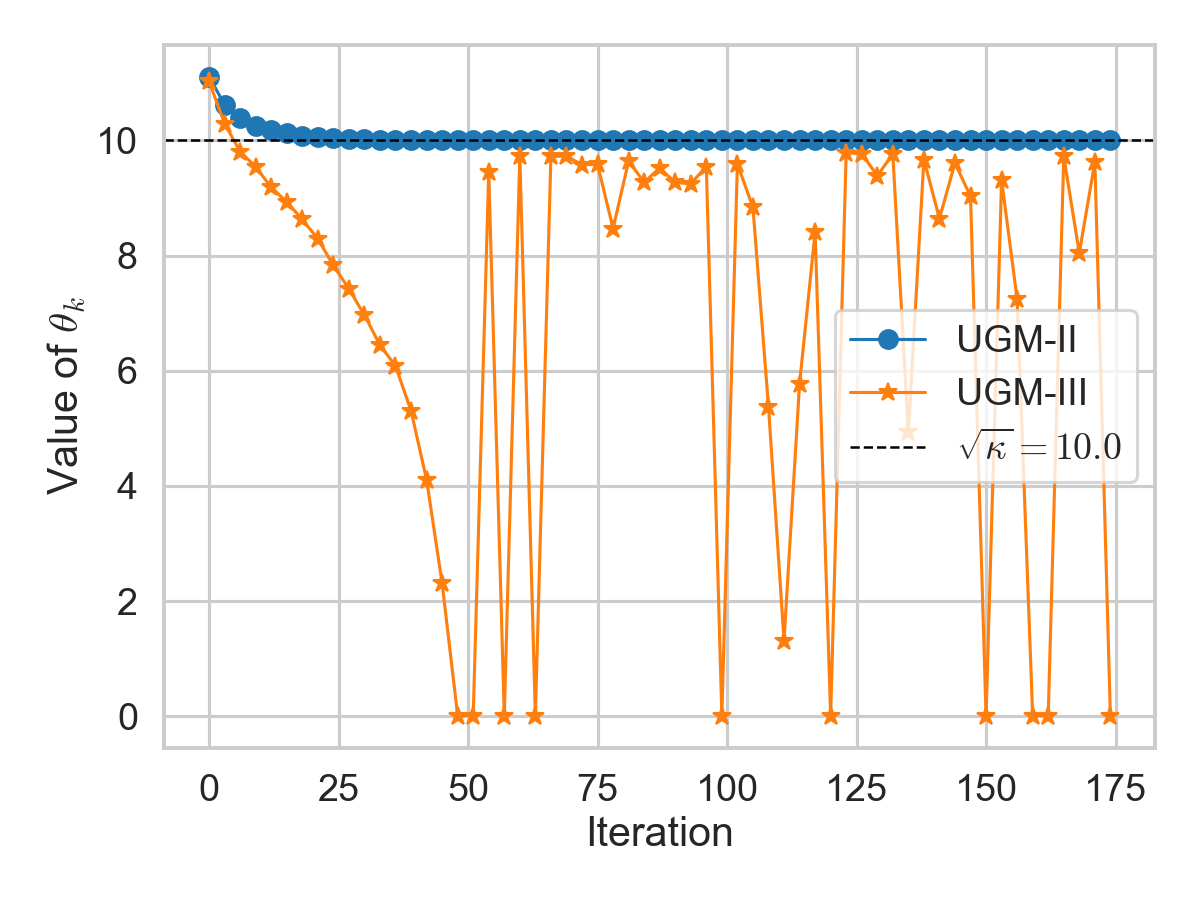}\caption{$\mu=10^2, L=10^4$}\end{subfigure}
\caption{Comparison of FGM, UGM-I, UGM-II, and UGM-III on Case 1. Rows show the objective suboptimality versus iterations, objective suboptimality versus CPU time, and the evolution of $\theta_k$, respectively. Columns correspond to different $(\mu, L)$ configurations.}
\label{fig:case1}
\end{figure}

\begin{figure}[!htbp]
\centering
\begin{subfigure}[b]{0.19\textwidth}\includegraphics[width=\textwidth]{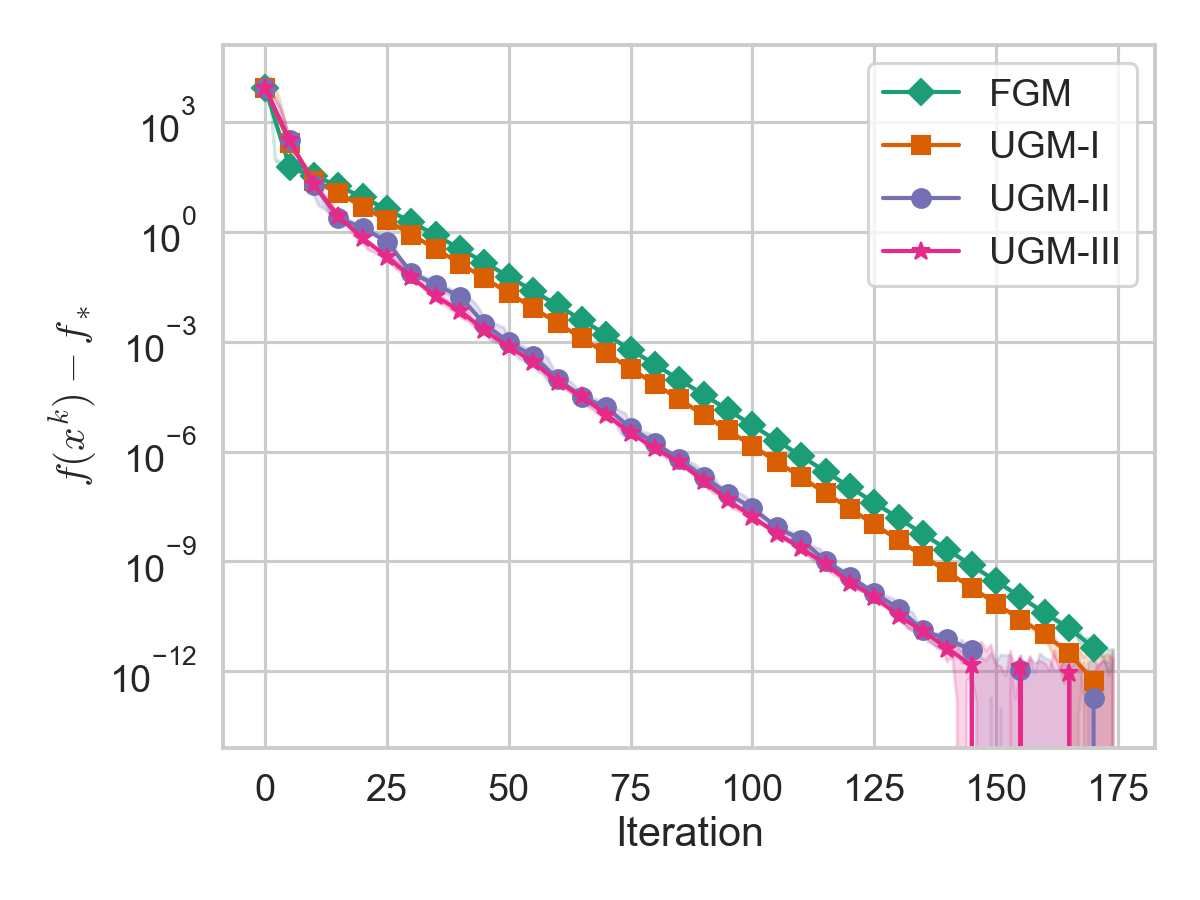}\end{subfigure}\hfill
\begin{subfigure}[b]{0.19\textwidth}\includegraphics[width=\textwidth]{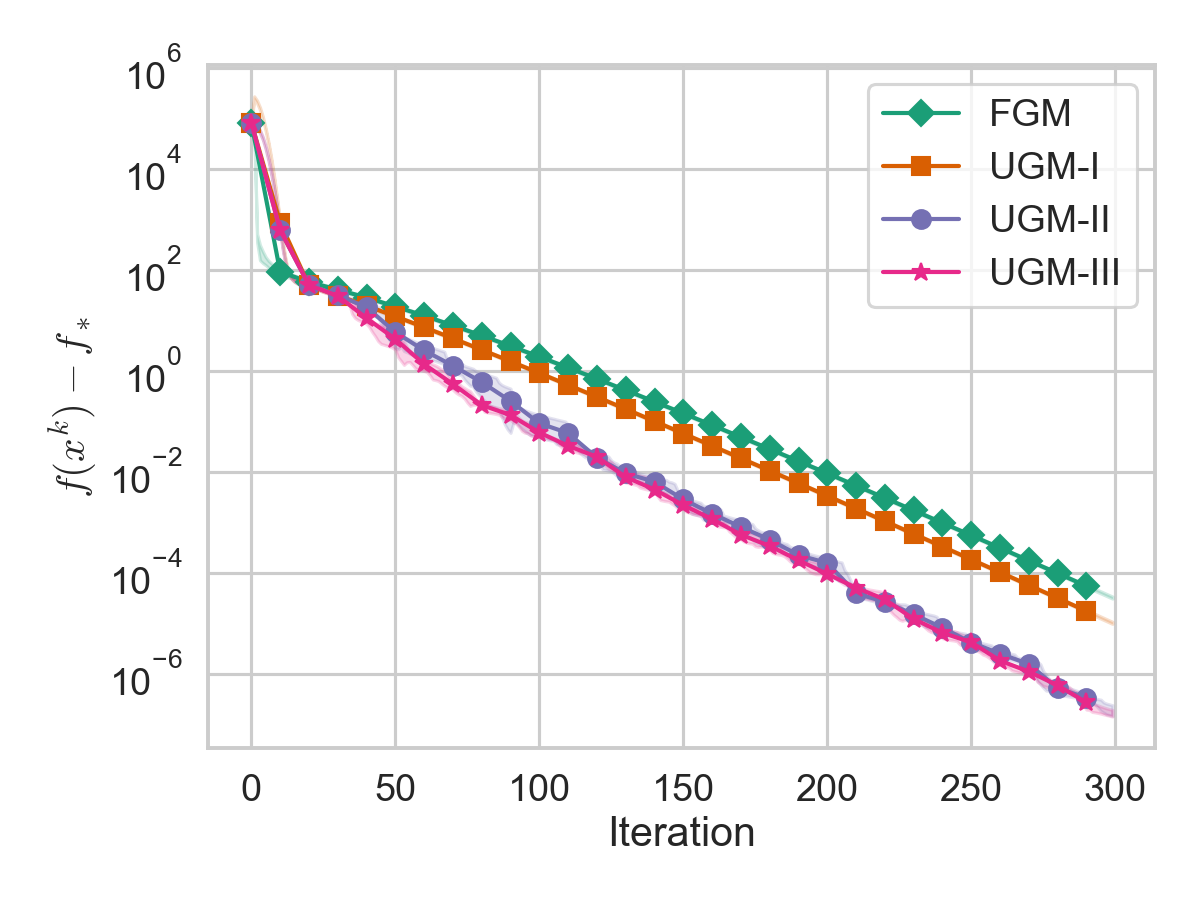}\end{subfigure}\hfill
\begin{subfigure}[b]{0.19\textwidth}\includegraphics[width=\textwidth]{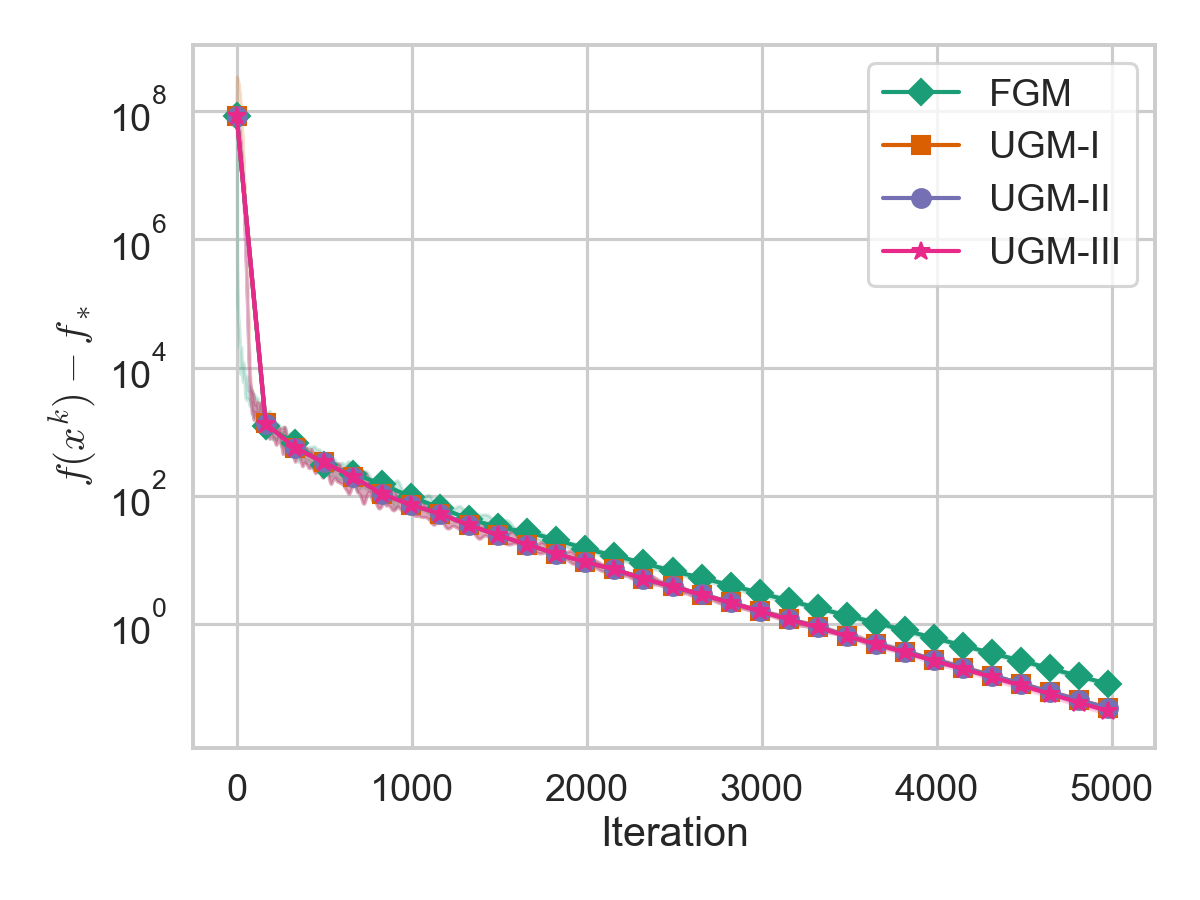}\end{subfigure}\hfill
\begin{subfigure}[b]{0.19\textwidth}\includegraphics[width=\textwidth]{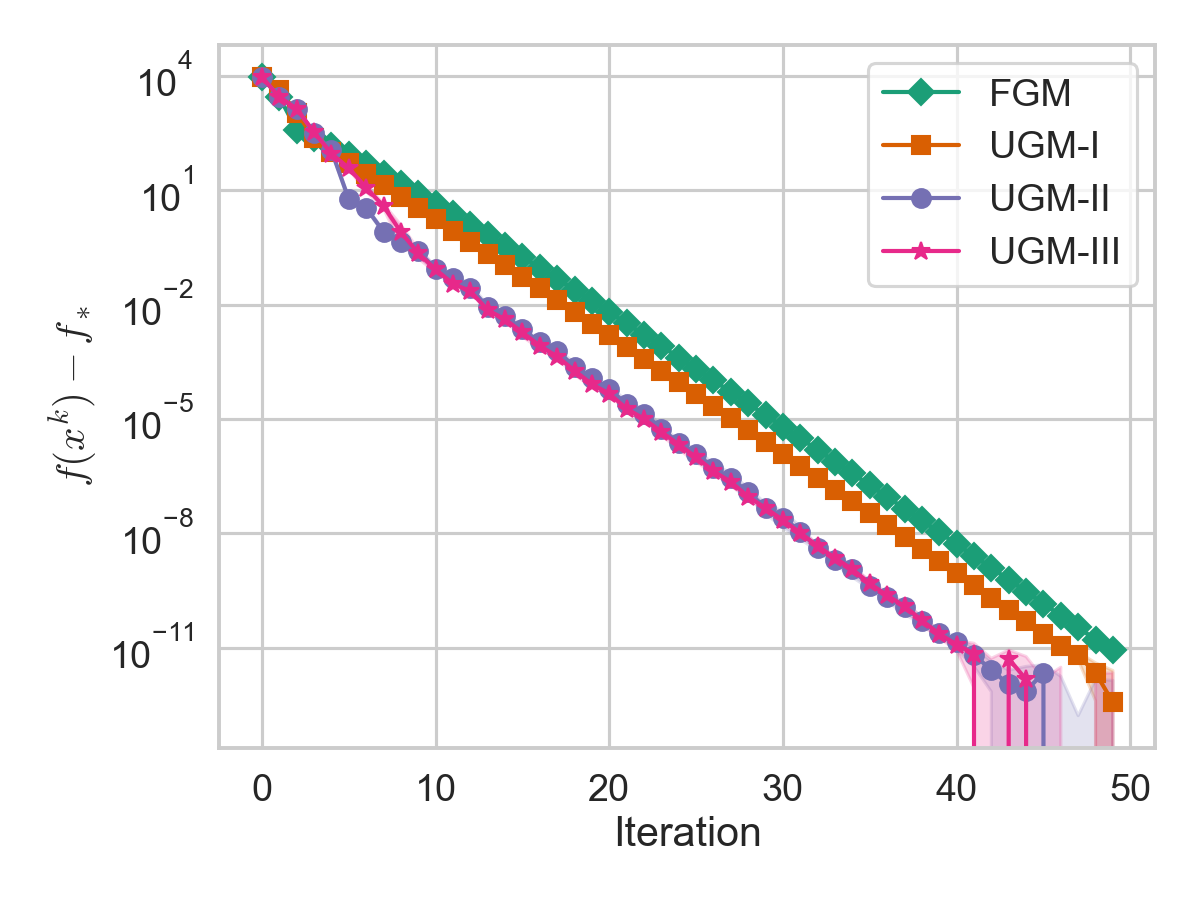}\end{subfigure}\hfill
\begin{subfigure}[b]{0.19\textwidth}\includegraphics[width=\textwidth]{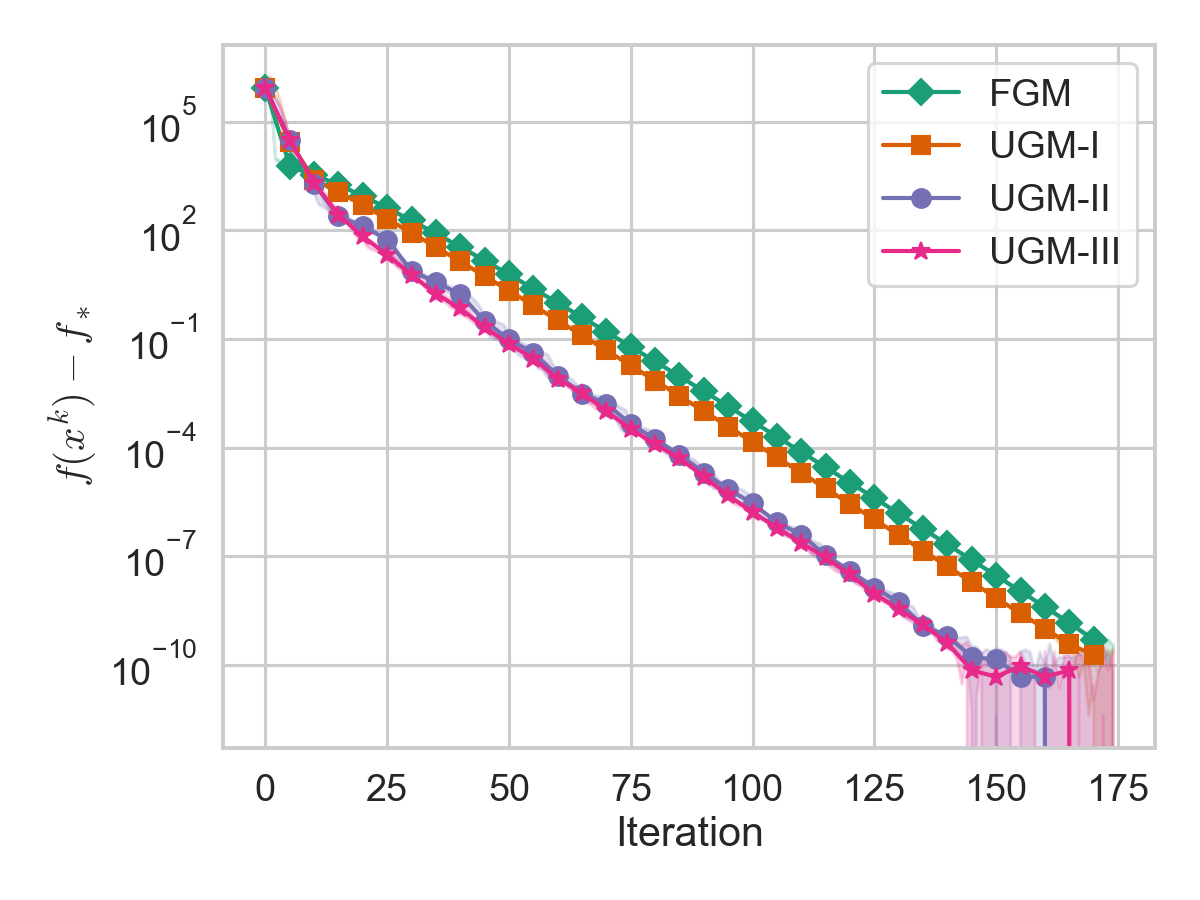}\end{subfigure}
\\
\begin{subfigure}[b]{0.19\textwidth}\includegraphics[width=\textwidth]{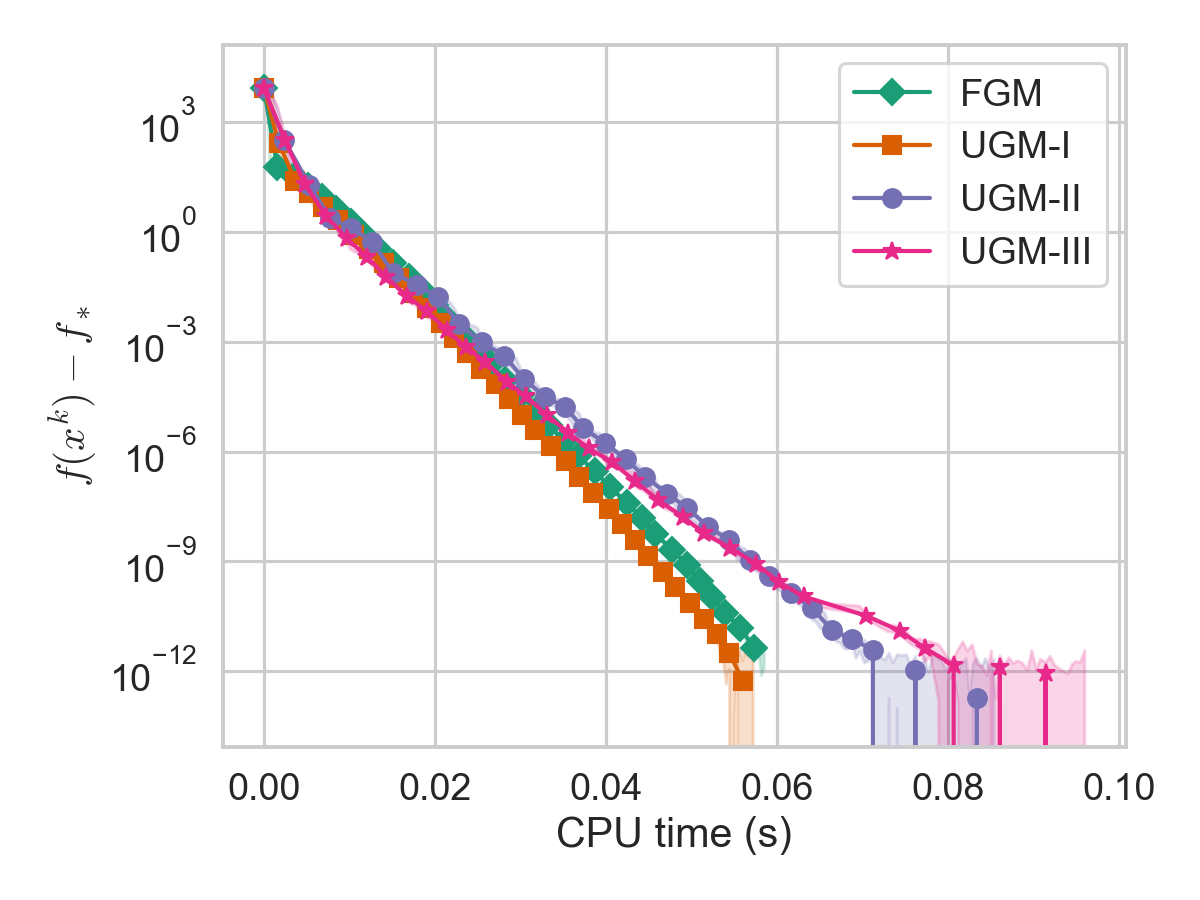}\end{subfigure}\hfill
\begin{subfigure}[b]{0.19\textwidth}\includegraphics[width=\textwidth]{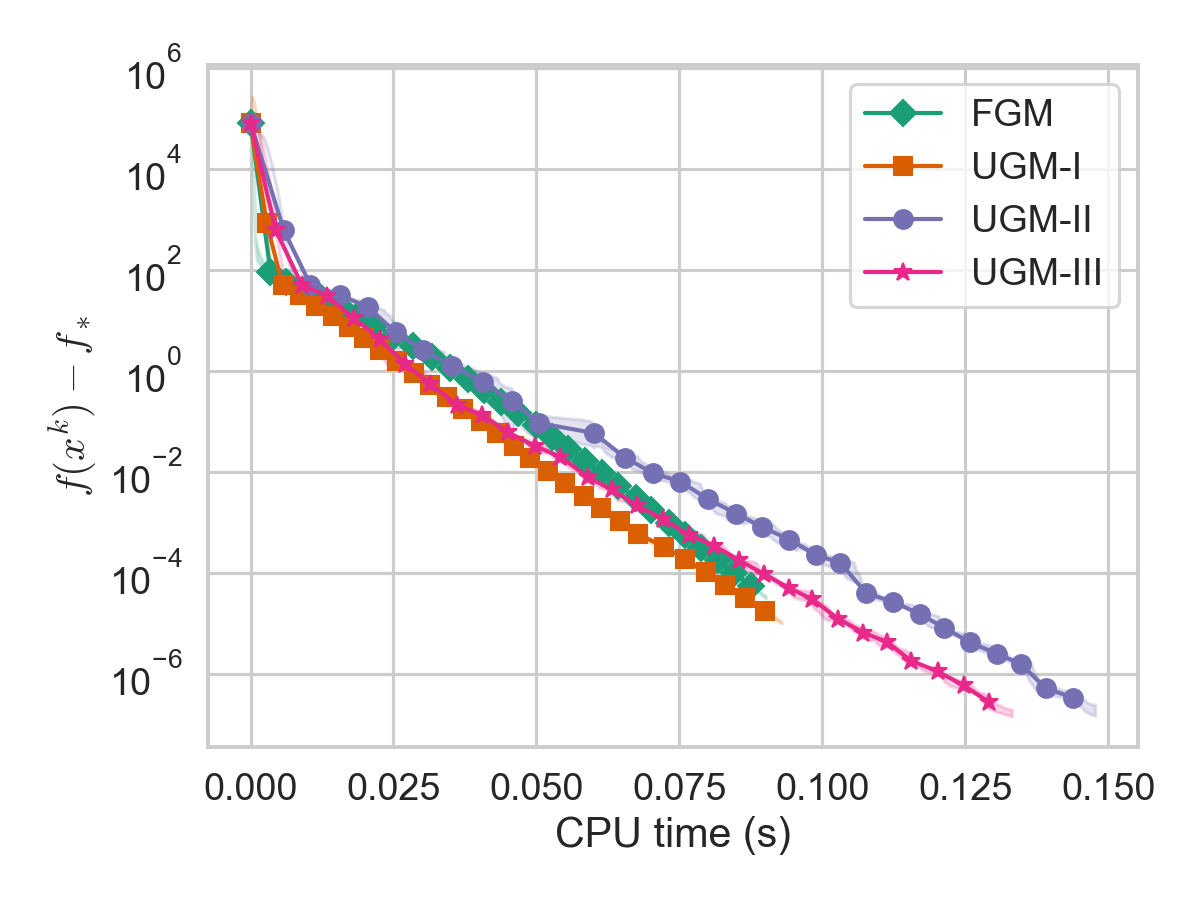}\end{subfigure}\hfill
\begin{subfigure}[b]{0.19\textwidth}\includegraphics[width=\textwidth]{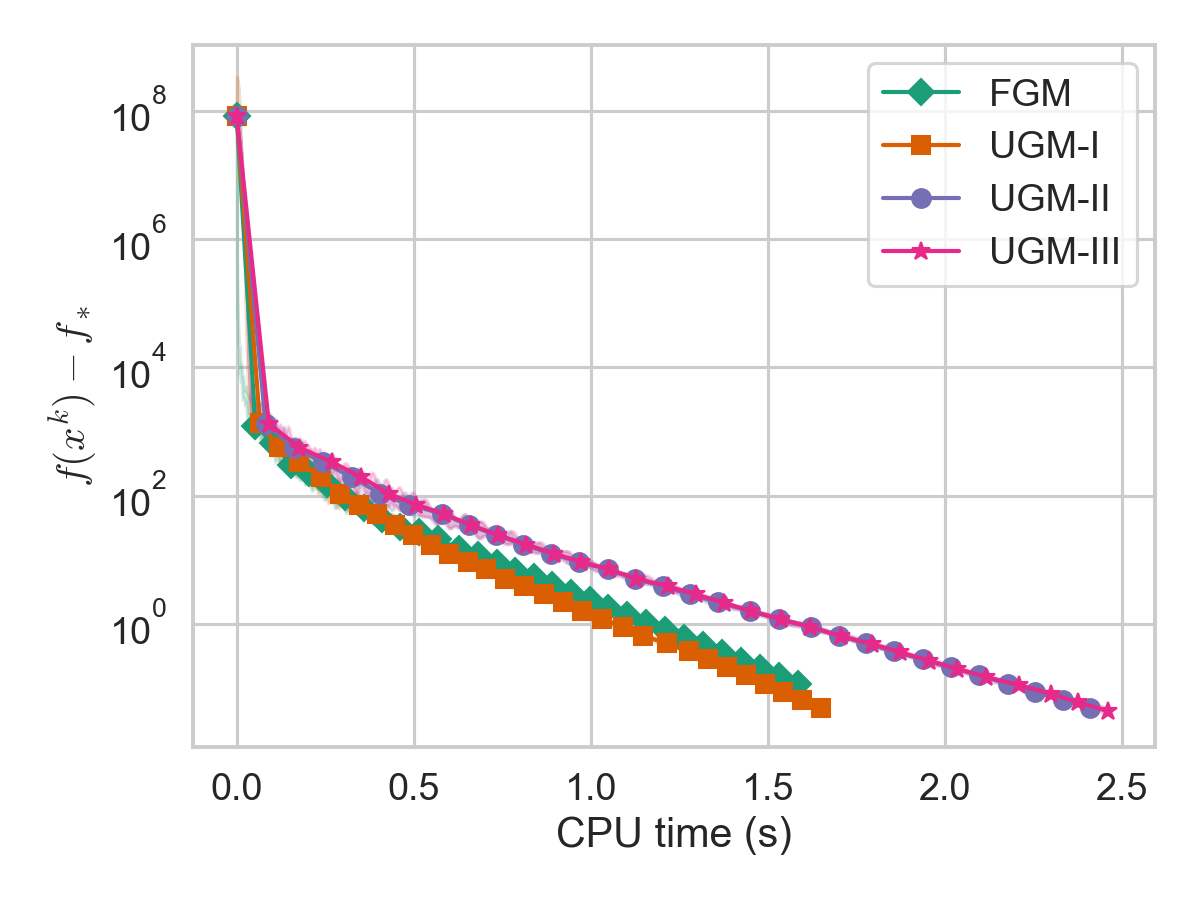}\end{subfigure}\hfill
\begin{subfigure}[b]{0.19\textwidth}\includegraphics[width=\textwidth]{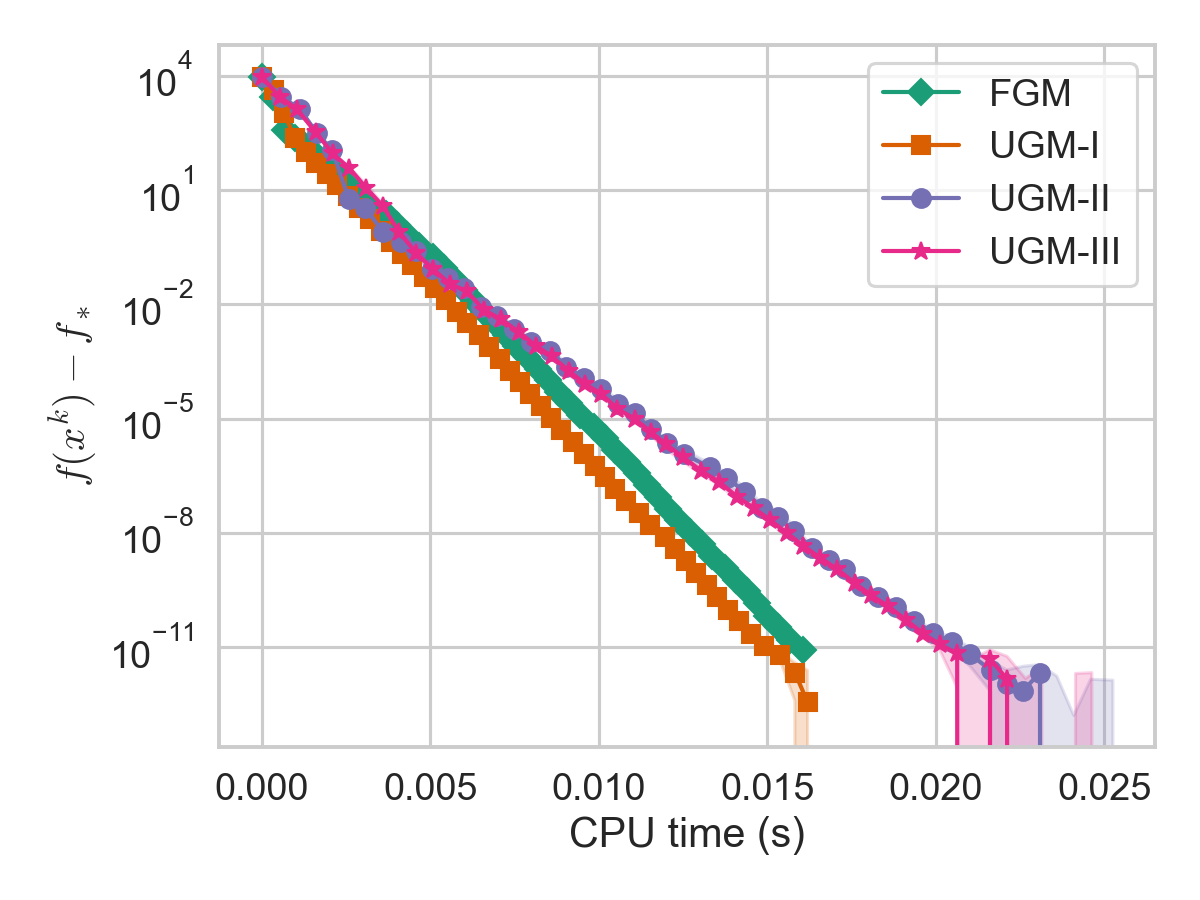}\end{subfigure}\hfill
\begin{subfigure}[b]{0.19\textwidth}\includegraphics[width=\textwidth]{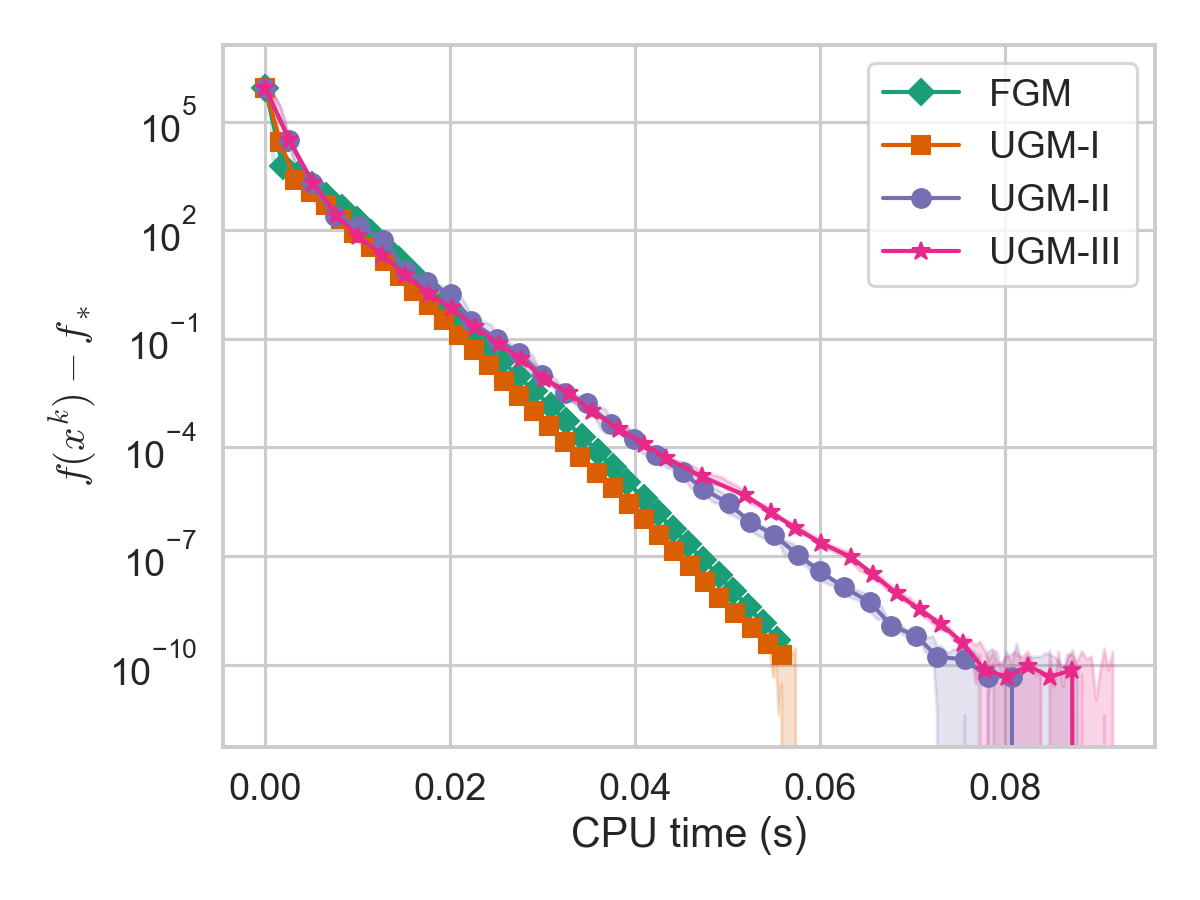}\end{subfigure}
\\
\begin{subfigure}[b]{0.19\textwidth}\includegraphics[width=\textwidth]{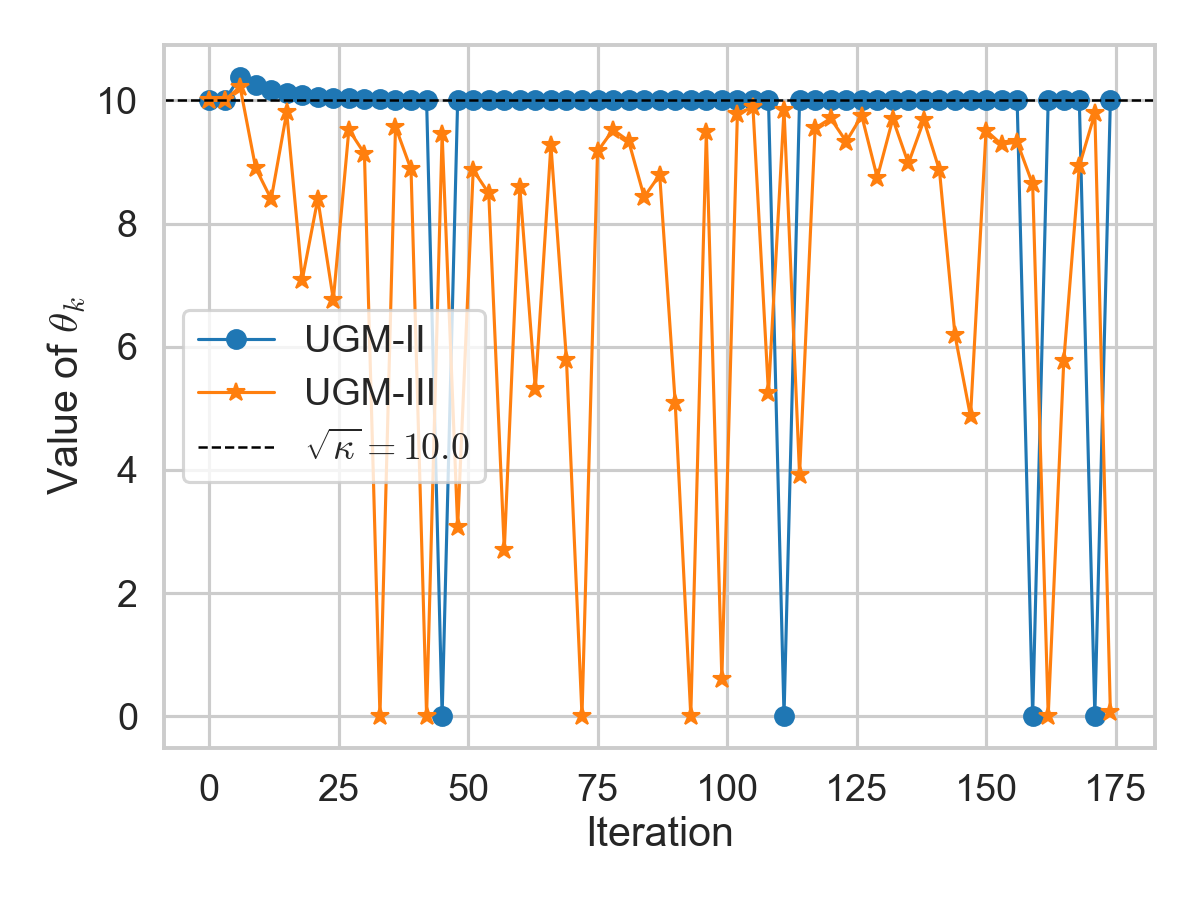}\caption{$\mu=1, L=10^2$}\end{subfigure}\hfill
\begin{subfigure}[b]{0.19\textwidth}\includegraphics[width=\textwidth]{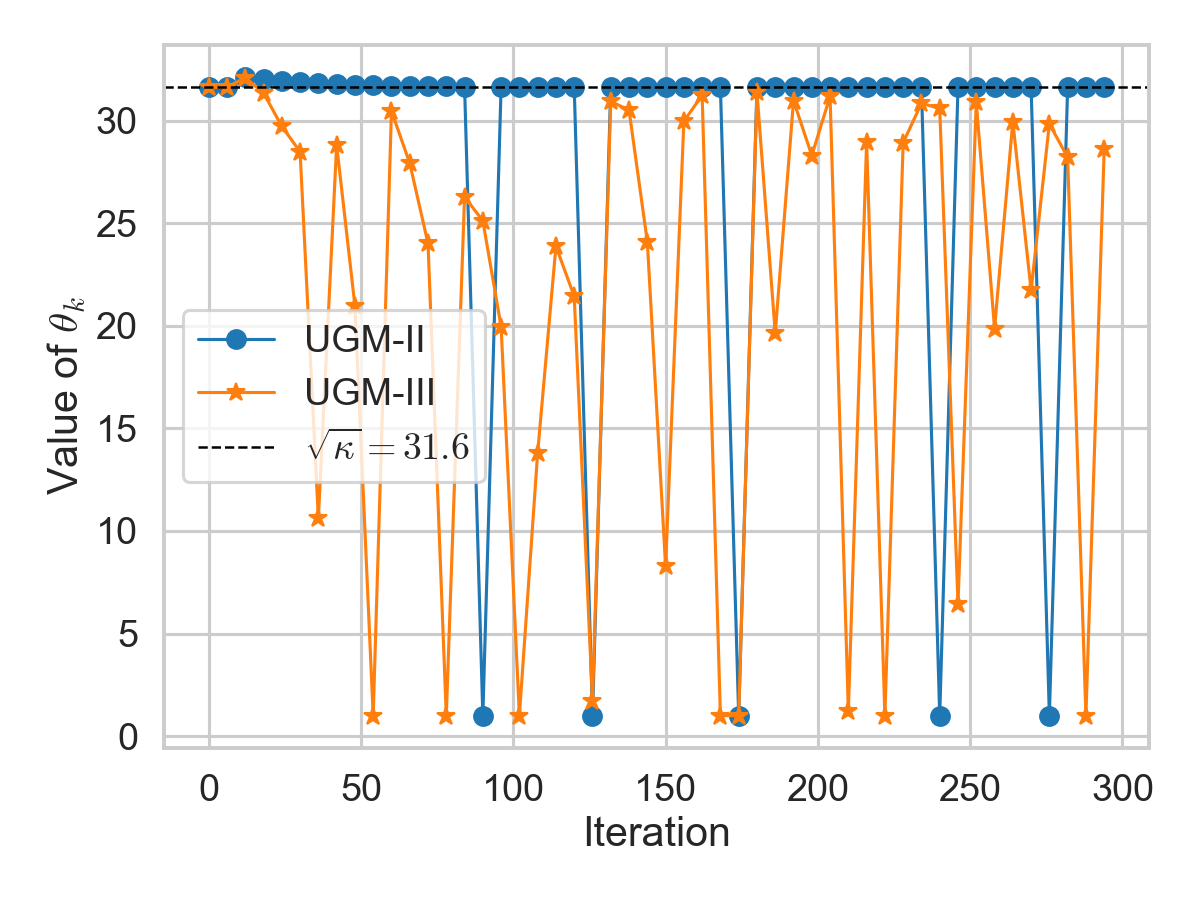}\caption{$\mu=1, L=10^3$}\end{subfigure}\hfill
\begin{subfigure}[b]{0.19\textwidth}\includegraphics[width=\textwidth]{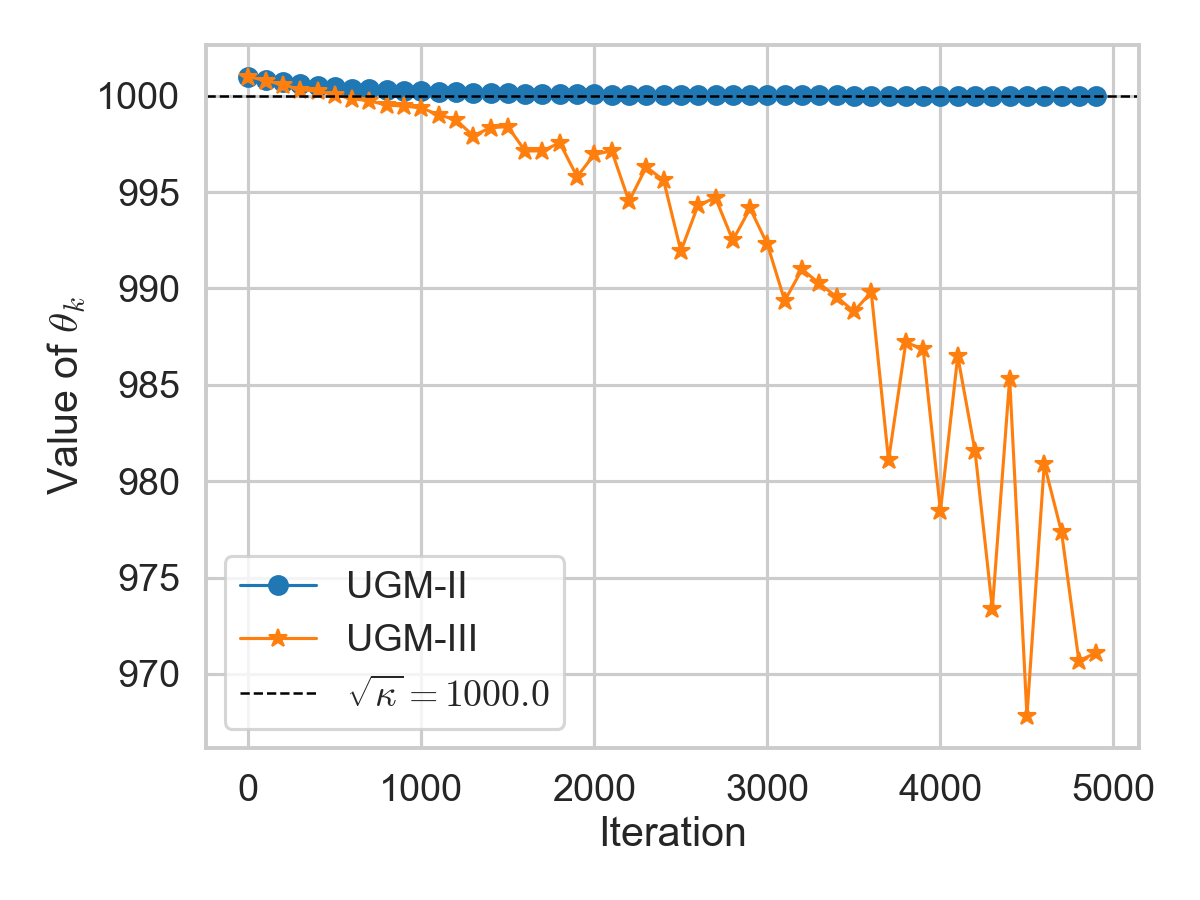}\caption{$\mu=1, L=10^6$}\end{subfigure}\hfill
\begin{subfigure}[b]{0.19\textwidth}\includegraphics[width=\textwidth]{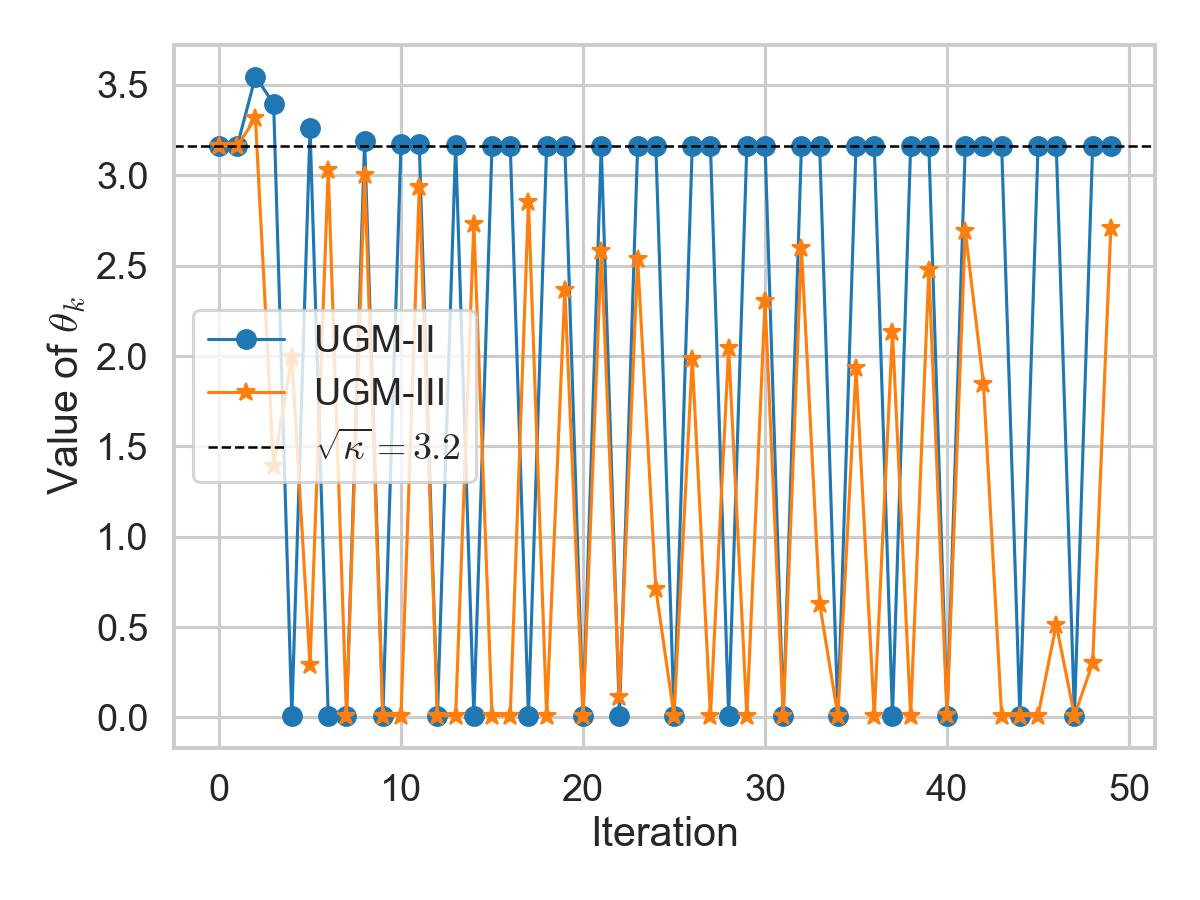}\caption{$\mu=10, L=10^2$}\end{subfigure}\hfill
\begin{subfigure}[b]{0.19\textwidth}\includegraphics[width=\textwidth]{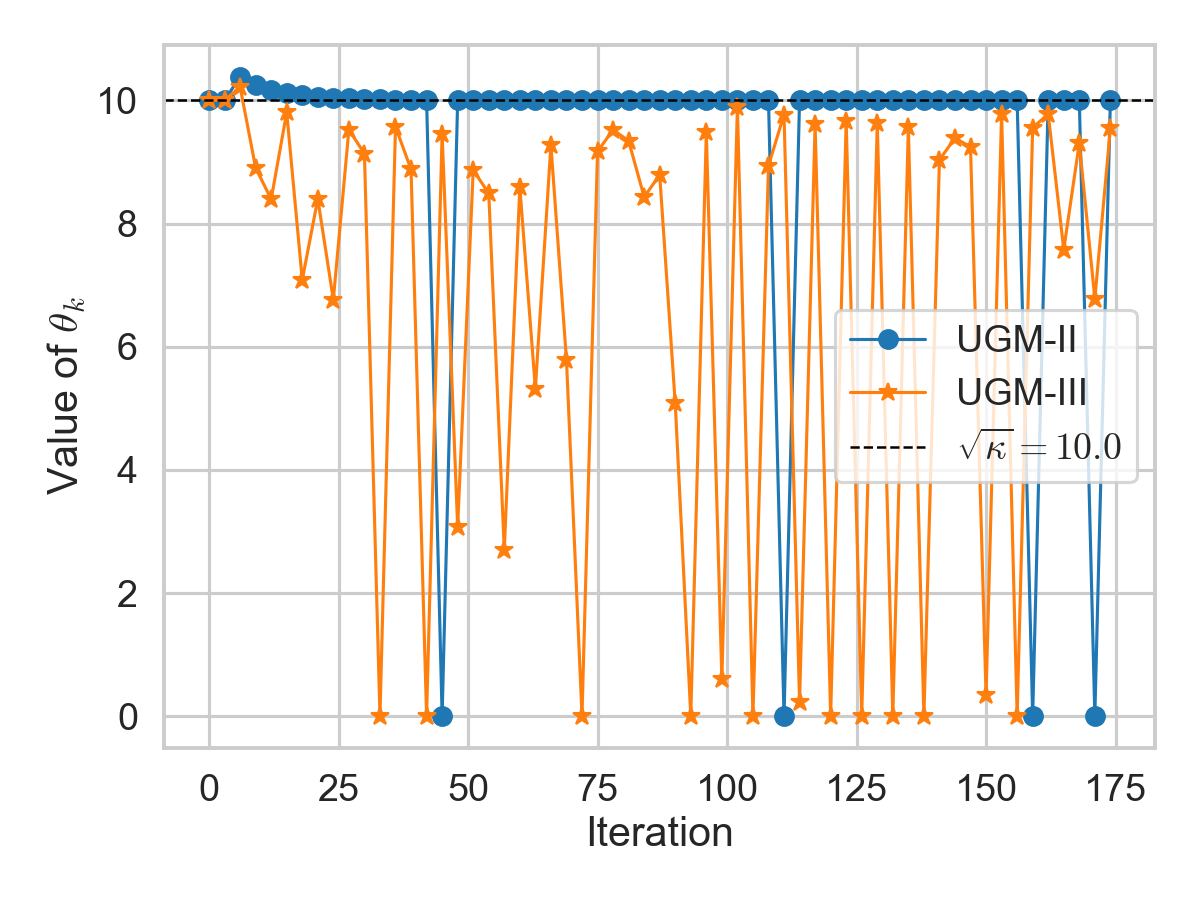}\caption{$\mu=10^2, L=10^4$}\end{subfigure}
\caption{Comparison of FGM, UGM-I, UGM-II, and UGM-III on Case 2. Rows show the objective suboptimality versus iterations, objective suboptimality versus CPU time, and the evolution of $\theta_k$, respectively. Columns correspond to different $(\mu, L)$ configurations.}
\label{fig:case2}
\end{figure}

\begin{figure}[!htbp]
\centering
\begin{subfigure}[b]{0.19\textwidth}\includegraphics[width=\textwidth]{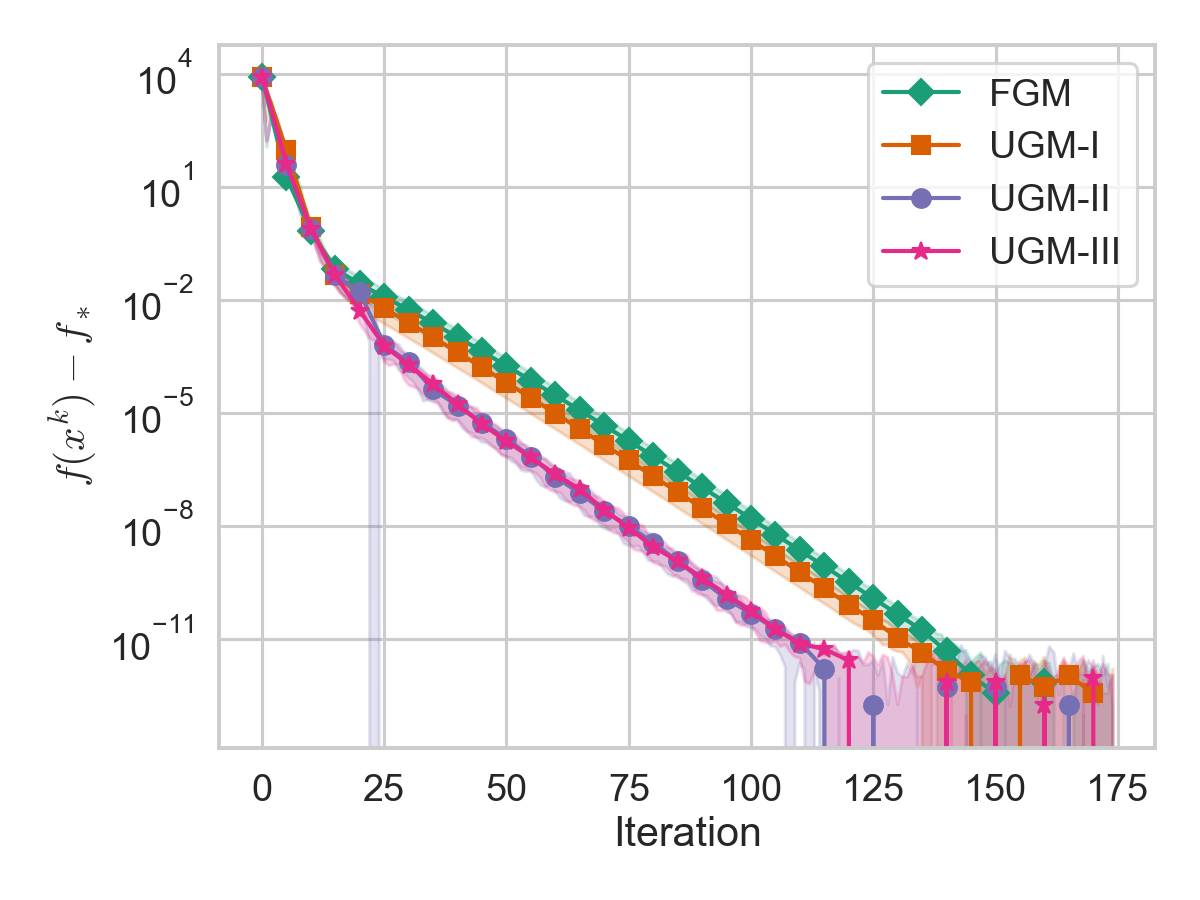}\end{subfigure}\hfill
\begin{subfigure}[b]{0.19\textwidth}\includegraphics[width=\textwidth]{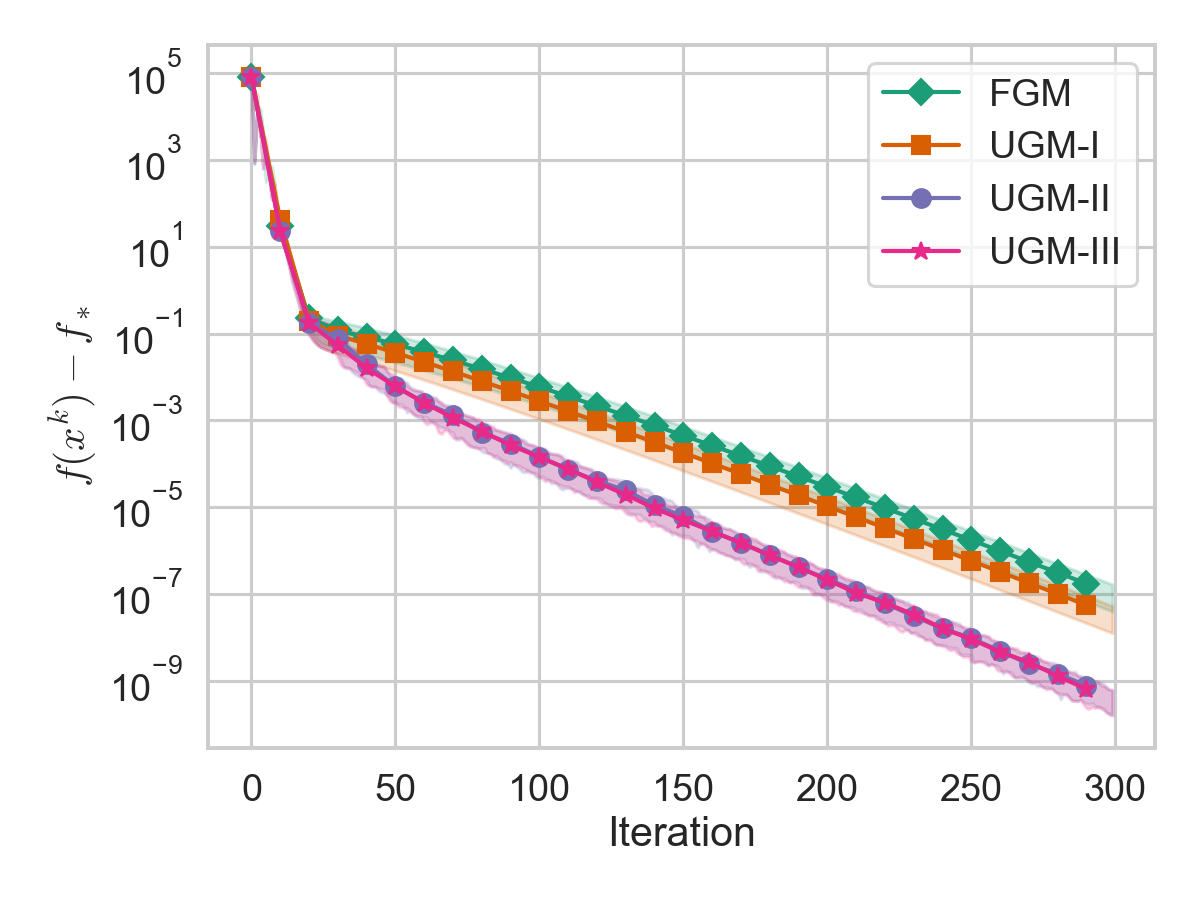}\end{subfigure}\hfill
\begin{subfigure}[b]{0.19\textwidth}\includegraphics[width=\textwidth]{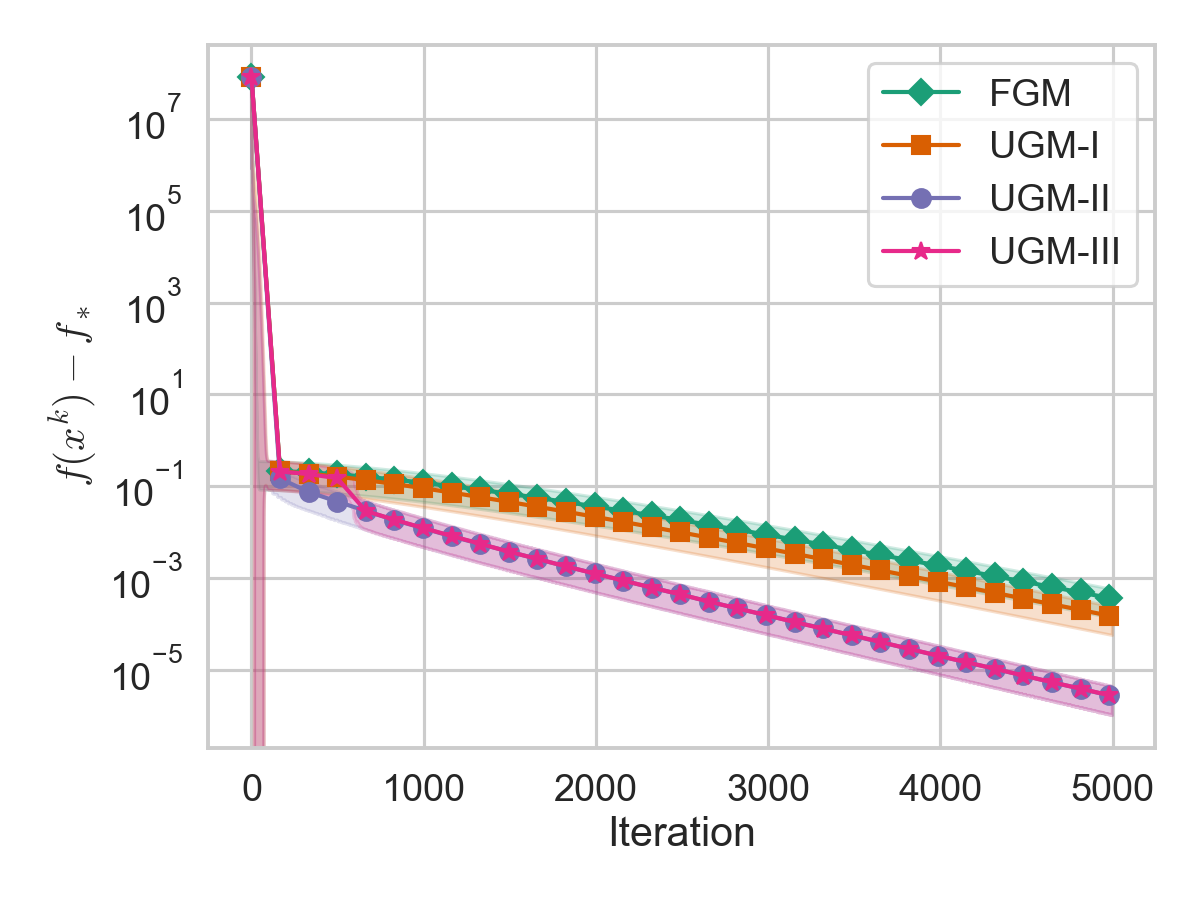}\end{subfigure}\hfill
\begin{subfigure}[b]{0.19\textwidth}\includegraphics[width=\textwidth]{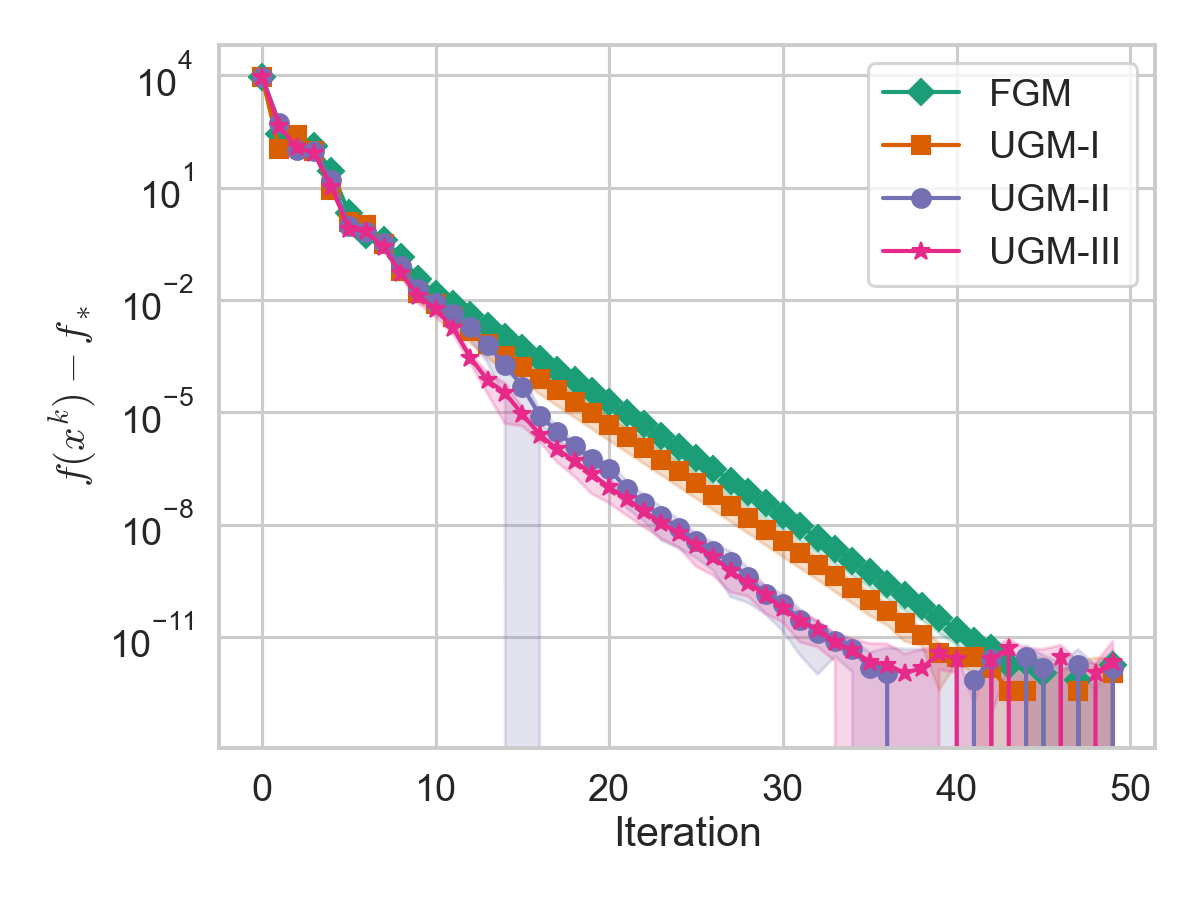}\end{subfigure}\hfill
\begin{subfigure}[b]{0.19\textwidth}\includegraphics[width=\textwidth]{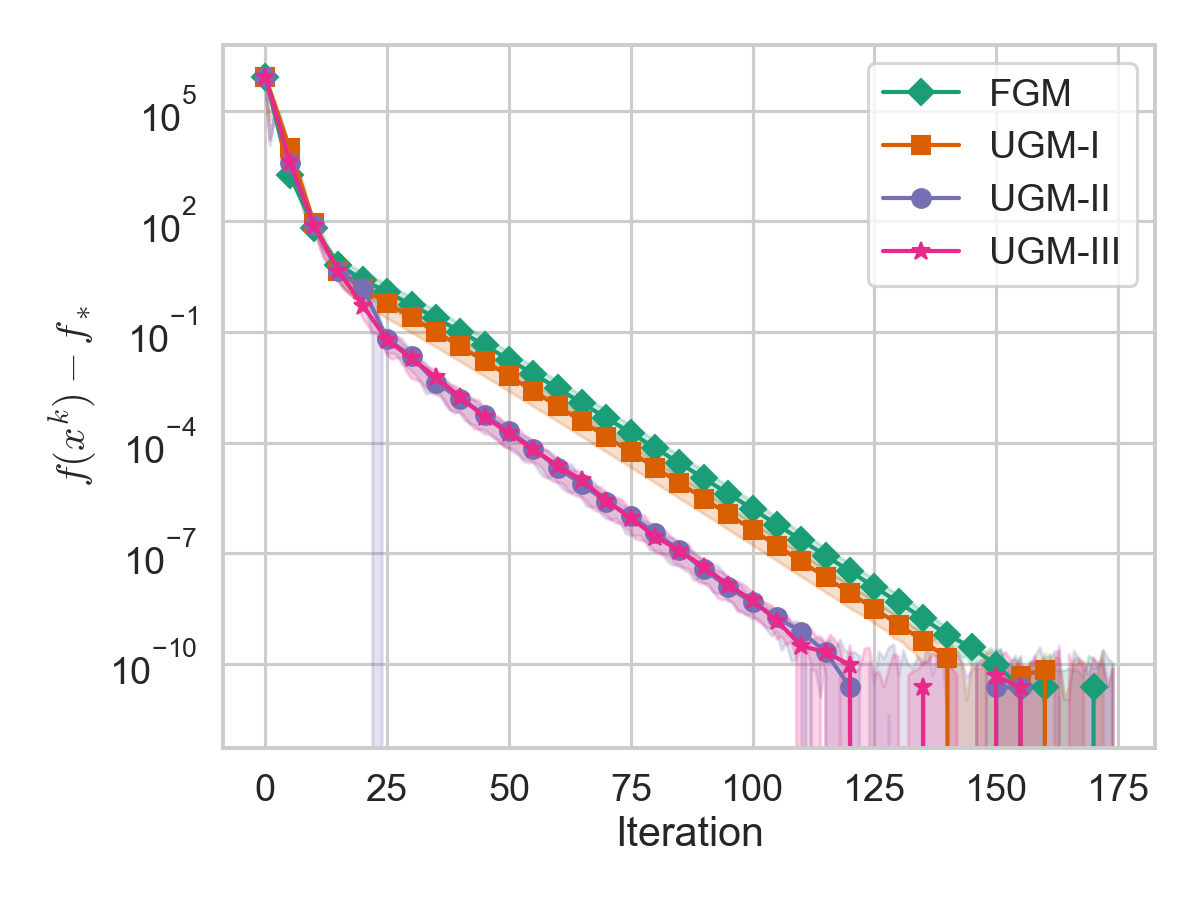}\end{subfigure}
\\
\begin{subfigure}[b]{0.19\textwidth}\includegraphics[width=\textwidth]{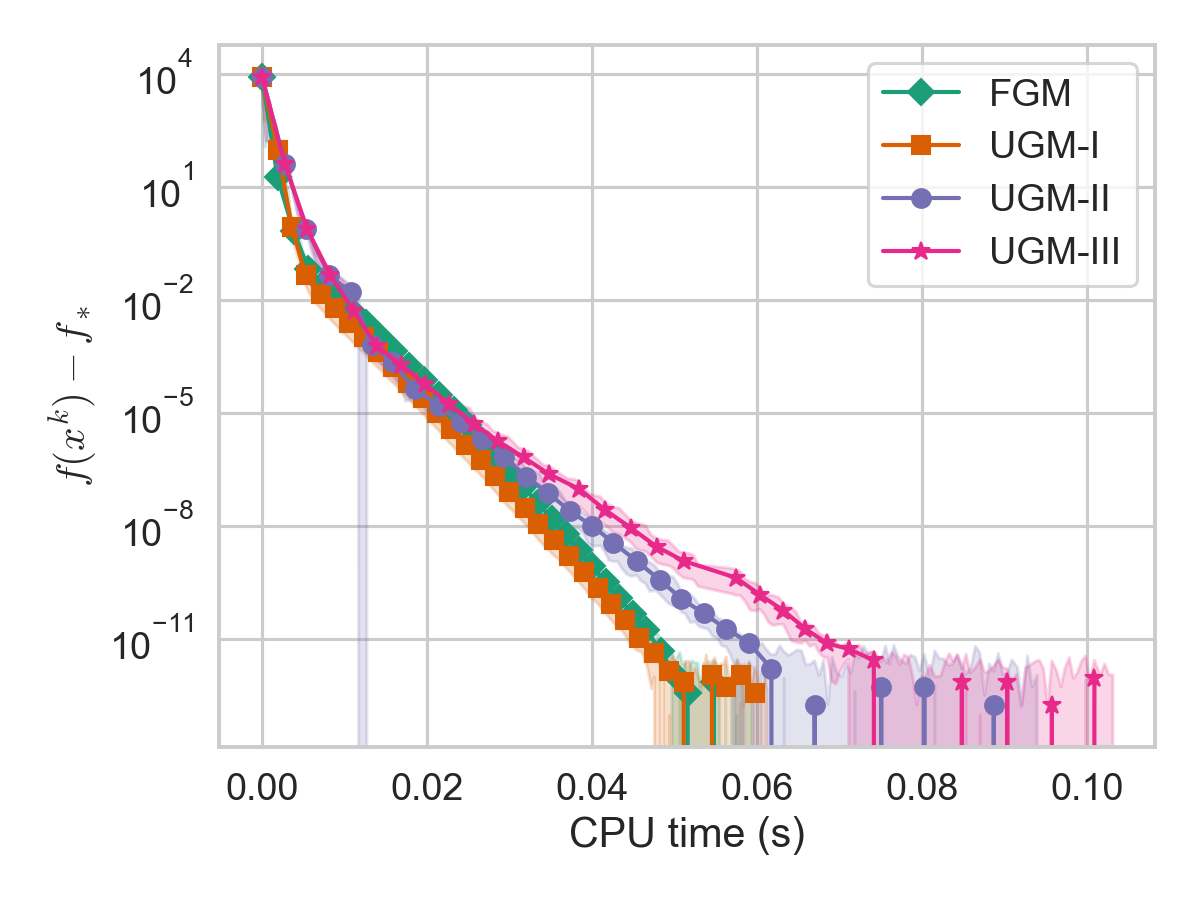}\end{subfigure}\hfill
\begin{subfigure}[b]{0.19\textwidth}\includegraphics[width=\textwidth]{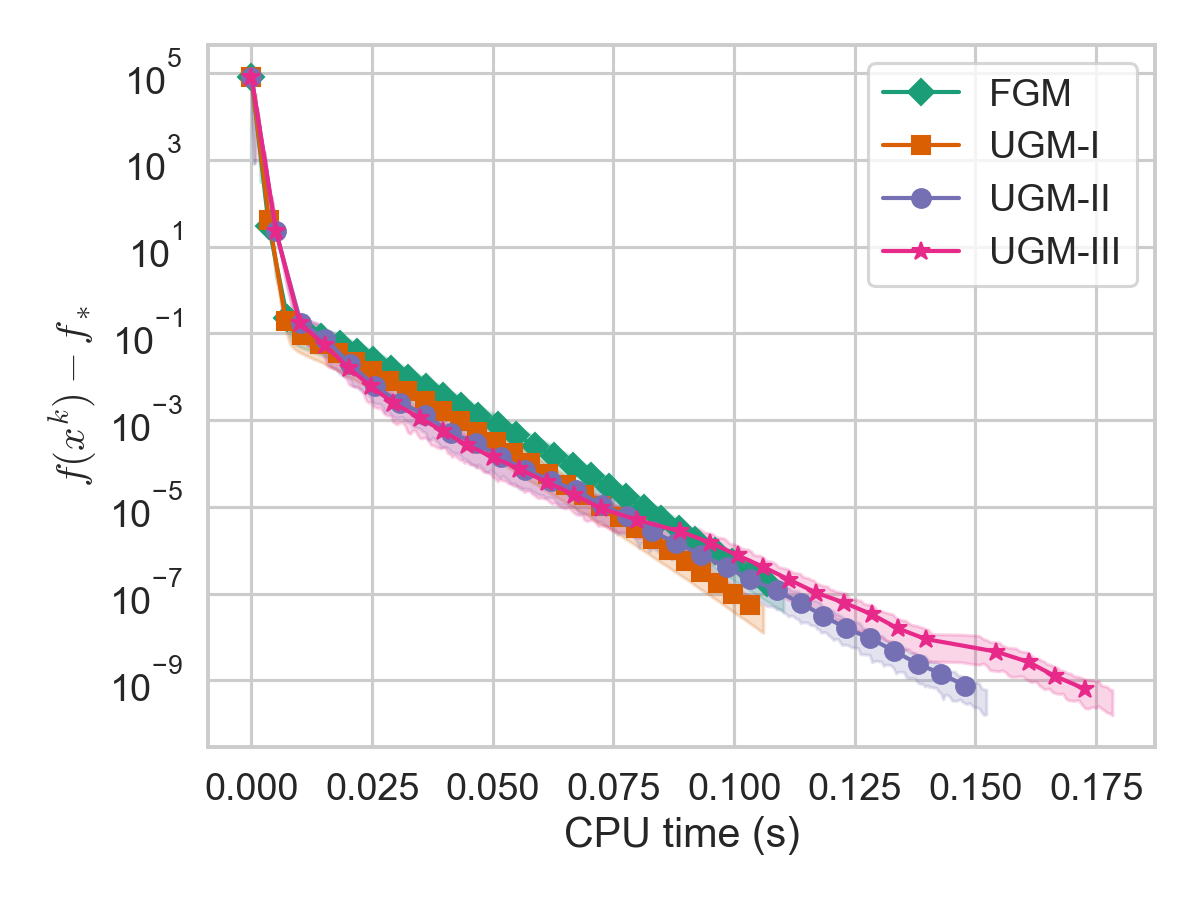}\end{subfigure}\hfill
\begin{subfigure}[b]{0.19\textwidth}\includegraphics[width=\textwidth]{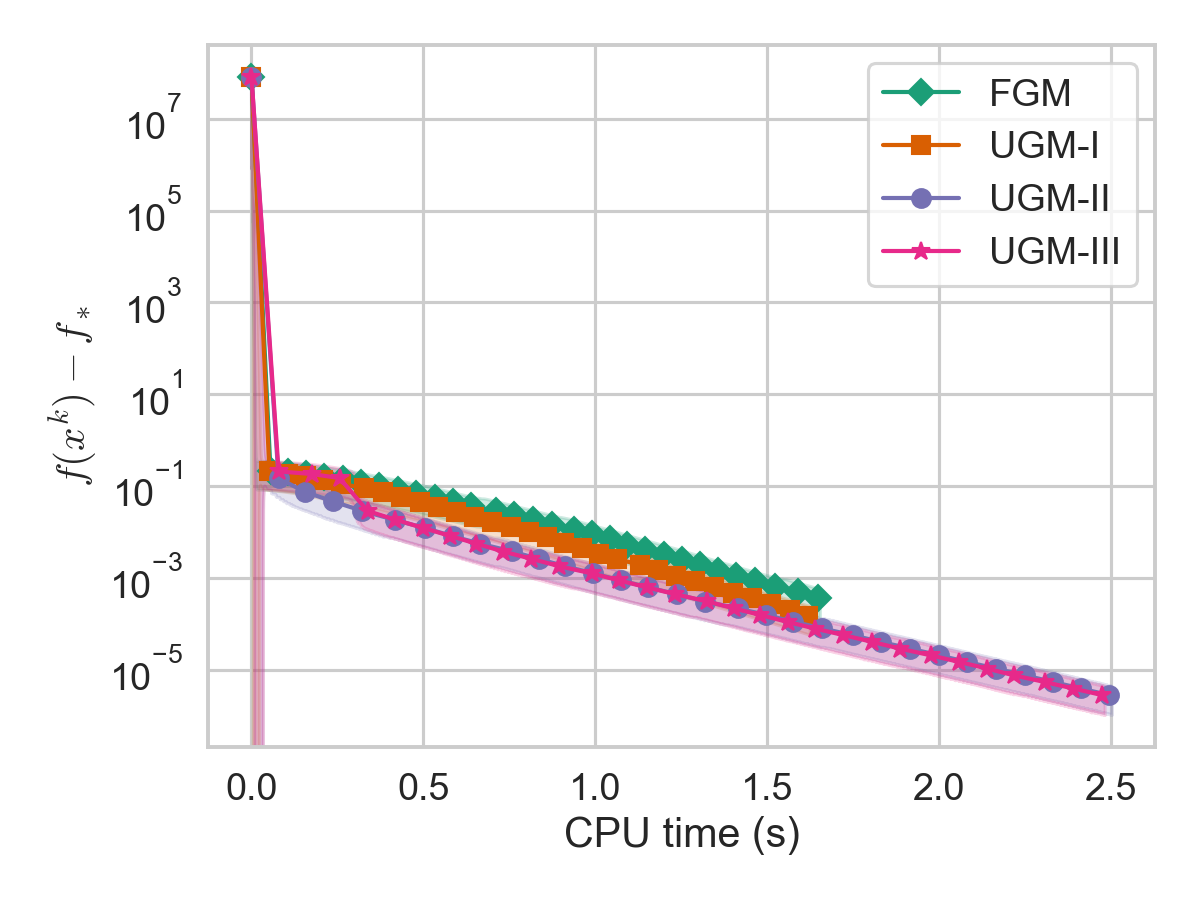}\end{subfigure}\hfill
\begin{subfigure}[b]{0.19\textwidth}\includegraphics[width=\textwidth]{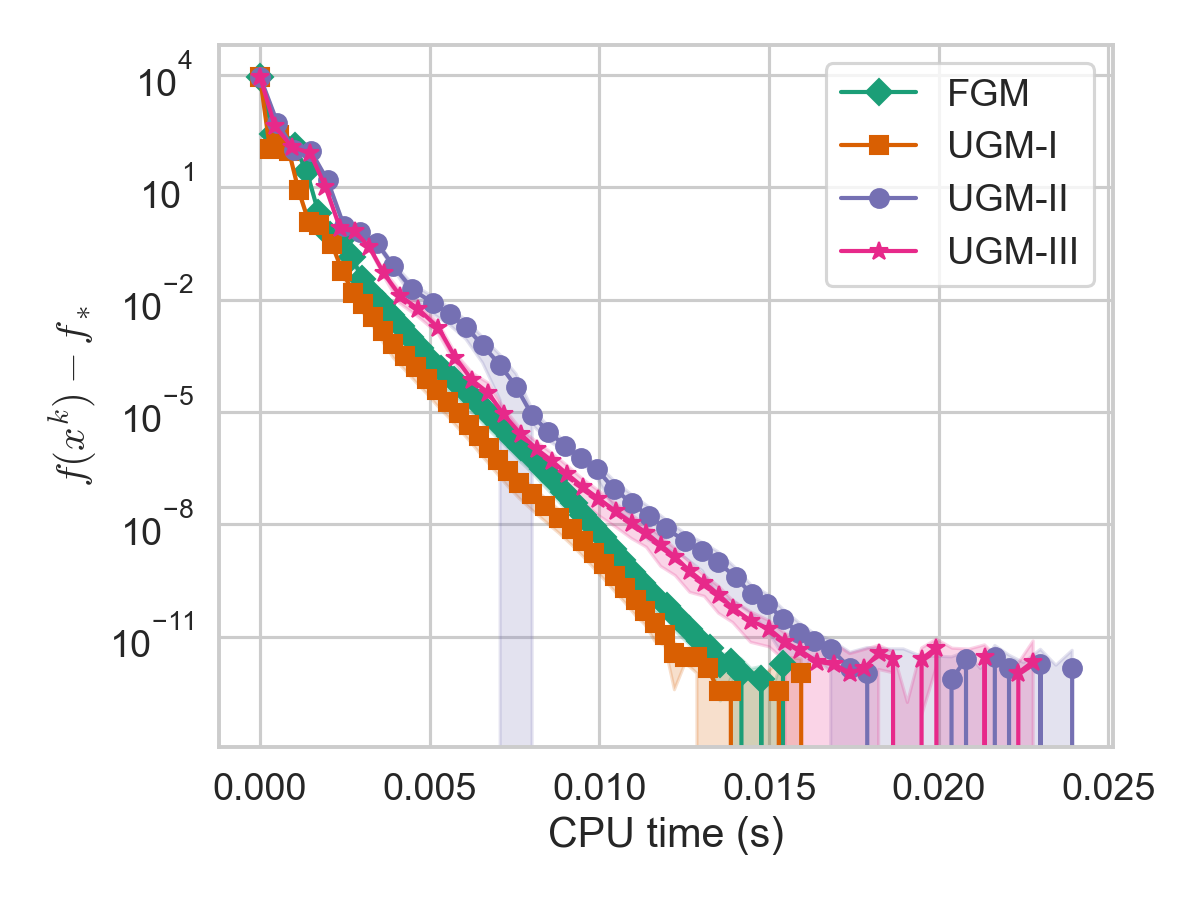}\end{subfigure}\hfill
\begin{subfigure}[b]{0.19\textwidth}\includegraphics[width=\textwidth]{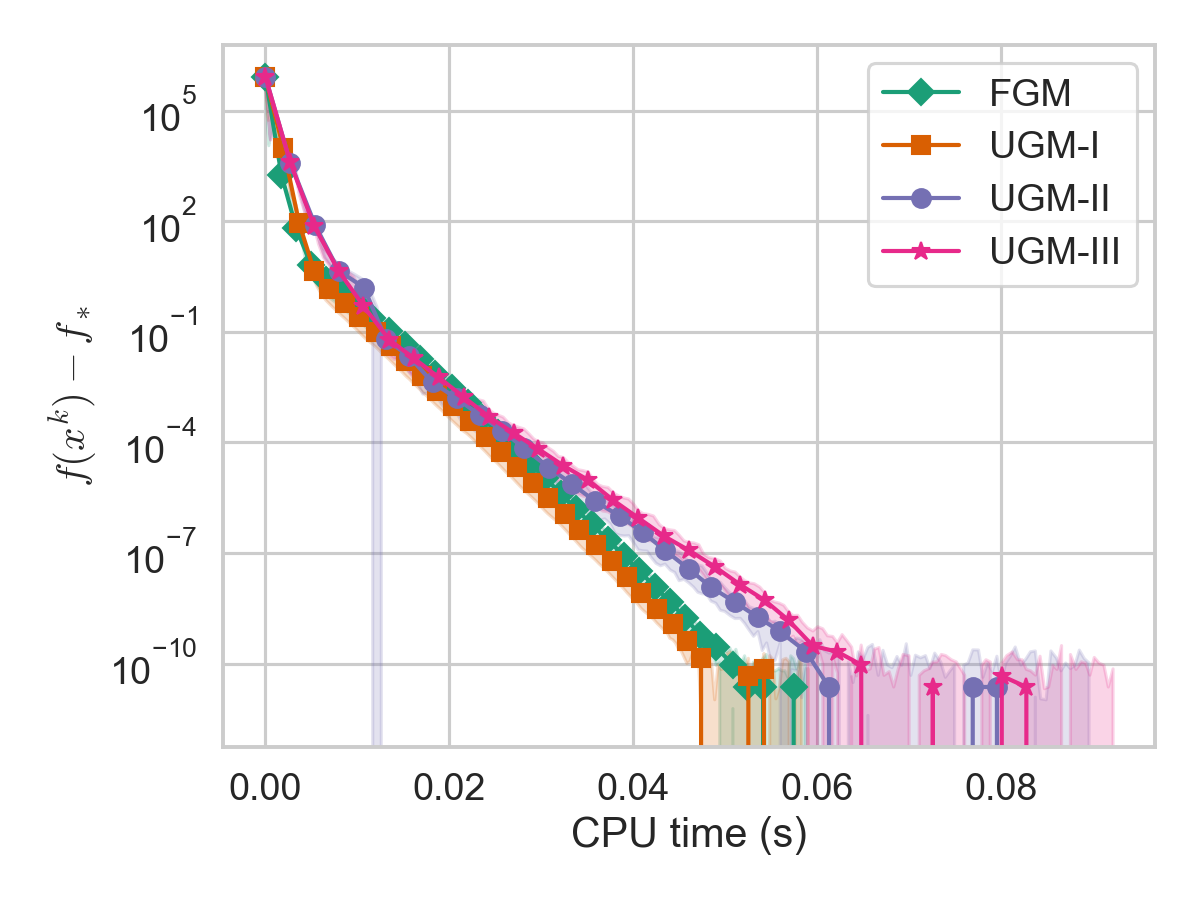}\end{subfigure}
\\
\begin{subfigure}[b]{0.19\textwidth}\includegraphics[width=\textwidth]{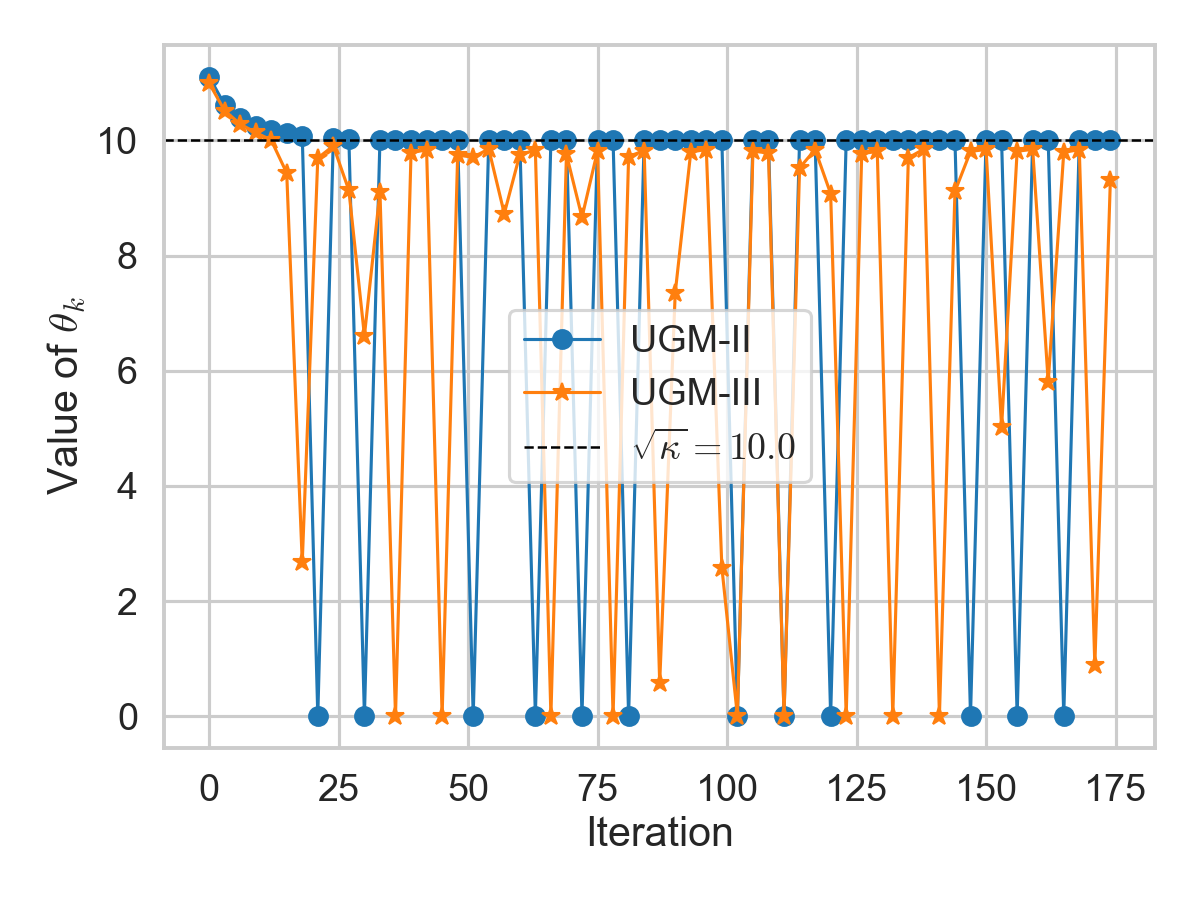}\caption{$\mu=1, L=10^2$}\end{subfigure}\hfill
\begin{subfigure}[b]{0.19\textwidth}\includegraphics[width=\textwidth]{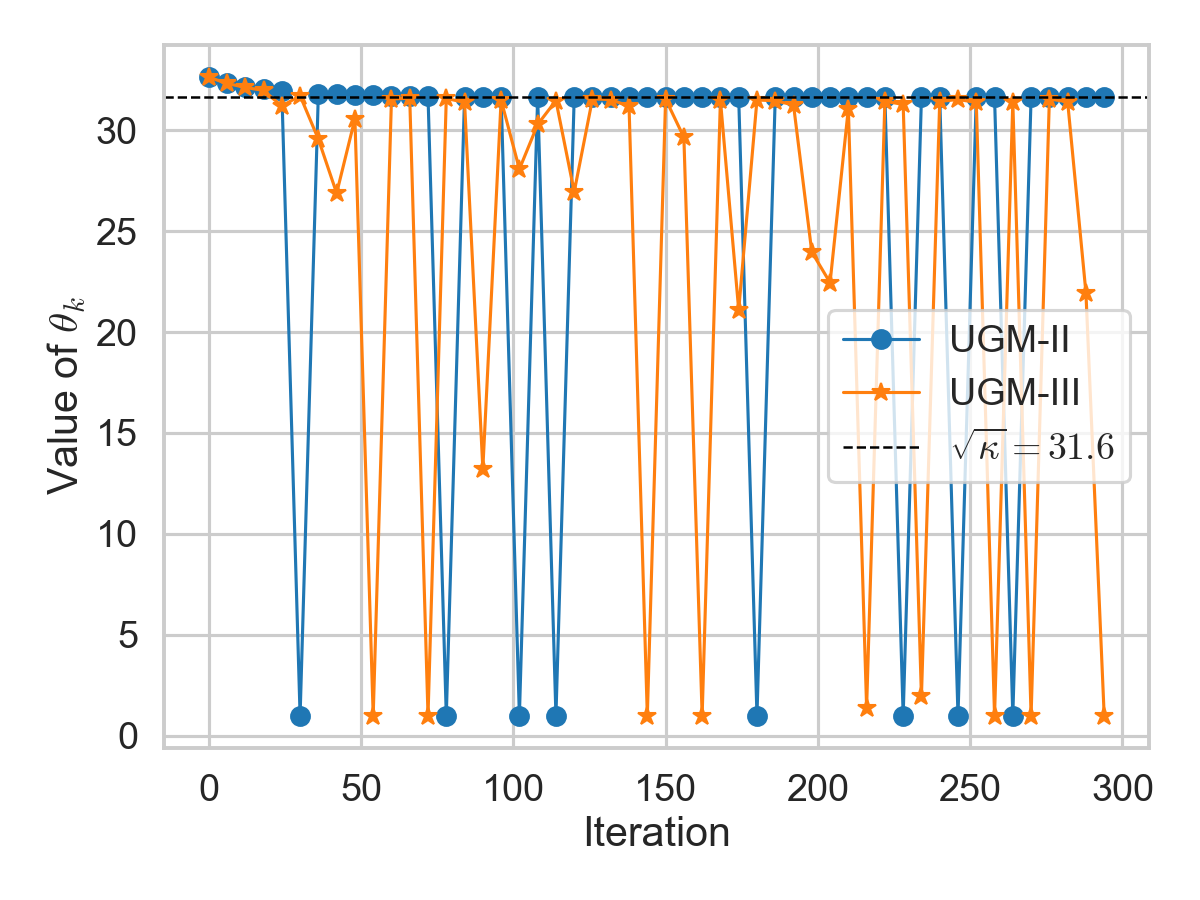}\caption{$\mu=1, L=10^3$}\end{subfigure}\hfill
\begin{subfigure}[b]{0.19\textwidth}\includegraphics[width=\textwidth]{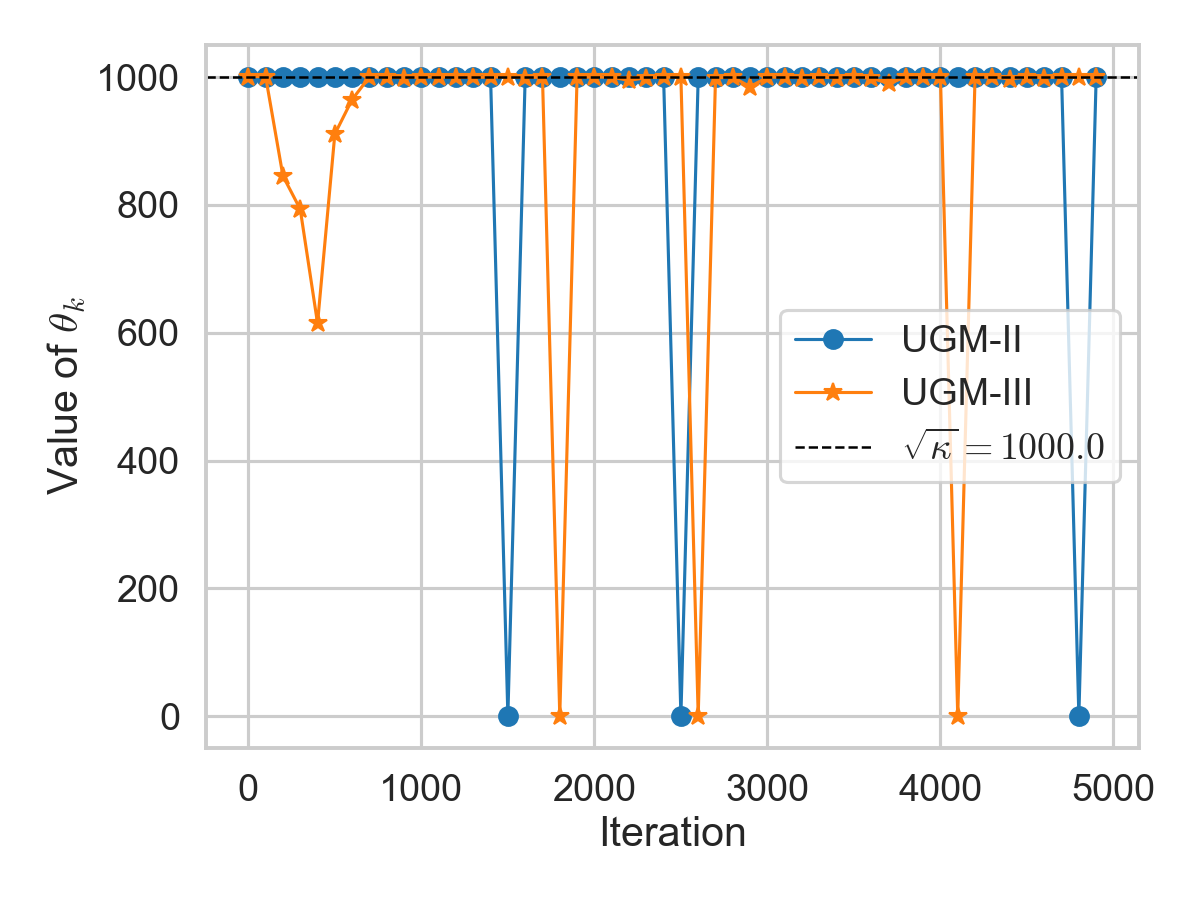}\caption{$\mu=1, L=10^6$}\end{subfigure}\hfill
\begin{subfigure}[b]{0.19\textwidth}\includegraphics[width=\textwidth]{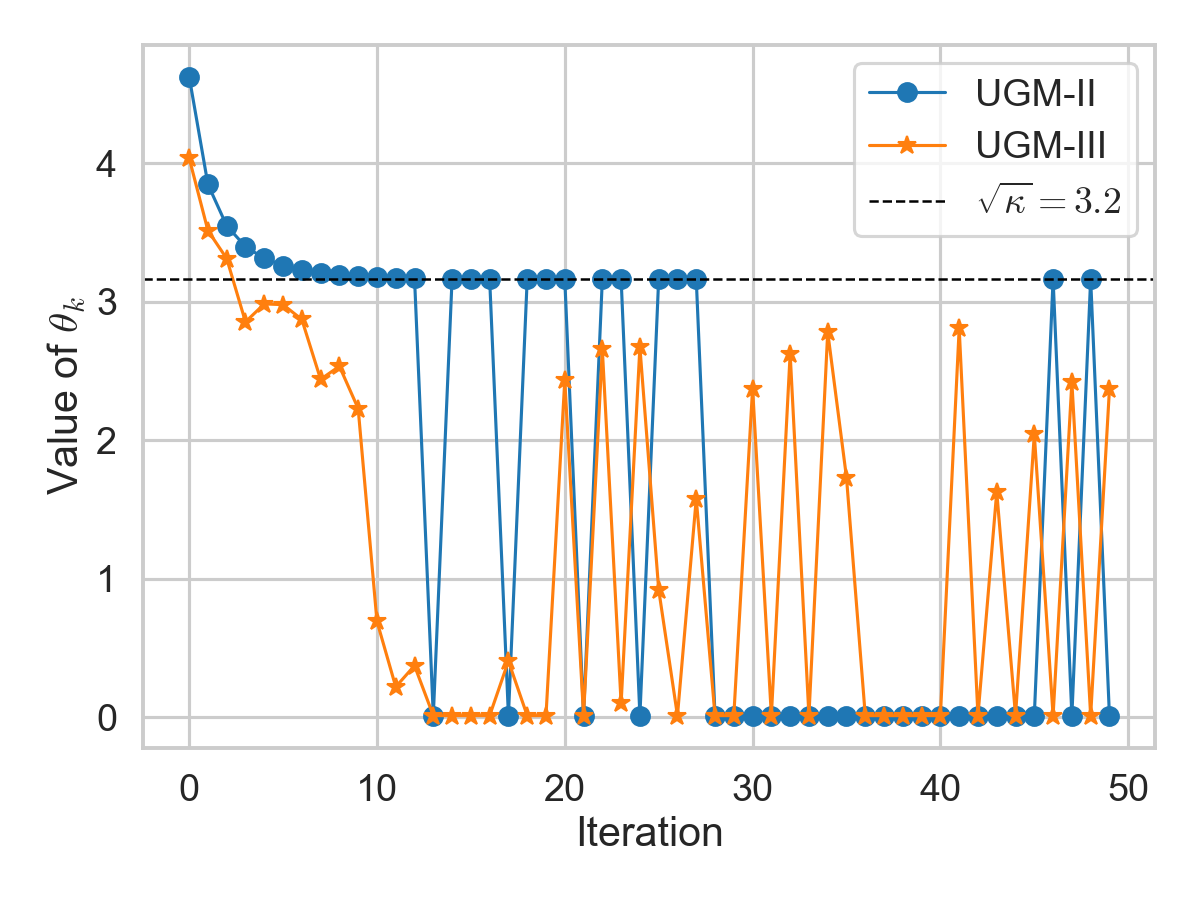}\caption{$\mu=10, L=10^2$}\end{subfigure}\hfill
\begin{subfigure}[b]{0.19\textwidth}\includegraphics[width=\textwidth]{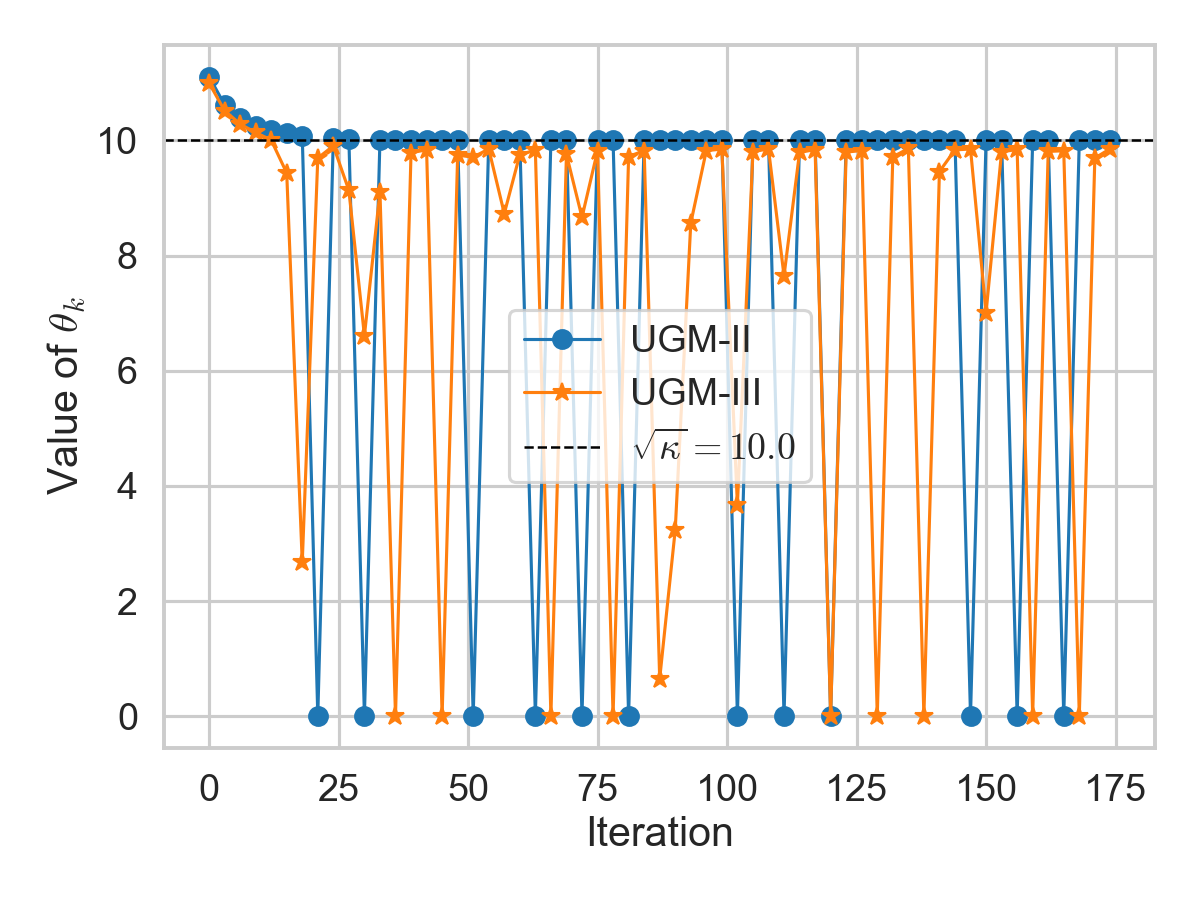}\caption{$\mu=10^2, L=10^4$}\end{subfigure}
\caption{Comparison of FGM, UGM-I, UGM-II, and UGM-III on Case 3. Rows show the objective suboptimality versus iterations, objective suboptimality versus CPU time, and the evolution of $\theta_k$, respectively. Columns correspond to different $(\mu, L)$ configurations.}
\label{fig:case3}
\end{figure}

\begin{figure}[!htbp]
\centering
\begin{subfigure}[b]{0.19\textwidth}\includegraphics[width=\textwidth]{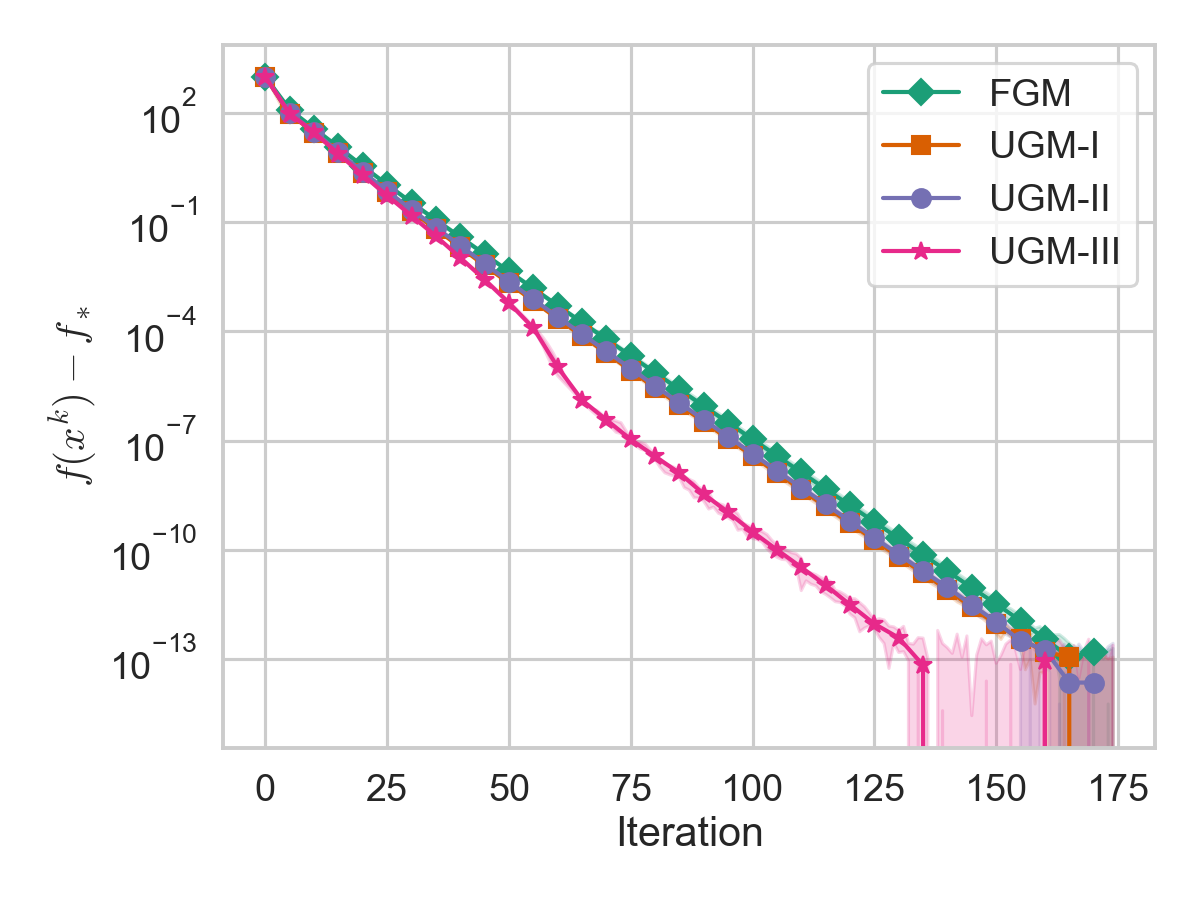}\end{subfigure}\hfill
\begin{subfigure}[b]{0.19\textwidth}\includegraphics[width=\textwidth]{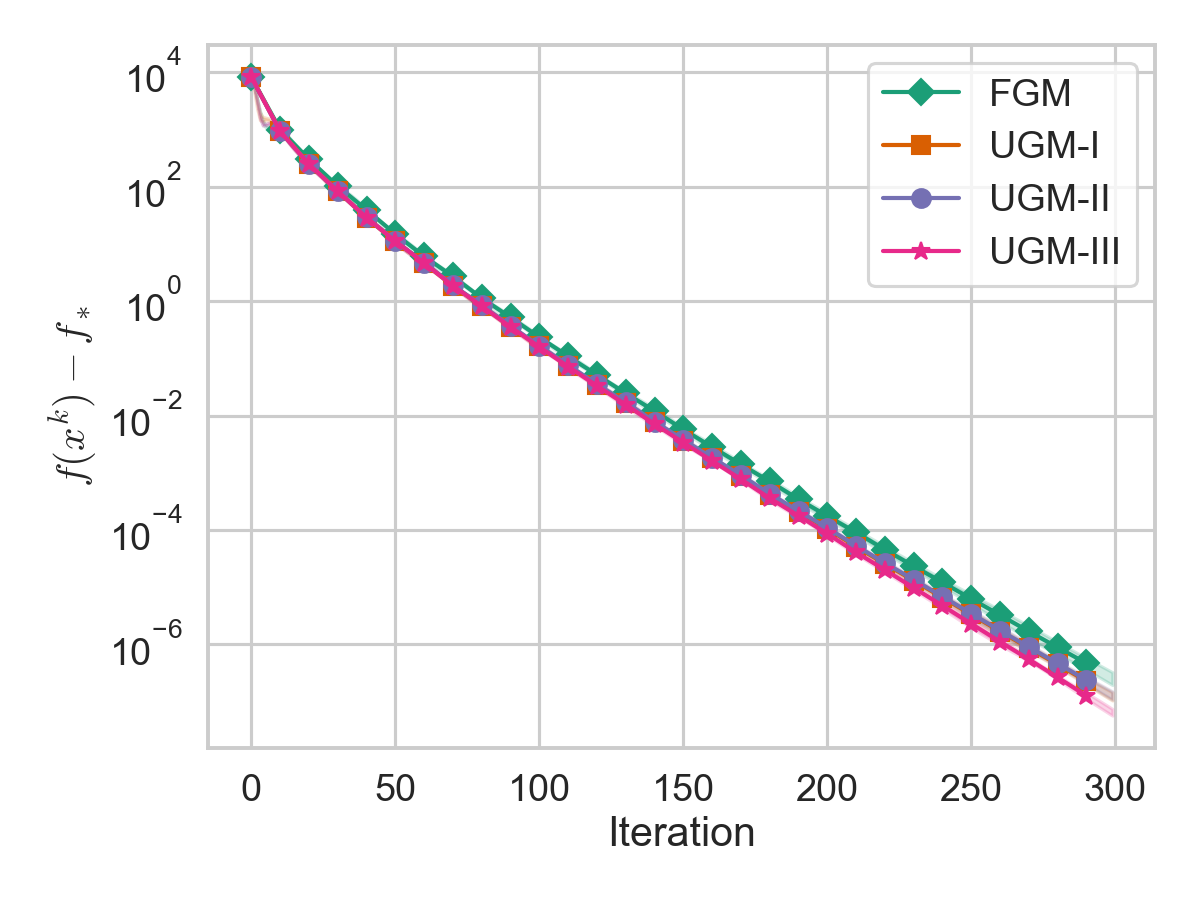}\end{subfigure}\hfill
\begin{subfigure}[b]{0.19\textwidth}\includegraphics[width=\textwidth]{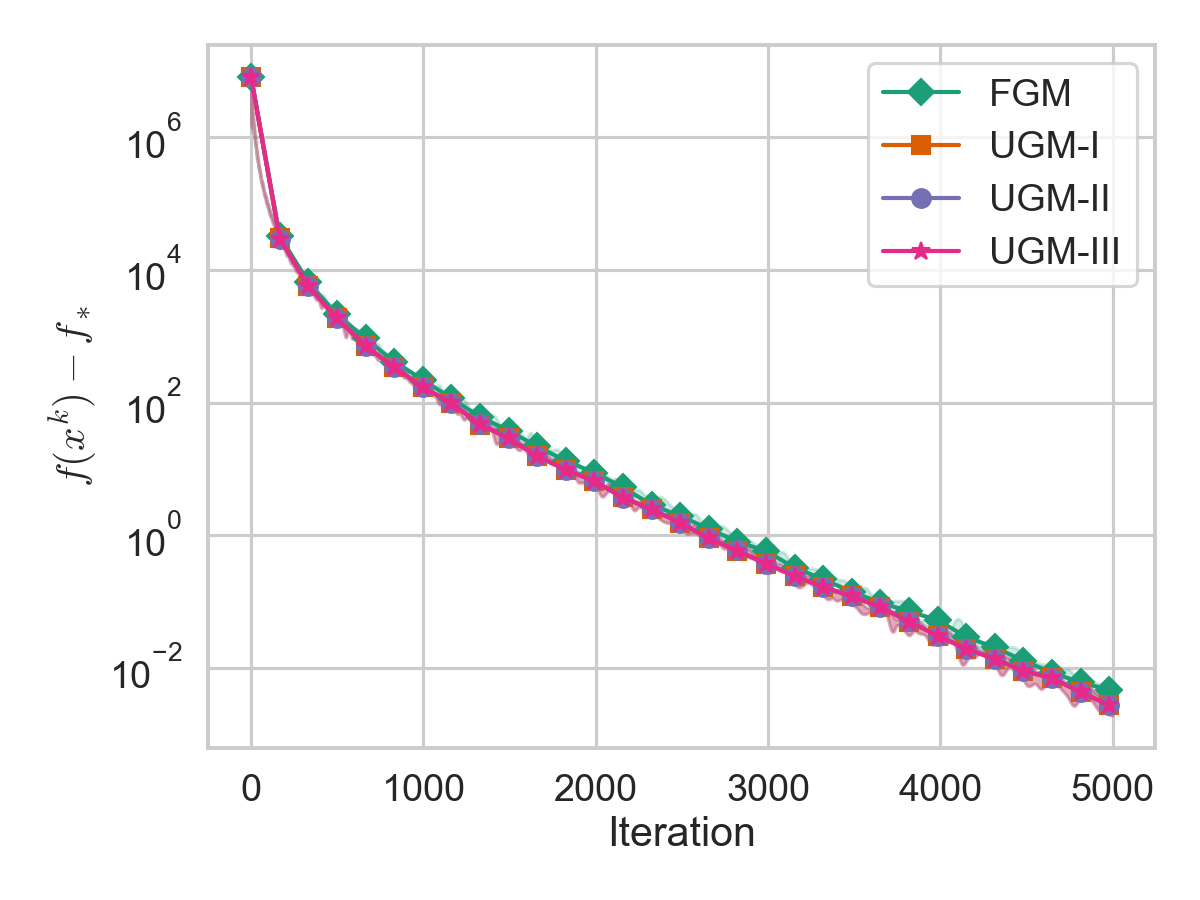}\end{subfigure}\hfill
\begin{subfigure}[b]{0.19\textwidth}\includegraphics[width=\textwidth]{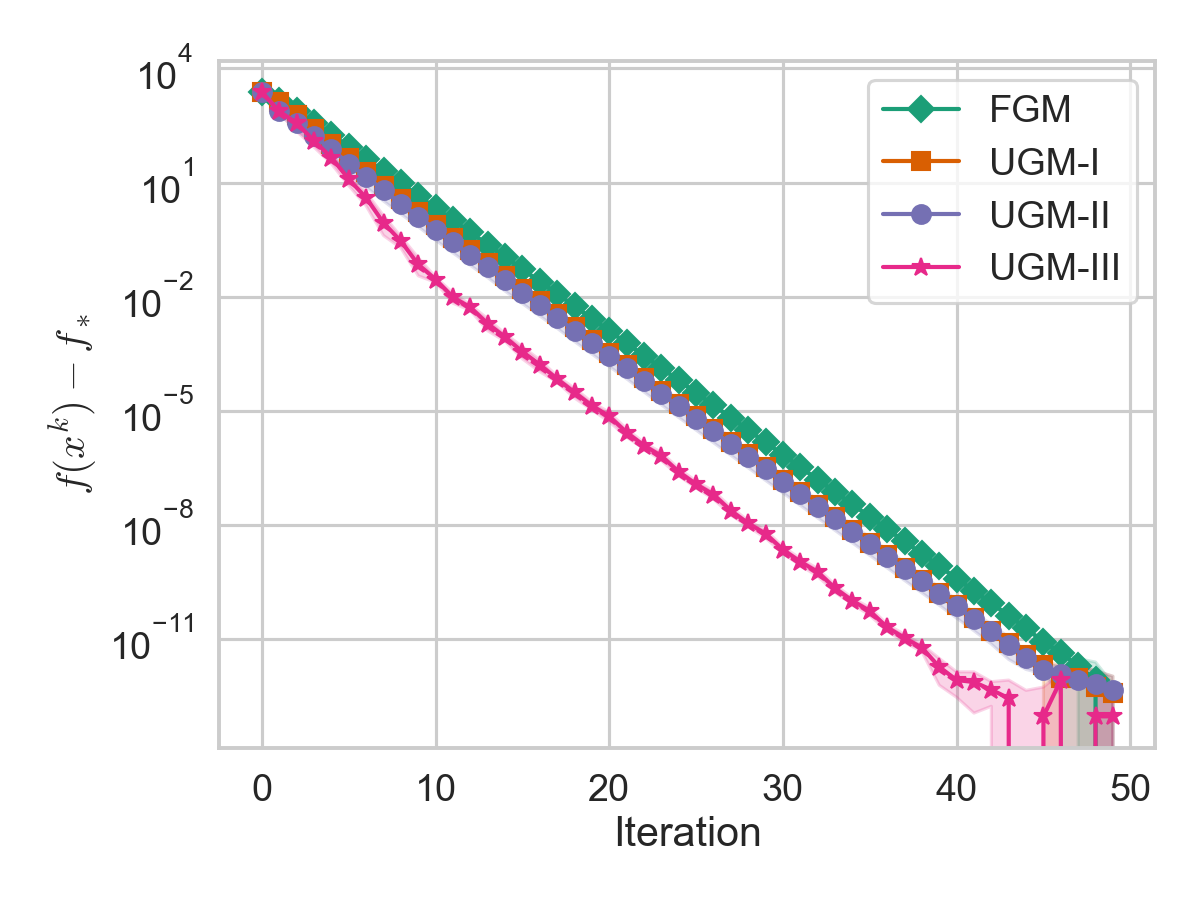}\end{subfigure}\hfill
\begin{subfigure}[b]{0.19\textwidth}\includegraphics[width=\textwidth]{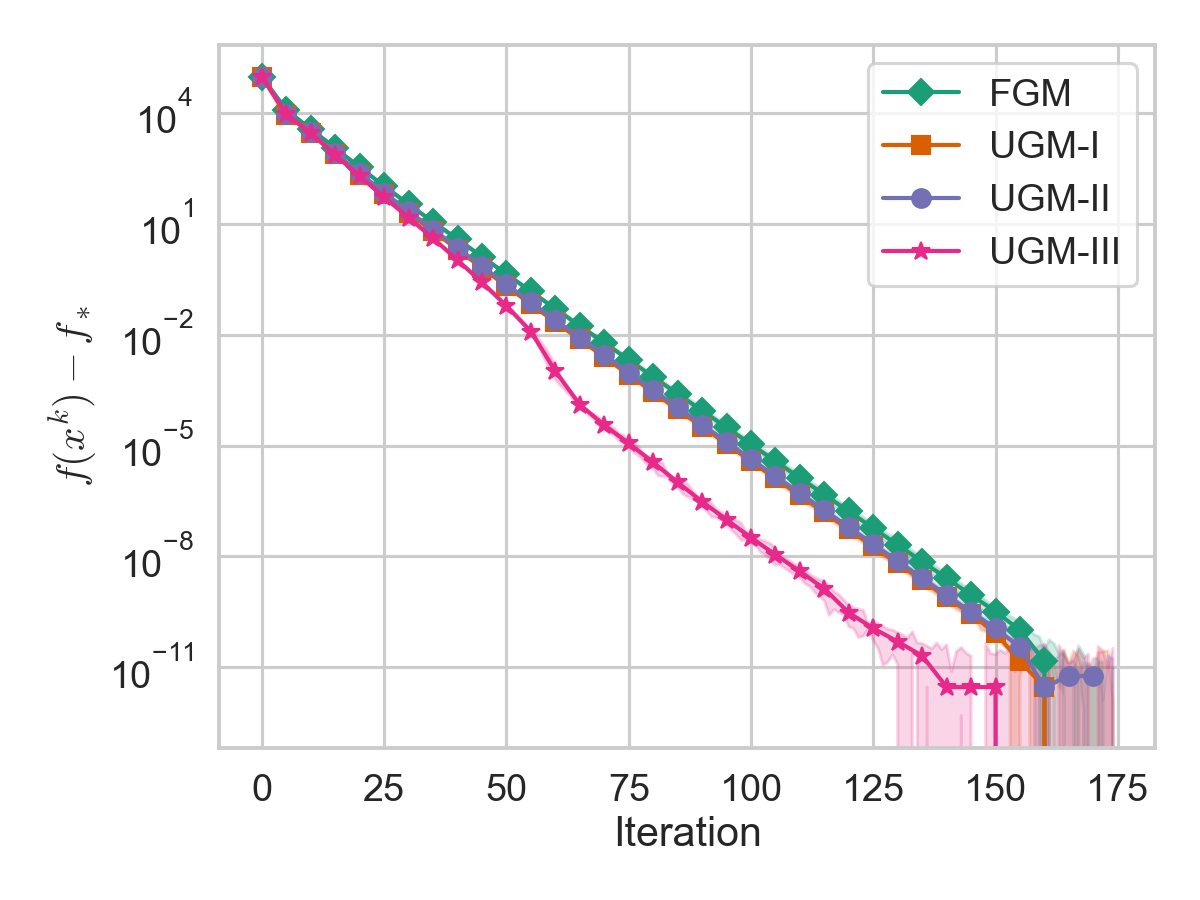}\end{subfigure}
\\
\begin{subfigure}[b]{0.19\textwidth}\includegraphics[width=\textwidth]{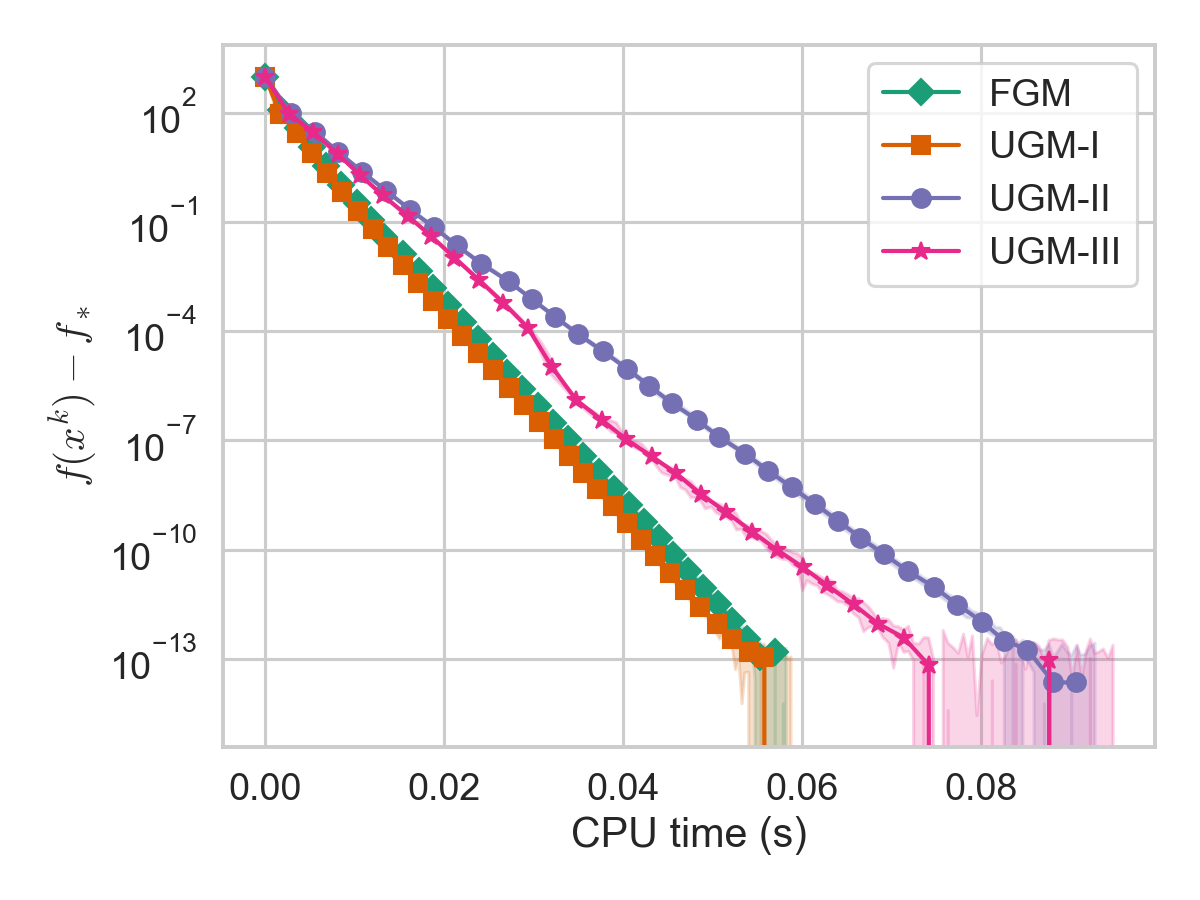}\end{subfigure}\hfill
\begin{subfigure}[b]{0.19\textwidth}\includegraphics[width=\textwidth]{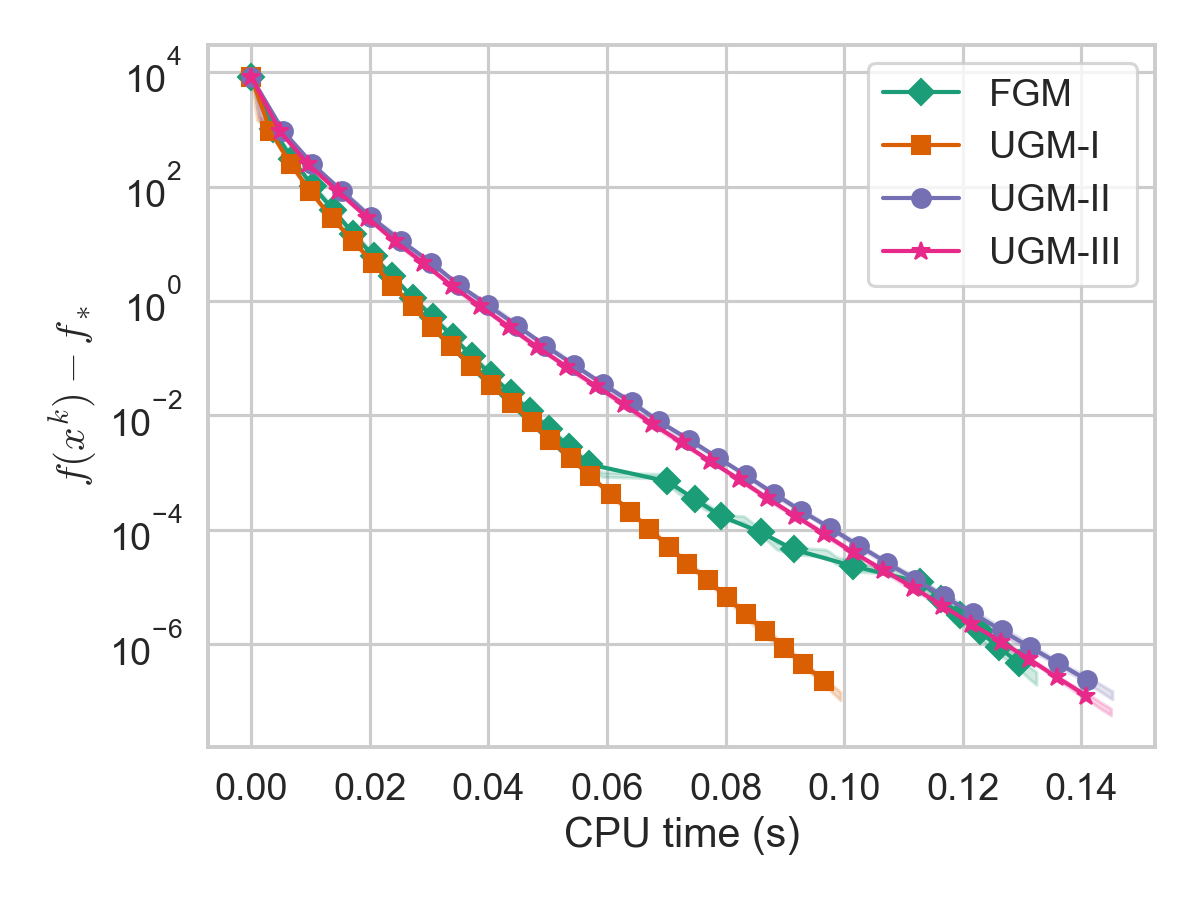}\end{subfigure}\hfill
\begin{subfigure}[b]{0.19\textwidth}\includegraphics[width=\textwidth]{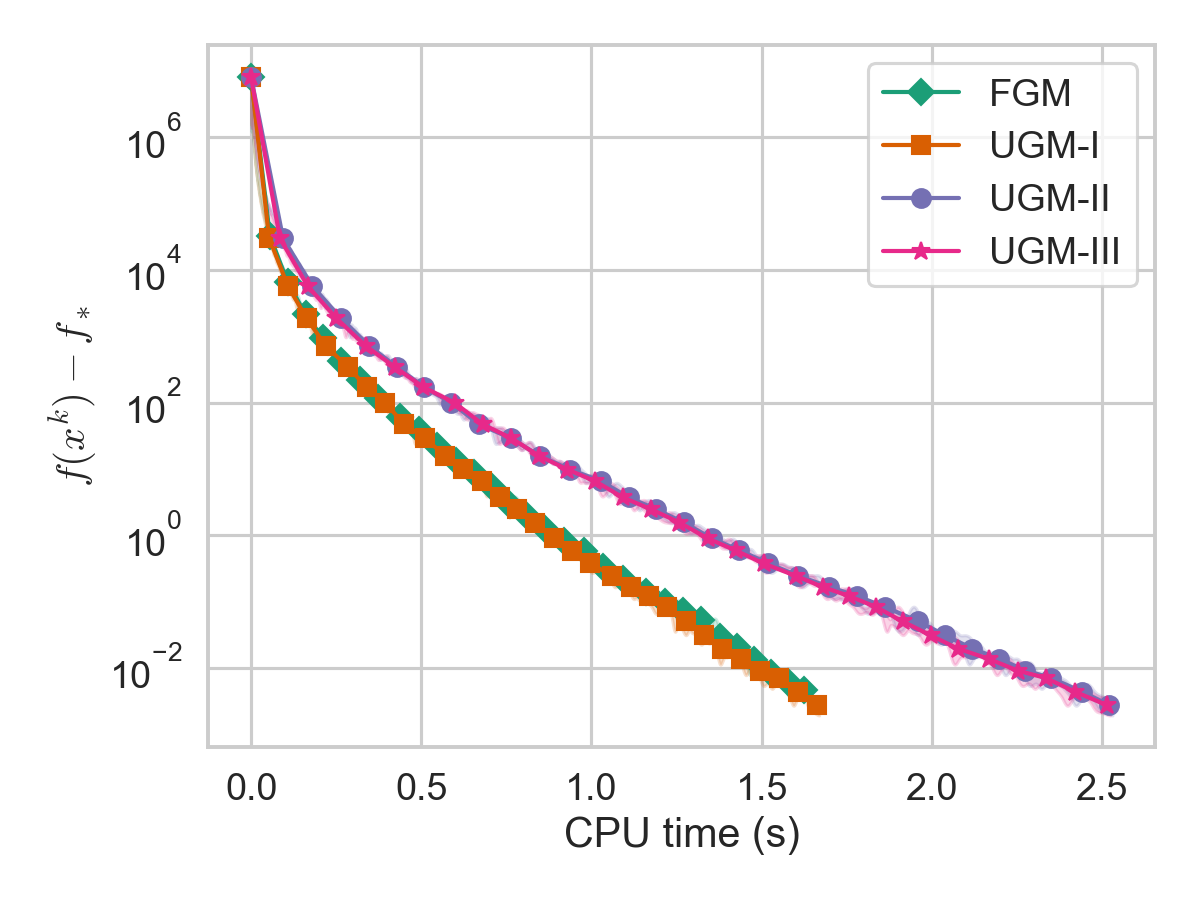}\end{subfigure}\hfill
\begin{subfigure}[b]{0.19\textwidth}\includegraphics[width=\textwidth]{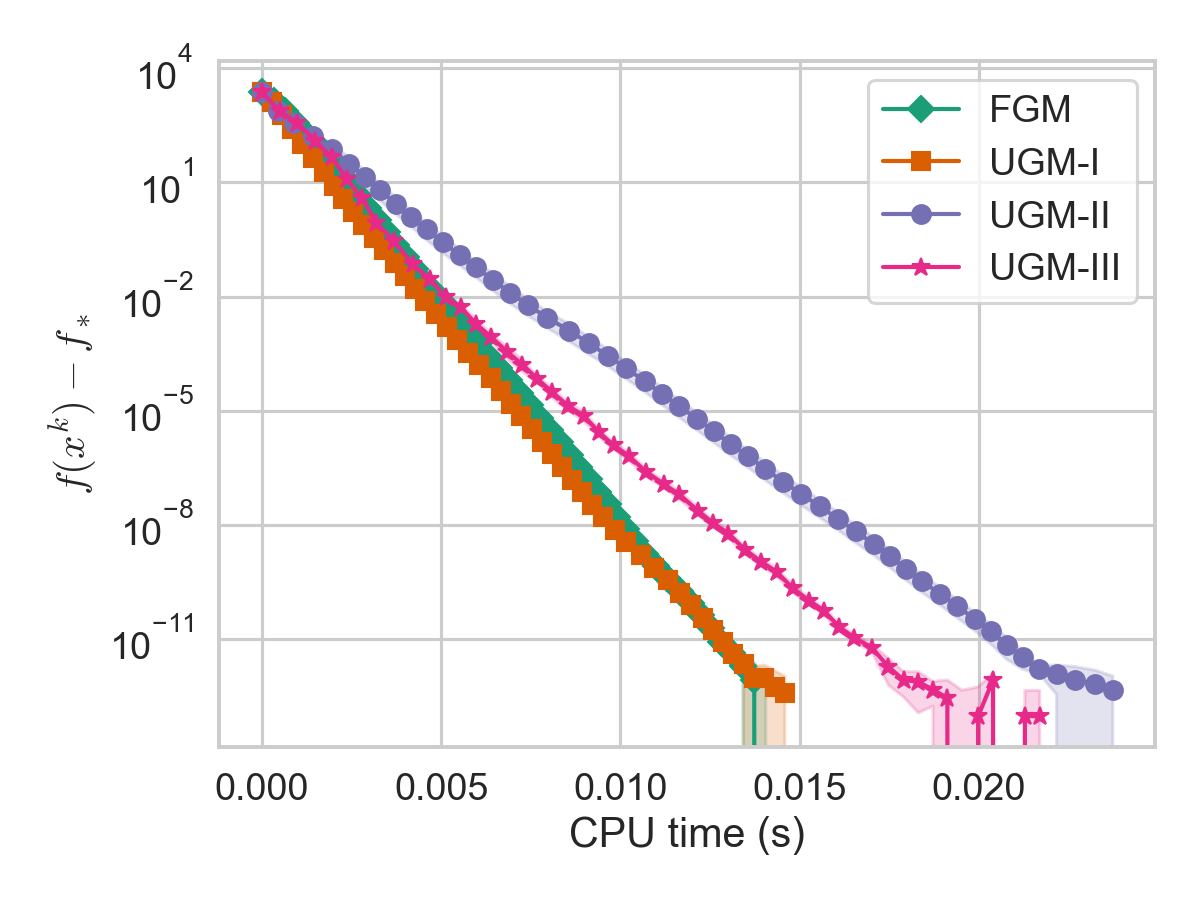}\end{subfigure}\hfill
\begin{subfigure}[b]{0.19\textwidth}\includegraphics[width=\textwidth]{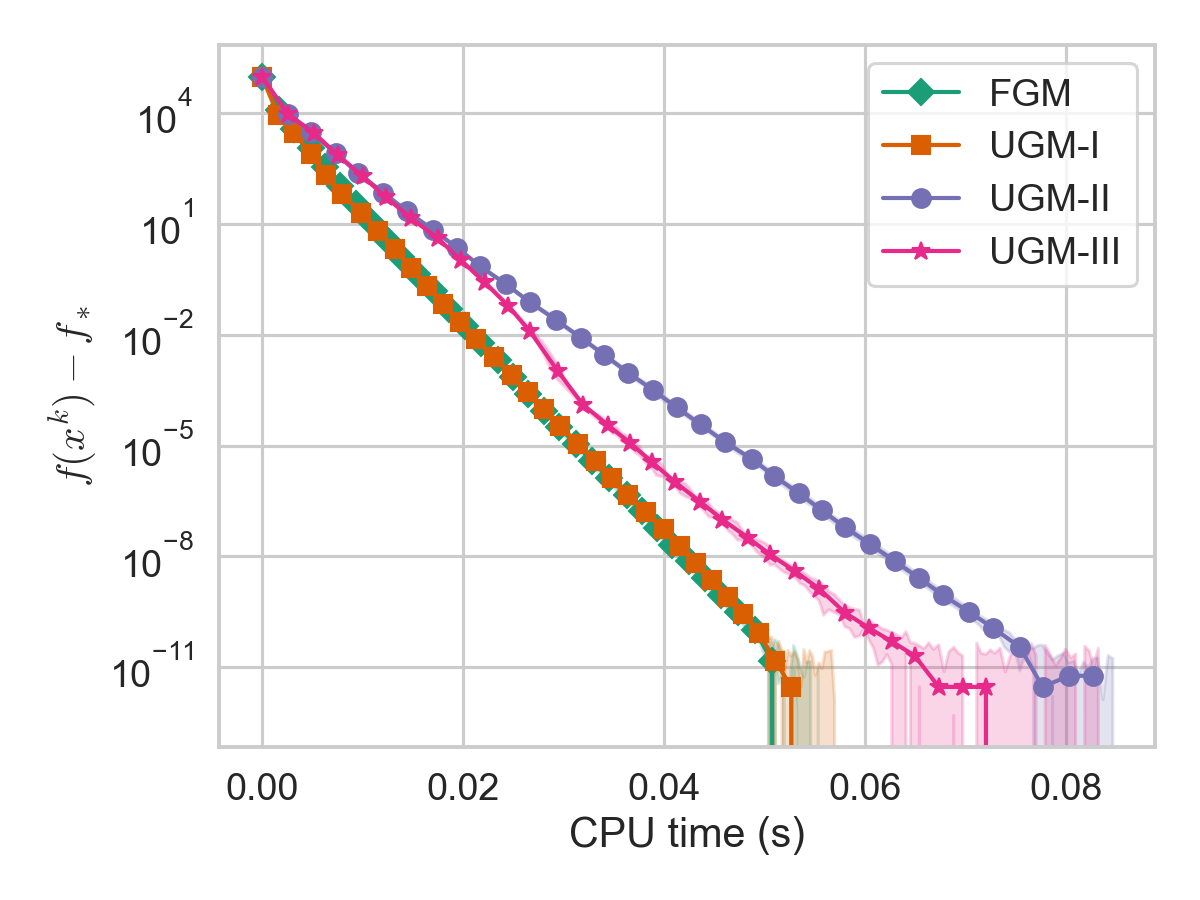}\end{subfigure}
\\
\begin{subfigure}[b]{0.19\textwidth}\includegraphics[width=\textwidth]{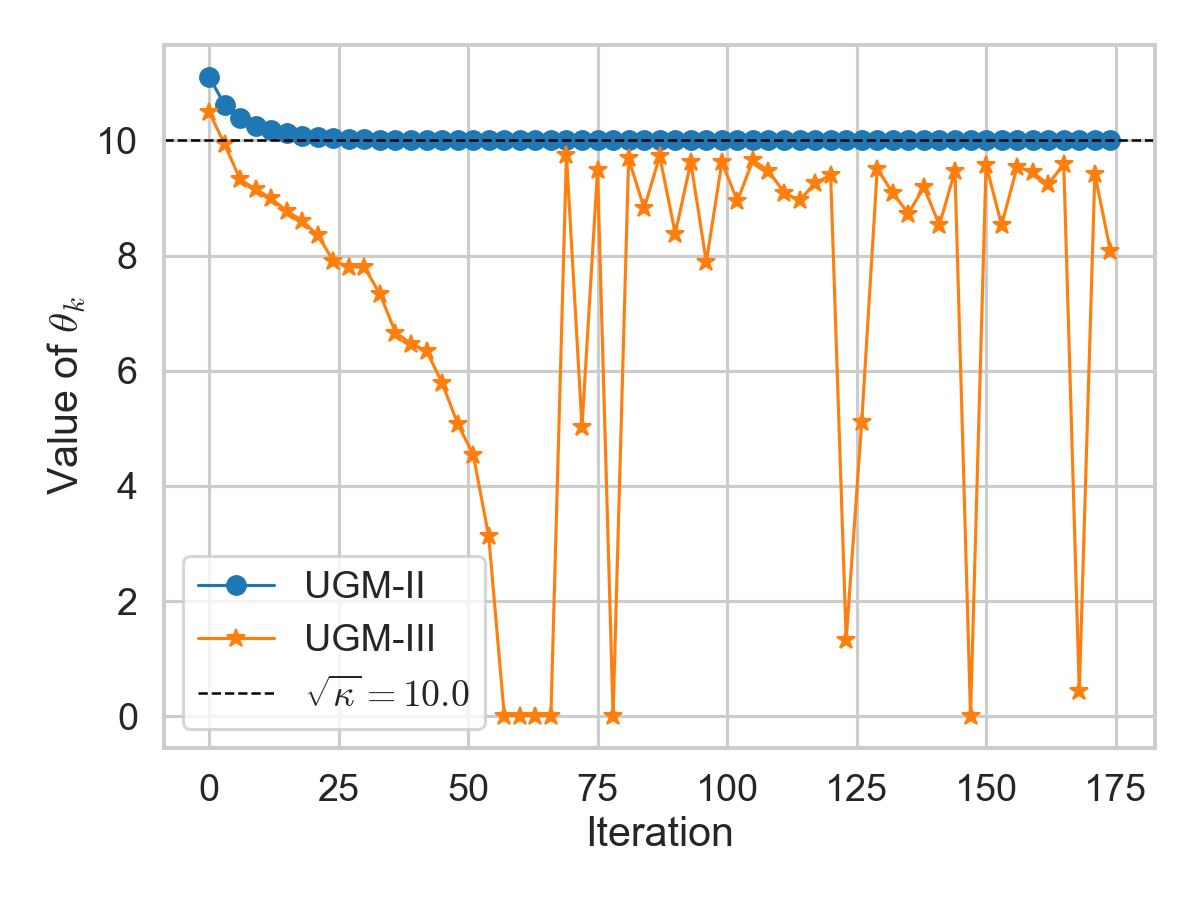}\caption{$\mu=1, L=10^2$}\end{subfigure}\hfill
\begin{subfigure}[b]{0.19\textwidth}\includegraphics[width=\textwidth]{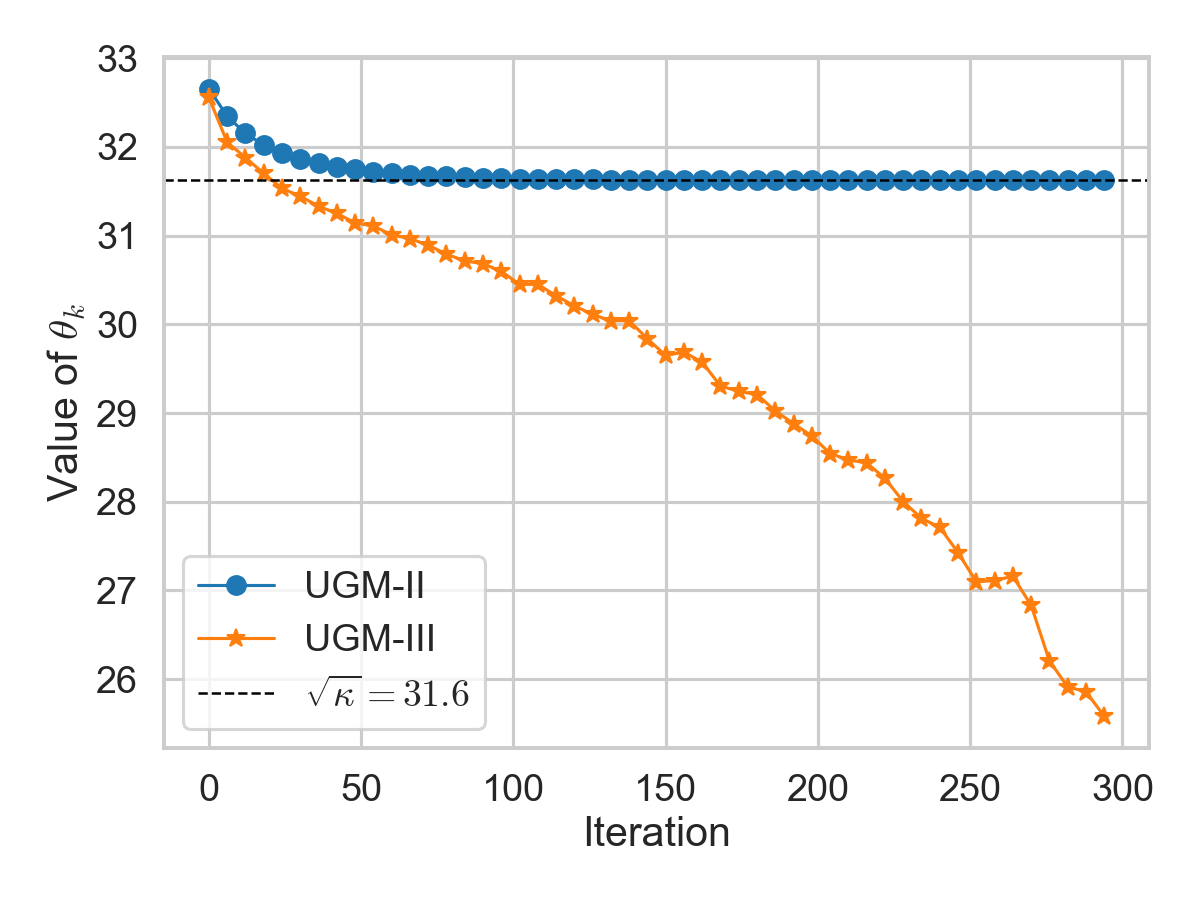}\caption{$\mu=1, L=10^3$}\end{subfigure}\hfill
\begin{subfigure}[b]{0.19\textwidth}\includegraphics[width=\textwidth]{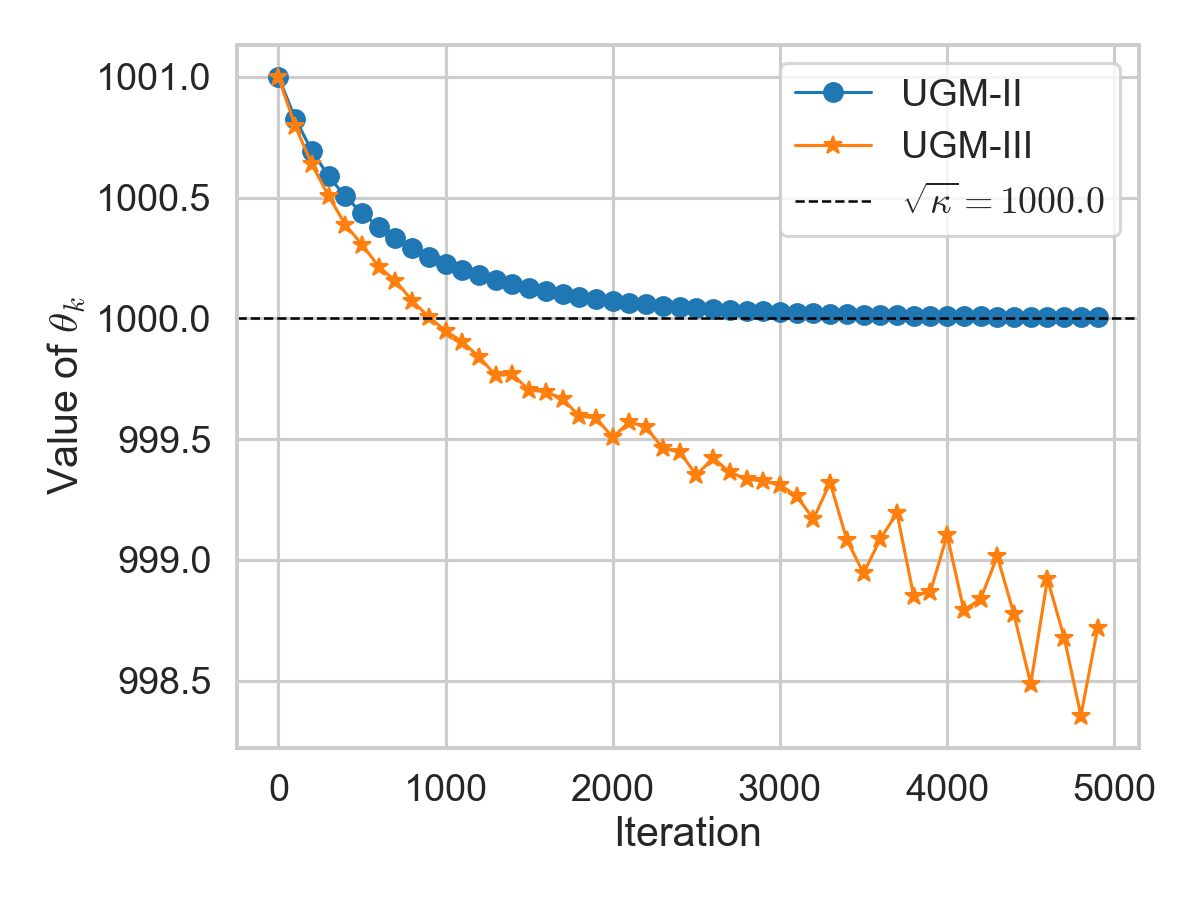}\caption{$\mu=1, L=10^6$}\end{subfigure}\hfill
\begin{subfigure}[b]{0.19\textwidth}\includegraphics[width=\textwidth]{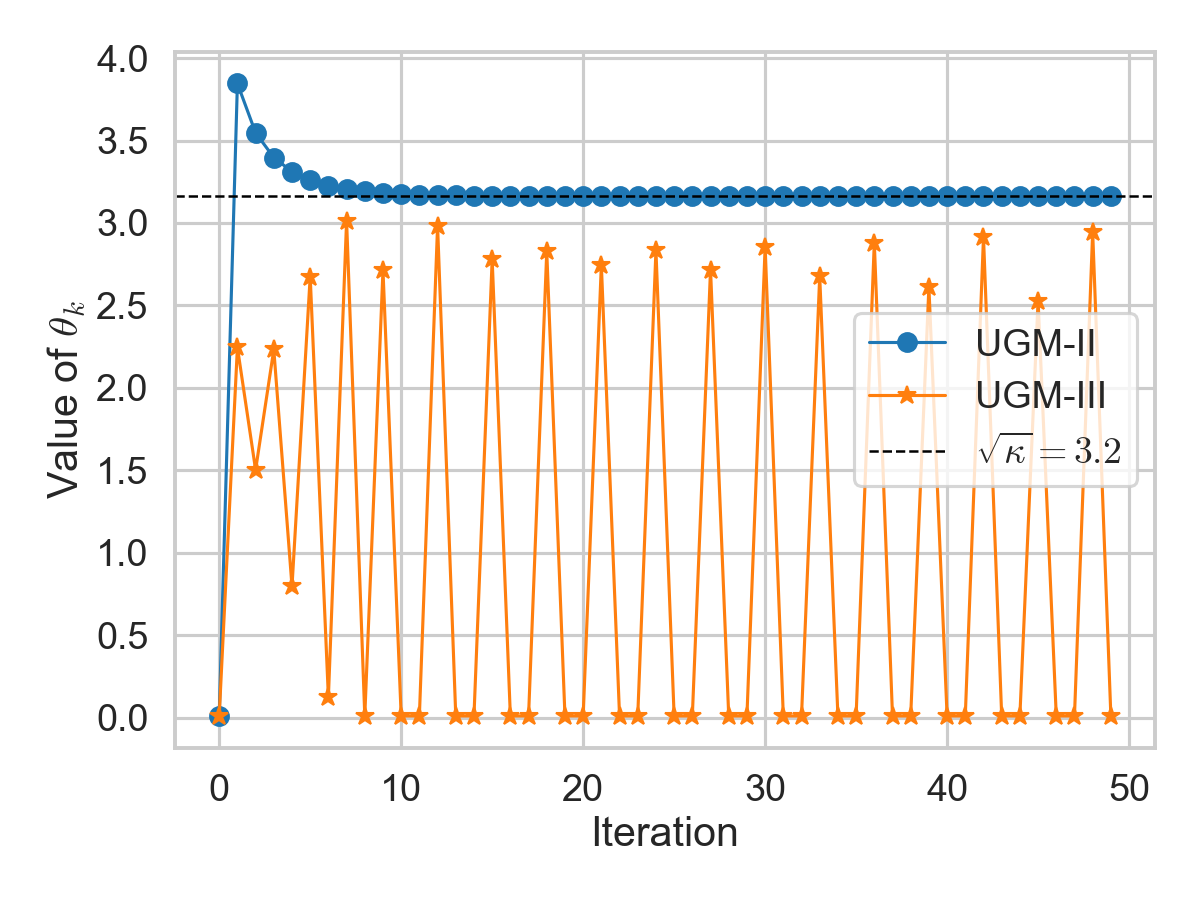}\caption{$\mu=10, L=10^2$}\end{subfigure}\hfill
\begin{subfigure}[b]{0.19\textwidth}\includegraphics[width=\textwidth]{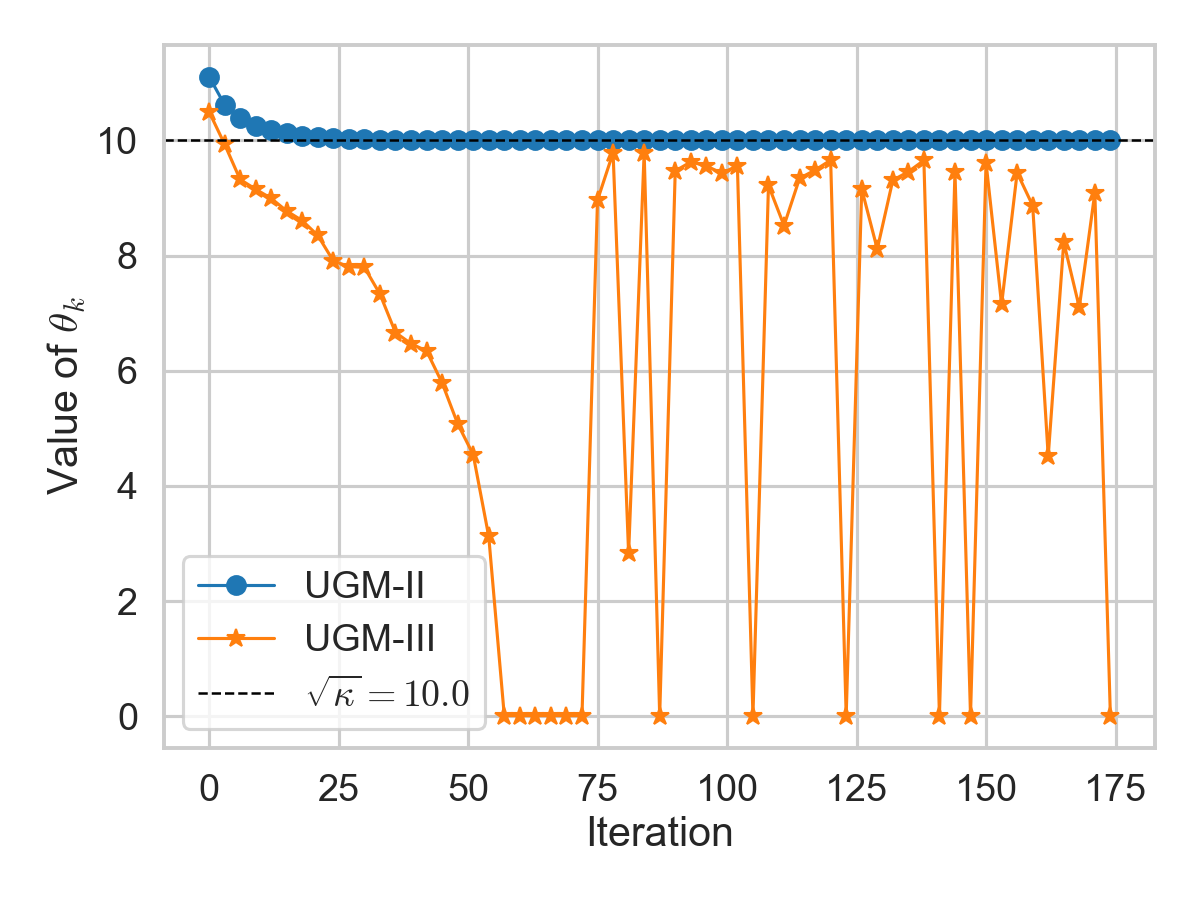}\caption{$\mu=10^2, L=10^4$}\end{subfigure}
\caption{Comparison of FGM, UGM-I, UGM-II, and UGM-III on Case 4. Rows show the objective suboptimality versus iterations, objective suboptimality versus CPU time, and the evolution of $\theta_k$, respectively. Columns correspond to different $(\mu, L)$ configurations.}
\label{fig:case4}
\end{figure}

\begin{figure}[!htbp]
\centering
\begin{subfigure}[b]{0.19\textwidth}\includegraphics[width=\textwidth]{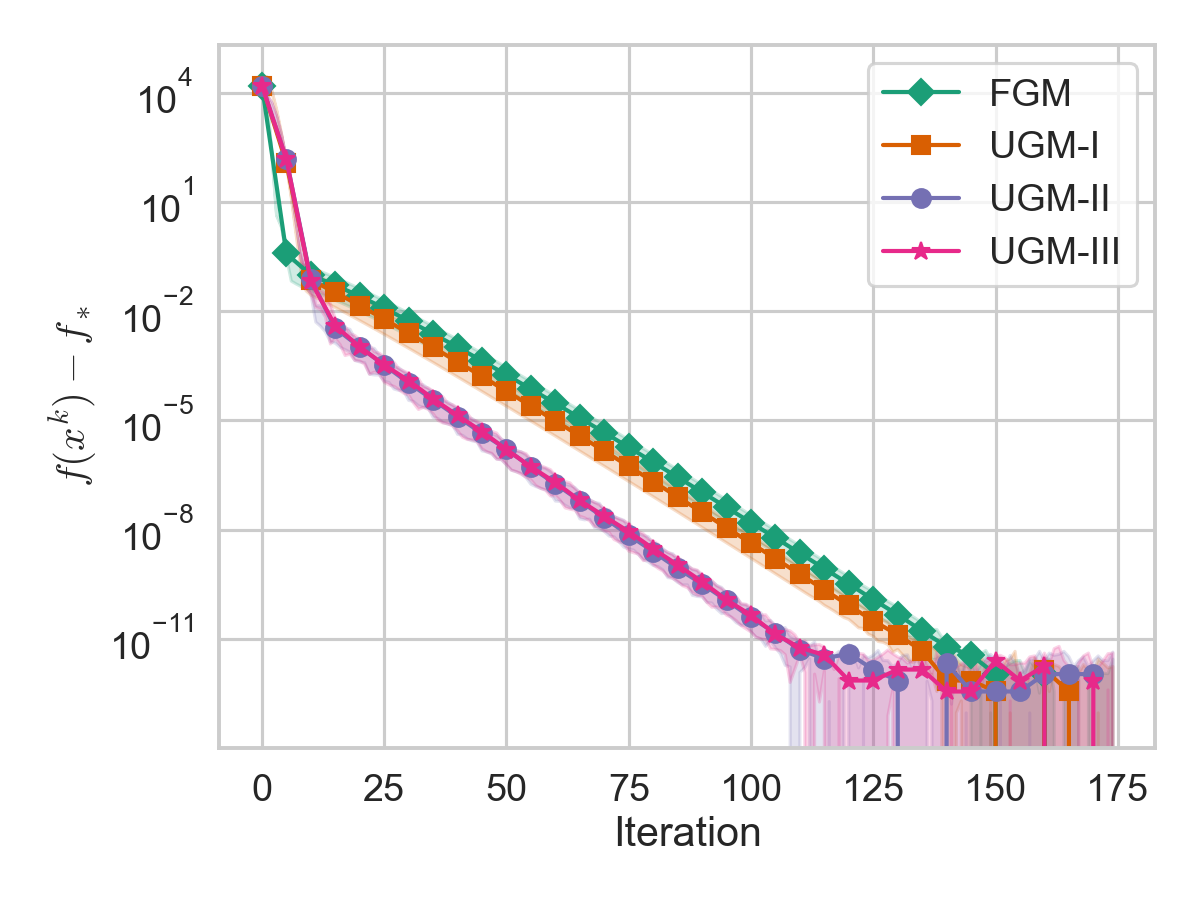}\end{subfigure}\hfill
\begin{subfigure}[b]{0.19\textwidth}\includegraphics[width=\textwidth]{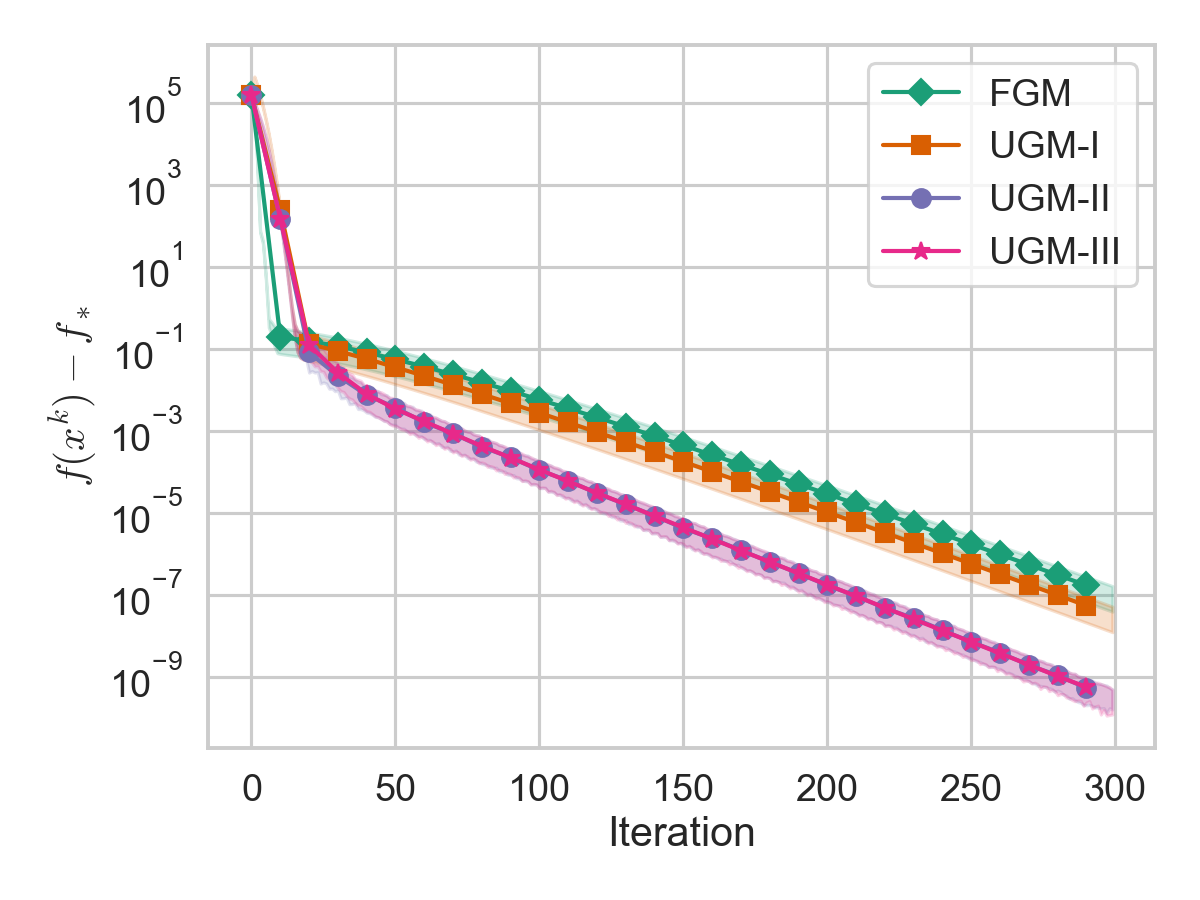}\end{subfigure}\hfill
\begin{subfigure}[b]{0.19\textwidth}\includegraphics[width=\textwidth]{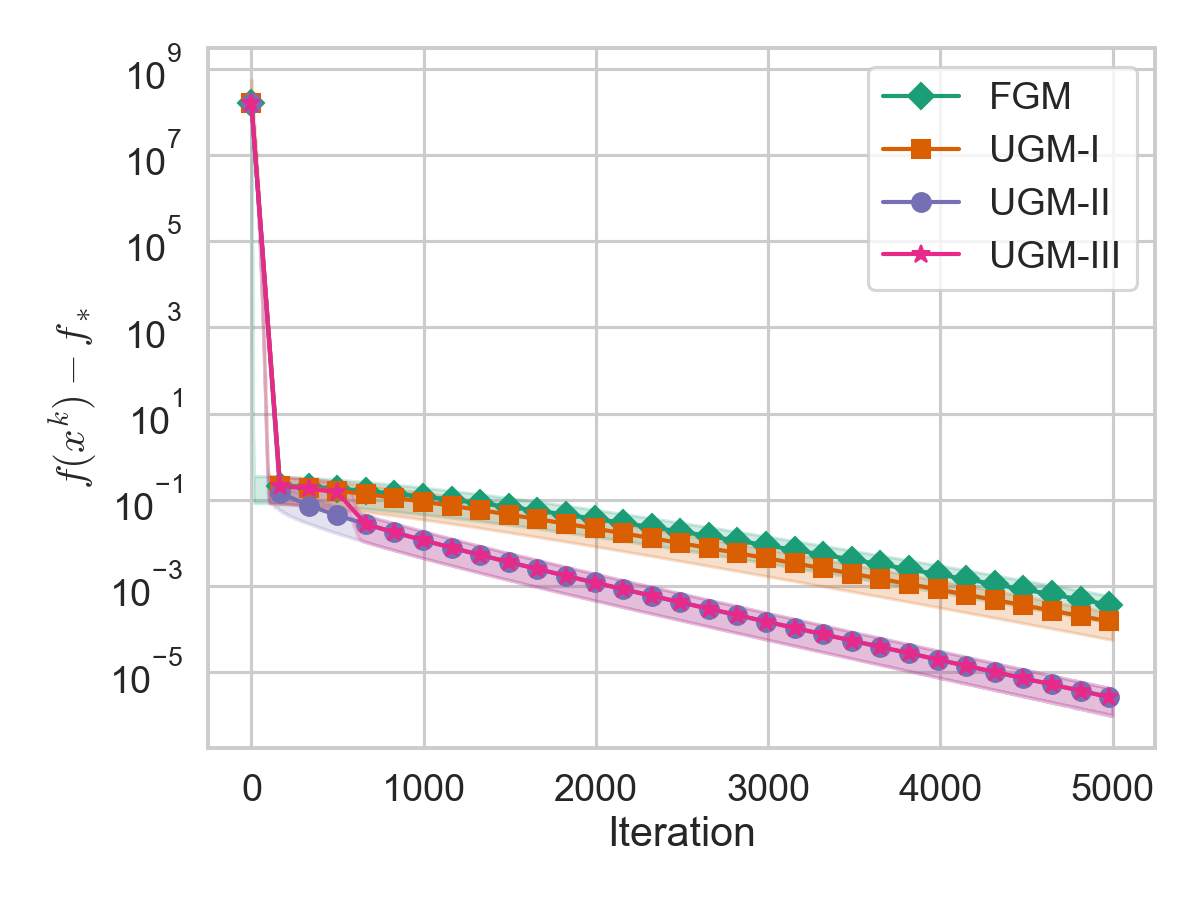}\end{subfigure}\hfill
\begin{subfigure}[b]{0.19\textwidth}\includegraphics[width=\textwidth]{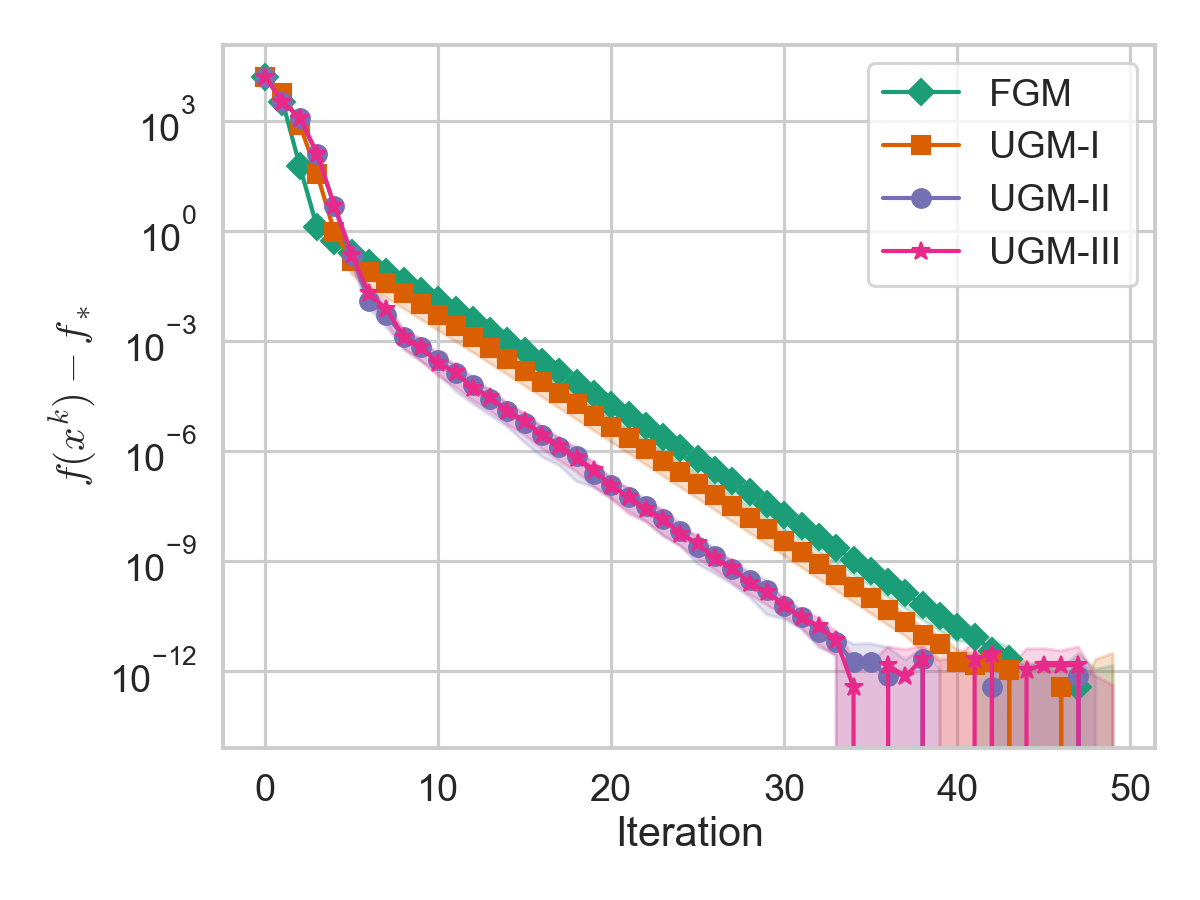}\end{subfigure}\hfill
\begin{subfigure}[b]{0.19\textwidth}\includegraphics[width=\textwidth]{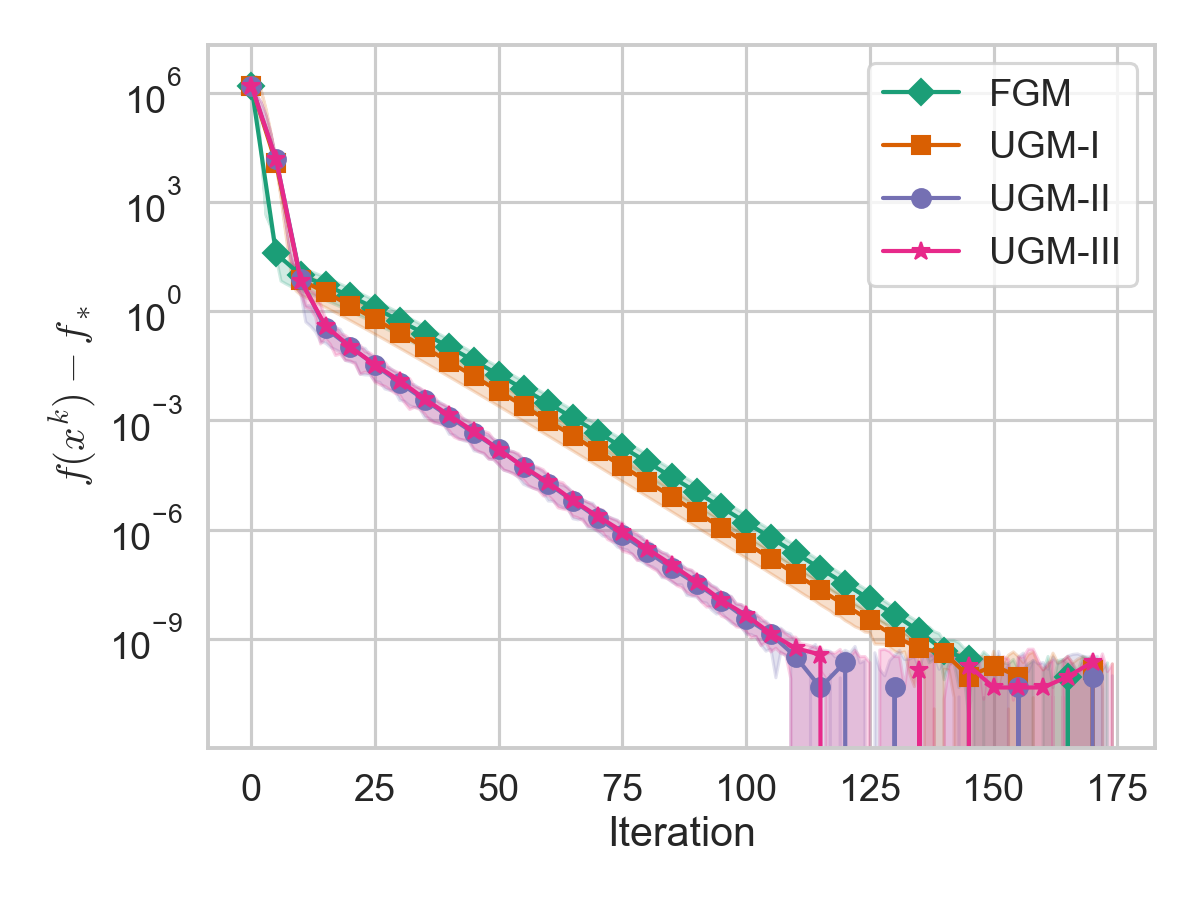}\end{subfigure}
\\
\begin{subfigure}[b]{0.19\textwidth}\includegraphics[width=\textwidth]{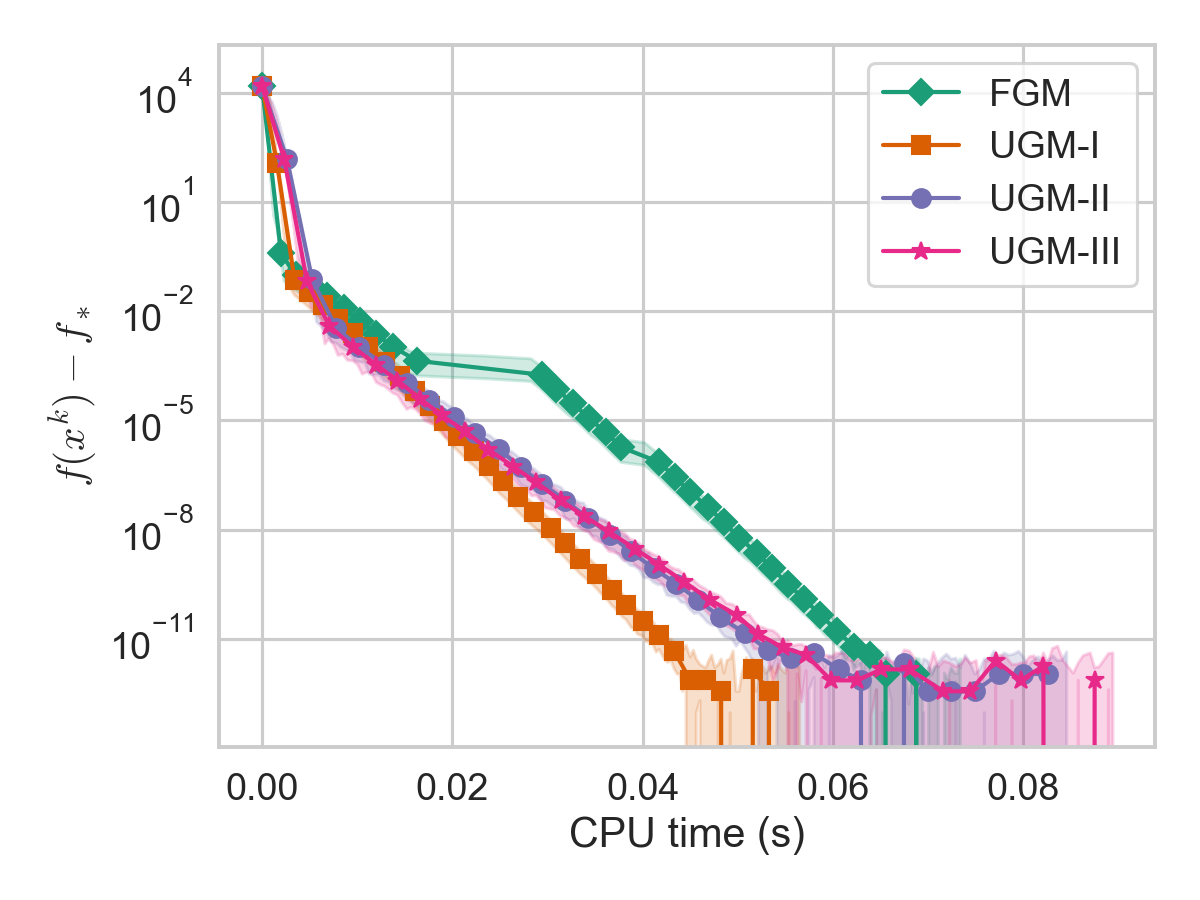}\end{subfigure}\hfill
\begin{subfigure}[b]{0.19\textwidth}\includegraphics[width=\textwidth]{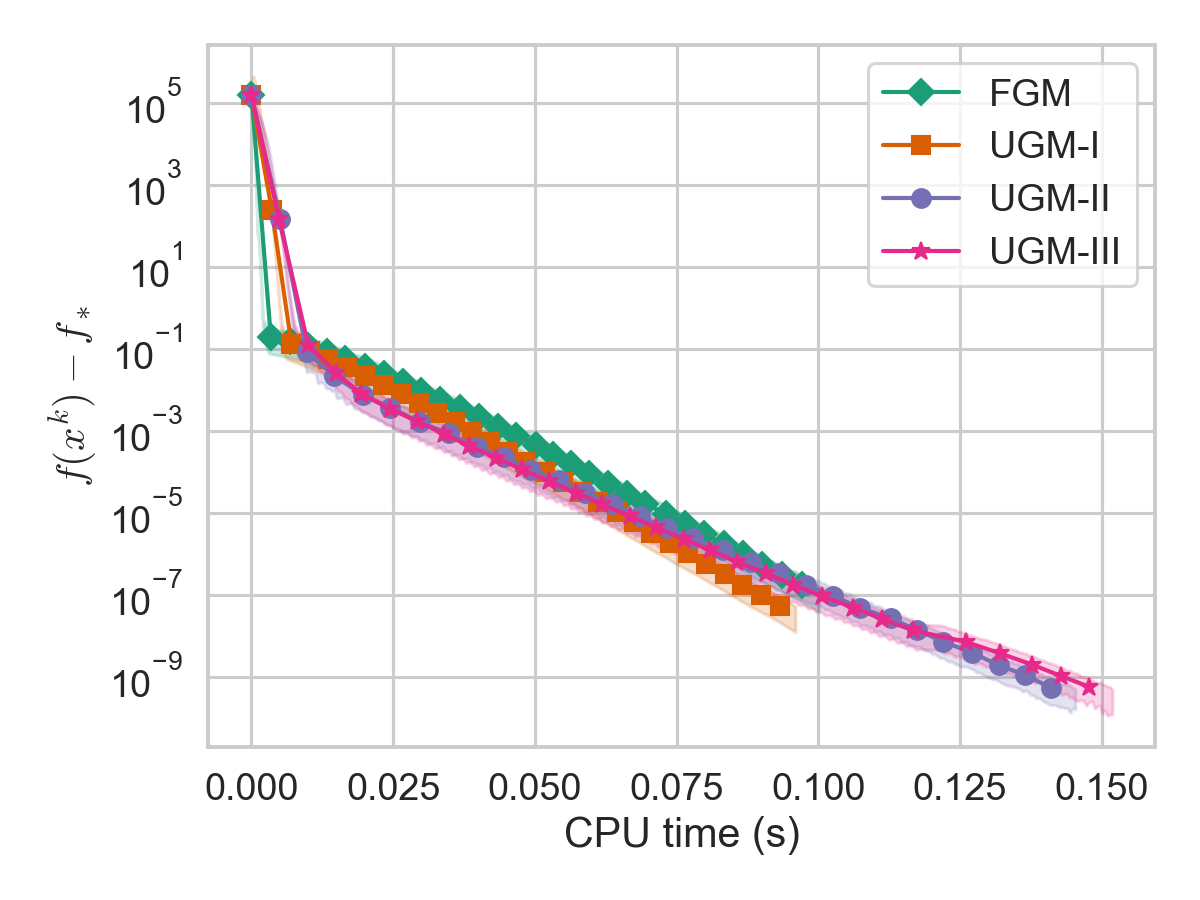}\end{subfigure}\hfill
\begin{subfigure}[b]{0.19\textwidth}\includegraphics[width=\textwidth]{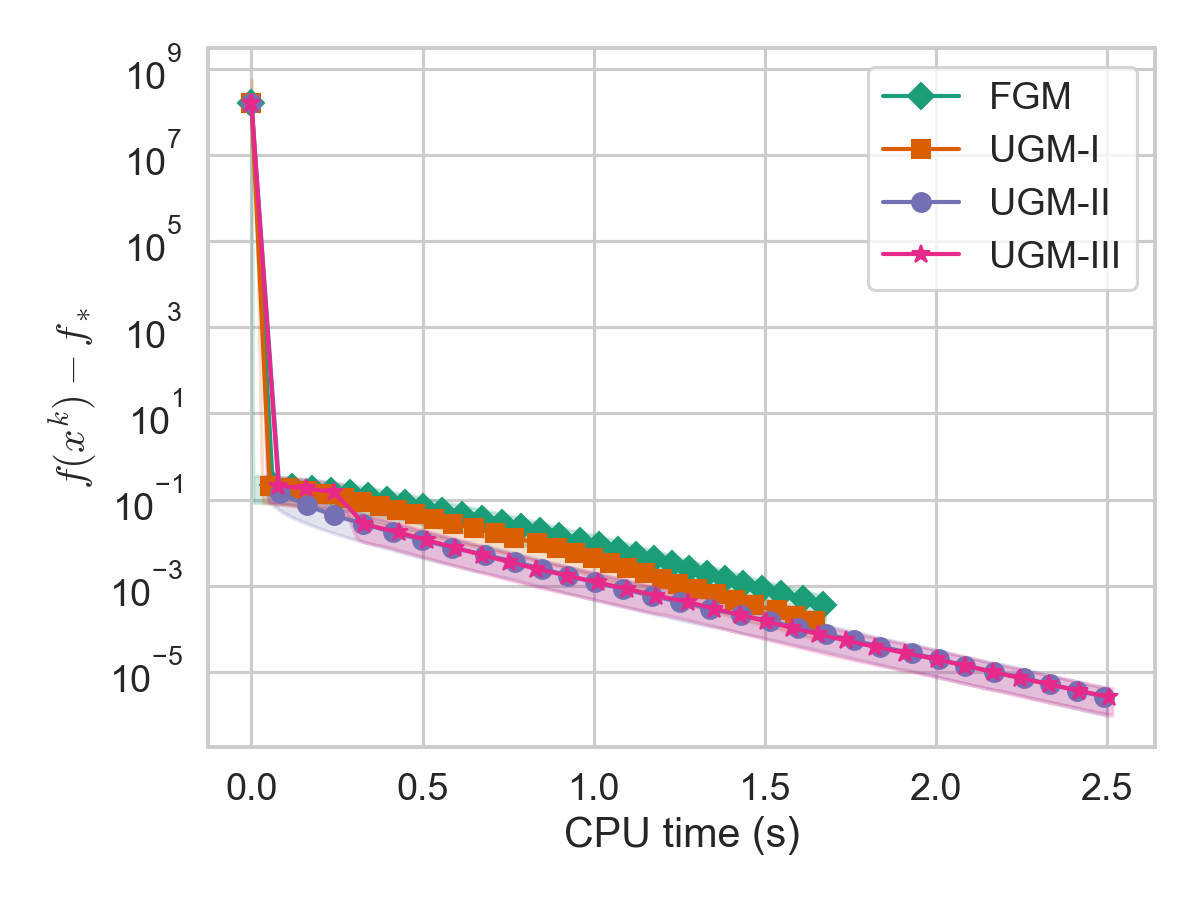}\end{subfigure}\hfill
\begin{subfigure}[b]{0.19\textwidth}\includegraphics[width=\textwidth]{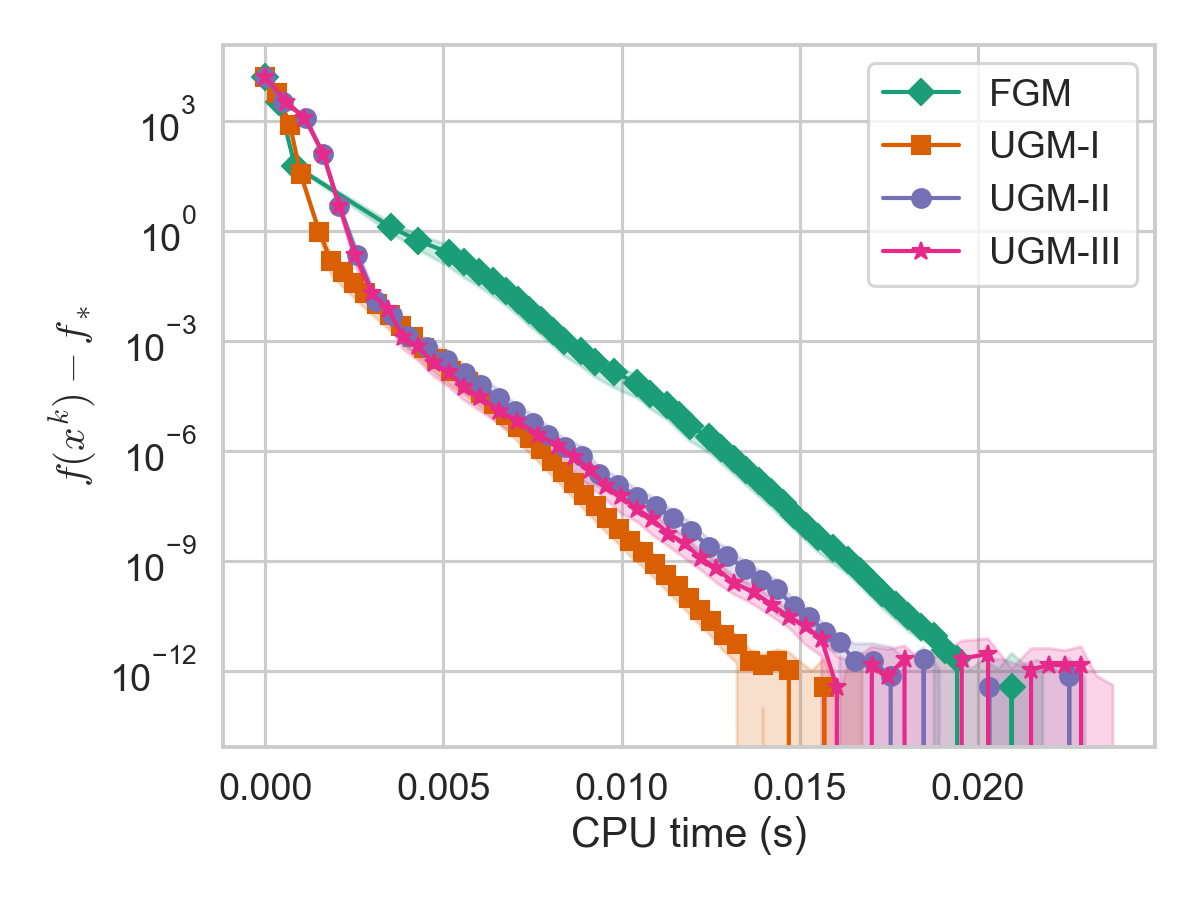}\end{subfigure}\hfill
\begin{subfigure}[b]{0.19\textwidth}\includegraphics[width=\textwidth]{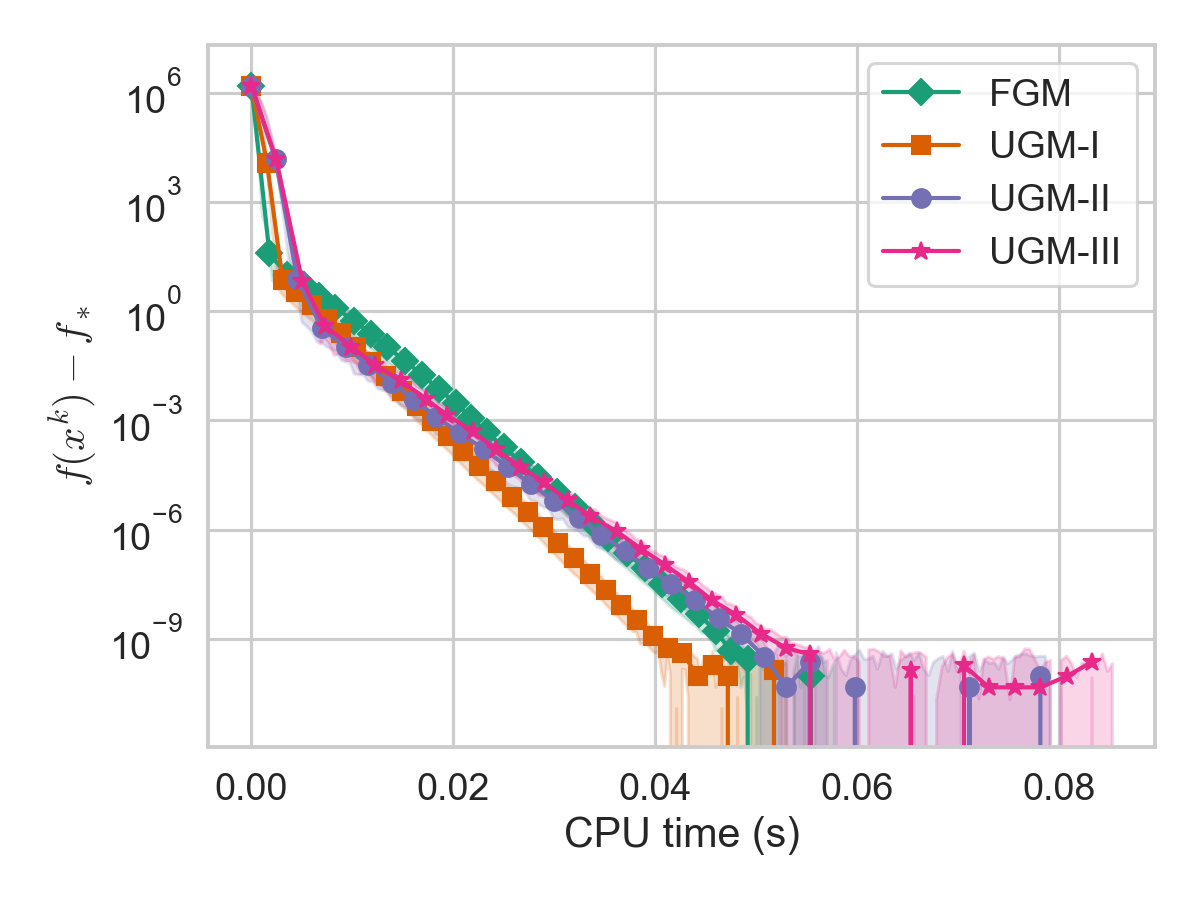}\end{subfigure}
\\
\begin{subfigure}[b]{0.19\textwidth}\includegraphics[width=\textwidth]{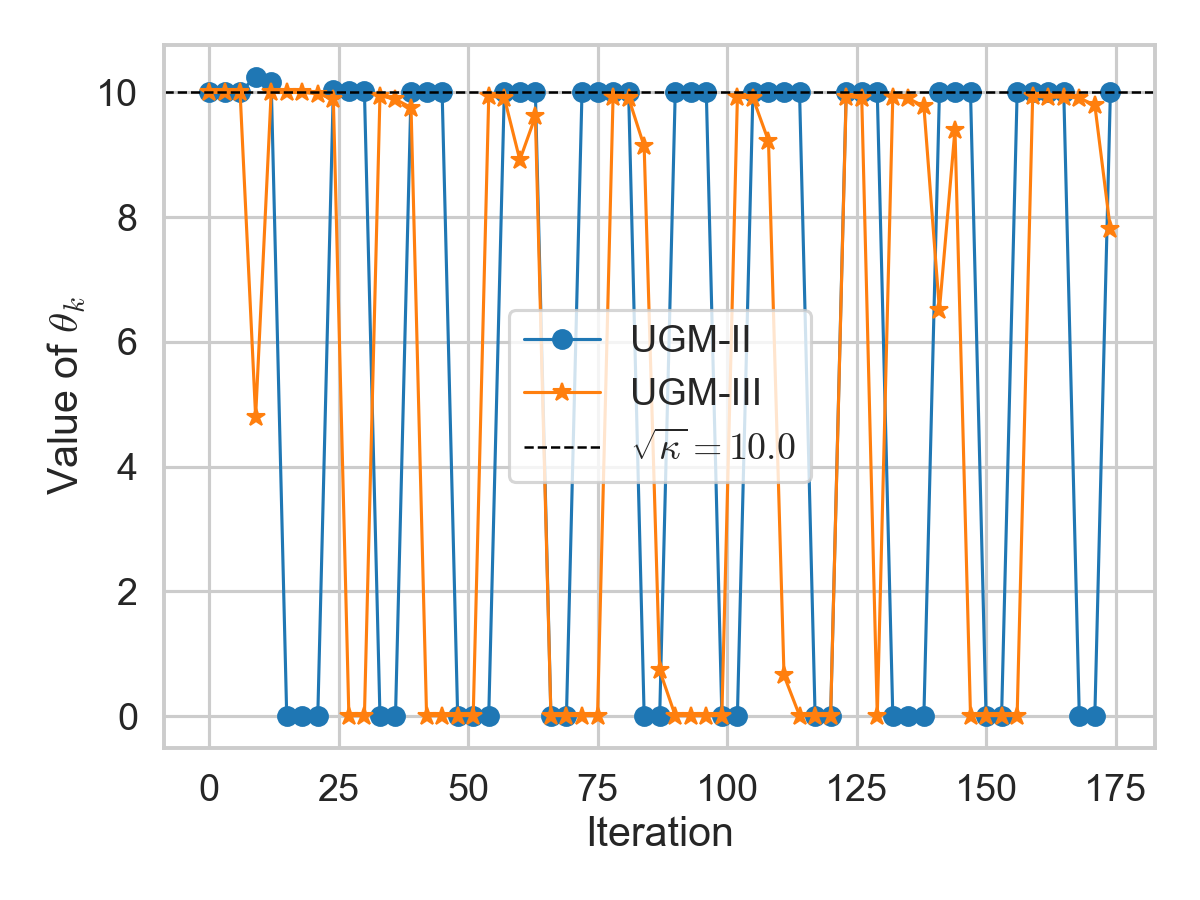}\caption{$\mu=1, L=10^2$}\end{subfigure}\hfill
\begin{subfigure}[b]{0.19\textwidth}\includegraphics[width=\textwidth]{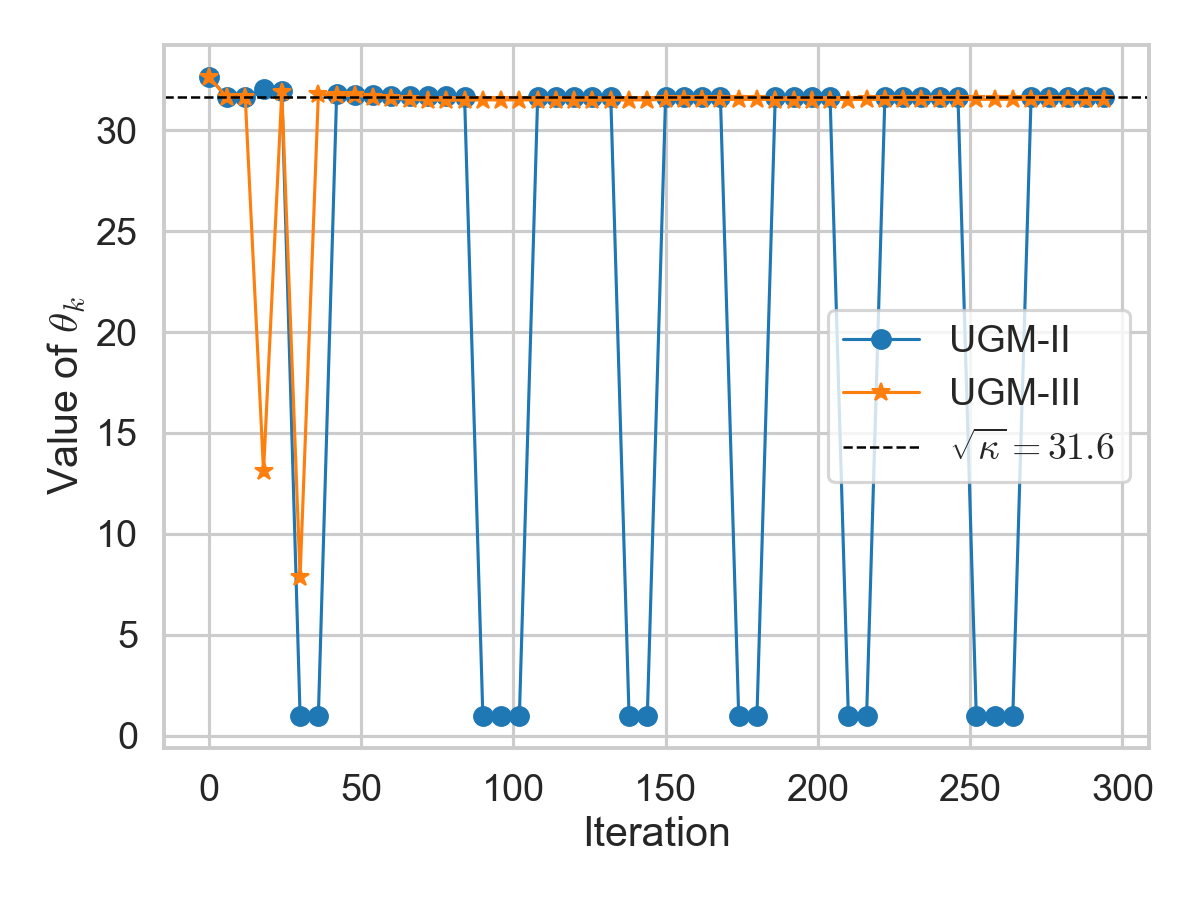}\caption{$\mu=1, L=10^3$}\end{subfigure}\hfill
\begin{subfigure}[b]{0.19\textwidth}\includegraphics[width=\textwidth]{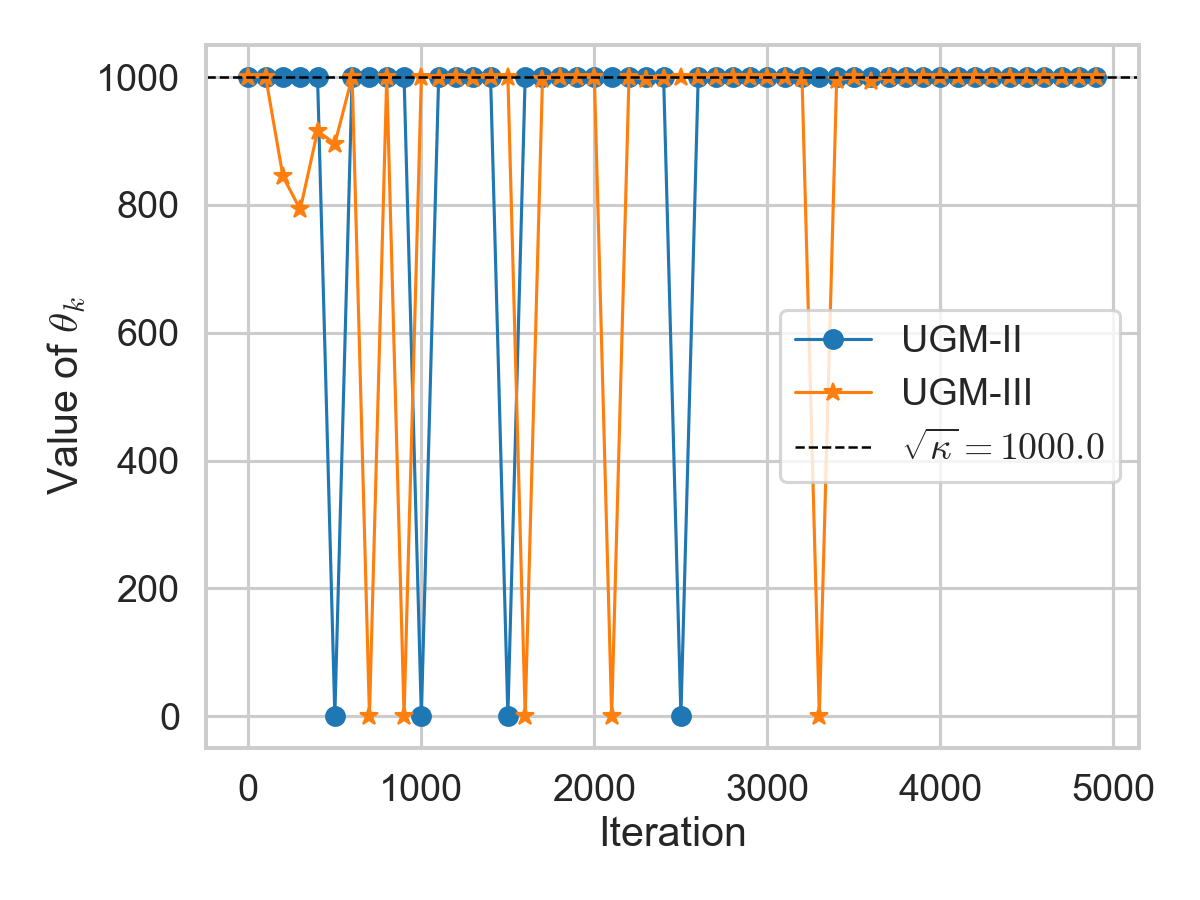}\caption{$\mu=1, L=10^6$}\end{subfigure}\hfill
\begin{subfigure}[b]{0.19\textwidth}\includegraphics[width=\textwidth]{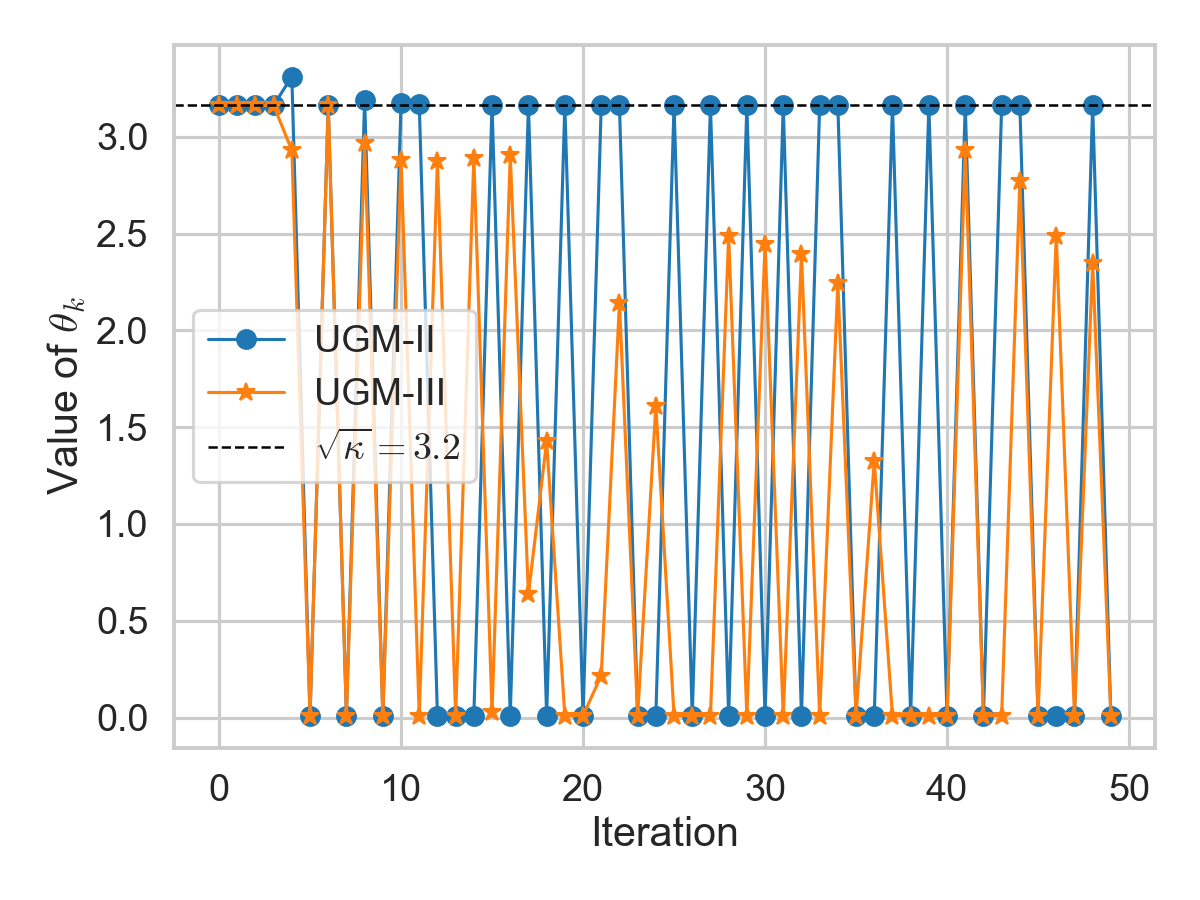}\caption{$\mu=10, L=10^2$}\end{subfigure}\hfill
\begin{subfigure}[b]{0.19\textwidth}\includegraphics[width=\textwidth]{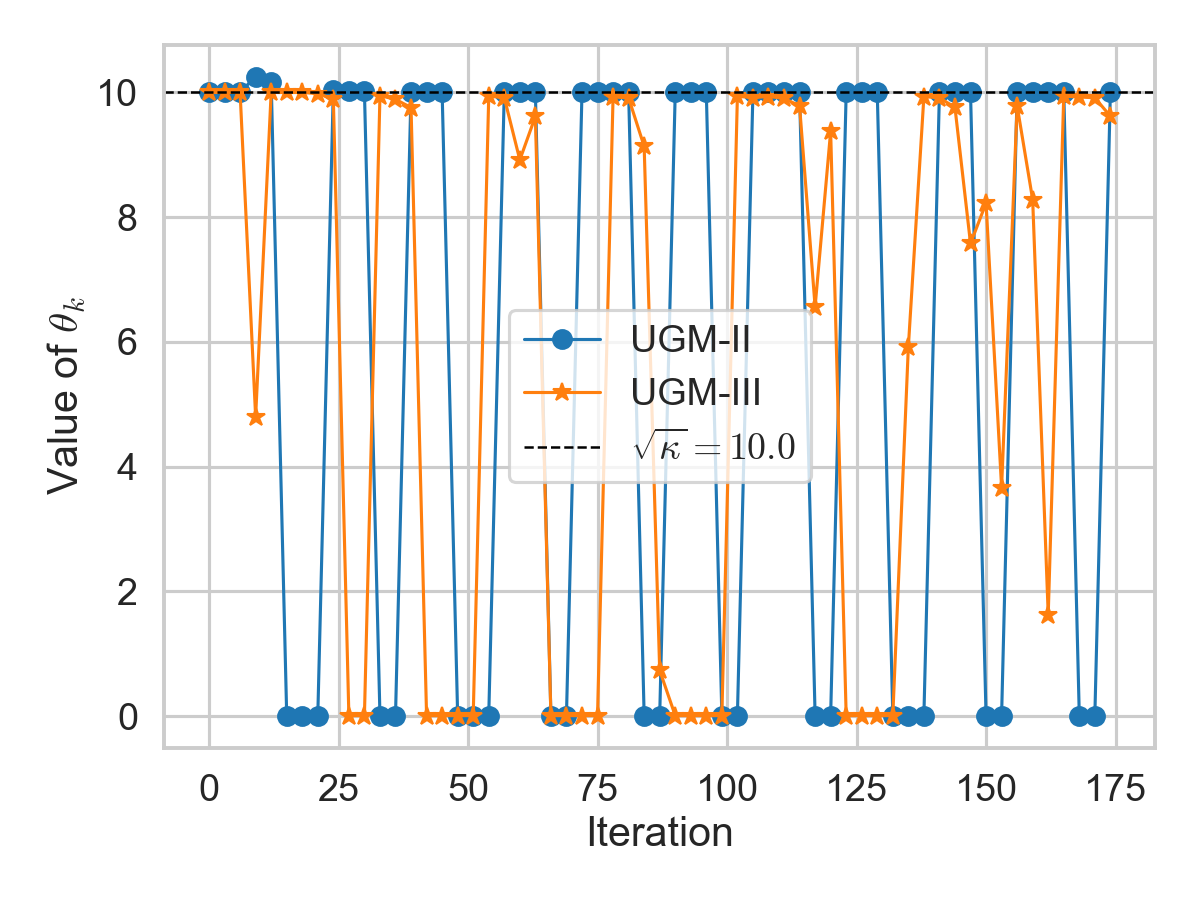}\caption{$\mu=10^2, L=10^4$}\end{subfigure}
\caption{Comparison of FGM, UGM-I, UGM-II, and UGM-III on Case 5. Rows show the objective suboptimality versus iterations, objective suboptimality versus CPU time, and the evolution of $\theta_k$, respectively. Columns correspond to different $(\mu, L)$ configurations.}
\label{fig:case5}
\end{figure}

In terms of iteration counts, UGM‑II and UGM‑III attain slightly lower suboptimality than FGM.
However, since UGM‑II and UGM‑III require one extra gradient evaluation per iteration, they do not offer a clear advantage in CPU time in most test cases.
Some exceptions include Figure \ref{fig:case3} (column (c)), and Figure \ref{fig:case5} (column (a), (c), (d)), where the faster convergence of UGM‑II and UGM‑III in terms of iterations offsets the additional gradient evaluations, rendering them competitive in CPU‑time performance. Across all test cases, the proposed UGMs achieve performance comparable to FGM.

\paragraph{The role of $\varrho$ in UGM‑I.} We next investigate the effect of the parameter $\varrho$ on the performance of UGM‑I.
We select four representative values of $\varrho$ from the interval
$[\frac{\sqrt{1+8\kappa}-3}{\sqrt{1+8\kappa}+1}, \frac{\kappa-1}{\kappa+1})$. 
The smallest value is taken at the lower bound endpoint, while the largest is set to $0.99(\kappa-1)/(\kappa+1)$ to stay safely below the strict upper limit. The remaining two values are equally spaced between these two extremes.
We compare the four resulting variants of UGM‑I on two contrasting eigenvalue distributions, Case 1 and Case 5, using three $(\mu, L)$ configurations: $(1,10^2)$, $(1,10^4)$, and $(10^2,10^4)$.  
The numerical results are presented in Figure~\ref{fig:rho_both}.

\begin{figure}[!htbp]
    \centering
    \begin{subfigure}[b]{0.32\textwidth}
        \includegraphics[width=\textwidth]{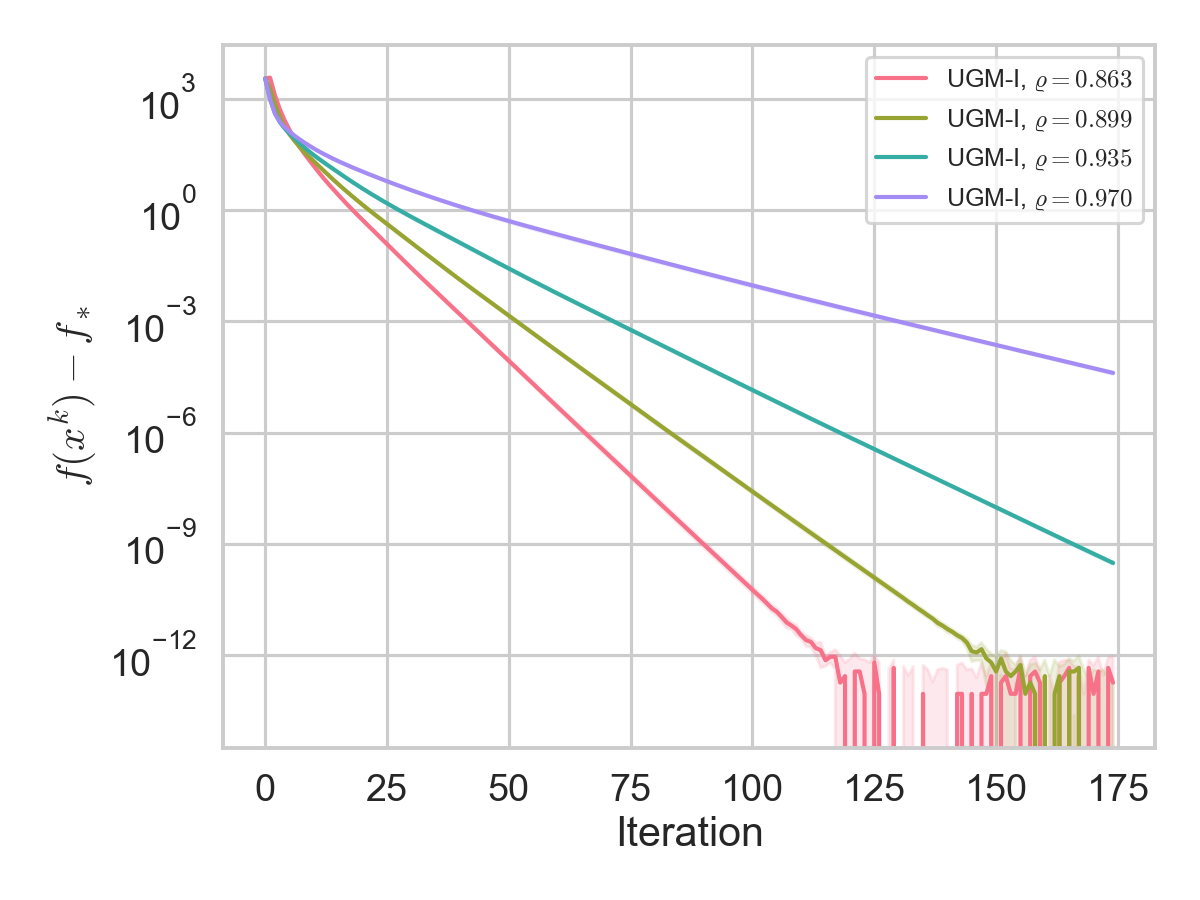}
    \end{subfigure}\hfill
    \begin{subfigure}[b]{0.32\textwidth}
        \includegraphics[width=\textwidth]{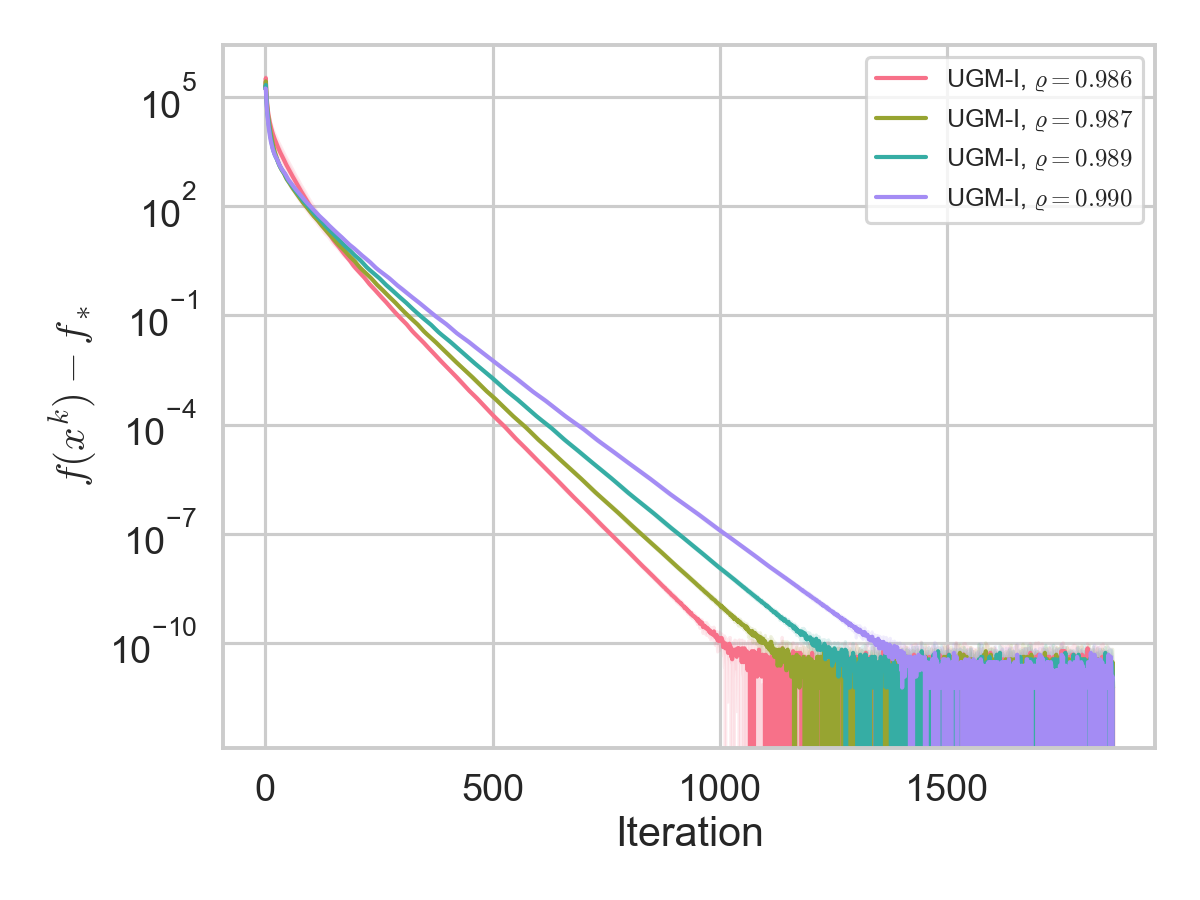}
    \end{subfigure}\hfill
    \begin{subfigure}[b]{0.32\textwidth}
        \includegraphics[width=\textwidth]{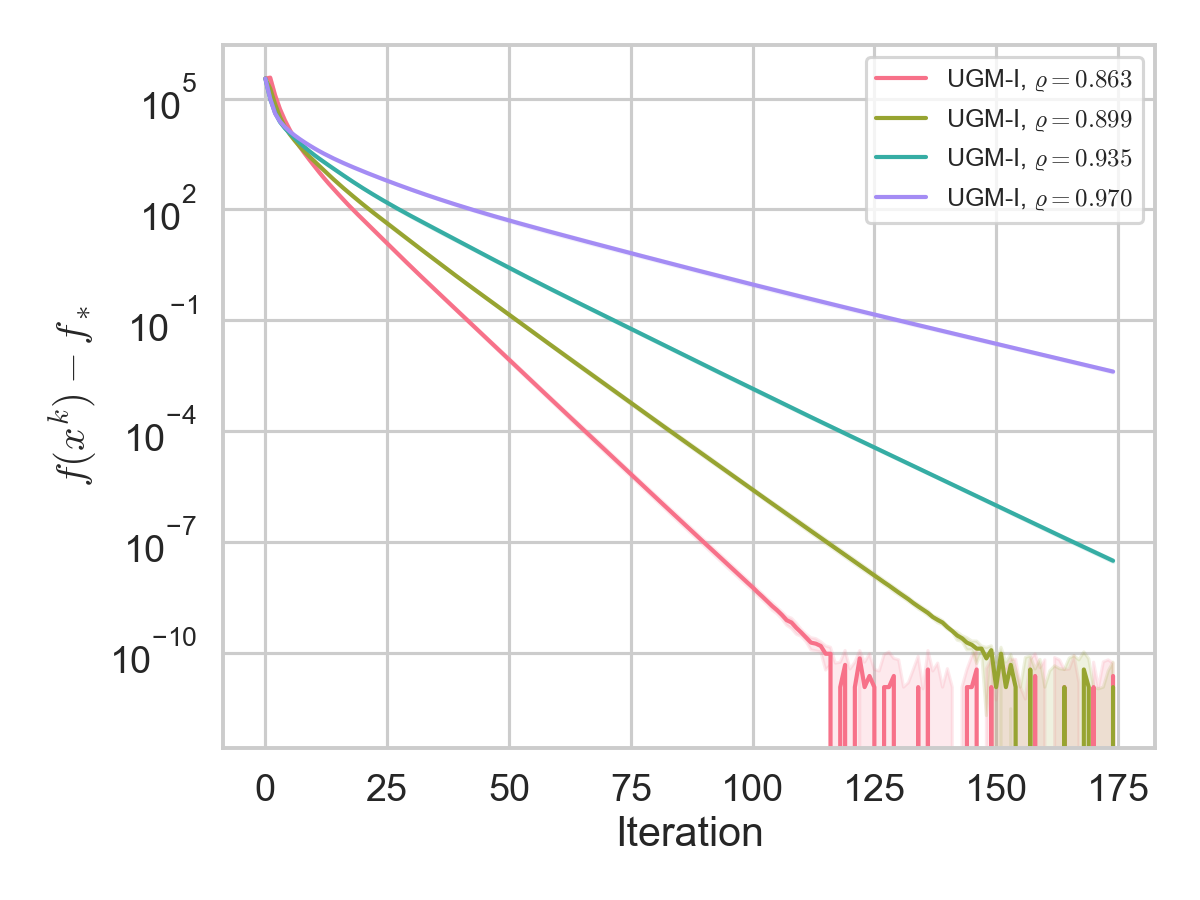}
    \end{subfigure}

    \begin{subfigure}[b]{0.32\textwidth}
        \includegraphics[width=\textwidth]{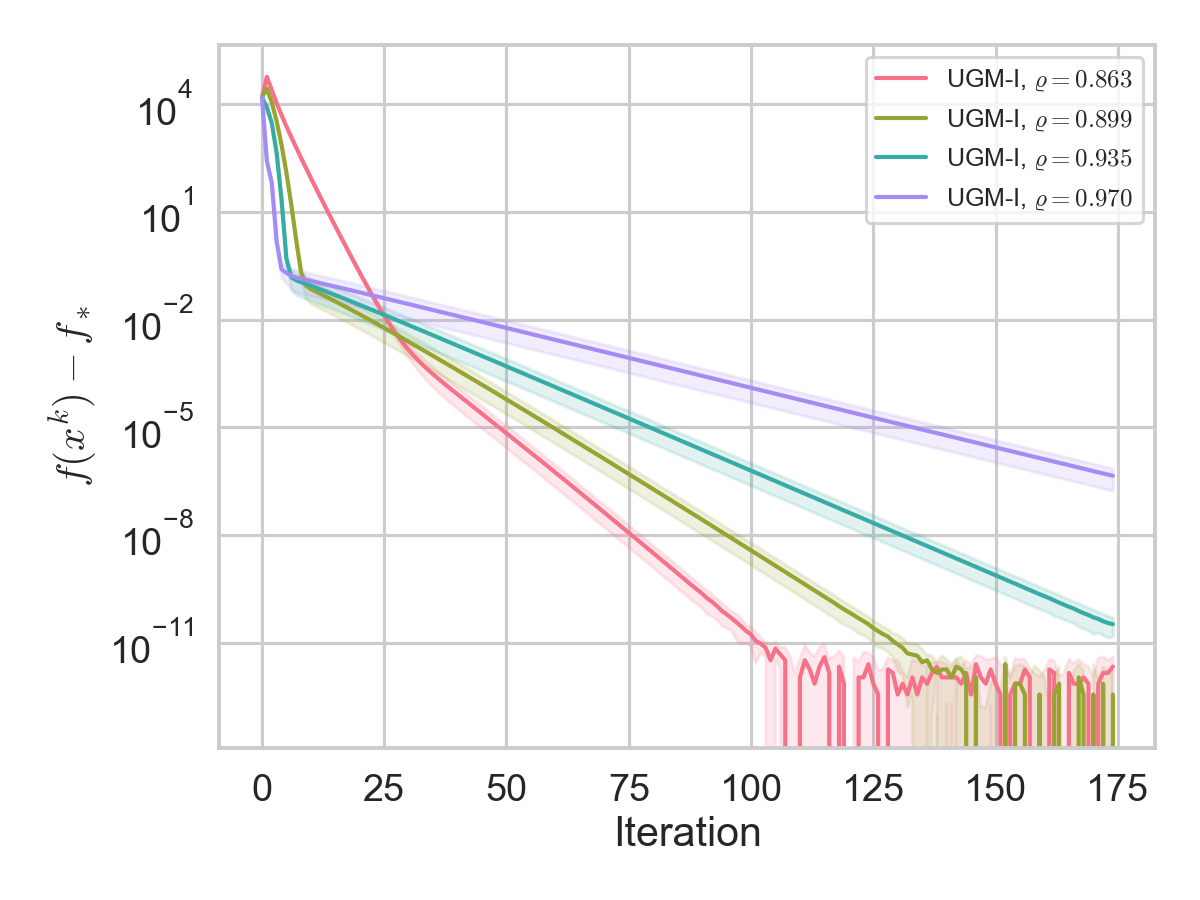}
        \caption{$\mu=1,\; L=10^2$}
    \end{subfigure}\hfill
    \begin{subfigure}[b]{0.32\textwidth}
        \includegraphics[width=\textwidth]{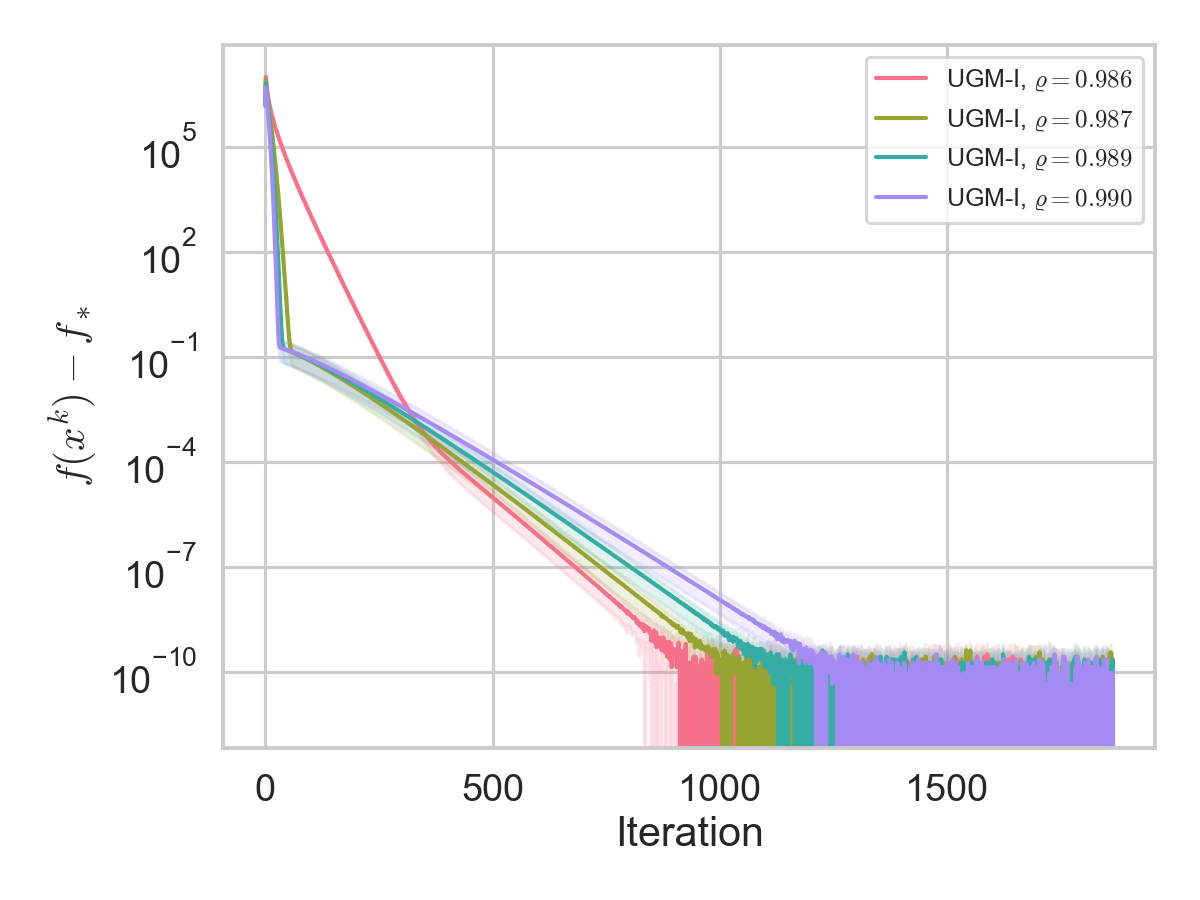}
        \caption{$\mu=1,\; L=10^4$}
    \end{subfigure}\hfill
    \begin{subfigure}[b]{0.32\textwidth}
        \includegraphics[width=\textwidth]{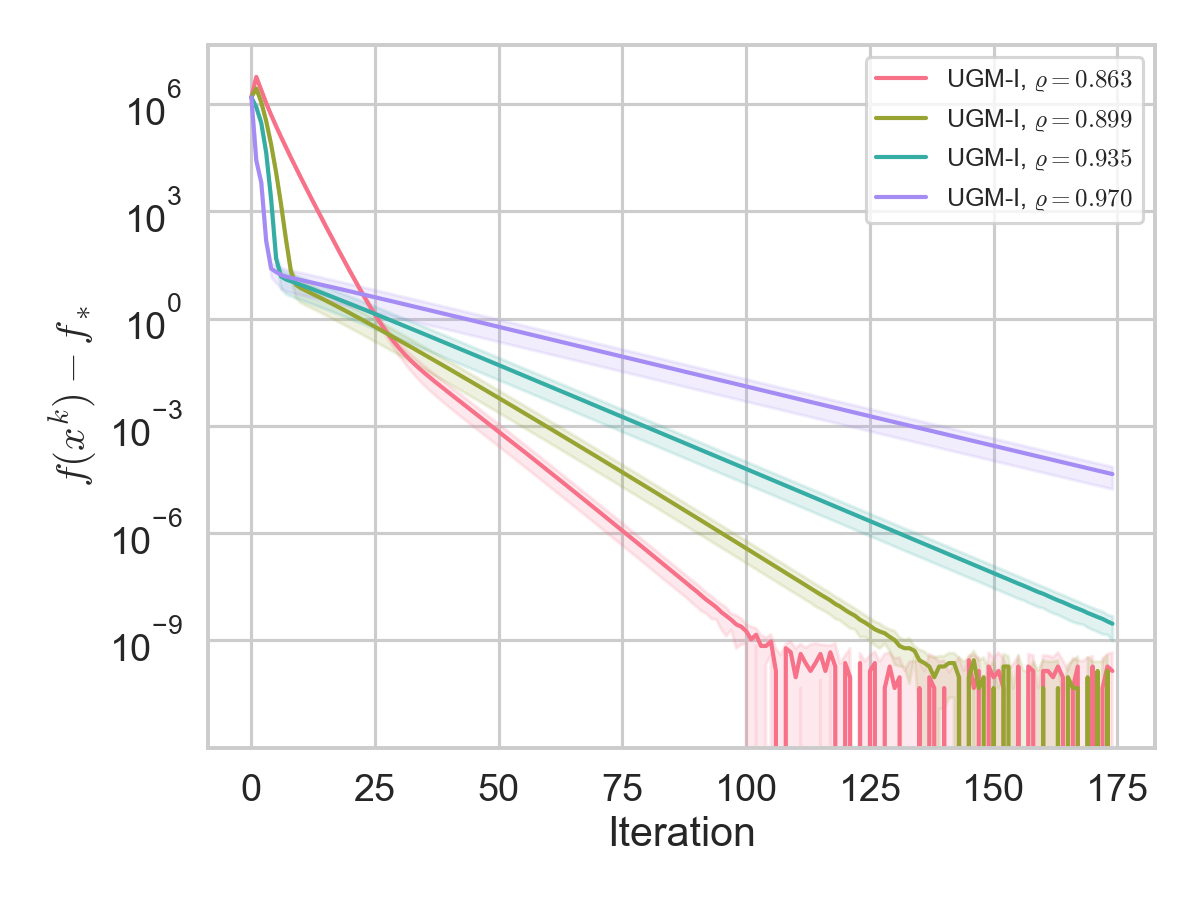}
        \caption{$\mu=10^2,\; L=10^4$}
    \end{subfigure}

    \caption{Comparison of UGM‑I with four choices of $\varrho$. Rows correspond to Cases 1 and 5, and columns correspond to different $(\mu, L)$ pairs.}
    \label{fig:rho_both}
\end{figure}

It is clear from the results in  Figure~\ref{fig:rho_both} that the parameter $\varrho$ dictates the convergence speed of UGM‑I: the smaller the value of $\varrho$, the steeper the decrease of the suboptimality $f(x^k)-f_*$.

\subsection{Logistic regression} \label{sec:lr}
In this subsection, we consider the regularized logistic regression problem
\begin{equation*}\label{log-regression}
\min_{x\in \mathbb R^{n}} f(x)=-\frac{1}{m}\sum_{i=1}^m \left(y_i \log\big(s(a_i^\top x)\big) + (1 - y_i) \log\big(1 - s(a_i^\top x)\big)\right) + \frac{\gamma}{2}\|x\|^2,
\end{equation*}
where $\{(a_i,y_i)\in  \mathbb R^{n} \times \{0, 1\}: i=1,\ldots,m\}$ is the dataset, $s(z)=1/\bigl(1+\exp(-z)\bigr)$ denotes the sigmoid function, $m$ is the number of samples, and $\gamma>0$ is the regularization parameter.
For this problem, the gradient of $f$ is given by $\nabla f(x) = \frac{1}{m}\sum_{i=1}^m a_i\big(s(a_i^\top x)-y_i\big) + \gamma x$. Consequently, $f$ is $L$-smooth with Lipschitz constant $L=\lambda_{\max}(A^\top A)/4 + \gamma$, where $A = (a_1^{\top},\ldots,a_m^{\top})^{\top}$ and $\lambda_{\max}(A^\top A)$ denotes the largest eigenvalue of $A^\top A$.
In this experiment, all competing algorithms are run for $3000$ iterations. We test on the mushrooms dataset from LIBSVM \cite{CL11}, which has sample size $m=8124$ and feature dimension $n=112$. The Lipschitz constant is $L=2.59$, and the regularization parameter is set to $\gamma=L/m=3.18\times 10^{-5}$.
Consistent with the experiments in Section \ref{sec:qp}, every algorithm is initialized at the origin.
For the proposed UGMs, the parameters are chosen as follows. We take $\varrho=1-1/\sqrt{\kappa}$, matching the convergence‑rate parameter of the FGM method. We further set $q_0 = (\sqrt{\kappa}-1)/(\sqrt{\kappa}+1)$ for UGM‑I and $q_0 = \varrho$ for UGM‑II and UGM‑III. We then let $\underline{\theta} = 1$, $a_0 = \frac{(1 - \varrho)[\kappa(1 - \varrho) - \varrho - 1]}{(\varrho + 1)\varrho}$, and $\bar{\theta} = \max\big\{\frac{1}{1-\varrho},\frac{\varrho a_0+1-\varrho}{(1-\varrho)a_0}\big\} = \sqrt{\kappa}$.

In Figure~\ref{fig:lr}, we plot the objective suboptimality $f(x^k)-f_*$ against iteration count and CPU time (in seconds), along with the evolution of the extrapolation parameter $\theta_k$ for UGM‑II and UGM‑III. 
From Figures~\ref{fig:lr}(a) and~\ref{fig:lr}(b), we observe that both UGM‑II and UGM‑III consistently outperform FGM in terms of both iteration efficiency and computational cost on the tested logistic regression instance. This performance gain arises from the adaptive updating rule of the extrapolation parameter $\theta_k$. As shown in Figure~\ref{fig:lr}(c), $\theta_k$ eventually converges to its lower bound $\underline{\theta}$ after a certain number of iterations, which suppresses parameter oscillations and substantially accelerates the decay of objective suboptimality.

\begin{figure}[!htbp]
    \centering
    \begin{subfigure}[b]{0.32\textwidth}
        \includegraphics[width=\textwidth]{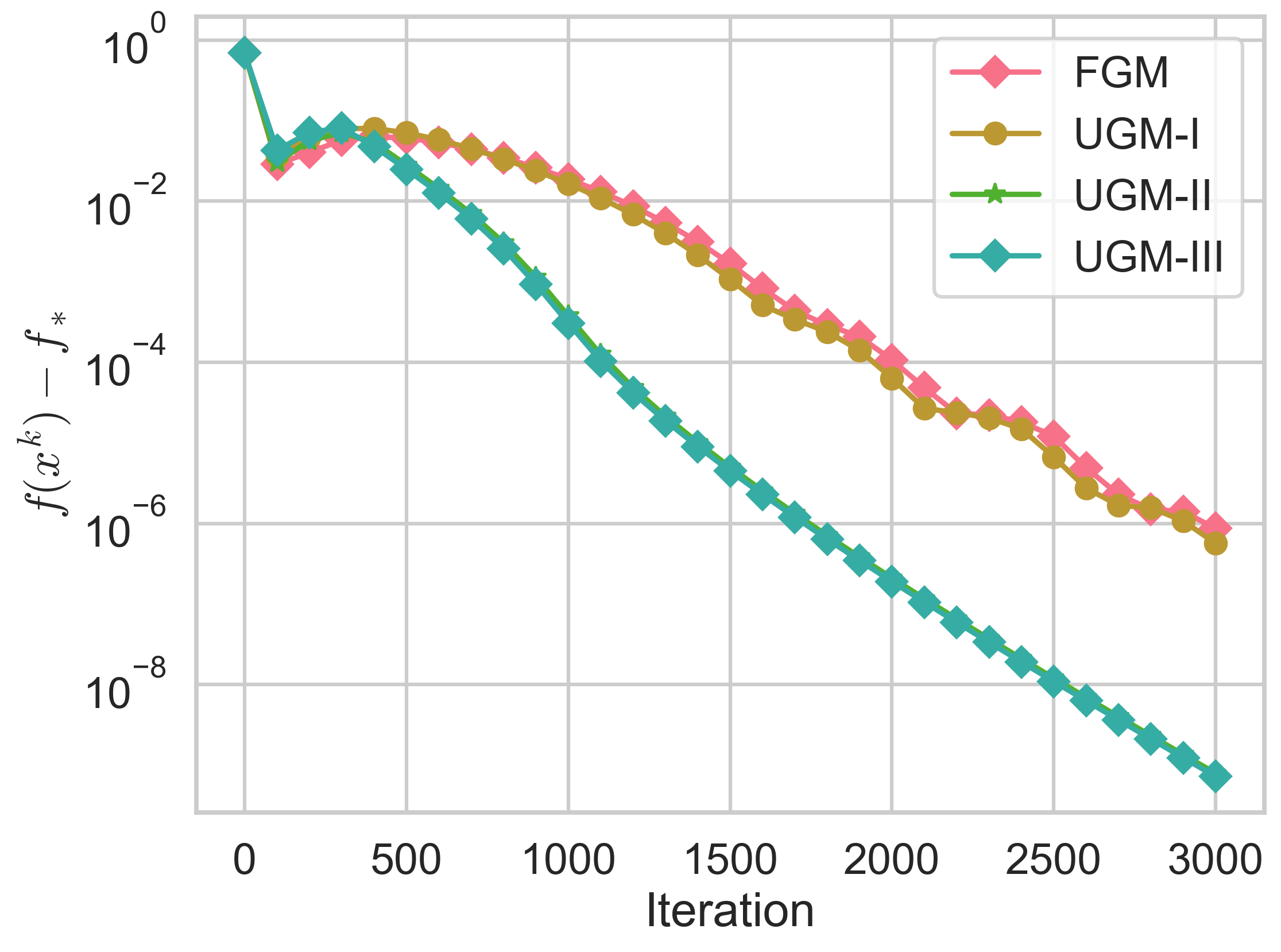}
    \end{subfigure}\hfill
    \begin{subfigure}[b]{0.32\textwidth}
        \includegraphics[width=\textwidth]{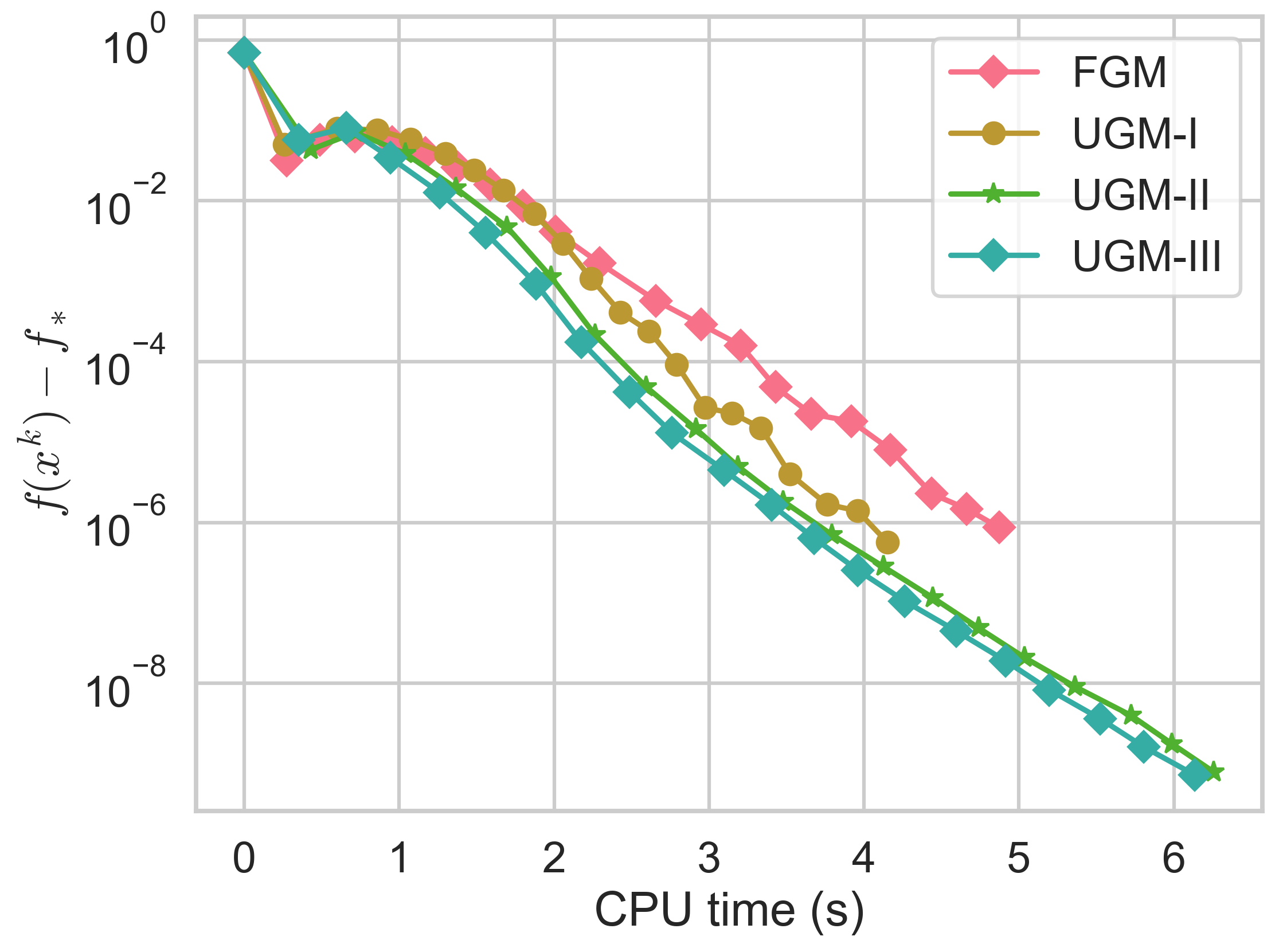}
    \end{subfigure}\hfill
    \begin{subfigure}[b]{0.32\textwidth}
        \includegraphics[width=\textwidth]{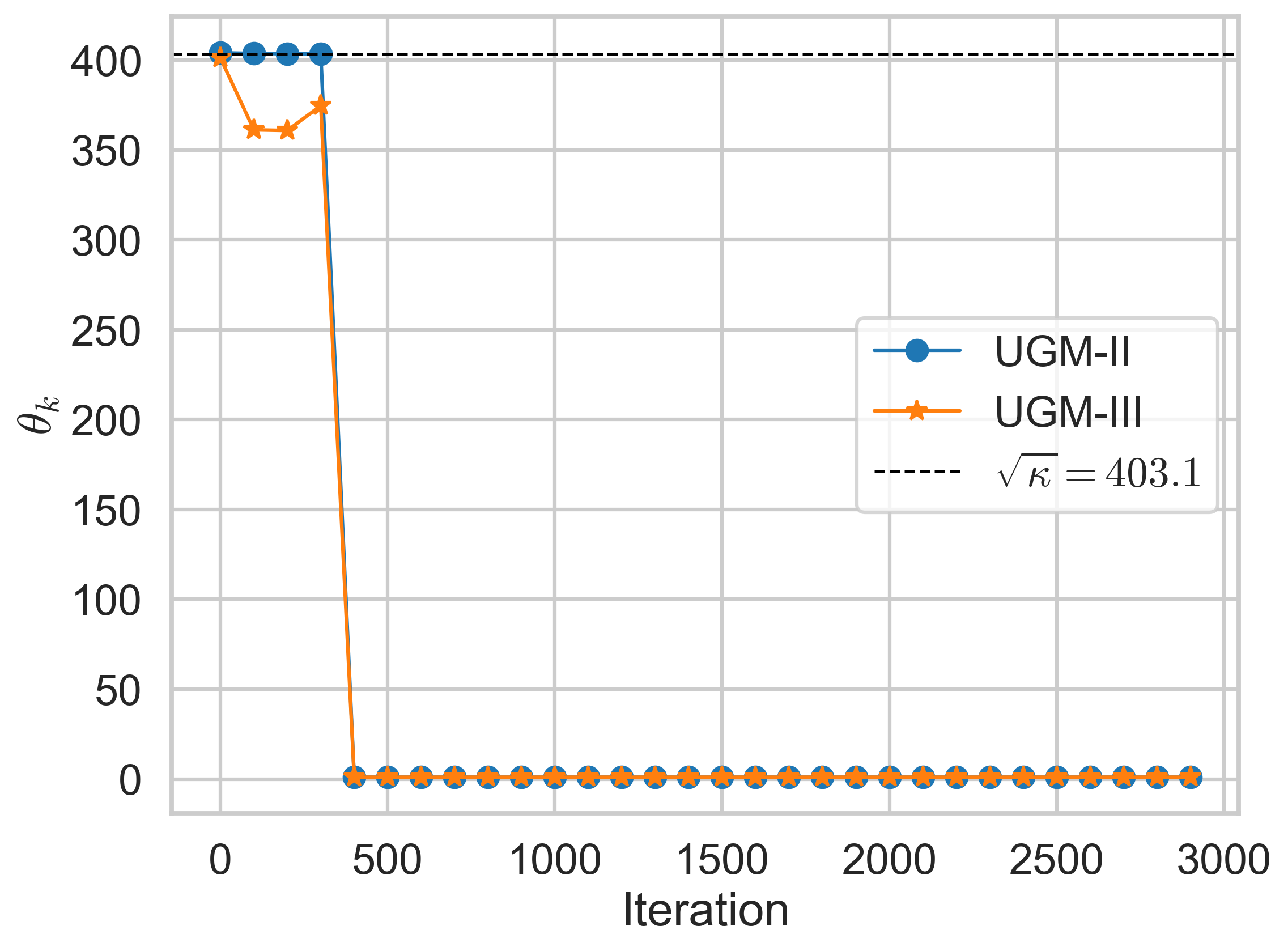}
    \end{subfigure}
    \caption{Comparison of FGM, UGM-I, UGM-II, and UGM-III on the logistic regression problem. From left to right: suboptimality versus iteration count, suboptimality versus CPU time, and the evolution of the extrapolation parameter $\theta_k$.}
    \label{fig:lr}
\end{figure}

\subsection{Softmax regression}
\label{sec:ne-cached-directional}
The numerical results above demonstrate that UGM‑II and UGM‑III achieve faster convergence in terms of iteration count, yet this benefit does not always yield a corresponding wall‑clock improvement, owing to the extra gradient evaluation required at each iteration. In this subsection, we examine whether their iteration‑wise advantage can be converted into better computational efficiency when the inner‑product scalar $\langle \nabla f(z^{k+1}),z^{k+1}-y^{k+1}\rangle$ is computed without constructing an additional full‑gradient vector. To this end, we consider the following $\ell_2$-regularized multiclass logistic regression problem, commonly referred to as softmax regression:
\begin{equation}
\label{eq:ne-softmax}
\min_{W\in\mathbb{R}^{d\times K}}
f(W) =
\frac{1}{m} \sum_{i=1}^{m}
\left[
\log\!\sum\nolimits_{c=1}^{K}
\exp\bigl((XW)_{i,c}\bigr)
-(XW)_{i,y_i}\right]+
\frac{\gamma}{2}\|W\|_F^2.
\end{equation}
Here $(XW)_{i,c}$ denotes the score of sample $i$ for class $c$, $(XW)_{i,y_i}$ is the score corresponding to its true class label $y_i$,  $\|\cdot\|_F$ denotes the Frobenius norm, and $\gamma > 0$. 
Given a prescribed condition number $\kappa=3000$, we set $\gamma:=L_{\mathrm{data}}/(\kappa-1)$, with $L_{\mathrm{data}} :=\frac{1}{2m}\|X\|_2^2$. 
With this choice, the objective function $f$ in \eqref{eq:ne-softmax} is $L$-smooth with Lipschitz constant $L=L_{\mathrm{data}}+\gamma$ and $\gamma$-strongly convex. Its resulting condition number is therefore $L/\gamma=\kappa$.

We investigate two ways to convert UGM‑II and UGM‑III’s iteration‑count advantage into CPU‑time savings.
The first uses cached logits to evaluate $\langle \nabla f(z^{k+1}),z^{k+1}-y^{k+1}\rangle$ without computing an additional full gradient. The second employs scalar aggregation in a distributed setting, replacing communication of high‑dimensional gradient matrices with a single scalar transmitted from each worker. Both experiments use the Digits dataset from \texttt{sklearn.datasets.load\_digits}, which contains $m=1797$ grayscale images belonging to $10$ digit classes. Each $8\times 8$ image is flattened into a 64‑dimensional vector with pixel values in $[0,16]$. We standardize the original $64$ features and use all samples to construct the finite‑sum problem. The two experiments differ only in the nonlinear feature map subsequently applied to the standardized data. Furthermore, the UGM variants in these experiments adopt the same parameter settings for $\varrho$, $q_0$, $\underline{\theta}$, $a_0$, and $\bar{\theta}$ as specified in Section~\ref{sec:lr}.

\paragraph{Scalar evaluation using cached logits.} Let $P(z^{k+1})\in\mathbb{R}^{n\times K}$ denote the matrix of softmax probabilities corresponding to the logits $Xz^{k+1}$, with entries
\[
\bigl[P(z^{k+1})\bigr]_{i,c}
=
\exp\bigl((Xz^{k+1})_{i,c}\bigr) \big/ \sum\nolimits_{j=1}^{K}\exp\bigl((Xz^{k+1})_{i,j}\bigr).
\]
Then, the scalar required by the update can be expressed as
\begin{equation}
\label{eq:ne-cached-Ck}
\begin{aligned}
\left\langle\nabla f(z^{k+1}),y^{k+1}-z^{k+1}\right\rangle
=\frac{1}{n}\left\langle P(z^{k+1})-Y,X\bigl(y^{k+1}-z^{k+1}\bigr)
\right\rangle+\mu\left\langle
z^{k+1},y^{k+1}-z^{k+1}\right\rangle,
\end{aligned}
\end{equation}
where $Y\in\mathbb{R}^{n\times K}$ is the one-hot encoding matrix of the class labels. Here, for matrix variables, $\langle\cdot,\cdot\rangle$ denotes the trace inner product.  
The quantity $X(y^{k+1}-z^{k+1})$ can be computed directly from the cached logits $Xy^{k+1}$ and $Xz^{k+1}$.
Thus, this directional implementation only evaluates the Frobenius inner products in \eqref{eq:ne-cached-Ck}, without explicitly constructing the full‑gradient term $X^\top(P(z^{k+1})-Y)$. The directional and full‑gradient implementations produce identical iterates up to floating‑point round‑off error, yet they incur different per‑iteration computational costs.

Our first experiment maps the $1797$ handwritten‑digit samples to $d=768$ random Fourier features (Digits‑RFF) and considers $K=10$ classes.

\begin{figure}[!htbp]
    \centering
    \begin{subfigure}[t]{0.32\textwidth}
        \centering
        \includegraphics[width=\linewidth]{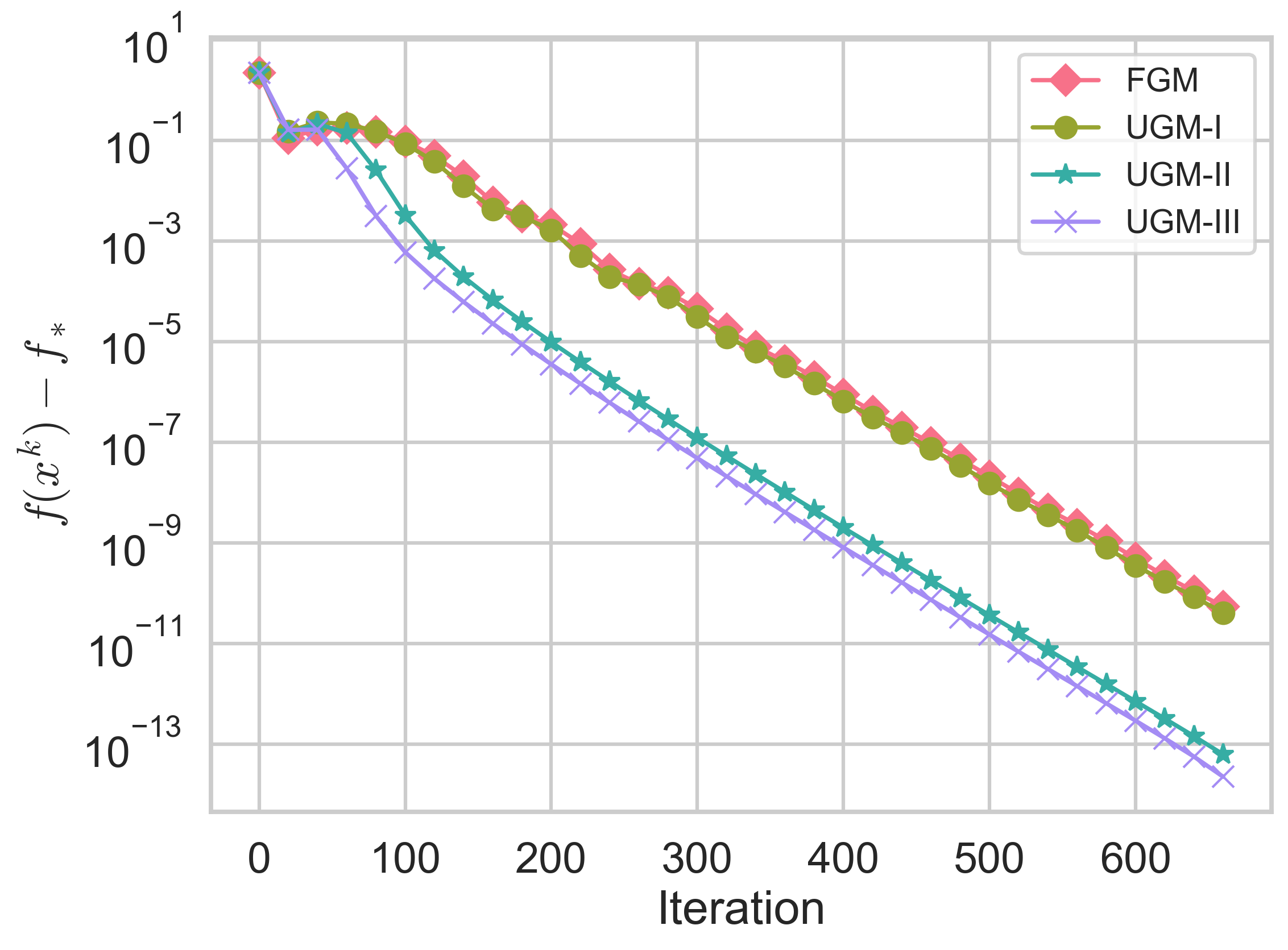}
    \end{subfigure}
    \hfill
    \begin{subfigure}[t]{0.32\textwidth}
        \centering
        \includegraphics[width=\linewidth]{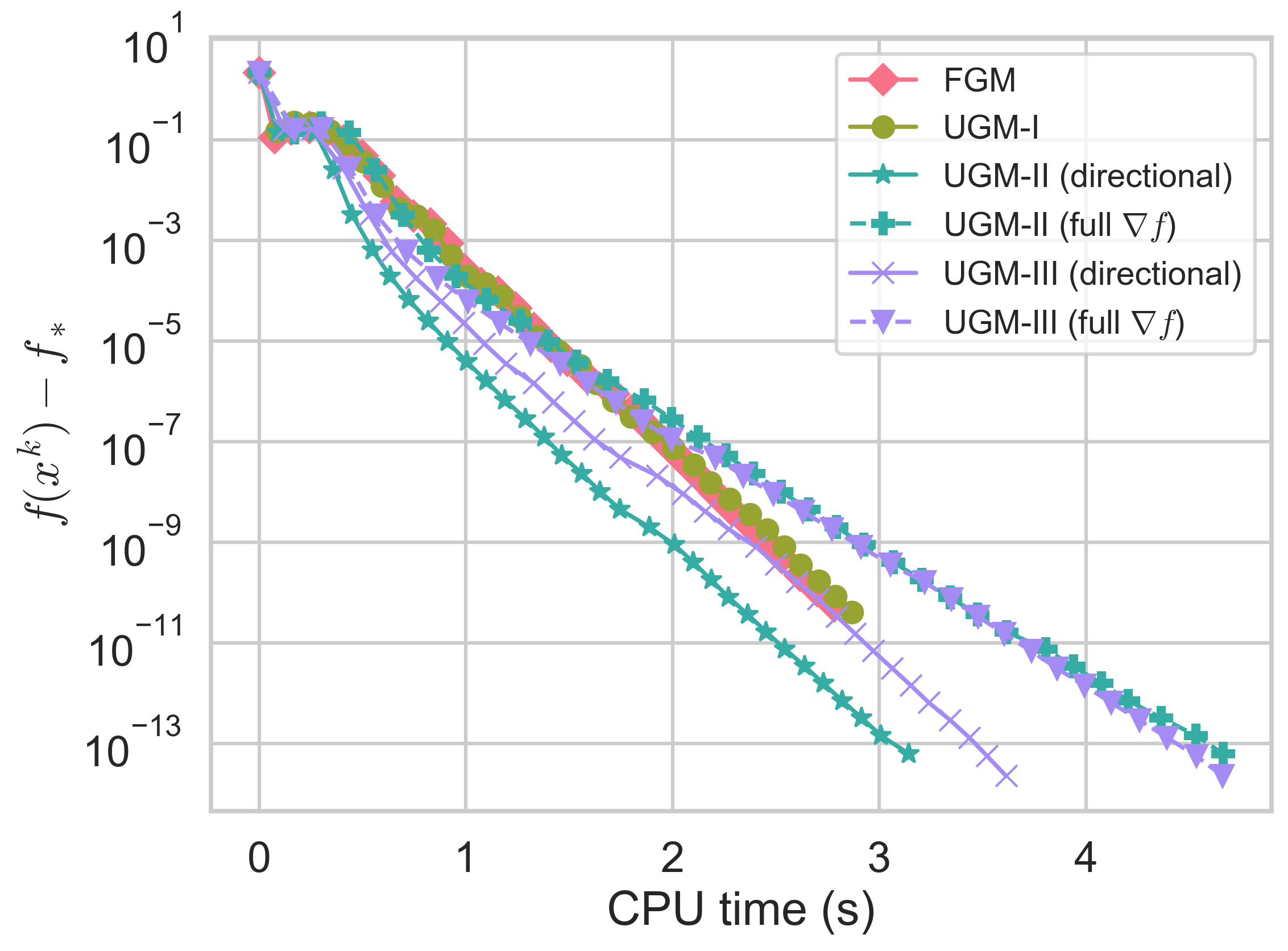}
    \end{subfigure}
    \hfill
    \begin{subfigure}[t]{0.32\textwidth}
        \centering
        \includegraphics[width=\linewidth]{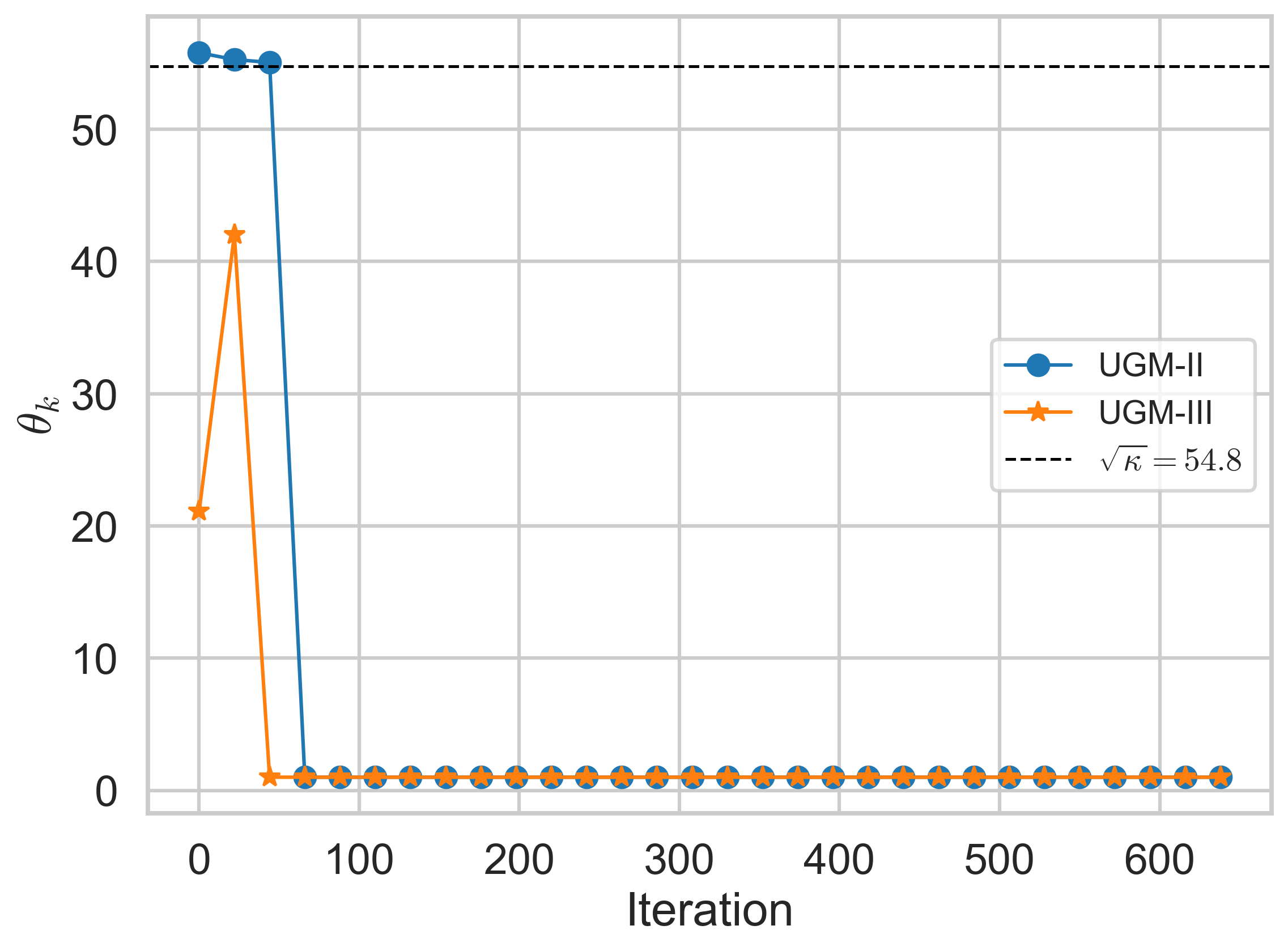}
    \end{subfigure}
    \caption{Comparison of FGM, UGM-I, UGM-II, and UGM-III on the Digits-RFF multiclass logistic regression problem. From left to right: suboptimality versus iteration count, suboptimality versus CPU time, and the evolution of the extrapolation parameter $\theta_k$.}
    \label{fig:ne-rff}
\end{figure}
The experimental results are presented in Figure~\ref{fig:ne-rff}.
In our tests, UGM‑II and UGM‑III attain the stopping criterion $f(x^k)-f_{*}\leq 10^{-6}$ in $255$ and $230$ iterations, respectively, compared with $400$ iterations for FGM and $390$ iterations for UGM‑I. With the cached‑logit directional implementation, their CPU times are $1.619$ and $1.340$ seconds. Compared with FGM, these results yield speedups of $1.52\times$ and $1.84\times$. These findings demonstrate that the cached‑logit implementation successfully converts the iteration‑count advantages of UGM‑II and UGM‑III into substantial CPU‑time savings.

\paragraph{Digits‑Poly2 with scalar aggregation.}
Our second experiment expands the same samples into $d=2144$ standardized degree‑2 polynomial features (Digits‑Poly2), again with $K=10$ classes, to further evaluate the benefits of scalar aggregation in a distributed setting. The samples are partitioned across four logical workers. For the directional computation, each worker returns a single double‑precision scalar, resulting in a total communication payload of $32$ bytes. In contrast, full‑gradient 
aggregation requires transmitting four $2144\times 10$ matrices, corresponding to $686080$ bytes. Therefore, scalar aggregation reduces the communication payload by a factor of $21440$.

\begin{figure}[!htbp]
    \centering
    \begin{subfigure}[t]{0.32\textwidth}
        \centering
        \includegraphics[width=\linewidth]{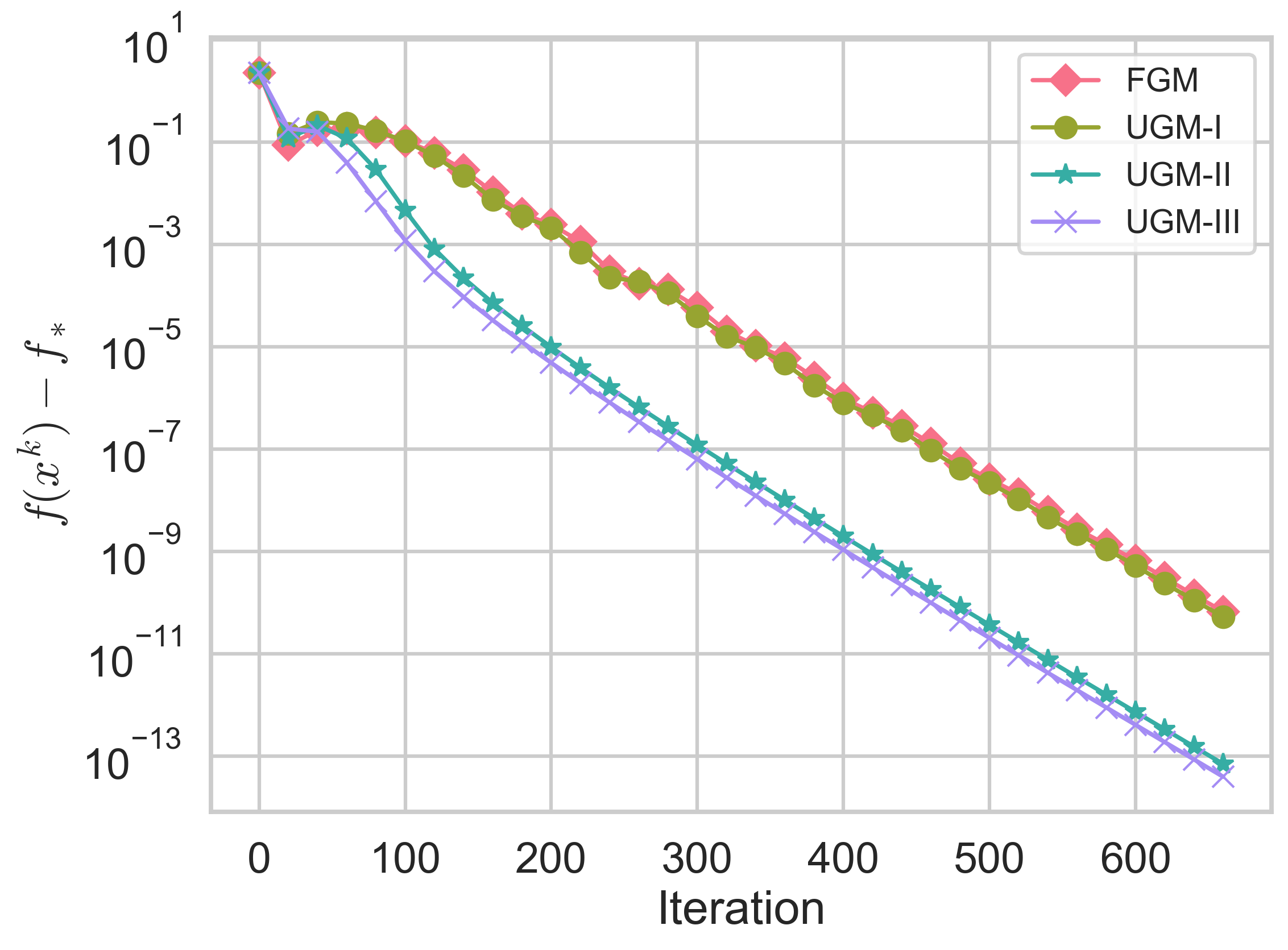}
    \end{subfigure}
    \hfill
    \begin{subfigure}[t]{0.32\textwidth}
        \centering
        \includegraphics[width=\linewidth]{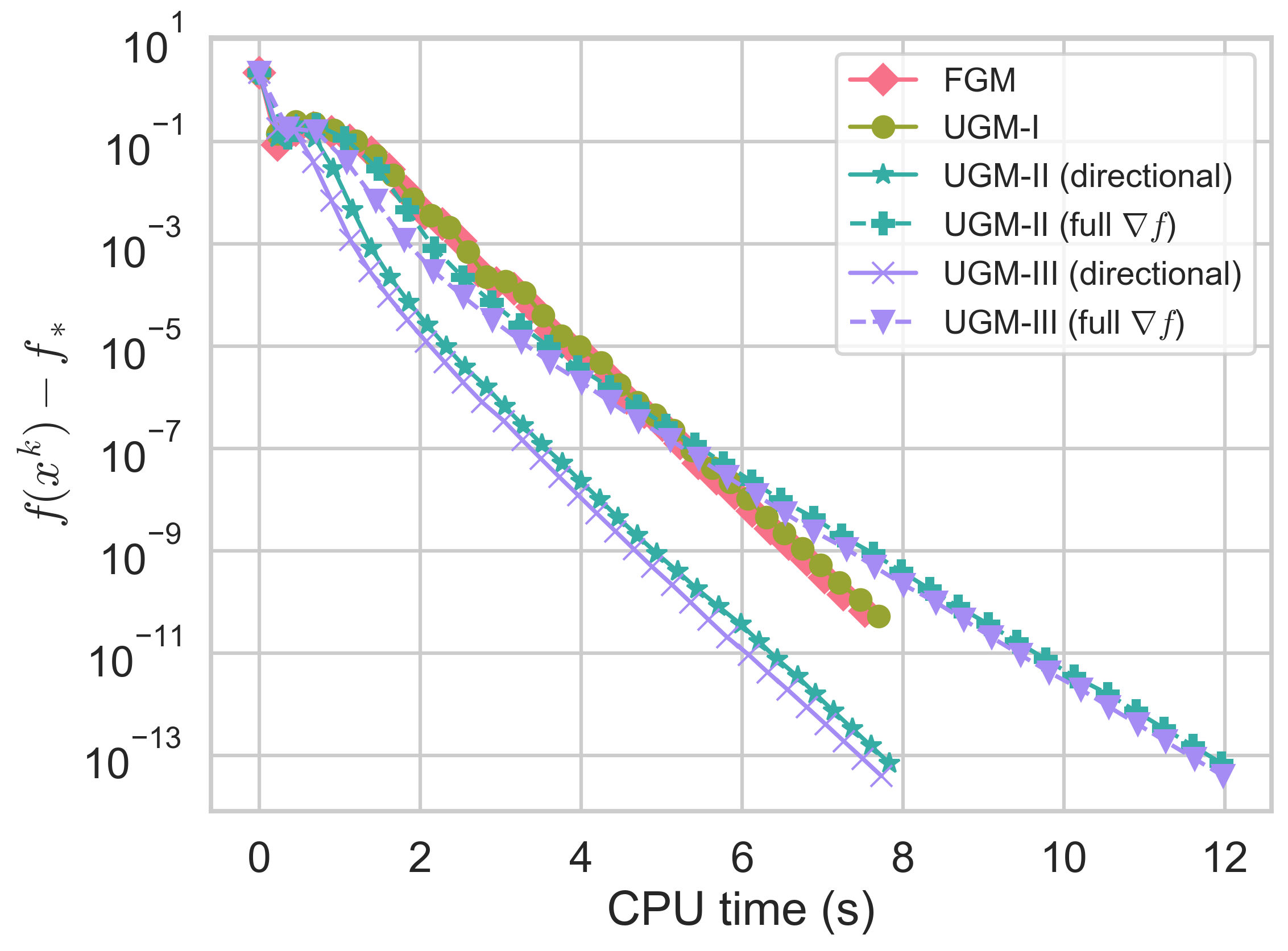}
    \end{subfigure}
    \hfill
    \begin{subfigure}[t]{0.32\textwidth}
        \centering
        \includegraphics[width=\linewidth]{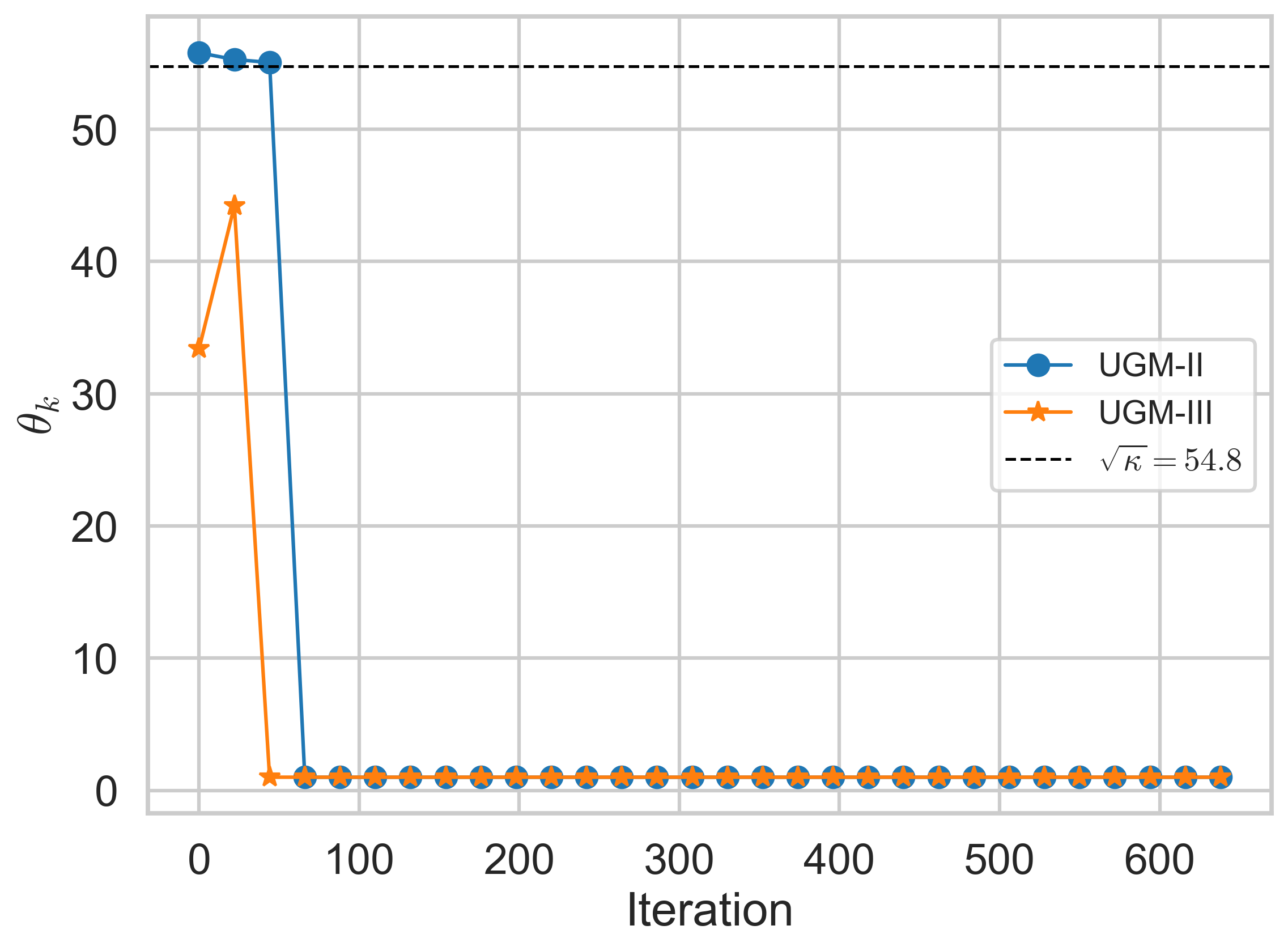}
    \end{subfigure}
    \caption{Comparison of FGM, UGM-I, UGM-II, and UGM-III on the Digits-Poly2 multiclass logistic regression problem. From left to right: suboptimality versus iteration count, suboptimality versus CPU time, and the evolution of the extrapolation parameter $\theta_k$.}
    \label{fig:ne-poly2}
\end{figure}

The experimental results are presented in Figure~\ref{fig:ne-poly2}. In our tests, UGM‑II and UGM‑III attain the stopping criterion $f(x^k)-f_{*}\leq 10^{-6}$ in $255$ and $240$ iterations, while FGM and UGM‑I require $400$ and $395$ iterations. With scalar aggregation across the distributed workers, their CPU times are reduced to $4.480$ and $3.918$ seconds, whereas FGM requires $6.128$ seconds. Relative to FGM, this yields speedups of $1.37\times$ and $1.56\times$, respectively. These findings demonstrate that scalar aggregation can convert the lower iteration counts of UGM‑II and UGM‑III into measurable computational gains in distributed settings.

Moreover, the $\theta_k$ panels (last plots in Figures~\ref{fig:ne-rff} and~\ref{fig:ne-poly2}) show a brief initial adaptive transient phase, after which the algorithm stably selects the lower bound $\underline{\theta}=1$. Therefore, the observed improvement in iteration‑count performance does not arise merely from employing a fixed momentum parameter. Instead, it stems from the branch‑aware update rule, while cached‑logit evaluation and scalar aggregation reduce the overhead of this extra selection information sufficiently to achieve an end‑to‑end runtime advantage.

\section{Concluding remarks}
In this paper, we study a unified first‑order method that encompasses a broad family of accelerated and non‑accelerated algorithms for minimizing $f \in \mathcal{F}_{\mu, L}$. We show that the proposed method can be naturally interpreted as a combination of the heavy‑ball method and gradient descent. This reveals that classical acceleration schemes effectively incorporate a conservative gradient‑descent step into the fast yet potentially divergent heavy‑ball dynamics to balance acceleration 
We then present a unified convergence analysis built upon a Lyapunov‑function framework, which further motivates the design of problem‑dependent accelerated algorithms and their variants. Applications to unconstrained quadratic optimization, logistic regression, and softmax regression demonstrate that the proposed methods can outperform existing approaches under specific structural conditions.

The following directions are worthy of further investigation.
\begin{enumerate}
    \item[(i)] One interesting direction is to estimate the sign of $\langle \nabla\! f(x^{k+1}),x^{k+1}-z^{k+1}  \rangle$ without relying on Lemma~\ref{lem:est}, which requires either additional full-gradient evaluations or the corresponding directional computations. Directly determining the sign of this inner product would reduce computational overhead and improve practicality. Such a technique is consistent with our unified Lyapunov framework and may support the design of additional accelerated optimization variants.

    \item[(ii)] Estimating the $L$-smoothness and $\mu$-strong convexity parameters without prior problem information remains challenging \cite{CLW25, XC26}. It is worth investigating whether adaptive iteration-dependent parameters $L_k$ and $\mu_k$ can be incorporated into the UGM framework \eqref{ugm} to achieve accelerated convergence, following the parameter strategy of Nesterov’s FGM \cite{N87}. The corresponding update scheme is provided below:
\begin{equation*}
\label{ugm-lk-mu-k}
\left\{ \quad 
\begin{aligned}
y^{k+1} &= x^{k} - \frac{1}{L_k} \nabla f(x^{k}), \\
w^{k+1} &= x^{k} - \frac{1}{\mu_k} \nabla f(x^{k}), \\
z^{k+1} &= \Big(1 - \sqrt{\mu_k/L_k} \Big) z^{k} + \sqrt{\mu_k/L_k} w^{k+1}, \\
x^{k+1} &= \frac{\sqrt{L_k/\mu_k}}{1 + \sqrt{L_k/\mu_k}} y^{k+1} + \frac{1}{1 + \sqrt{L_k/\mu_k}} z^{k+1}.
\end{aligned}
\right.
\end{equation*}

    \item[(iii)] We note that if $\langle \nabla f(x^{k+1}),x^{k+1}-z^{k+1}\rangle\leq 0$, any negative $\theta_k$ satisfies the condition \eqref{sat-con}, which allows negative momentum while maintaining convergence guarantees. Previous studies have shown that negative momentum may cause oscillation and instability in minimax optimization \cite{GHR19, SWA26}. Examining its behavior under the current algorithmic setting constitutes another interesting research direction.
\end{enumerate}

\bibliographystyle{alpha}
\bibliography{ref}

\newpage
\appendix

\section{Proof on the equivalence between \texorpdfstring{\eqref{hb-method}}{hb-method} and \texorpdfstring{\eqref{hb}}{hb}} 
\label{proof-of-hb}
\begin{lemma}
The iteration scheme \eqref{hb-method} is equivalent to \eqref{hb} with $z^0=x^0$, $q = \frac{\sqrt{\kappa}-1}{\sqrt{\kappa}+1}$, and $\theta = \frac{\sqrt{\kappa}-1}{2}$.
\end{lemma}
\begin{proof}
First, it follows from \eqref{hb-x} that $z^{i+1}=(1+\theta)x^{i+1}-\theta x^{i}$ for all $i\geq 0$. 
Substituting this relation with $i=k$ and $i=k+1$ into \eqref{hb-z} and using \eqref{hb-w}, we obtain
\[
\begin{aligned}
(1+\theta) x^{k+1} - \theta x^{k} = z^{k+1} 
&
\xlongequal{\eqref{hb-z},\,\eqref{hb-w}}
q \left( (1+\theta) x^{k} - \theta x^{k-1} \right)  + (1-q) \big(x^{k} -  \mu^{-1}\nabla f(x^{k})  \big)\\
& = x^{k}   +  q\theta ( x^{k} -  x^{k-1}) - (1-q)\mu^{-1} \nabla f(x^{k}).
\end{aligned}
\]
By reformulating the above equality, we obtain $x^{k+1} = x^{k} - \frac{1-q}{\mu(1+\theta)} \nabla f(x^{k}) + \frac{q\theta}{1+\theta} (x^{k}-x^{k-1})$. 
By setting $q = \frac{\sqrt{\kappa}-1}{\sqrt{\kappa}+1}$ and $\theta = \frac{\sqrt{\kappa}-1}{2}$, one can readily verify that
\[
1-q = \frac{1}{1+\theta} = \frac{2\sqrt{\mu}}{\sqrt{L}+\sqrt{\mu}}, \quad \frac{1-q}{\mu(1+\theta)}=\frac{4}{(\sqrt{L}+\sqrt{\mu})^2} \text{~~and~~}
\frac{q\theta}{1+\theta} = \frac{(\sqrt{L}-\sqrt{\mu})^2}{(\sqrt{L}+\sqrt{\mu})^2}.
\]
Therefore, the two schemes \eqref{hb-method} and \eqref{hb} are equivalent.
\end{proof}

\section{Reformulation of several algorithms} \label{sec:refor}
In this section, we present equivalent formulations of several well‑known algorithms within the UGM framework \eqref{ugm}.
Note that all cumbersome identities used in the proofs of Lemmas \ref{sec-b2}, \ref{sec-b3}, \ref{sec-b4} and \ref{sec-b5}  have been validated by hand-written derivation and symbolic computation, where \texttt{check\_proof.ipynb} is used exclusively for the symbolic-computation verification.

\begin{lemma} \label{lemma:ugm-2}
The UGM framework \eqref{ugm} with initialization $z^0=x^0$ is equivalent to
    \begin{subequations}  \label{ugm-2}
    \begin{align}
    y^{k+1} &= x^{k}-\frac{1}{L}\nabla f(x^{k}), \label{ugm-2a}\\
    x^{k+1} & = \frac{\theta + (1-q)\kappa}{1+\theta} y^{k+1} + \frac{1 + q\theta - (1-q)\kappa}{1+\theta} x^{k} - \frac{q\theta}{1+\theta} y^{k}, \label{ugm-2b}
    \end{align}
    \end{subequations}
with initialization $y^0=x^0$.
\end{lemma}

\begin{proof}
First, \eqref{ugm-y} is the same as  \eqref{ugm-2a}, from which
we obtain $\nabla f(x^{k}) = L(x^{k} - y^{k+1})$. Substituting this relation into \eqref{ugm-w} yields 
\begin{equation}\label{jy-0808a}
w^{k+1} = x^{k} - \kappa (x^{k} - y^{k+1}) = \left(1 - \kappa\right) x^{k} + \kappa y^{k+1}.
\end{equation}
Additionally, from \eqref{ugm-x}, we have $z^{i} = (1+\theta) x^{i} - \theta y^{i}$ for all $i\geq 1$. 
Substituting this relation with $i=k$ and $i=k+1$ into \eqref{ugm-z} and using \eqref{jy-0808a}, we obtain
\begin{align*}
(1+\theta) x^{k+1} - \theta y^{k+1} &= z^{k+1} \xlongequal{\eqref{ugm-z}, \eqref{jy-0808a}} 
q\big((1+\theta) x^{k} - \theta y^{k}\big) + (1-q)\big(  (1 - \kappa) x^{k} + \kappa y^{k+1} \big)\\
&= \big( q(1+\theta) + (1-q)(1 - \kappa) \big) x^{k} - q\theta y^{k} + (1-q)\kappa y^{k+1}\\
&= \big(1 + q\theta - (1-q)\kappa\big) x^{k} - q\theta y^{k} + (1-q)\kappa y^{k+1}.
\end{align*}
By rearranging the above equality to explicitly solve for $x^{k+1}$, we obtain \eqref{ugm-2b}, which establishes the desired equivalence.
\end{proof}

\subsection{Gradient descent methods}
\begin{lemma} \label{gd-1l}
With the initial condition $z^{0}=x^{0}$ and parameters $q=1-1/\kappa$ and arbitrary $\theta\neq -1$, 
the UGM framework \eqref{ugm} is equivalent to GD with step size $1/L$, i.e., \eqref{gd}.
\end{lemma}

\begin{proof}
Let $\{(y^k, w^k, z^k, x^k): k\geq 1\}$ be the sequence generated by UGM \eqref{ugm} with initialization $z^0=x^0$ and parameters $q=1-1/\kappa$ and arbitrary $\theta\neq -1$.
We prove by induction that $z^{k} = x^{k}$ for all $k\geq 1$.  
Assume that $z^{k}=x^{k}$ for some $k\geq 0$. 
Then, from \eqref{ugm-z}, \eqref{ugm-w} and \eqref{ugm-y}, we obtain
\begin{equation*}
z^{k+1} \xlongequal{\eqref{ugm-z}, \eqref{ugm-w}} q x^{k} + (1-q)\big(x^{k} - \mu^{-1}\nabla f(x^{k})\big) 
= x^{k} - L^{-1}\nabla f(x^{k}) \stackrel{\eqref{ugm-y}}{=} y^{k+1},
\end{equation*}
where the second ``$=$" follows from  $q = 1-\frac{1}{\kappa} = 1 - \frac{\mu}{L}$. 
By plugging $z^{k+1}=y^{k+1}$ into \eqref{ugm-x}, we obtain $x^{k+1} = y^{k+1}$.
Hence,  $z^{k+1}=y^{k+1}=x^{k+1} = x^{k} - L^{-1}\nabla f(x^{k})$.
The proof is completed by induction.
\end{proof}

\begin{lemma} \label{gd-2l}
With the initial condition $z^{0}=x^{0}$ and parameters $q=\frac{\kappa-1}{\kappa+1}$ and $\theta = 0$, 
the UGM framework \eqref{ugm} is equivalent to GD \eqref{gd-alpha} with step size $\alpha = 2/(L+\mu)$.
\end{lemma}

\begin{proof}
Since $\theta=0$ and $z^{0}=x^{0}$, \eqref{ugm-x} yields $x^{k} = z^{k}$ for all $k\geq 0$. 
Hence, for every $k\geq 0$, we obtain
\[
x^{k+1}=z^{k+1}
 \xlongequal{\eqref{ugm-z}, \eqref{ugm-w}}  q x^{k}+(1-q)\big(x^{k}- \mu^{-1}\nabla f(x^{k})\big)
      = x^{k}-\frac{2}{L+\mu}\nabla f(x^{k}),
\]
where the final equality follows from $\frac{1-q}{\mu} = \frac{2}{L+\mu}$, which holds with  $q = \frac{\kappa-1}{\kappa+1}$.
This completes the proof. 
\end{proof}

\subsection{Nesterov's FGM}
\begin{lemma} \label{sec-b2}
With the initial condition $z^{0}=x^{0}$ and parameters $q=1-\frac{1}{\sqrt{\kappa}}$ and  $\theta=\sqrt{\kappa}$, 
the UGM framework \eqref{ugm} is equivalent to \eqref{fgm}.
\end{lemma}
\begin{proof} 
To prove the equivalence between  \eqref{ugm} and \eqref{fgm}, we only need to verify that the update $x^{k+1}$ in \eqref{ugm-x} reduces to that given in \eqref{fgm}. 
Utilizing the given values of $q$ and $\theta$, the following identities hold:
\begin{equation}\label{jy-0808b}
\frac{\theta + (1-q)\kappa}{1+\theta} = \frac{2\sqrt{\kappa}}{\sqrt{\kappa}+1},\quad
\frac{1 + q\theta - (1-q)\kappa}{1+\theta} = 0 \text{~~and~~}
\frac{q\theta}{1+\theta} = \frac{\sqrt{\kappa}-1}{\sqrt{\kappa}+1}.
\end{equation}
We have validated these equalities through manual derivation, while the symbolic‑computation check is also implemented in the script \texttt{check\_proof.ipynb}.
It then follows from \eqref{ugm-2b} and \eqref{jy-0808b} that
\[
x^{k+1} = \frac{2\sqrt{\kappa}}{\sqrt{\kappa}+1} y^{k+1}
- \frac{\sqrt{\kappa}-1}{\sqrt{\kappa}+1} y^{k} = y^{k+1}+\frac{\sqrt{\kappa}-1}{\sqrt{\kappa}+1}(y^{k+1}-y^{k}).
\]
This exactly matches the update of $x^{k+1}$ in \eqref{fgm}, which completes the proof.
\end{proof}

\subsection{TM method}

\begin{lemma} \label{sec-b3}
With the initial condition $z^{0}=x^{0}$ and parameters $q=\rho=1-\frac{1}{\sqrt{\kappa}}$ and $\theta=\frac{\rho}{2(1-\rho)}$, 
the sequence $\{(z^k,x^k): k\geq 1\}$ generated by the UGM framework \eqref{ugm} 
coincides with the sequence generated by the TM method \eqref{tm-method}.
\end{lemma}
\begin{proof}
To begin with, we note that all cumbersome identities used in this proof have been checked by manual derivation, while symbolic‑computation verification is carried out in \texttt{check\_proof.ipynb}.

First, we establish the equivalence between \eqref{tm-x} and \eqref{ugm-x}. 
To prove the equivalence,  it suffices to show that $x^{k+1}$ from both \eqref{tm-x} and \eqref{ugm-x} satisfies
\begin{equation} \label{eq-tm}
x^{k+1}-x^{k}=\frac{\rho^2}{2-\rho}(y^{k+1}-y^{k})+(1+\rho)(y^{k+1}-x^{k}),
\end{equation}
where $y^{k+1} = x^k - \frac{1}{L} \nabla f(x^k)$.
\begin{enumerate}
    \item[(i)] Consider $\{x^k: k\geq 1\}$ generated by the TM method \eqref{tm-method}. 
    From \eqref{tm-x}, the left-hand side of \eqref{eq-tm} is equal to
\begin{align}
    x^{k+1}-x^{k} \stackrel{\eqref{tm-x}}{=} (1+\gamma)\xi^{k+1}-\gamma \xi^{k} - \big( (1+\gamma)\xi^{k} - \gamma \xi^{k-1}\big) = (1+\gamma)\xi^{k+1}-(2\gamma+1)\xi^{k}+\gamma \xi^{k-1}. \label{x-x}
\end{align}
On the other hand, from the parameter settings of \eqref{tm-method}, we obtain $\lambda/\gamma=1+\rho$
and \eqref{tm-x} yields $\xi^{k}-\xi^{k-1}=(x^{k}-\xi^{k})/\gamma$.
We then derive 
\begin{align*}
\xi^{k+1} & \stackrel{\eqref{tm-xi}}{=} \xi^{k}+\lambda(\xi^{k}-\xi^{k-1})-\nu \nabla f(x^k) \stackrel{\eqref{tm-x}}{=} \xi^{k} + (1+\rho)(x^{k}-\xi^{k}) -\nu \nabla f(x^k)\\
&= -\rho \xi^{k} + (1+\rho)(x^{k}-\tfrac{1}{L}\nabla f(x^{k})) \stackrel{\eqref{ugm-y}}{=} -\rho \xi^{k} + (1+\rho) y^{k+1},
\end{align*}
where the third ``$=$'' follows from $\nu={(1+\rho)}/{L}$.
Thus, we obtain $y^{k+1}=(\xi^{k+1}+  \rho \xi^{k}) / (1+\rho)$. 
Using this relation together with $\gamma=\frac{\rho^2}{(1+\rho)(2-\rho)}$, we compute the two terms on the right-hand side of \eqref{eq-tm}
as follows:
\begin{subequations}\label{jy-0809a}
\begin{align}
y^{k+1}-y^{k} &=  \frac{1}{1+\rho}\big(\xi^{k+1}+  \rho \xi^{k}\big) - \frac{1}{1+\rho} \big(\xi^{k} + \rho \xi^{k-1} \big) =   \frac{1}{1+\rho} \big( \xi^{k+1} + (\rho-1)\xi^{k}- \rho \xi^{k-1}\big), \label{jy-y-y}\\
y^{k+1}-x^{k} &\stackrel{\eqref{tm-x}}{=} \frac{1}{1+\rho}\big(\xi^{k+1}+  \rho \xi^{k}\big) -(1\!+\!\gamma)\xi^{k}+\gamma \xi^{k-1}
=\frac{1}{1+\rho}\Big[ \xi^{k+1}\!-\!\frac{\rho^2\!-\!\rho\!+\!2}{2\!-\!\rho}\xi^{k}+\frac{\rho^2}{2\!-\!\rho}\xi^{k-1}\Big]. \label{y-x}
\end{align}
\end{subequations}
Utilizing $\gamma=\frac{\rho^2}{(1+\rho)(2-\rho)}$ again, the following identities can be verified:
\begin{equation} \label{jy-0809b}
\frac{\rho^2(\rho-1)}{(1+\rho)(2-\rho)}-\frac{\rho^2-\rho+2}{2-\rho}=-(2\gamma+1) \text{~~and~~} \frac{-\rho^3}{(1+\rho)(2-\rho)}+\frac{\rho^2}{2-\rho}=\gamma.
\end{equation}
Then, substituting \eqref{jy-0809a} into the right-hand side of \eqref{eq-tm} and using the identities in \eqref{jy-0809b}, we derive
\begin{align}\label{jy-0809c}
\frac{\rho^2}{2-\rho}(y^{k+1}-y^{k})+(1+\rho)(y^{k+1}-x^{k})= (1+\gamma)\xi^{k+1}-(2\gamma+1)\xi^{k}+\gamma \xi^{k-1}.
\end{align}
By comparing \eqref{x-x} and \eqref{jy-0809c}, we confirm that \eqref{eq-tm} is satisfied by \eqref{tm-method}.

\item[(ii)] Consider $\{x^k: k\geq 1\}$ generated by the UGM framework \eqref{ugm}. Utilizing the given parameters $q = \rho=1-\frac{1}{\sqrt{\kappa}}$ and $\theta = \frac{\rho}{2(1-\rho)}$, it can be verified that
\[
\frac{\theta + (1-q)\kappa}{1+\theta} = \frac{2+\rho}{2-\rho},\quad
\frac{1 + q\theta - (1-q)\kappa}{1+\theta} = -\rho \text{~~and~~}
\frac{q\theta}{1+\theta} = \frac{\rho^2}{2-\rho}.
\]
It then follows from the above identities and \eqref{ugm-2b} that 
$x^{k+1} = \frac{2+\rho}{2-\rho} y^{k+1} -\rho x^{k} - \frac{\rho^2}{2-\rho} y^{k}$, from which it is trivial to derive 
\eqref{eq-tm}.
\end{enumerate}
This completes the proof of the equivalence between \eqref{tm-x} and \eqref{ugm-x}. 

We next establish the equivalence between \eqref{tm-z} and \eqref{ugm-z} by verifying that $z^{k+1}$  from both \eqref{tm-z} and \eqref{ugm-z} satisfies
\begin{equation} \label{eq-tm-z}
z^{k+1}-x^{k+1}=\frac{\rho}{2(1-\rho)}(x^{k+1}-y^{k+1}).
\end{equation}
\begin{enumerate}
\item[(i)] Consider $\{z^k: k\geq 1\}$ generated by the TM method \eqref{tm-method}. 
Subtracting each‑hand‑side of \eqref{y-x} from the corresponding side of \eqref{x-x} and utilizing $\gamma=\frac{\rho^2}{(1+\rho)(2-\rho)}$ alongside  elementary algebraic manipulations, we arrive at
$x^{k+1}-y^{k+1}=\frac{2\rho}{(1+\rho)(2-\rho)}(\xi^{k+1}-\xi^{k})$.
Using \eqref{tm-x} and \eqref{tm-z}, together with this relation, we further derive
\begin{equation*}
z^{k+1}-x^{k+1} \xlongequal{\eqref{tm-x},\eqref{tm-z}} (\delta-\gamma)(\xi^{k+1}-\xi^{k}) = (\delta-\gamma)\frac{(1+\rho)(2-\rho)}{2\rho}(x^{k+1}-y^{k+1})
= \frac{\rho (x^{k+1}-y^{k+1})}{2(1-\rho)},
\end{equation*}
where the last equality follows from the parameter settings of $\delta$ and $\gamma$ in \eqref{tm-method} and elementary computations. This confirms \eqref{eq-tm-z}. 

\item[(ii)] Consider $\{z^k: k\geq 1\}$ generated by the UGM framework \eqref{ugm}. It follows straightforwardly from \eqref{ugm-x} and $\theta = \frac{\rho}{2(1-\rho)}$ that
$z^{k+1} - x^{k+1} \stackrel{\eqref{ugm-x}}{=} \theta(x^{k+1}-y^{k+1})=\frac{\rho}{2(1-\rho)}(x^{k+1}-y^{k+1})$, i.e., \eqref{eq-tm-z} is confirmed.
\end{enumerate}
This completes the proof of the equivalence between \eqref{tm-z} and \eqref{ugm-z}.
\end{proof}

\subsection{SC-OGM}
\begin{lemma} \label{sec-b4}
With the initial condition $z^{0}=x^{0}$ and parameters $q=\frac{1}{\gamma+1}$ and $\theta=\frac{\gamma+1}{\gamma}$, 
the UGM framework \eqref{ugm} is equivalent to the SC-OGM scheme \eqref{sc-ogm}.
\end{lemma}
\begin{proof}
To begin with, we note that all cumbersome identities used in this proof have been checked by manual derivation, while symbolic‑computation verification is carried out in \texttt{check\_proof.ipynb}.

To establish this equivalence, it suffices to verify that the update for $x^{k+1}$ given in \eqref{ugm-x} satisfies
\begin{equation}\label{jy-0809d}
x^{k+1} = y^{k+1} +\frac{1}{2\gamma+1}(y^{k+1}-y^{k})+\frac{1}{2\gamma+1}(y^{k+1}-x^k),
\end{equation}
which exactly matches the iteration rule for $x^{k+1}$ in \eqref{sc-ogm}.
With the given parameters $q=\frac{1}{\gamma+1}$, $\theta=\frac{\gamma+1}{\gamma}$, $\gamma=\frac{\sqrt{8\kappa+1}+3}{2\kappa-2}$, and
noting that $\kappa = 1 + 3/\gamma + 2/\gamma^2$, we can verify the following identities:
\[
\frac{\theta + (1 - q)\kappa}{1 + \theta} = \frac{2\gamma + 3}{2\gamma + 1}, \quad \frac{1 + q\theta - (1 - q)\kappa}{1 + \theta} = -\frac{1}{2\gamma + 1} 
\text{~~and~~} \frac{q\theta}{1 + \theta} = \frac{1}{2\gamma + 1}.
\]
It then follows from \eqref{ugm-2b} that 
$x^{k+1} = \frac{2\gamma + 3}{2\gamma + 1} y^{k+1} - \frac{1}{2\gamma + 1} x^{k} - \frac{1}{2\gamma + 1} y^k$,
which implies \eqref{jy-0809d} immediately. 
\end{proof}

\subsection{AOR-HB}
\begin{lemma} \label{sec-b5}
With the initial condition 
$z^{0}\!=\!x^{0}$, and parameters 
$q \!=\! \frac{2\kappa + 2\sqrt{\kappa} \!-\! 1 + \sqrt{4\sqrt{\kappa} + 5}}{2(\sqrt{\kappa} + 1)^2}$ and $\theta = \frac{\kappa \left( 2\sqrt{\kappa} + 3 \!-\! \sqrt{4\sqrt{\kappa} + 5} \right)}{2 \left( \kappa + \sqrt{4\sqrt{\kappa} + 5} - 2 \right)}$, 
the UGM framework \eqref{ugm} is equivalent to the AOR-HB scheme \eqref{aor-hb}. 
Moreover, we have $q \in [\frac{3}{4},1)$ and $\theta>0$. 
\end{lemma}
\begin{proof}
To begin with, we note that all cumbersome identities used in this proof have been checked by manual derivation, while symbolic‑computation verification is carried out in \texttt{check\_proof.ipynb}.
To prove the equivalence, we only need to verify that the update for $x^{k+1}$ from both \eqref{ugm-x} and \eqref{aor-hb} satisfies 
\begin{equation} \label{aor-hb-ugm}
x^{k+1} = \frac{2\kappa}{(\sqrt{\kappa}+1)^2} y^{k+1}
+ \frac{2\sqrt{\kappa}+1}{(\sqrt{\kappa}+1)^2} x^{k}
- \frac{\kappa}{(\sqrt{\kappa}+1)^2} y^{k},
\end{equation}
where $y^{k+1} = x^k - \frac{1}{L}\nabla f(x^{k})$. 
Using the given values for $q$ and $\theta$, it can be verified that
\[
\frac{\theta + (1-q)\kappa}{1+\theta} =  \frac{2\kappa}{(\sqrt{\kappa}+1)^2},\quad
\frac{1 + q\theta - (1-q)\kappa}{1+\theta} = \frac{2\sqrt{\kappa}+1}{(\sqrt{\kappa}+1)^2} \text{~~and~~}
\frac{q\theta}{1+\theta} = \frac{\kappa}{(\sqrt{\kappa}+1)^2}.
\]
Hence, for the UGM framework \eqref{ugm}, \eqref{aor-hb-ugm} is implied by \eqref{ugm-2b}, together with the above identities. 
On the other hand, for the AOR-HB scheme \eqref{aor-hb}, noting that $\nabla f(x^{k}) = L(x^{k} - y^{k+1})$, we obtain
\begin{equation*}
\begin{aligned}
x^{k+1} & \stackrel{\eqref{aor-hb}}{=}x^{k}-\frac{1}{(\sqrt{L}+\sqrt{\mu})^2}\big(2\nabla f(x^{k})-\nabla f(x^{k-1})\big)+\frac{L}{(\sqrt{L}+\sqrt{\mu})^2}(x^{k}-x^{k-1})\\
& = x^{k}-\frac{L}{(\sqrt{L}+\sqrt{\mu})^2}\big(2(x^{k}-y^{k+1})-(x^{k-1}-y^k)\big)+\frac{L}{(\sqrt{L}+\sqrt{\mu})^2}(x^{k}-x^{k-1})\\
& = \frac{2\kappa}{(\sqrt{\kappa}+1)^2} y^{k+1}
+ \frac{2\sqrt{\kappa}+1}{(\sqrt{\kappa}+1)^2} x^{k}
- \frac{\kappa}{(\sqrt{\kappa}+1)^2} y^{k}.
\end{aligned}
\end{equation*}
Here, elementary computations are involved in deriving the last equality. 
This confirms \eqref{aor-hb-ugm} and completes the proof of the equivalence.

To complete the proof for the remaining claims of this lemma, we set $t = \sqrt{\kappa} \geq 1$, which yields $q = q(t):=\frac{2t^2 + 2t - 1 + \sqrt{4t + 5}}{2(t+1)^2}$. 
Direct computation yields $q'(t) = \frac{(t + 2)\sqrt{4t + 5} - (3t + 4)}{\sqrt{4t + 5} (t+1)^3}$.
Given $t\geq 1$, we obtain $(t+2)\sqrt{4t+5} \geq 3(t+2) > 3t+4$. Hence, $q'(t)>0$ holds for every $t\geq 1$. 
Accordingly, $q(t)$ is strictly-increasing over $[1,+\infty)$, and $\frac{3}{4}=q(1)\leq q(t) < \lim_{t\rightarrow +\infty} q(t)=1$. 
Finally, $\theta = \frac{t^2 \left( 2t + 3 - \sqrt{4t + 5} \right)}{2 \left( t^2 + \sqrt{4t + 5} - 2 \right)} > 0$ holds trivially since $t\geq 1$.
\end{proof}

\section{Supplementary proofs for two inequalities}\label{appendix-two-ineqs}
In the analysis for the TM-method in Section \ref{sec:uni-ana}, the following inequality \eqref{ht} is used, and we now provide its proof. 
\begin{lemma}\label{appendix-ineq-a}
Let $P(t)=\frac{1}{2} + \frac{1}{1 + t^2} - \frac{2}{\sqrt{1 + t^2}(\sqrt{1 + t} + \sqrt{1 + t^2})}$ as introduced right below \eqref{Aarho3}.
Then, for any $t \in (0,1)$, we have
\begin{equation} \label{ht}
\frac{1}{2}\leq P(t) \leq \frac{1+t}{2(1+t^2)}.
\end{equation}
\end{lemma}
\begin{proof}
First, we show that for any $t\in(0,1)$, there holds
\begin{equation} \label{ineq:1}
1<\frac{\sqrt{1+t}}{\sqrt{1+t^2}}<\frac{2+t-t^2}{2-t+t^2}.
\end{equation}
The left-hand-side inequality in \eqref{ineq:1} holds trivially since $t>t^2$ for $t\in(0,1)$.
To verify the right-hand-side inequality, it is sufficient to show that 
$(2+m)\sqrt{1+t^2}>(2-m)\sqrt{1+t}$ for $t\in(0,1)$, where $m := t(1-t)$.
By noting that $m = t(1-t) \in (0, 1/4]$ for $t\in (0,1)$, one can readily show that
$(2+m)\sqrt{1+t^2}> (2-m)\sqrt{1+t}$ is equivalent to $4 - m^2 +4(t^2+t)>0$, and the latter inequality obviously holds since $m^2 \leq 1/16$.
This confirms that \eqref{ineq:1} holds for any $t\in (0,1)$.

Now, we prove \eqref{ht}. On the one hand, it follows from the right-hand-side inequality in \eqref{ineq:1} that
\[
P(t)  \leq  \frac{1}{2} + \frac{1}{1 + t^2} - \frac{2}{\sqrt{1 + t^2}(\sqrt{1 + t^2}\frac{2+t-t^2}{2-t+t^2} + \sqrt{1 + t^2})}  =\frac{1+t}{2(1+t^2)}.
\]
On the other hand, the left-hand-side inequality in \eqref{ineq:1} implies
\[
P(t) \geq  \frac{1}{2} + \frac{1}{1 + t^2} - \frac{2}{\sqrt{1 + t^2}(\sqrt{1 + t^2} + \sqrt{1 + t^2})} =\frac{1}{2}.
\]
The two inequalities right above confirm \eqref{ht}.
\end{proof}











The following inequality is employed in the proof of Lemma \ref{lem:ak_qk}, and we now present its proof.

\begin{lemma} \label{sec3-detail-1}
If $\varrho$ satisfies $\frac{\sqrt{1+8\kappa}-3}{\sqrt{1+8\kappa}+1} \leq \varrho <1$, then 
$\frac{(1 - \varrho)[\kappa(1 - \varrho) - \varrho - 1]}{(\varrho + 1)\varrho} \leq 1$.
\end{lemma}
\begin{proof}
Since $\frac{\sqrt{1+8\kappa}-3}{\sqrt{1+8\kappa}+1} \leq \varrho <1$, we have $\varrho + 3 \geq \sqrt{1+8\kappa}(1-\varrho) > 0$. 
By squaring both sides, it is straightforward to show that the latter inequality is equivalent to $\varrho+1\geq \kappa (1-\varrho)^2$.
Subtracting $(1-\varrho^2)$ from both-hand sides yields 
$\varrho(\varrho + 1) \ge (1 - \varrho) \bigl[ \kappa(1 - \varrho) - \varrho - 1 \bigr] $, which establishes the desired inequality.
\end{proof}

\end{document}